\documentclass[a4paper, 12pt]{article}
\usepackage{amsmath,amsthm,amssymb}
\usepackage[french]{babel}
\usepackage[utf8]{inputenc}
\usepackage[T1]{fontenc}
\usepackage{hyperref}
\usepackage{imakeidx}
\input xy
\xyoption{all}

\usepackage{color}
\usepackage{makeidx}
\usepackage{hyperref}
\usepackage{enumitem}
\makeindex[title=Index des notations]

\usepackage{a4wide}
\swapnumbers
\numberwithin{equation}{section}

\newcommand{\Titre}{Titre}
{\theoremstyle{definition}
\newtheorem{definition}{D\'efinition}[section]
\newtheorem{defimc}[definition]{\Titre}

\newtheorem{notation}[definition]{Notation}

\newtheorem{remark}[definition]{Remarque}
\newtheorem{remarks}[definition]{Remarques}

\newtheorem{examples}[definition]{Exemples}}

\newtheorem{theomc}[definition]{\Titre}

\newtheorem{rappels}[definition]{Rappels}
\newtheorem{notations}[definition]{Notations}
\newtheorem{proposition}[definition]{Proposition}
\newtheorem{proposition-definition}[definition]{Proposition-D\'efinition}
\newtheorem{notation-definition}[definition]{Notation-d\'efinition}
\newtheorem{lemme}[definition]{Lemme}
\newtheorem{lemme-notations}[definition]{Lemme-Notations}
\newtheorem{lemme-notation}[definition]{Lemme-Notation}

\newtheorem{theorem}[definition]{Th\'eor\`eme}
\newtheorem{corollary}[definition]{Corollaire}

\newcommand{\R}{\mathbb{R}}

\newcommand{\cG}{\mathcal{G}}
\newcommand{\cE}{\mathcal{E}}

\newcommand{\cL}{\mathcal{L}}

\newcommand{\cK}{\mathcal{K}}
\newcommand{\C}{\mathbb{C}}

\newcommand{\cC}{\mathcal{C}}

\newcommand{\id}{{\hbox{id}}}
\newcommand{\ie}{{\it i.e.}\/ }
\newcommand{\lc}{l.c.\/ }
\newcommand{\mq}{m.q.\/ }

\newcommand{\cf}{cf.\/ }

\newcommand{\reltens}[4]{{#1}{\,_{#2}\otimes_{#3}\,}{#4}}
\newcommand{\fprod}[4]{{#1}{\,_{#2}\ast_{#3}\,}{#4}}
\newcommand{\restr}[2]{{#1}\!\!\restriction_{#2}}

\renewcommand\theenumii{\roman{enumii}}

\renewcommand\labelenumii{({\theenumii})}

\begin{document}

\everymath{\displaystyle}
%\printindex

\newcommand{\Sp}{{\rm Sp}}      
\newcommand{\M}{{\cal M}}
 \newcommand{\Mor}{{\rm Mor}}
\renewcommand{\Re}{{\rm Re}}
\newcommand{\End}{{\rm End}}

\newcommand{\bic}{''}
\newcommand{\dom}{{\rm dom}}  
\newcommand{\im}{{\rm im}}

\begin{center}
{\large{\bf \'EQUIVALENCE MONO\"IDALE DE GROUPES QUANTIQUES ET $K$-TH\'EORIE BIVARIANTE}}
\end{center}

\begin{center}
par
\end{center}

\begin{center}  
%\footnote{AMS subject classification: Primary 47G30, 57R30. Secondary 46L87.} 

{\sc Saad Baaj et Jonathan Crespo  }

\end{center}

\bigskip

{\footnotesize
Laboratoire de Math\'ematiques, UMR 6620 - CNRS
\vskip-4pt Universit\'e Blaise Pascal, Campus des C\'ezeaux
\vskip-4pt 3, place Vasarely
\vskip-4pt BP 80026
\vskip-4pt 63171 Aubi\`ere cedex, France
\medbreak
\vskip-4pt {\it Adresses \'electroniques}
\vskip-4pt \texttt{saad.baaj@math.univ-bpclermont.fr}
\vskip-4pt \texttt{jonathan.crespo@math.univ-bpclermont.fr}

\bigbreak

\bigskip
\bigskip

\centerline{\bf R\'esum\'e}

\bigskip
Dans cet article, nous g\'en\'eralisons au cas localement compact et r\'egulier, deux r\'esultats fondamentaux \cite{RV} \cite{V1}  portant sur les actions des groupes quantiques compacts.  Soient  $G_1$ et $G_2$ deux groupes quantiques localement compacts mono\"idalement \'equivalents \cite{DeC1, DeC2}  au sens  de De Commer,  et  r\'eguliers. Par un proc\'ed\'e d'induction que nous introduisons, nous \'etablissons une \'equivalence des cat\'egories  $A^{G_1}$ et  $A^{G_2}$  form\'ees par les actions des groupes $G_1$ et $G_2$ dans les C*-alg\`ebres. Comme application de ce r\'esultat, nous d\'eduisons  l'\'equivalence des cat\'egories $KK^{G_1}$ et  $KK^{G_2}$. La preuve   s'appuie  sur  une  version de la dualit\'e de Takesaki-Takai  pour les actions continues dans les C*-alg\`ebres  d'un groupo\"ide mesur\'e quantique de base finie.

\bigskip
{\bf Mots-cl\'es} 
: groupes quantiques localement compacts, \'equivalence mono\"idale, K-th\'eorie bivariante.
\bigskip

\bigbreak

\centerline{\bf Abstract}

\bigskip
In this article, we generalize to the case of regular locally compact quantum groups, two important results concerning actions of compact quantum groups (see \cite{RV} and \cite{V1}). Let $G_1$ and $G_2$ be two mono\"idally equivalent regular locally compact quantum groups in the sense of De Commer (see \cite{DeC1,DeC2}). We introduce an induction procedure and we build an equivalence of the categories $A^{G_1}$ and $A^{G_2}$ consisting of continuous actions of $G_1$ and $G_2$ on C*-algebras. As an application of this result, we derive a canonical equivalence of the categories $KK^{G_1}$ and $KK^{G_2}$. We introduce and investigate a notion of actions on C*-algebras of measured quantum groupoids (see \cite{E}) on a finite basis. The proof of the equivalence between $KK^{G_1}$ and
$KK^{G_2}$ relies on a version of the Takesaki-Takai duality theorem for continuous actions on C*-algebras of measured quantum groupoids on a finite basis. 

\bigskip

\newpage

\tableofcontents

\section*{Introduction}
\addcontentsline{toc}{section}{Introduction}

L'\'equivalence mono\"idale  des groupes quantiques compacts  a \'et\'e d\'evelopp\'ee par Bichon, De Rijdt et Vaes (\cf\cite{BRV}). Deux groupes quantiques compacts $G_1$ et $G_2$ sont dits mono\"idalement \'equivalents si les  cat\'egories des repr\'esentations  de $G_1$ et $G_2$ sont \'equivalentes comme C*-cat\'egories mono\"idales.
\hfill\break
Dans \cite{BRV}, les auteurs ont montr\'e que $G_1$ et $G_2$ sont mono\"idalement \'equivalents si et seulement s'il existe une C*-alg\`ebre unitale $B$, munie d'une action continue \`a gauche de  $G_1$  et d'une  action continue \`a droite de  $G_2$, qui commutent et sont ergodiques de multiplicit\'e pleine.
\hfill\break
Plusieurs r\'esultats importants de la th\'eorie g\'eom\'etrique des groupes quantiques discrets libres (marches al\'eatoires et fronti\`eres associ\'ees, propri\'et\'e de Haagerup, moyennabilit\'e faible, $K$-moyennabilit\'e...),  reposent sur l'\'equivalence mono\"idale de leurs duaux compacts. Parmi les applications de l'\'equivalence mono\"idale \`a cette th\'eorie, citons les suivantes :
\hfill\break
$\bullet$ Dans \cite{VV}, en  exploitant l'\'equivalence mono\"idale \cite{BRV} de ${\rm A}_{\rm o}(F)$ avec un ${\rm SU}_q(2)$ convenable (et les r\'esultats de \cite{I}, \cite{VVe}),  Vaes et Vander Vennet ont calcul\'e les fronti\`eres de Poisson et de Martin pour les duaux des groupes orthogonaux libres ${\rm A}_{\rm o}(F)$ ;

$\bullet$ Dans \cite{RV}, De Rijdt et Vander  Vennet ont \'etabli une correspondance bijective entre les actions continues de deux  groupes quantiques compacts mono\"idalement \'equivalents. De plus, cette correspondance \'echangent les fronti\`eres de Poisson ou de Martin des duaux discrets de ces  deux groupes. Il en r\'esulte que si on connait la fronti\`ere de Poisson ou de Martin pour un groupe quantique discret $\widehat G$, on peut d\'eduire celles des groupes dont les duaux compacts sont mono\"idalement \'equivalents \`a $G$. Ce principe a permis aux auteurs de calculer les fronti\`eres de Poisson des duaux des groupes quantiques d'automorphismes ;

$\bullet$ Dans \cite{DFY}, et en utilisant le principe pr\'ec\'edent, les auteurs ont \'etabli la propri\'et\'e CCAP et celle de Haagerup pour le dual de n'importe quel ${\rm A}_{\rm o}(F)$.  Gr\^ace \`a    la compatibilit\'e de l'\'equivalence mono\"idale avec certaines op\'erations, ils ont \'etendu ces propri\'et\'es pour les groupes quantiques  discrets libres ;

$\bullet$ Dans \cite{V1}, Voigt a montr\'e que les cat\'egories $KK^{G_1}$ et $KK^{G_2}$ de deux groupes quantiques compacts  mono\"idalement \'equivalents,   sont \'equivalentes. Ce r\'esultat entra\^ine l'invariance par \'equivalence mono\"idale de la conjecture de Baum-Connes pour les duaux. En \'etablissant cette conjecture pour le dual d'un ${\rm SU}_q(2)$ convenable, Voigt a prouv\'e cette conjecture pour les duaux des groupes orthogonaux libres ${\rm A}_{\rm o}(F)$, ainsi que la $K$-moyennabilit\'e de ces groupes. Ce r\'esultat a permis \`a Vergnioux et 
Voigt \cite{VV1}  d'\'etablir cette conjecture pour  les duaux des groupes unitaires  libres ${\rm A}_{\rm u}(F)$. 

Dans sa th\`ese \cite{DeC2}, De  Commer a \'etendu la notion  d'\'equivalence mono\"idale  au cas localement compact. Deux  groupes quantiques localement compacts $G_1$ et $G_2$  (au sens de Kustermans et Vaes \cite{KustV}), sont dits mono\"idalement \'equivalents s'il existe une action galoisienne \`a gauche $\gamma$  de $G_1$,   et
une action galoisienne \`a droite   $\alpha$ de $G_2$, dans la m\^eme alg\`ebre de von Neumann $N$,  qui commutent.    
Il a montr\'e que cette notion se d\'ecrit par un groupo\"ide mesur\'e quantique (au sens \cite{Le} et \cite{E}), de base $\C^2$, dont $G_1$ et $G_2$ sont des ''sous-groupes''. Un tel groupo\"ide est appel\'e un groupo\"ide de co-liaison.

Les groupo\"ides mesur\'es quantiques ont \'et\'e introduits et \'etudi\'es par Lesieur et Enock (\cf\cite{Le}, \cite{E}).
Un groupo\"ide mesur\'e  quantique au sens de  \cite{E}, est un octuplet ${\cal    G} = (N, M,  \alpha, \beta, \Gamma, T, T', \nu)$, o\`u $N$ et $M$ sont des alg\`ebres de von Neumann ($N$ est la base   et $M$ est l'alg\`ebre   du groupo\"ide; elles  correspondent respectivement  \`a l'espace  des unit\'es  et \`a l'espace total d'un groupo\"ide classique), $\alpha$ et $\beta$  sont des morphismes  normaux et fid\`eles de $N$ et $N^{\rm o}$ (l'alg\`ebre oppos\'ee de $N$) dans $M$ dont les images commutent, $\Gamma$ est le coproduit, $\nu$ est un poids normal semi-fini sur $N$, et $T$ , $T'$ sont des poids op\'eratoriels de $M$ dans $N$. Ces objets  v\'erifient une liste d'axiomes.\hfill\break 
Dans le cas o\`u la base $N$ est de dimension finie, les axiomes liants les objets d'un tel octuplet ${\cal   G}$, ont \'et\'e simplifi\'es par De Commer (\cf\cite{DeC1}, \cite{DeC2}) et nous utiliserons  cette version dans la suite. Plus pr\'ecis\'ement,  on peut prendre pour  poids $\nu$, la trace de Markov (non normalis\'ee)    de   la C*-alg\`ebre $N = \oplus_{l=1}^k M_{n_l}$.
Le produit tensoriel relatif d'espaces de Hilbert (resp. le produit fibr\'e d'alg\`ebres de von Neumann) est  remplac\'e   par le produit tensoriel ordinaire d'espaces de Hilbert (resp. d'alg\`ebres de von Neumann). Le coproduit $\Gamma$ est \`a valeurs dans $M \otimes M$, mais n'est pas unital.
\hfill\break
%Si  $N= \C^2$ et si les applications source et but $\alpha$ et $\beta$ coincident,  et  prennent leurs valeurs \`a l'ext\'erieur %du centre de $M$,  De Commer \cite{DeC1}  a montr\'e que le groupo\"ide quantique ${\cal G}$ a pour    dual au sens de %\cite{E}, le groupo\"ide quantique canoniquement associ\'e  \`a  deux groupes quantiques localement compacts $G_1$ et %$G_2$   mono\"idalement \'equivalents.  \hfill\break
Dans cet article, nous introduisons la notion d'action continue d'un groupo\"ide mesur\'e quantique  ${\cal G}$ de  base une C*-alg\`ebre $N$ de dimension finie,  dans une C*-alg\`ebre $A$. Nous \'etendons la construction du produit crois\'e et d'action duale. Dans le cas o\`u ${\cal G}$ est r\'egulier, nous \'etendons la dualit\'e de Takesaki-Takai \cite{BaSka2} \`a ce cadre.\hfill\break
Si un  groupo\"ide de co-liaison    ${\cal G}$,  associ\'e  \`a l'\'equivalence mono\"idale de deux groupes quantiques localement compacts $G_1$ et $G_2$, agit continument   dans une   C*-alg\`ebre $A$, alors $A$ est une somme directe $A = A_1 \oplus A_2$  et      par restriction de l'action de ${\cal G}$, les groupes quantiques $G_1$ et $G_2$ agissent continument dans $A_1$ et $A_2$ respectivement.  R\'eciproquement, si $G_1$ et $G_2$ sont r\'eguliers, nous associons canoniquement \`a    toute action continue de $G_1$ dans une C*-alg\`ebre $A_1$, une C*-alg\`ebre $A_2$  munie d'une action continue de $G_2$. Comme cons\'equences importantes de cette construction, nous \'etablissons : \hfill\break
 -  une correspondance  fonctorielle    bijective,  entre les  actions continues des groupes quantiques $G_1$ et $G_2$ qui g\'en\'eralise le cas compact \cite{RV}, ainsi que le cas des d\'eformations \cite{NT}  par un $2$-cocycle ;\hfill\break
 -   une \'equivalence de Morita  des produits crois\'es $A_1 \rtimes G_1$ et $A_2 \rtimes G_2$ ;\hfill\break
 -  une description compl\`ete des actions continues d'un groupo\"ide de co-liaison ; \hfill\break
 -  l'\'equivalence des cat\'egories $KK^{G_1}$ et  $KK^{G_2}$ de  Kasparov,  g\'en\'eralisant au cas localement compact et r\'egulier, un r\'esultat \cite{V1} de Voigt.\hfill\break
Les preuves des r\'esultats ci-dessus utilisent de fa\c con cruciale  la r\'egularit\'e des   groupes quantiques $G_1$ et $G_2$. Nous montrons que la r\'egularit\'e  (au sens de \cite{Ec}, voir aussi \cite{Tim1,Tim2}),  d'un  groupo\"ide de co-liaison    ${\cal G}$,  associ\'e  \`a l'\'equivalence mono\"idale de deux groupes quantiques localement compacts $G_1$ et $G_2$, est \'equivalente  \`a la r\'egularit\'e des groupes quantiques $G_1$ et $G_2$. Ce r\'esultat r\'esout aussi quelques questions pos\'ees dans \cite{NT} dans le cas des d\'eformations par un $2$-cocycle, qui a \'et\'e \'etudi\'e aussi dans \cite{Kasprzak}.

L'organisation de ce travail est la suivante : 

  - Au  premier  paragraphe, nous rappelons la notion de groupe quantique \cite{KustV}  localement compact (l.c.)\index{la@l.c.}, ainsi que la d\'efinition  d'une action continue  \cite{BaSkaV} d'un tel objet, dans une C*-alg\`ebre ;
\hfill\break
 - Au deuxi\`eme paragraphe, nous  montrons que les groupo\"ides quantiques mesur\'es \cite{E} v\'erifient une propri\'et\'e  d'irr\'eductibilit\'e.  Ce r\'esultat \'etend au cas des groupo\"ides quantiques mesur\'es, l'irr\'eductibilt\'e de la repr\'esentation r\'eguli\`ere \cite{BaSka2}  d'un groupe quantique l.c.  Nous en d\'eduisons l'\'equivalence entre la r\'egularit\'e d'un groupo\"ide de co-liaison associ\'e \`a deux groupes quantiques mono\"idalement \'equivalents $G_1$ et $G_2$, et celle des deux groupes ;
\hfill\break
- Au troisi\`eme paragraphe, nous  introduisons la notion d'action continue d'un groupo\"ide  quantique de base de dimension finie,  dans une C*-alg\`ebre. Nous montrons que le produit crois\'e admet une action continue du groupo\"ide dual et nous g\'en\'eralisons, dans le cas r\'egulier, la dualit\'e \cite{BaSka2} de Takesaki-Takai  pour le double produit crois\'e. Nous d\'ecrivons aussi le double produit crois\'e par l'action d'un groupo\"ide de co-liaison ;\hfill\break
- Au quatri\`eme paragraphe,  nous \'etablissons l'\'equivalence des actions continues de deux groupes quantiques \lc $G_1$ et $G_2$ mono\"idalement \'equivalents et  r\'eguliers ;\hfill\break
 - Au cinqui\`eme paragraphe, nous construisons une \'equivalence des cat\'egories $KK^{G_1}$ et  $KK^{G_2}$ associ\'ees \`a deux groupes quantiques $G_1$ et $G_2$ mono\"idalement \'equivalents et  r\'eguliers ;
\hfill\break
- En appendice, nous avons regroup\'e quelques r\'esultats utilis\'es dans l'article.

{ \it  Le premier auteur remercie G. Skandalis pour d'utiles discussions sur divers points de cet article.}

\section{Rappels et notations}

\subsection{Notations pr\'eliminaires}

Nous introduisons quelques notations et conventions utilis\'ees dans l'article.

$\bullet$ Pour tout sous-ensemble $X$ d'un espace de Banach $E$, on note $[X]$ le sous-espace vectoriel ferm\'e  de $E$ engendr\'e par $X$.

$\bullet$ Tous les produits tensoriels de C*-alg\`ebres sont suppos\'es munis de la norme spatiale (produits tensoriels \og 
min \fg).

$\bullet$  Si $A$ est une C*-alg\`ebre, on note $M(A)$ la C*-alg\`ebre des multiplicateurs de $A$.

$\bullet$  Soient $H$, $K$ deux espaces de Hilbert, on note $\Sigma_{H \otimes K}$, ou plus simplement 
$\Sigma$, la volte, c'est-\`a-dire l'op\'erateur unitaire $H \otimes K \rightarrow K \otimes H : \xi \otimes \eta  \mapsto \eta \otimes \xi$.

Dans cet article, nous utiliserons la notion de C*-module hilbertien sur une C*-alg\`ebre, ainsi que leurs produits tensoriels (interne et externe). 
Les d\'efinitions et conventions utilis\'ees sont celles de \cite{K1}. En particulier, soient ${\cal E}$ et ${\cal F}$ deux 
C*-modules hilbertiens sur une C*-alg\`ebre $A$.

$\bullet$ On note $\cL({\cal E}, {\cal F})$ l'espace de Banach  form\'e par les   op\'erateurs 
$T : {\cal E} \rightarrow  {\cal F}$ qui ont un adjoint et   $\cL({\cal E})$ la C*-alg\`ebre  $\cL({\cal E}, {\cal E})$. 
\hfill\break
$\bullet$ Si  $\xi \in {\cal F}$ et 
$ \eta \in {\cal E}$, on note 
$\theta_{\xi, \eta}$ l'op\'erateur  $\zeta \mapsto \xi \langle \eta , \zeta\rangle_A$ et on pose $\cK({\cal E}, {\cal F})  = [ \theta_{\xi, \eta} \,| \, \xi \in {\cal F}\, , \,
 \eta \in {\cal E}]$ et  $\cK({\cal E}) = \cK({\cal E}, {\cal E})$ la sous-C*-alg\`ebre des op\'erateurs compacts engendr\'ee par les op\'erateurs $\theta_{\xi, \eta}$ dans $\cL({\cal E})$. On rappelle \cite{K1} qu'on a  $\cL({\cal E}) = M(\cK({\cal E}))$.

\subsection{Poids sur une alg\`ebre de von Neumann \texorpdfstring{\cite{Co}}{}}  \noindent
Soit $\varphi$ un poids normal semi-fini et fid\`ele (nsff)\index{nsff} sur une alg\`ebre de von Neumann $M$. On note $\mathfrak N_\varphi$, $H_\varphi$, $\Lambda_\varphi$, $J_\varphi$, $\Delta_\varphi$, $(\sigma_t^\varphi)$ 
  les objets canoniques de la th\'eorie de Tomita-Takesaki, associ\'es au poids $\varphi$. Rappelons simplement qu'on a : 
\hfill\break
 -  $\mathfrak N_\varphi := \{x \in M \,\,|\,\, \varphi(x^*x) < \infty\}$  et l'espace de Hilbert $H_\varphi$ est le  compl\'et\'e de
$\mathfrak N_\varphi$ pour le produit scalaire $\langle x , y\rangle := \varphi(x^*y)$, pour $x,y\in\mathfrak{N}_{\varphi}$ ;
\hfill\break
 - $\Lambda_\varphi : \mathfrak N_\varphi  \rightarrow H_\varphi$ est l'injection de $\mathfrak N_\varphi$ dans son compl\'et\'e ;
\hfill\break
 - $\pi_\varphi : M \rightarrow \cL(H_\varphi)$ est la repr\'esentation d\'efinie par 
$\pi_\varphi(a) \Lambda_\varphi x :=   \Lambda_\varphi (a x)$.
\hfill\break
Le triplet $(H_\varphi , \pi_\varphi , \Lambda_\varphi)$ est par d\'efinition la repr\'esentation GNS du poids $\varphi$. 

\subsection{Groupes quantiques localement compacts  \texorpdfstring{\cite{{KustV},{BaSka2}}}{}} 
\begin{definition}\cite{KustV}
Un groupe quantique localement compact (l.c.)\index{la@l.c.} est un couple  $G = (M , \delta)$ tel que : 
\begin{enumerate}
\item $M$ est une alg\`ebre de von Neumann et $\delta : M \rightarrow M \otimes M$ est un *-morphisme unital,  injectif,   normal et coassociatif, \ie $(\delta \otimes \id_M) \delta= (\id_M \otimes \delta) \delta$ ;
\item Il existe des poids nsff $\varphi$ et  $\psi$ sur  $M$ v\'erifiant :  
\hfill\break
  -  $\varphi$  est  invariant \`a  gauche, \ie $\varphi((\omega \otimes {\rm id}_M)\delta(x)) = \omega(1) \varphi(x)$ pour tout $x\in \mathfrak M_\varphi^+$ et tout $\omega\in M_\ast^+$,
\hfill\break
  -  $\psi$ est  invariant \`a  droite, \ie $\psi(({\rm id}_M  \otimes \omega)\delta(x)) = \omega(1) \psi(x)$ pour tout $x\in \mathfrak M_\psi^+$ et tout $\omega\in M_\ast^+$.
\end{enumerate}
\end{definition}
\noindent
Pour un groupe quantique \lc $(M, \delta)$,  un poids nsff sur $M$  invariant \`a gauche (resp. invariant \`a droite) est unique \cite{KustV} \`a une constante strictement positive pr\`es.
\hfill\break
Soit $G =(M, \delta)$ un groupe quantique l.c. Fixons un poids nsff sur $M$  invariant \`a gauche (une  mesure de Haar \`a gauche) $\varphi$. La repr\'esentation GNS  de $\varphi$ est not\'ee 
$(L^2(G) ,  \pi , \Lambda)$.  La repr\'esentation r\'eguli\`ere gauche du groupe quantique $G$ est l'unitaire multiplicatif \cite{BaSka2} $W\in B(L^2(G) \otimes L^2(G))$ d\'efini par
$$W^*(\Lambda x \otimes \Lambda y) = (\Lambda \otimes \Lambda)(\delta(y)(x \otimes 1)), \quad x , y \in  \mathfrak N_\varphi.$$
\noindent
En identifiant $L^\infty(G) := M$ \`a son image par $\pi$, on obtient :
\hfill\break
$\bullet$ $M$ est l'adh\'erence forte de l'alg\`ebre $\{({\rm id} \otimes \omega)(W) \,|\, \omega \in B(L^2(G))_\ast\}$ ;
\hfill\break
$\bullet$ $\delta(x) = W^*(1 \otimes x) W$ pour tout $x\in M$.
\hfill\break
L'alg\`ebre de Hopf-von Neumann $(M , \delta)$ admet  \cite{KustV} une antipode unitaire $R_M : M \rightarrow M$ et on peut prendre $\psi := \varphi \circ R_M$ pour mesure de Haar  \`a droite. Le cocycle  de Connes \cite{{Co},{V}}
   de $\psi$  relativement \`a $\varphi$  est de la forme :
$$(D\psi : D\varphi)_t  =  \nu^{i t^2\!/ 2} d^{\,it},\quad t\in\R \quad ;\quad \nu >0,\,\,d \eta M.$$
\noindent
Posons $\mathfrak N_\varphi^{d} = \{x\in M \,\,/\,\, x d^{1/ 2} \,\,\hbox{born\'e et } \,\overline{xd^{1/ 2}} \in \mathfrak N_\varphi\}$. La repr\'esentation GNS \cite{V} de $\psi$ est donn\'ee par $(L^2(G) , \id , \Lambda_\psi)$, o\`u 
$\Lambda_\psi$ est  la fermeture pour les topologies ultraforte/normique de l'application 
$\mathfrak N_\varphi^{d} \rightarrow H :  x \mapsto\Lambda  \overline{xd^{1/ 2}}$. Avec ce choix de repr\'esentation GNS pour $\psi$, on a $J_\psi = \nu^{i/ 4} J$.
\hfill\break
La repr\'esentation r\'eguli\`ere droite de $G$ est l'unitaire multiplicatif 
$V\in B(L^2(G) \otimes L^2(G))$ d\'efini par
$$V(\Lambda_\psi x \otimes \Lambda_\psi  y) = (\Lambda_\psi  \otimes \Lambda_\psi )(\delta(x)(1  \otimes y)), \quad x , y \in  \mathfrak N_\psi.$$
\noindent
Le groupe quantique  $\widehat G$ dual de $G$ est d\'efini par l'alg\`ebre de Hopf-von Neumann
$(L^\infty(\widehat G) , \widehat \delta)$, o\`u 
$L^\infty(\widehat G)$ est la  fermeture forte de l'alg\`ebre $\{(\id \otimes \omega)(V) \, | \, \omega  \in B(L^2(G))_\ast\}$ et le coproduit $\widehat \delta : L^\infty(\widehat G)  \rightarrow     L^\infty(\widehat G)\otimes  L^\infty(\widehat G)$ est défini par $\widehat \delta(x) = V ^*(1 \otimes x )V$ pour tout $x\in L^\infty(\widehat G)$.
\hfill\break
Le groupe quantique $\widehat G$ admet des mesures de Haar \cite{KustV},  dont on peut prendre $L^2(G)$  pour espace de Hilbert des  repr\'esentations GNS, on note dans la suite  $\widehat J$ l'op\'erateur de Tomita de la mesure de Haar \`a gauche  de $\widehat G$.

\subsubsection{C*-alg\`ebres de Hopf associ\'ees \`a un groupe quantique}

Au groupe quantique $G$, on associe  \cite{BaSka2, KustV} deux C*-alg\`ebres de Hopf $(S, \delta)$ et 
$(\widehat S, \widehat \delta)$ d\'efinies par : 
\hfill\break
$\bullet$ $S$ (resp. $\widehat S$)\index{sa@$S$, $\widehat{S}$} est la fermeture normique de l'alg\`ebre 
$$\{(\omega \otimes \id)(V) \, | \, \omega  \in B(L^2(G))_\ast\}\quad  \quad   ( {\rm resp.} \,\,
 \{(\id \otimes \omega)(V) \, | \, \omega  \in B(L^2(G))_\ast\}) \, ;$$
\noindent
$\bullet$ $ \delta : S  \rightarrow     M(S \otimes  S) : x  \mapsto V(x \otimes 1)V^*  ,\quad 
 \widehat \delta : \widehat S  \rightarrow     M(\widehat  S\otimes  \widehat S) : x  \mapsto V^*(1 \otimes x)V$.

Posons $U:= \widehat J J = \nu^{i/ 4} J \widehat J$\index{u@$U$}.  Les C*-alg\`ebres $S$ et $\widehat S$ admettent les repr\'esentations :
\hfill\break
 -  $L : S \rightarrow B(L^2(G)) : x \mapsto L(x) = x  , \quad 
R : S \rightarrow B(L^2(G)) : x \mapsto  UL(x)U^*$ ;
\hfill\break
 -  $\rho  : \widehat S \rightarrow B(L^2(G)) : x \mapsto \rho(x) = x , \quad 
\lambda : \widehat S \rightarrow B(L^2(G)) : x \mapsto  U\rho(x)U^*$.\index{lb@$L$, $R$, $\rho$, $\lambda$}
\hfill\break
On d\'eduit de (\cite{KustV}, Proposition 2.15) qu'on a 
$$W =\Sigma (U \otimes 1) V (U^* \otimes 1)\Sigma  \quad \text{et} \quad [W_{12} , V_{23}] = 0.$$
\noindent
La repr\'esentation r\'eguli\`ere droite de $\widehat G$ est l'unitaire multiplicatif  $\widetilde V = 
\Sigma (1 \otimes U) V (1 \otimes U^*) \Sigma$.
\begin{notations}
\begin{enumerate}
\item $\cK \subset B(L^2(G))$ d\'esigne la C*-alg\`ebre des op\'erateurs compacts.
\item $\cC(V)$\index{c@$\cC(-)$} (resp. $\cC(W)$) est la fermeture normique de l'alg\`ebre \cite{BaSka2} 
$\{({\rm id} \otimes \omega)(\Sigma  V) \, | \, \omega  \in B(L^2(G))_\ast\}$ (resp.
$\{({\rm id} \otimes \omega)(\Sigma W) \, | \, \omega  \in B(L^2(G))_\ast\}$).
\end{enumerate}
\end{notations}
\begin{definition}\cite{BaSka2, BaSkaV}
Le groupe quantique $G$ est dit r\'egulier (resp. semi-r\'egulier) si $\cK = \cC(V)$ (resp. 
$\cK \subset \cC(V))$.
\end{definition}
Notons que $G$ est   r\'egulier (resp. semi-r\'egulier) si et seulement si $\cK = \cC(W)$ (resp. 
$\cK \subset \cC(W))$ (\cf \cite{BaSkaV}).

\subsubsection{Actions continues de groupes quantiques l.c. Produits crois\'es }
 
Conservons les notations pr\'ec\'edentes. Soit $A$ une C*-alg\`ebre. Une coaction de $(S,\delta)$ dans $A$ est un *-morphisme $\delta_A  :  A  \rightarrow M(A \otimes S)$ non d\'eg\'en\'er\'e v\'erifiant $(\delta_A \otimes \id_S) \delta_A = (\id_A \otimes \delta) \delta_A$.  
\begin{definition}
Une action fortement (resp. faiblement) continue de $G$ dans $A$ est  une coaction $\delta_A$  de $(S,\delta)$ dans $A$ v\'erifiant 
$$[\delta_A(A)(1_A \otimes S)] = A \otimes S \quad (\text{resp. } A = [(\id_A \otimes \omega)\delta_A(A)\,|\,  \omega\in B(L^ 2(G))_\ast]).$$
Une $G$-alg\`ebre est un couple $(A, \delta_A)$, o\`u $A$ est une C*-alg\`ebre et $\delta_A  :  A  \rightarrow M(A \otimes S)$ est une action fortement continue  de $G$ dans $A$.
\end{definition}

Si $G$ est r\'egulier, toute action de $G$ faiblement continue est fortement continue (\cf\cite{BaSkaV}). Dans ce cas, nous dirons que $\delta_A$ est une action continue de $G$ dans $A$.

\begin{definition}
 Soit $\delta_A  :  A  \rightarrow M(A \otimes S)$ (resp. $\delta_A  :  A  \rightarrow M(A \otimes \widehat S)$) une action  fortement continue de $G$ (resp. $\widehat G$) dans une C*-alg\`ebre $A$.  Notons $\pi_L$ (resp. $\widehat\pi_\lambda$) la repr\'esentation de $A$ dans le C*-module $A \otimes L^2(G))$ d\'efinie par 
$\pi_L := (\id_A \otimes L) \delta_A$ (resp. $\widehat\pi_\lambda := (\id_A \otimes \lambda) \delta_A$).\index{pa@$\pi_L$, $\widehat{\pi}_{\lambda}$}
\hfill\break
On appelle produit crois\'e  (r\'eduit) de $A$ par $G$ (resp. $\widehat G$), la  sous-C*-alg\`ebre not\'ee $A \rtimes G$ (resp. $A \rtimes \widehat G$)\index{a@$A \rtimes G$, $A \rtimes \widehat G$} de $\cL(A \otimes L^2(G))$ engendr\'ee par les produits 
$\pi_L(a)(1_A \otimes \rho(x))$, $a\in A$ et $x\in\widehat S$ (resp. $\widehat\pi_\lambda(a) (1_A \otimes L(x))$, $a\in A$ et $x\in S$).
\end{definition}

Le produit crois\'e $A \rtimes G$ (resp. $A \rtimes \widehat G$),  admet une action fortement continue de $\widehat G$ (resp. $G$). Dans le cas o\`u $G$ est r\'egulier, la dualit\'e de Takesaki-Takai s'\'etend \`a ce cadre, \cf\cite{BaSka2}.

\begin{notation} On note $A^G$\index{ab@$A^G$} la cat\'egorie dont les objets sont les $G$-alg\`ebres et les morphismes sont les *-morphismes 
$f : A \rightarrow M(B)$ non d\'eg\'en\'er\'es et $G$-\'equivariants, \ie $(f \otimes \id)\delta_A  = \delta_B \circ f$.
\end{notation}

\subsubsection{Bimodules hilbertiens \'equivariants}\label{dbimod}

Dans ce qui suit, nous rappelons bri\`evement la bidualit\'e pour les produits crois\'es de bimodules hilbertiens \'equivariants, \cf\cite{BaSka1, BaSka2}.

\'Etant donné un groupe quantique \lc r\'egulier $G$, soient $(A, \delta_{A})$ et $(B, \delta_{B})$ deux $G$-alg\`ebres et $({\cal E}, \delta_{{\cal E}})$ un $A-B$ bimodule hilbertien $G$-\'equivariant, \ie ${\cal E}$ est un $A-B$ bimodule hilbertien et la C*-alg\`ebre de liaison $\cK({\cal E} \oplus B)$ est munie d'une action continue 
$$\delta_{\cK({\cal E} \oplus B)} : \cK({\cal E} \oplus B) \rightarrow M(\cK({\cal E} \oplus B) \otimes S)$$
\noindent
compatible avec les coactions $\delta_A$ et $\delta_B$ au sens suivant :
\hfill\break
 - le *-morphisme \'evident $B \rightarrow \cK({\cal E} \oplus B)$ est suppos\'e $G$-\'equivariant ;
\hfill\break
 - l'action \`a gauche $A \rightarrow \cL({\cal E})$ est suppos\'ee $G$-\'equivariante.
\hfill\break
On note $\delta_{{\cal E}}$ la restriction de $\delta_{\cK({\cal E} \oplus B)}$ \`a ${\cal E}$ (\cf\cite{BaSka1}).

On d\'eduit alors de (\cite{BaSka2}, paragraphe 7) : 

1) Le $A \rtimes G \rtimes \widehat {G}- B \rtimes G \rtimes \widehat {G}$ bimodule hilbertien $G$-\'equivariant ${\cal E} \rtimes G \rtimes \widehat {G}$ s'identifie au $A \otimes {\cal K}(L^2(G)) - B \otimes {\cal K}(L^2(G))$ bimodule  hilbertien $G$-\'equivariant ${\cal E} \otimes {\cal K}(L^2(G))$.
\hfill\break
2) Les $G$-alg\`ebres $A  \otimes {\cal K}(L^2(G))$ et $B  \otimes {\cal K}(L^2(G))$ sont munies de l'action biduale de $G$. L'action de $G$ dans le $B \otimes {\cal K}(L^2(G))$-module  hilbertien   ${\cal E} \otimes {\cal K}(L^2(G))$
est donn\'ee  par la coaction $\delta_{{\cal E} \otimes {\cal K}(L^2(G))}$  d\'efinie par :
$$\delta_{{\cal E} \otimes {\cal K}(L^2(G))}(\xi)  = V_{23} ({\rm id}_{\cE} \otimes \sigma)(\delta_{{\cal E}} \otimes {\rm id}_{{\cal K}(L^2(G))})(\xi) V_{23}^*  , \quad \xi \in {\cal E} \otimes {\cal K}(L^2(G)),$$
\noindent
o\`u $\sigma : S \otimes {\cal K}(L^2(G)) \rightarrow {\cal K}(L^2(G)) \otimes S : x \otimes y \mapsto y\otimes x$ et $V\in M({\cal K}(L^2(G)) \otimes S)$ est la repr\'esentation r\'eguli\`ere droite du groupe quantique $G$.

\section{Compl\'ements sur les groupo\"ides mesur\'es quantiques }

Dans ce paragraphe, nous \'etablissons un r\'esultat d'irr\'eductbilit\'e valable pour les groupo\"ides mesur\'es quantiques. Dans le cas o\`u la base du groupo\"ide est de dimension finie,      ce r\'esultat   joue un r\^ole crucial pour \'etablir la dualit\'e de Takesaki-Takai. Dans le cas d'un groupo\"ide de co-liaison \cite{{DeC1}, {DeC2}}  il permet de relier la r\'egularit\'e de ce groupo\"ide \`a celle des deux groupes quantiques \lc mono\"idalement \'equivalents sous-jacents  \`a ce groupo\"ide.

Nous commen\c cons par rappeler les d\'efinitions et r\'esultats portant sur les  groupo\"ides mesur\'es quantiques dont  on a besoin.    On se r\'ef\`ere principalement \`a \cite{E}.  Les notions  de produits tensoriels relatifs d'espaces de Hilbert et de produits fibr\'es  d'alg\`ebres de von Neumann sont rappel\'es dans l'appendice.

\subsection{Irr\'eductibilit\'e }

Un groupo\"ide  mesur\'e   quantique   (m.q.)\index{ma@m.q.} de base $N$,   est un octuplet    ${ \cal   G} = (N, M,  \alpha, \beta, \Gamma, T, T', \nu)$, o\`u   :\hfill\break
 -  $M$ et $N$ sont  des    alg\`ebres de von Neumann,    $\alpha : N  \rightarrow M $ et $\beta : N^{\rm o} \rightarrow M$ sont    des morphismes  normaux et fid\`eles ($N^{\rm o}$ d\'esigne l'alg\`ebre de von Neumann oppos\'ee) ;  \hfill\break
 - le coproduit $\Gamma : M \rightarrow \fprod{M}{\beta}{\alpha}{M}$ est un morphisme   normal et fid\`ele ($\fprod{M}{\beta}{\alpha}{M}$ est un produit fibr\'e d'alg\`ebres de von Neumann, \cf \ref{prodfib}) ;\hfill\break
 - $T : M^+ \rightarrow \alpha(N)^{+,{\rm ext}}$  et $T' : M^+ \rightarrow \beta(N^{\rm o})^{+,{\rm ext}}$  sont des   poids op\'eratoriels normaux et fid\`eles ;\hfill\break
 - $\nu$ est un poids nsff sur $N$.

Ces objets sont assujettis aux conditions suivantes : \hfill\break
1. les images des morphismes $\alpha$ et $\beta$ commutent ;\hfill\break
2. $\Gamma$ est coassociatif, \ie $(\fprod{\Gamma}{\beta}{\alpha}{{\rm id}}) \Gamma = (\fprod{{\rm id}}{\beta}{\alpha}{\Gamma}) \Gamma$ ; \hfill\break
3. pour tout $n\in N$, on a $\Gamma(\alpha(n))  =   \reltens { \alpha(n) } { \beta } { \alpha } { 1 }$ et $\Gamma(\beta(n^{\rm o}))  =  \reltens { 1 } { \beta } { \alpha } { \beta(n^{\rm o}) }$ ;\hfill\break
4. les poids  nsff $\varphi$ et $\psi$   sur $M$ d\'efinis respectivement par $\varphi = \nu \circ  \alpha^{-1} \circ   T$ et $\psi = 
\nu \circ   \beta^{-1} \circ   T'$ v\'erifient 
$$T = (\fprod{{\rm id}}{\beta}{\alpha}{\varphi}) \Gamma \quad {\rm et} \quad T' =(\fprod{\psi}{\beta}{\alpha}{{\rm id}} )\Gamma\,;$$\noindent
5. les groupes modulaires des poids $\varphi$ et $\psi$ commutent.

\noindent
Soit $(H, \pi, \Lambda)$ la repr\'esentation GNS du poids $\varphi$, notons   $(\sigma_t)$, $\Delta$ et $J$ respectivement le  groupe modulaire, l'op\'erateur modulaire   et  l'op\'erateur de Tomita associ\'es au  poids $\varphi$. Dans ce qui suit, on identifie $M$ et son image par $\pi$ dans $B(H)$.\hfill\break  
D'apr\`es \cite{E},  on a : \hfill\break
 - $R_{\cal G}  : M \rightarrow M$ un *-antiautomorphisme v\'erifiant $R_{\cal G}^2 = {\rm id} _M$ et coinvolutif ;
\hfill\break
 - on peut supposer que $T' = R_{\cal G}  T R_{\cal G} $, et donc aussi $\psi = \varphi \circ   R_{\cal G} $ ;
\hfill\break
 -  pour tout $t\in\R$, on  a  $(D\psi : D\varphi)_t  = \lambda^{i t^2\!/ 2} d^{\,it}$, o\`u $\lambda, d$ sont des op\'erateurs autoadjoints positifs affili\'es \`a $M$, l'op\'erateur $\lambda$ \'etant central.
\hfill\break
 -  la repr\'esentation GNS du poids $\psi$ est donn\'ee par $(H,  \id , \Lambda_\psi)$, o\`u $\Lambda_\psi$ est la fermeture pour les topologies ultraforte/normique de l'op\'erateur   non born\'e   :
$$\{ x\in M\,|\,\, x d^{1/ 2} \,\,\hbox{born\'e et } \,\,\overline{xd^{1/ 2}}\in \mathfrak N_\varphi \} \rightarrow  H  : x   \mapsto \Lambda \overline{xd^{1/ 2}}\,;$$
\noindent
 - l'op\'erateur de Tomita $J_\psi$ associ\'e au poids $\psi$ est donn\'e par $J_\psi = \lambda^{i/ 4} J$ ;\hfill\break
 - le groupo\"ide mesur\'e dual $\widehat{\cal  G}  = (N, \widehat M, \alpha, \widehat\beta, \widehat\Gamma, \widehat T, \widehat T', \nu)$ est d\'efinie de la fa\c con suivante : \hfill\break
$\bullet$ $\widehat M$\index{mb@$\widehat{M}$} est l'alg\`ebre de von Neumann engendr\'ee par les op\'erateurs $(\omega \ast {\rm id})({\cal  W}_{\cal G})$, o\`u $\omega\in B(H)_\ast$ (${\cal  W}_{\cal G}$\index{w@${\cal W}_{\cal G}$} est l'unitaire pseudo-multiplicatif \cite{Va} associ\'e au groupo\"ide ${ \cal    G}$). \hfill\break
$\bullet$ $\widehat\beta : N^{\rm o}  \rightarrow  \widehat M$\index{ba@$\widehat{\beta}$} est le *-morphisme fid\`ele et normal d\'efini par $\widehat\beta(n^{\rm o}) := J \alpha (n)^* J$, pour tout $n\in N$.\hfill\break
$\bullet$ $\widehat\Gamma  : \widehat M \rightarrow  \fprod { \widehat M } { \widehat\beta } { \alpha } { \widehat M } :  x \mapsto 
\widehat\Gamma(x) := \sigma_{\alpha\widehat{\beta}} {\cal   W}_{\cal G}(\reltens { x } { \beta } { \alpha } { 1 })  {\cal   W}_{\cal G}^* \sigma_{\widehat{\beta}\alpha}$ (o\`u $\sigma_{\alpha\widehat{\beta}}$ et $\sigma_{\widehat{\beta}\alpha}$ d\'esignent les voltes).\hfill\break
$\bullet$ On a un unique poids nsff $\widehat\varphi$ sur $\widehat M$ avec comme repr\'esentation GNS 
  $(H, {\rm id}, \Lambda_{\widehat\varphi})$, o\`u     $\Lambda_{\widehat\varphi}$  est   la fermeture de l'op\'erateur 
$(\omega \ast  {\rm id})({\cal  W}_{\cal G}) \mapsto  a_\varphi(\omega)$, o\`u $\omega$ appartient \`a  
 un sous-espace dense de  
 $$I_\varphi :=\{ \omega\in B(H)_\ast\, | \, \exists k \geqslant 0,\,\, \forall x\in \mathfrak N_\varphi,   \,\,|\omega(x^*)|^2\leqslant k \varphi(x^* x)\},$$  
le vecteur   $a_\varphi(\omega)\in H$ \'etant  d\'efini par  
$$\omega(x^*) =  \langle \Lambda x, a_\varphi(\omega)\rangle,\quad \text{pour tout }x\in \mathfrak N_\varphi.$$
$\widehat T$ est alors l'unique poids op\'eratoriel nsff de $\widehat M$ sur $\alpha(N)$ v\'erifiant 
 $\widehat\varphi= \nu \circ   \alpha^{-1} \circ   \widehat T$.
Soit $\widehat T' =  R_{\widehat{\cal  G}} \widehat T R_{\widehat{\cal  G}}$, o\`u $R_{\widehat{\cal  G}}: \widehat M  \rightarrow \widehat M$ est la coinvolution unitaire d\'efinie par $R_{\widehat{\cal  G}}(x) = J x^* J$, pour $x\in\widehat{M}$.
L'op\'erateur de Tomita du poids $\widehat\varphi$ est not\'e $\widehat J$.
\medbreak
Dans la suite, nous utiliserons pour dual du  groupo\"ide ${\cal G}$,  le groupo\"ide 
$\widehat{\cal  G} ^{\rm c} = (N^{\rm o}, \widehat M',  \beta, \widehat\alpha, \widehat\Gamma^{\rm c}, \widehat T^{\rm c}, \widehat T^{\rm c}{'}, \nu^{\rm o})$ car il est mieux adapt\'e pour les actions \`a droite du groupo\"ide ${\cal  G}$. Notons qu'on a  : \hfill\break
 $\bullet$  $\widehat\Gamma^{\rm c}(x) = ({\cal   W}_{(\widehat{\cal  G} )^{\rm c}})^* (\reltens{1}{\beta}{\alpha}{x}) { \cal   W}_{(\widehat{\cal  G} )^{\rm c}}$ pour tout $x\in \widehat M'$ ;\hfill\break
$\bullet$ $\widehat T^{\rm c}  = C_{\widehat M} \circ \widehat T \circ C_{\widehat M}^{-1}$, o\`u 
$C_{\widehat M}  : \widehat M  \rightarrow \widehat M' : x \mapsto \widehat J x^* \widehat J$ ;\hfill\break
$\bullet$   $\widehat T^{\rm c}{'}  = R_{\widehat{\cal  G} ^{\rm c}} \circ  \widehat T^{\rm c} \circ R_{\widehat{\cal  G} ^{\rm c}}$ ;\hfill\break
$\bullet$ le poids commutant  $\widehat\varphi' = \nu^{\rm o} \circ \beta^{-1} \circ \widehat T^{\rm c}$    d\'eduit du poids $\widehat\varphi$,  est invariant \`a gauche pour le coproduit $\widehat\Gamma^{\rm c}$.

\begin{notations}
Notons ${\cal W}_{\cal G}$ l'unitaire pseudo-multiplicatif associ\'e \cite{E} \`a un groupo\"ide m.q. ${\cal G}$.
Posons avec les notations de \cite{E} :\index{va@$\widehat{\cal V}$, $\cal V$, $\widetilde{\cal V}$}\index{u@$U$} 
$$\widehat{\cal V}  := {\cal W}_{\cal G} , \quad   {\cal V}  := {\cal W}_{\widehat{{\cal G}^{\rm o}}} =   {\cal W}_{(\widehat{\cal G})^{\rm c}} , \quad 
  \widetilde{\cal V}  := {\cal W}_{ ({\cal G}^{\rm o})^{\rm c}} , \quad  U:= \widehat J J.$$
\end{notations} 

\begin{lemme} Nous avons : 
\begin{enumerate}
\item
  $U^* = \lambda^{-{ i/ 4}} U$ ;
\item $\widehat{{\cal V}  } = \sigma_{\widehat\beta \alpha} (\reltens{ U }{ \beta }{ \alpha }{ 1 })   {{\cal V}  } (\reltens{U^*}{\alpha}{\beta}{1}) \sigma_{\beta \alpha}$ ;
\item $\widetilde{{\cal V}  } = \sigma_{\beta \widehat\alpha} (\reltens{ 1 }{ \beta }{ \widehat\alpha } { U })    {{\cal V}  }  (\reltens{ 1 }{ \widehat\alpha }{ \beta }{ U^* }) \sigma_{\widehat\beta \widehat\alpha}$ ;
\end{enumerate}
o\`u $\sigma_{\beta \alpha} : \reltens{H}{\beta}{\alpha}{H}  \rightarrow  \reltens{H}{\alpha}{\beta}{H}$, ainsi que $\sigma_{\widehat\beta \alpha}$, $\sigma_{\widehat\beta \widehat\alpha}$ et $\sigma_{\beta \widehat\alpha}$\index{sb@$\sigma_{\widehat\beta \alpha}$, $\sigma_{\widehat\beta \widehat\alpha}$, $\sigma_{\beta \widehat\alpha}$} d\'esignent les voltes, et les repr\'esentations $\widehat\alpha$ et $\widehat\beta$ sont donn\'ees par $\widehat\alpha := {\rm Ad}\,U \circ \alpha$\index{ac@$\widehat{\alpha}$} et $\widehat\beta := {\rm Ad}\,U \circ \beta$.
\begin{proof} 
Le a) r\'esulte de \cite{E} 3.11 (iv). Rappelons que $\lambda$  est central dans $M$ et $\widehat M$.

Le b) (resp. le c)) r\'esulte de \cite{E} 3.12 (v) (resp. 3.12 (vi)).
\end{proof}
\end{lemme}

\begin{lemme} 
Soient ${\cal G} = (N, M,  \alpha, \beta, \Gamma, T, T', \nu)$    un groupo\"ide \mq et  $\widehat{\cal G} = (N, M,  \alpha, \widehat\beta, \widehat\Gamma, \widehat T, \widehat R \circ \widehat T  \circ \widehat R, \nu)$    le  groupo\"ide dual \cite{E} de ${\cal G}$.  
\begin{enumerate}
\item  Pour tout $x\in M'$, on a :
$$\widehat{{\cal V}  }(\reltens { x }{ \beta }{ \alpha }{ 1 }) =(\reltens{ x }{ \alpha }{ \widehat\beta }{ 1 })\widehat{{\cal V}  } , \quad  {{\cal V}  } (\reltens{ 1 }{ \widehat\alpha }{ \beta }{ x }) = (\reltens{ 1 }{ \beta }{ \alpha }{ x }) {{\cal V}  }.$$
\item Pour tout $x\in \widehat M$, on a :
$${{\cal V}  } (\reltens{x}{\widehat\alpha}{\beta}{1}) = (\reltens{x}{\beta}{\alpha}{1}) {{\cal V}  }, \quad 
\widetilde{{\cal V}  } (\reltens{1}{\widehat\beta}{\widehat\alpha}{x}) = (\reltens{1}{\widehat\alpha}{\beta}{x}) \widetilde{{\cal V}  }.$$
\item Pour tout $x\in M$, on a :
$${{\cal V}  } (\reltens{x}{\widehat\alpha}{\beta}{1}){{\cal V}  }^* = \widehat{{\cal V}  }^*(\reltens{1}{\alpha}{\widehat\beta}{x}) \widehat{{\cal V}  } , \quad \widetilde{{\cal V}  }  (\reltens{x}{\widehat\beta}{\widehat\alpha}{1}) = (\reltens{x}{\widehat\alpha}{\beta}{1})\widetilde{{\cal V}  }.$$
\item  Pour tout $x\in \widehat M'$, on a :
$${{\cal V}  }^* (\reltens{1}{\beta}{\alpha}{x}){{\cal V}  }  = \widetilde{{\cal V}  }(\reltens{x}{\widehat\beta}{\widehat\alpha}{1}) \widetilde{{\cal V}  }^* , \quad 
\widehat{{\cal V}  } (\reltens{1}{\beta}{\alpha}{x}) = (\reltens{1}{\alpha}{\widehat\beta}{x}) \widehat{{\cal V}  }.$$
\end{enumerate}
\begin{proof}
Les op\'erateurs apparaissant dans l'\'enonc\'e  ont un sens gr\^ace aux relations \cite{E} :
\begin{equation}\label{Eq:commut}
M \cap \widehat M = \alpha(N),\quad 
M \cap \widehat M' = \beta(N^{\rm o}),\quad M' \cap \widehat M = \widehat\beta(N^{\rm o}),\quad M' \cap \widehat M' = \widehat\alpha(N).
\end{equation}
\noindent
La preuve repose uniquement sur les r\'esultats 3.8, 3.11 et 3.12 de  \cite{E}. Par exemple,  
 la premi\`ere  relation du a) r\'esulte du fait   que $M$ est l'adh\'erence faible de l'alg\`ebre $\{   
({\rm id} \ast \omega)(\widehat{\cal V}) \,\,|\,\,\omega\in B(H)_\ast\}$, \cf 3.8. La deuxi\`eme relation r\'esulte alors de la   premi\`ere et de la formule $\widehat{{\cal V}  } = \sigma_{\widehat\beta \alpha} (\reltens{U}{\beta}{\alpha}{1})     {{\cal V}  }  (\reltens{U^*}{\alpha}{\beta}{1}) \sigma_{\beta \alpha}$.
\hfill\break
La premi\`ere  relation du c) (resp.  du d)) est exactement la premi\`ere (resp. derni\`ere) relation de 3.12 (v)  (resp. 3.12 (vi)).
\end{proof}
\end{lemme}
\begin{corollary} \label{2.4} Nous avons :
\begin{enumerate}[label=\arabic*)]
\item ${{\cal V}  }^* (\reltens{UxU^*}{\beta}{\alpha}{1}) {{\cal V}  } =  (\reltens{U}{\alpha}{\beta}{1}) \sigma_{\beta \alpha} {{\cal V}  } (\reltens{x}{\widehat\alpha}{\beta}{1}) {{\cal V}  }^*\sigma_{\alpha \beta} (\reltens{U^*}{\widehat\alpha}{\beta}{1})$, \quad $x\in M$ ;
\item ${{\cal V}  } (\reltens{1}{\widehat\alpha}{\beta}{UyU^*}) {{\cal V}  }^*  =  (\reltens{1}{\beta} {\alpha}{U}) \sigma_{\widehat\alpha \beta}  {{\cal V}  }^* (\reltens{1}{\beta}{\alpha}{y}) {{\cal V}  }\sigma_{\beta \widehat\alpha} (\reltens{1}{\beta}{\widehat\alpha}{U^*})$, \quad $y\in \widehat M'$ ;
\item $\widetilde{{\cal V}  } (\reltens{1}{\widehat\beta}{\widehat\alpha}{UyU^*}) \widetilde{{\cal V}  }^* =   (\reltens{1}{\widehat\alpha}{\beta}{U}) \sigma_{\widehat\beta \widehat\alpha}  \widetilde{{\cal V}  }^* (\reltens{1}{\widehat\alpha}{\beta}{y}) \widetilde{{\cal V}  }\sigma_{\widehat\alpha \widehat\beta} (\reltens{1}{\widehat\alpha}{\widehat\beta}{U^*})$, \quad $y\in  M'$ ;
\item $\widehat{{\cal V}  }_{12} {{\cal V}  }_{23} ={{\cal V}  }_{23} \widehat{{\cal V}  }_{12}$,  \quad   ${{\cal V}  }_{12} \widetilde{{\cal V}  }_{23} =\widetilde{{\cal V}  }_{23}{{\cal V}  }_{12}$.
\end{enumerate}
\end{corollary}

\begin{proof}
Le 1) (resp. le 2)) est une  cons\'equence de la prem\`ere relation du c) (resp. du d). Le 3) s'obtient \`a  partir du 2) en rempla\c cant le groupo\"ide ${\cal   G}$ par le groupo\"ide $(\widehat{\cal  G} )^{\rm c}$. La premi\`ere (resp. la deuxi\`eme) \'egalit\'e du 4) est \'equivalente \`a la premi\`ere  relation du b) (resp. deuxi\`eme du c).
\end{proof}

Nous avons :

\begin{proposition}\label{irr1}
$(\reltens{1}{\widehat\beta}{\widehat\alpha}{U}) \sigma_{\alpha \widehat\beta} \widehat{{\cal V}  } {{\cal V}  } \widetilde{{\cal V}  } \in
\fprod { \widehat\beta(N^{\rm o}) } { \widehat\beta } { \widehat\alpha } {\widehat\alpha(N)}$.
\end{proposition}

\begin{proof} 
Posons $T = (\reltens{1}{\widehat\beta}{\widehat\alpha}{U}) \sigma_{\alpha \widehat\beta} \widehat{{\cal V}  } {{\cal V}  } \widetilde{{\cal V}  }$ et remarquons que $T\in B(\reltens{H}{\widehat\beta}{\widehat\alpha}{H})$ est un unitaire. Utilisant la d\'efinition \cite{E} du produit fibr\'e $\fprod { \widehat\beta(N^{\rm o}) } { \widehat\beta } { \widehat\alpha } { \widehat\alpha(N) }$, il revient au m\^eme de montrer que pour tout $x\in \widehat\beta(N^{\rm o})'$ et tout $y\in\widehat\alpha(N)'$, on a $[T, \reltens{x}{\widehat\beta}{\widehat\alpha}{y}] = 0$.

Montrons qu'on a $[T, \reltens{x}{\widehat\beta}{\widehat\alpha}{1}] = 0$ pour tout $x\in 
\widehat\beta(N^{\rm o})'$.
\hfill\break
 - Si $x\in M$, on a  
\begin{align} 
T(\reltens{x}{\widehat\beta}{\widehat\alpha}{1}) &=  (\reltens{1}{\widehat\beta}{\widehat\alpha}{U}) \sigma_{\alpha \widehat\beta}  \widehat{{\cal V}  }{{\cal V}  }(\reltens{x}{\widehat\alpha}{\beta}{1}) \widetilde{{\cal V}  } & (2.3 \, c)) \nonumber\\
&= (\reltens{1}{\widehat\beta}{\widehat\alpha}{U}) \sigma_{\alpha \widehat\beta}  \widehat{{\cal V}  }{{\cal V}  }(\reltens{x}{\widehat\alpha}{\beta}{1}) {{\cal V}  }^*{{\cal V}  }\widetilde{{\cal V}  }  &  \nonumber   \\
 & = (\reltens{1}{\widehat\beta}{\widehat\alpha}{U}) \sigma_{\alpha \widehat\beta}  \widehat{{\cal V}  }\widehat{{\cal V}  }^*(\reltens{1}{\alpha}{\widehat\beta}{x}) \widehat{{\cal V}  }{{\cal V}  }\widetilde{{\cal V}  } & (2.3 \, c)) \nonumber\\
& = (\reltens{1}{\widehat\beta}{\widehat\alpha}{U}) \sigma_{\alpha \widehat\beta}  (\reltens{1}{\alpha} {\widehat\beta}{x}) \widehat{{\cal V}  }{{\cal V}  }\widetilde{{\cal V}  } = (\reltens{x}{\widehat\beta}{\widehat\alpha}{1}) T.  & {}\nonumber 
\intertext{- Si $x\in \widehat M'$, on a}
T(\reltens{x}{\widehat\beta}{\widehat\alpha}{1}) &=  (\reltens{1}{\widehat\beta}{\widehat\alpha}{U}) \sigma_{\alpha \widehat\beta} \widehat{{\cal V}  }{{\cal V}  }\widetilde{{\cal V}  } (\reltens{x}{\widehat\beta}{\widehat\alpha}{1})\widetilde{{\cal V}  }^*\widetilde{{\cal V}  } & \nonumber\\
&=    (\reltens{1}{\widehat\beta}{\widehat\alpha}{U}) \sigma_{\alpha \widehat\beta}  \widehat{{\cal V}  }{{\cal V}  } 
{{\cal V}  }^* (\reltens{1}{\beta}{\alpha}{x}){{\cal V}  } \widetilde{{\cal V}  } & (2.3 \,d))\nonumber \\ 
&= (\reltens{1}{\widehat\beta}{\widehat\alpha}{U}) \sigma_{\alpha \widehat\beta} (\reltens{1}{\alpha}{\widehat\beta}{x})
\widehat{{\cal V}  }{{\cal V}  }\widetilde{{\cal V}  } 
= (\reltens{x}{\widehat\beta}{\widehat\alpha}{1}) T. & {} \nonumber
 \nonumber 
\end{align}
On a donc montr\'e que pour tout $x\in \widehat\beta(N^{\rm o})'$, on a $[T, \reltens{x}{\widehat\beta}{\widehat\alpha}{1}] = 0$. De fa\c con similaire,  montrons  que  pour tout $y\in \widehat\alpha(N)'$, on a $[T, \reltens{1}{\widehat\beta}{\widehat\alpha}{y}] = 0$.

- Si $y\in \widehat M$, on a
\begin{align} 
T(\reltens{1}{\widehat\beta}{\widehat\alpha}{y})  &= (\reltens{1}{\widehat\beta}{\widehat\alpha}{U}) \sigma_{\alpha \widehat\beta} \widehat{{\cal V}  }{{\cal V}  }(\reltens{1}{\widehat\alpha}{\beta}{y}) \widetilde{{\cal V}  } =
(\reltens{1}{\widehat\beta}{\widehat\alpha}{U}) \sigma_{\alpha \widehat\beta} \widehat{{\cal V}  }{{\cal V}  }(\reltens{1}{\widehat\alpha}{\beta}{y}){{\cal V}  }^*{{\cal V}  } \widetilde{{\cal V}  }  &  {(2.3 \,\,b))} \nonumber\\ 
&=(\reltens{1}{\widehat\beta}{\widehat\alpha}{U}) \sigma_{\alpha \widehat\beta} \widehat{{\cal V}  }(\reltens{1}{\beta}{\alpha}{U}) \sigma_{\widehat\alpha \beta}  {{\cal V}  }^* (\reltens{1}{\beta}{\alpha}{UyU^*})  {{\cal V}  }\sigma_{\beta \widehat\alpha} (\reltens{1}{\beta}{\widehat\alpha}{U^*}){{\cal V}  } \widetilde{{\cal V}  } &  {(\ref{2.4}\, 2))} \nonumber\\ 
&= (\reltens{U}{\beta}{\widehat\alpha}{U}) 
 (\reltens{1}{\beta}{\alpha}{U^*yU})  {{\cal V}  }\sigma_{\beta \widehat\alpha} (\reltens{1}{\beta} {\widehat\alpha}{U^*}){{\cal V}  } \widetilde{{\cal V}  } &  {(2.2\,\, b))}\nonumber\\
&=(\reltens{1}{\widehat\beta}{\widehat\alpha}{y}) (\reltens{1}{\widehat\beta}{\widehat\alpha}{U}) (\reltens{U}{\beta}{\alpha}{1})
{{\cal V}  }\sigma_{\beta \widehat\alpha} (\reltens{1}{\beta}{\widehat\alpha}{U^*}){{\cal V}  } \widetilde{{\cal V}  } &  {(2.2 \,\,b))}\nonumber\\
 &=(\reltens{1}{\widehat\beta}{\widehat\alpha}{y}) (\reltens{1}{\widehat\beta}{\widehat\alpha}{U}) \sigma_{\alpha \widehat\beta} \sigma_{\widehat\beta \alpha}
(\reltens{U}{\beta}{\alpha}{1})
{{\cal V}  }\sigma_{\beta \widehat\alpha} (\reltens{1}{\beta}{\widehat\alpha}{U^*}){{\cal V}  } \widetilde{{\cal V}  }  = (\reltens{1}{\widehat\beta}{\widehat\alpha}{y})  T. \nonumber &  {(2.2 \,\,b))} 
\intertext{- Si $y\in M$, on a}
T(\reltens{1}{\widehat\beta}{\widehat\alpha}{y}) &= 
 (\reltens{1}{\widehat\beta}{\widehat\alpha}{U}) \sigma_{\alpha \widehat\beta} \widehat{{\cal V}  }{{\cal V}  } 
(\reltens{1}{\widehat\alpha}{\beta}{U}) \sigma_{\widehat\beta \widehat\alpha}  \widetilde{{\cal V}  }^* (\reltens{1}{\widehat\alpha}{\beta}{UyU^*}) \widetilde{{\cal V}  }\sigma_{\widehat\alpha \widehat\beta} (\reltens{1}{\widehat\alpha}{\widehat\beta}{U^*})
\widetilde{{\cal V}  } & (\ref{2.4}\,3)) \nonumber\\
&=(\reltens{1}{\widehat\beta}{\widehat\alpha}{U}) \sigma_{\alpha \widehat\beta} \widehat{{\cal V}  }
(\reltens{1}{\beta}{\alpha}{U})\sigma_{\widehat\alpha  \beta}
(\reltens{1}{\widehat\alpha}{\beta}{UyU^*}) \widetilde{{\cal V}  }\sigma_{\widehat\alpha \widehat\beta} (\reltens{1}{\widehat\alpha}{\widehat\beta}{U^*})
\widetilde{{\cal V}  } & (2.2\,c)) \nonumber\\
&=
(\reltens{1}{\widehat\beta}{\widehat\alpha}{U}) \sigma_{\alpha \widehat\beta} \widehat{{\cal V}  }(\reltens{UyU^*}{\beta}{\alpha}{U})
\sigma_{ \widehat\alpha \beta}
\widetilde{{\cal V}  }\sigma_{\widehat\alpha \widehat\beta} (\reltens{1}{\widehat\alpha}{\widehat\beta}{U^*})
\widetilde{{\cal V}  } &  {}\nonumber\\
&=(\reltens{1}{\widehat\beta}{\widehat\alpha}{U}) (\reltens{1}{\widehat\beta}{\alpha}{U^*yU}) \sigma_{\alpha \widehat\beta} \widehat{{\cal V}  } (\reltens{1}{\beta}{\alpha}{U})
\sigma_{ \widehat\alpha \beta}
\widetilde{{\cal V}  }\sigma_{\widehat\alpha \widehat\beta} (\reltens{1}{\widehat\alpha}{\widehat\beta} {U^*})
\widetilde{{\cal V}  } & (2.3\,a)) \nonumber\\
&=
(\reltens{1}{\widehat\beta}{\widehat\alpha}{y})  (\reltens{1}{\widehat\beta}{\widehat\alpha}{U}) \sigma_{\alpha \widehat\beta}
\widehat{{\cal V}  } (\reltens{1}{\beta}{\alpha}{U})
\sigma_{ \widehat\alpha \beta}
\widetilde{{\cal V}  }\sigma_{\widehat\alpha \widehat\beta} (\reltens{1}{\widehat\alpha}{\widehat\beta}{U^*})
\widetilde{{\cal V}  }& {}\nonumber\\
& =
(\reltens{1}{\widehat\beta}{\widehat\alpha}{y})  (\reltens{1}{\widehat\beta}{\widehat\alpha}{U}) \sigma_{\alpha \widehat\beta}
\widehat{{\cal V}  } {{\cal V}  }
\widetilde{{\cal V}  } & (2.2\,c))\nonumber \\ 
&=(\reltens{1}{\widehat\beta}{\widehat\alpha}{y}) T. & {}
\nonumber 
\end{align}
\end{proof}

\subsection{Cas o\`u la base  est de dimension finie}\label{2.2}

\noindent
Si la base d'un groupo\"ide   \mq   est de dimension finie, De Commer \cite{DeC1,DeC2} a donn\'e une d\'efinition \'equivalente plus simple d'un tel groupo\"ide. En se basant sur cette d\'efinition que nous rappelons,  nous am\'eliorons le r\'esultat d'irr\'eductibilt\'e \ref{irr1}  et nous g\'en\'eralisons aussi \`a ce cadre, quelques r\'esultats utiles du cas des groupes quantiques localement compacts.  

\noindent
Fixons dans ce qui suit, $\displaystyle{ N = \oplus_{l=1}^k M_{n_l}}$ une C*-alg\`ebre de dimension finie, munie de la trace de Markov non normalis\'ee $\displaystyle \epsilon =\oplus_{l=1}^k n_l {\rm Tr}_{M_{n_l}}$\index{ea@$\epsilon$}.
\smallbreak
Le r\'esultat suivant sera  utilis\'e plusieurs fois :

\begin{lemme}\label{q}  Soient  $A$ et $B$ deux C*-alg\`ebres et soient $\alpha : N  \rightarrow M(A)$ et $\beta : N^{\rm o} \rightarrow M(B)$ deux *-morphismes non d\'eg\'en\'er\'es. Alors il existe un unique projecteur $p\in M(B \otimes A)$ tel que pour tout syst\`eme d'unit\'es matricielles (s.u.m.)\index{sc@s.u.m.} $(e_{ij}^{(l)})_{i,j = 1, \cdots, n_l,\, l = 1, \cdots,k}$\index{eb@$e_{ij}^{(l)}$} de $N$, on ait 
$$p = \sum_{l=1}^k {\frac{1}{n_l}} \sum_{i,j =1}^{n_l} \beta(e_{ji}^{(l){\rm o}}) \otimes \alpha(e_{ij}^{(l)}).$$
\noindent
Dans toute la suite, le projecteur $p$ sera not\'e $q_{\beta, \alpha}$\index{qa@$q_{\beta, \alpha}$}.
\begin{proof}
 On peut supposer que $B$ (resp. $A$) est une C*-alg\`ebre non d\'eg\'en\'er\'ee de $B(K)$ (resp. $B(H)$)), avec $K, H$ des espaces de Hilbert.\hfill\break
Notons ${\cal E}$ le C*-module sur $N$ canoniquement associ\'e (\ref{pr})  au   morphisme $\beta : N^{\rm o} \rightarrow  B(K)$. Soit $v$ la co-isom\'etrie  d\'efinie par :
$$v  : K \otimes H  \rightarrow {\cal  E} \otimes_{\alpha} H : \xi \otimes \eta  \mapsto 
\xi \otimes_{\alpha} \eta.$$
\noindent
On v\'erifie que $p = v^* v$ convient (\cf \cite{DeC1}).
\end{proof}
\end{lemme}  

\noindent
La d\'efinition \'equivalente d'un groupo\"ide mesur\'e quantique  de base $N$, donn\'ee par De Commer \cite{DeC1, DeC2},  est alors la suivante : 
\begin{definition}\label{m.qfini}
Un groupo\"ide  mesur\'e quantique (m.q.) de base $N$,   est un octuplet    
$$
{\cal G} = (N, M,  \alpha, \beta, \delta, T, T', \epsilon),
$$
o\`u  :\hfill\break
 -  $M$ est   une   alg\`ebre de von Neumann,    $\alpha : N  \rightarrow M $ et $\beta : N^{\rm o} \rightarrow M$ sont    des morphismes  normaux et fid\`eles ;\hfill\break
 - le coproduit $\delta : M \rightarrow M \otimes M$ est un morphisme   normal et fid\`ele ;\hfill\break
 - $T : M^+ \rightarrow \alpha(N)^{+,{\rm ext}}$  et $T' : M^+ \rightarrow \beta(N^{\rm o})^{+,{\rm ext}}$  sont des   poids op\'eratoriels  normaux, semi-finis  et fid\`eles.
\smallbreak
Ces objets sont assujettis aux conditions suivantes : 
\begin{enumerate}[label=\arabic*)]
\item Les images des morphismes $\alpha$ et $\beta$ commutent ;
\item $\delta(1) = q_{\beta, \alpha}$ et $\delta$ est coassociatif, \ie $(\delta \otimes {\rm id}_M) \delta = ({\rm id}_M \otimes \delta) \delta$ ; 
\item Pour tout $n\in N$, on a $\delta(\alpha(n))  = \delta(1) (\alpha(n)\otimes 1)$ et $\delta(\beta(n^{\rm o}))  = \delta(1) (1 \otimes \beta(n^{\rm o}))$ ;
\item Les poids  nsff $\varphi$ et $\psi$   sur $M$ d\'efinis respectivement par $\varphi =  \epsilon \circ  \alpha^{-1} \circ T$ et $\psi = 
\epsilon \circ \beta^{-1} \circ T'$ v\'erifient 
$$T = ({\rm id}_M \otimes \varphi) \delta \quad {\rm et} \quad T' =(\psi \otimes {\rm id}_M)\delta \text{ ;}$$ \item Pour tout $t\in\R$, on a 
$\sigma_t^T \circ \beta = \beta$ et $\sigma_t^{T'} \circ \alpha= \alpha$.
\end{enumerate}
\end{definition}

\begin{notations}\strut
\begin{enumerate}
\item
  $V$  d\'esigne  l'image par l'injection canonique $\iota_{\widehat\alpha \alpha}^\beta  : 
B(\reltens{H}{\widehat\alpha}{\beta}{H},\reltens{H}{\beta}{\alpha}{H})  \rightarrow  B(H \otimes H)$  de l'unitaire pseudo-multiplicatif \cite{E}
${\cal V}  := {\cal W}_{\widehat{{\cal G}^{\rm o}}} =   {\cal W}_{(\widehat{\cal G})^{\rm c}}$.
\item
 $W$ d\'esigne  l'image par l'injection canonique  
$\iota_{\beta  \widehat\beta} ^\alpha  : 
B(\reltens{H}{\beta}{\alpha}{H}, \reltens{H}{\alpha}{\widehat\beta}{H})  \rightarrow  B(H \otimes H)$ de l'unitaire pseudo-multiplicatif \cite{E} $\widehat{\cal V}  := {\cal W}_{\cal G}$. 
\item $\widetilde V$ d\'esigne l'image par l'injection canonique 
$\iota_{\widehat\beta  \beta}^{\widehat\alpha} : 
B( \reltens{H}{\widehat\beta}{\widehat\alpha}{H}, \reltens{H}{\widehat\alpha}{\beta}{H})  \rightarrow  B(H \otimes H)$ 
de l'unitaire pseudo-multiplicatif \cite{E}  $\widetilde{\cal V}  := {\cal W}_{ ({\cal G}^{\rm o})^{\rm c}}$.\index{vb@$V$, $W$, $\widetilde{V}$}
\end{enumerate}
\end{notations}
Dans ce qui suit, nous r\'ecapitulons les principales propri\'etes v\'erifi\'ees par $W, V$ et $\widetilde V$ qui seront utilis\'ees dans la suite. La preuve de ces r\'esultats d\'ecoule  des propri\'et\'es  \cite{E} des unitaires pseudo-multiplicatifs dont ils sont les images (\cf \cite{E, DeC2}).
\begin{proposition}\label{2.9} Les op\'erateurs $V , W , \widetilde{V}$ sont des isom\'etries partielles pentagonales agissant dans $H \otimes H$ et v\'erifient :
\begin{enumerate}
%$$V(\Lambda_\psi x \otimes \Lambda_\psi y) = (\Lambda_\psi   \otimes \Lambda_\psi ) \delta(x)(1 \otimes y) \,|\, x, y \in{\cal  N}_\psi
% \quad | \quad W^*(\Lambda  x \otimes \Lambda  y) = (\Lambda    \otimes \Lambda) \delta(y)(x \otimes 1)\,;\, x, y \in{\cal  N}_\varphi$$
%\noindent
\item $W = \Sigma (U \otimes 1) V (U^* \otimes 1)\Sigma$, \quad $\widetilde V = \Sigma (1 \otimes U) V (1 \otimes U^*)\Sigma$,\quad avec\quad  $U:= \widehat J J$ \quad ;
\item $V^* = (J \otimes \widehat J) V (J \otimes \widehat J), \quad  W^* = (\widehat J \otimes J) W (\widehat J \otimes J)$ \quad ;
\item
le support initial et le support final de chacune des isom\'etries partielles pentagonales $W , V , \widetilde V$ sont donn\'es par les formules :
$$V^* V = \widetilde V \widetilde V^* = q_{\widehat\alpha, \beta}, \quad W^*W  = V V^* = q_{\beta, \alpha} , \quad W W^* = q_{\alpha, \widehat\beta} , \quad \widetilde V^*\widetilde V  = q_{\widehat\beta, \widehat\alpha}.$$
\end{enumerate}
\end{proposition}

\begin{proposition}\label{2.10}
\begin{enumerate}
\item $W_{12} V_{23} =  V_{23} W_{12}  , \quad V_{12} \widetilde V_{23} =  \widetilde V_{23} V_{12}$.
\item
Pour tout $n\in N$, on a :
$$[V, \alpha(n) \otimes 1] =       [V, \widehat \beta(n^{\rm o}) \otimes 1] =    [V, 1 \otimes \widehat\alpha(n)] =  [V,  1 \otimes \widehat \beta(n^{\rm o})] = 0,$$
$$V(1 \otimes \alpha(n)) =  (\widehat\alpha(n) \otimes 1) V  \,,\,  V(\beta(n^{\rm o}) \otimes 1) =  (1 \otimes \beta(n^{\rm o}))V \quad ;$$
$$[W, \widehat \beta(n^{\rm o}) \otimes 1] =  [W, \widehat\alpha(n) \otimes 1] =  [W, 1 \otimes \beta(n^{\rm o})]=  [W, 1 \otimes \widehat\alpha(n)] = 0,$$
$$W(1 \otimes \widehat \beta(n^{\rm o})) =  (\beta(n^{\rm o}) \otimes 1) W  \,,\,  W(\alpha(n) \otimes 1) =  (1 \otimes \alpha(n))W \quad ;$$
$$[\widetilde V, \alpha(n) \otimes 1] =  [\widetilde V, \beta(n^{\rm o}) \otimes 1] =  [\widetilde V, 1 \otimes \alpha(n)] =  [\widetilde V, 1 \otimes \widehat \beta(n^{\rm o})] = 0,$$
$$\widetilde V(1 \otimes \beta(n^{\rm o})) =  (\widehat \beta(n^{\rm o}) \otimes 1) \widetilde V  \,,\,  \widetilde V(\widehat\alpha(n) \otimes 1) =  (1 \otimes \widehat\alpha(n))\widetilde V.$$
\item
Le coproduit  de l'alg\`ebre de von Neumann $M$ v\'erifie :  
$$\delta(x) = V (x \otimes 1_H) V^* = W^* (1_H \otimes x) W  , \quad x \in M.$$
\item Le coproduit  de l'alg\`ebre de von Neumann $\widehat M'$, qu'on note $\widehat\delta$,  v\'erifie : 
$$\widehat\delta(x) = V^* (1_H \otimes x) V = \widetilde V (x \otimes 1_H) \widetilde V  , \quad x \in\widehat M'.$$
\end{enumerate}
\end{proposition}

\begin{notation}\label{RemNot}
Contrairement \`a \cite{E}, on adopte pour    l'alg\`ebre de von Neumann $\widehat M$, le  coproduit d\'efini par :\index{da@$\widehat\delta_\lambda$}
$$\widehat\delta_\lambda(x)  = W (x \otimes 1_H) W^*  , \quad x\in  \widehat M.$$
{\bf  Dans toute la suite, on note $\widehat{\cal G}$ le groupo\"ide dual 
$(N^{\rm o}, \widehat M',  \beta, \widehat\alpha, \widehat\delta, \widehat T, \widehat {T'}, \epsilon)$} avec 
$\widehat T$ le poids op\'eratoriel d\'efini  par $\widehat\varphi' = \epsilon \circ \beta^{-1} \circ \widehat T$, $\widehat {T'} = R_{\widehat M'} \circ \widehat T \circ R_{\widehat M'}$ et $R_{\widehat M'} : \widehat M' \rightarrow   \widehat M'  : x \mapsto J x^* J$.
 \end{notation}
 
\subsection{C*-alg\`ebres de Hopf  faibles associ\'ees \`a un groupo\"ide mesur\'e quantique\label{Hopff} \texorpdfstring{\cite{DeC1, DeC2}}{}}

\noindent
Nous rappelons (avec des notations et des conventions diff\'erentes),  les d\'efinitions des C*-alg\`ebres de Hopf faibles associ\'ees par De Commer \cite{DeC1, DeC2} \`a un groupo\"ide \mq ${\cal G}$ de base $N$  de dimension finie, ainsi  que les principaux r\'esultats de \cite{DeC1, DeC2}. 

Avec les notations du paragraphe \ref{2.2}, 
posons :\index{sa@$S$, $\widehat{S}$}
$$S = [(\omega \otimes {\rm id} )(V) \,\, | \,\, \omega \in B(H)_\ast] , \quad \widehat S = [({\rm id}  \otimes \omega)(V) \,\, | \,\, \omega \in B(H)_\ast].$$
\noindent
Nous avons :
\hfill\break
 $\bullet$  $S$ et $\widehat S$ sont des sous-C*-alg\`ebres non d\'eg\'en\'er\'ees de $B(H)$, faiblement denses dans $M$ et $\widehat M'$ respectivement ; 
 
$\bullet$ On munit $S$   et  $\widehat S$  des  repr\'esentations fid\`eles et non d\'eg\'en\'er\'ees  :
\hfill\break
 -  $L  : S \rightarrow B(H) : x \mapsto x  , \quad R  : S \rightarrow B(H) : x \mapsto U x U^*$,
\hfill\break
 -  $\rho   : \widehat S \rightarrow B(H) : x \mapsto x  , \quad \lambda  : \widehat S \rightarrow B(H) : x \mapsto U x U^*.$\index{lb@$L$, $R$, $\rho$, $\lambda$}
\hfill\break
On rappelle que les *-automorphismes  ${\rm Ad}\,U$ et ${\rm Ad}\,U^*$ co\"incident sur $M$ et $\widehat M$ et on a  (\cf\cite{DeC1,DeC2}) :
\begin{equation}\label{mult} 
V\in M( \widehat S \otimes S)  , \quad    W \in M( S \otimes \lambda(\widehat S))  , \quad  \widetilde V \in M(R(S) \otimes \widehat S) \quad ;
\end{equation}
$\bullet$ $[S , R(S)] = [\widehat S , \lambda(\widehat S)] = 0$ ;

$\bullet$  Les *-morphismes 
$$\delta : S \rightarrow M(S \otimes S) :  x \mapsto V(x \otimes 1_H) V^* \quad \text{et} \quad \widehat\delta : \widehat S  \rightarrow M(\widehat S \otimes \widehat S) : x \mapsto V^*(1_H \otimes x) V$$
\noindent
se prolongent  de fa\c con unique \`a $M(S)$ et   $M(\widehat S)$ respectivement en des  *-morphismes
$\delta : M(S) \rightarrow M(S \otimes S)$  et $\widehat\delta : M(\widehat S) \rightarrow M(\widehat S \otimes \widehat S)$, strictement continus et v\'erifiant 
$$\delta(1_S) = q_{\beta, \alpha} , \quad \widehat\delta(1_{\widehat S}) = q_{\widehat \alpha, \beta} \quad ;$$
\noindent
$\bullet$ 
$\delta$ et $\widehat\delta$ sont coassociatifs et v\'erifient 
\begin{equation}\label{simp}
[\delta(S)(1_S \otimes S)] = [\delta(S)(S \otimes 1_S)] = \delta(1_S) (S \otimes S) , \quad [\delta(\widehat S)(1_{\widehat S} \otimes \widehat S)] = [\widehat\delta(\widehat S)(\widehat S \otimes 1_{\widehat S})] = \widehat\delta(1_{\widehat S}) (\widehat S \otimes \widehat S) \, ;
\end{equation}
\noindent
$\bullet$
Les *-morphismes $\alpha : N \rightarrow M(S)$ et $\beta : N^{\rm o} \rightarrow M(S)$ v\'erifient 
$$\delta(\alpha(n))  = \delta(1_S)(\alpha(n) \otimes 1_S) , \quad \delta(\beta(n^{\rm o}))  = \delta(1_S)(1_S \otimes \beta(n^{\rm o}))  , \quad n\in N\quad ;$$
$\bullet$
Les *-morphismes $\beta : N^{\rm o} \rightarrow M(\widehat S)$ et $\widehat \alpha  : N  \rightarrow M(\widehat S)$ v\'erifient 
$$\widehat\delta(\beta(n^{\rm o}))  = \delta(1_{\widehat S})(\beta(n^{\rm o}) \otimes 1_S)  , \quad \widehat\delta(\widehat \alpha(n))  = \delta(1_{\widehat S})(1_S \otimes \widehat \alpha(n ))  , \quad n\in N.$$

\noindent
\begin{notations} On d\'esigne par $S \widehat S$ l'espace vectoriel {\it ferm\'e} engendr\'e par les produits $x y$, o\` u $x\in S$ et $y\in \widehat S$. On a donc $S \widehat S =[(\omega_1 \otimes {\rm id} \otimes  \omega_2)(V_{12} V_{23}) \,\,|\,\,\omega_1, \omega_2\in B(H)_\ast]$. 

Posons aussi ${\cal C}(V) := [({\rm id}  \otimes \omega)(\Sigma V) \,\,|\,\,\omega\in B(H)_\ast]$\index{c@$\cC(-)$}.  
Il est facile de voir que ${\cal C}(V)$ est une alg\`ebre, la preuve est identique \`a celle de (\cite{BaSka2} 3.2).
 
On note $(e^{(l)})_ {l =1, \cdots,  k}$, les projecteurs centraux minimaux de la C*-alg\`ebre $N=\oplus_{l=1}^k M_{n_l}$.
\end{notations}
\noindent
Le principal r\'esultat de ce paragraphe est  :

\begin{theorem}\label{th13}Il existe des scalaires $\lambda_l\in\C $ tels que $|\lambda_l|= 1$, pour $1\leqslant l \leqslant k$, uniques v\'erifiant 
$$(1 \otimes U) \Sigma W V \widetilde V = \sum_{l=1}^k \lambda_l (\widehat\beta(e^{(l){\rm o}}) \otimes \widehat\alpha(e^{(l)}))q_{\widehat\beta , \widehat\alpha}.$$
\end{theorem}

Pour la preuve du th\'eor\`eme, on a besoin du  lemme suivant :

\begin{lemme}\label{lemth13} Soient $\alpha : N \rightarrow B(H)$, $\beta  : N^{\rm o} \rightarrow B(H)$ deux *-morphismes unitals. 
\begin{enumerate}
\item $q_{\beta, \alpha} (\beta(N^{\rm o})  \otimes \alpha(N)) q_{\beta, \alpha} = \bigoplus_{l=1}^k \C\cdot(\beta(e^{(l){\rm o}}) \otimes \alpha(e^{(l)})) q_{\beta, \alpha}$.
\item$q_{\alpha, \beta} (\alpha(N)  \otimes \beta(N^{\rm o})) q_{\alpha, \beta} = \bigoplus_{l=1}^k \C \cdot(\alpha(e^{(l)}) \otimes \beta(e^{(l){\rm o}})) q_{\alpha, \beta}$.
\end{enumerate}
\end{lemme}

\begin{proof}  Montrons le a), la preuve du b) est similaire. 
\hfill\break
Soient $x, y\in N$, on a :
$$q_{\beta, \alpha} (\beta(x^{\rm o})  \otimes \alpha(y)) q_{\beta, \alpha} = 
\sum_{l=1}^k {\frac{1}{n_l^2}}\,\epsilon(yxe^{(l)}) (\beta(e^{(l){\rm o}}) \otimes \alpha(e^{(l)}))  q_{\beta, \alpha}.$$
Par bilin\'earit\'e, il suffit de v\'erifier cette relation pour $x= e_{ij}^{(l)}$ et $y = e_{rs}^{(l')}$. 
\begin{align}
q_{\beta, \alpha} (\beta(e_{ij}^{(l){\rm o}} )  \otimes \alpha(e_{rs}^{(l')})) q_{\beta, \alpha} &= \displaystyle
\delta_l^{l'} \delta_i^s \delta_j^r {\frac{1}{n_l^2}} \sum_{a,b=1}^{n_l}  (\beta(e_{ba}^{(l){\rm o}} e^{(l){\rm o}}) \otimes \alpha(e_{ab}^{(l)}) 
\nonumber\\
&= \sum_{l''=1}^k  \frac{1}{n_{l''}^2}\,  \epsilon(e_{rs}^{(l')}  e_{ij}^{(l)} e^{(l'')}) (\beta(e^{(l''){\rm o}}) \otimes \alpha(e^{(l'')})) q_{\beta, \alpha}.\nonumber
\end{align}
\end{proof}

\noindent
\begin{proof}[D\'emonstration du th\'eor\`eme] Posons $T:= (1 \otimes U) \Sigma \widehat V V \widetilde V$. On a $T\in B(H \otimes H)$ est une isom\'etrie partielle v\'erifiant   (\ref{2.9} c)) :
$$   T \in q_{\widehat\beta , \widehat\alpha} (\widehat\beta(N^{\rm o}) \otimes \widehat\alpha(N))q_{\widehat\beta , \widehat\alpha}  , \quad   T^* T = \widetilde V^* \widetilde V = q_{\widehat\beta , \widehat\alpha}.$$
 \noindent
Il existe alors  (\ref{lemth13}) des scalaires $\lambda_l\in\mathbb{C}$ v\'erifiant 
$$T = \sum_{l=1}^k \lambda_l (\widehat\beta(e^{(l){\rm o}}) \otimes \widehat\alpha(e^{(l)}))q_{\widehat\beta , \widehat\alpha}.$$ 
Par injectivit\'e de $\widehat\alpha$ et $\widehat\beta$,    $(\widehat\beta(e^{(l){\rm o}}) \otimes \widehat\alpha(e^{(l)}))q_{\widehat\beta , \widehat\alpha} \not= 0$ pour tout $l$. On en d\'eduit que
$$T^* T = \sum_{l=1}^k |\lambda_l|^2(\widehat\beta(e^{(l){\rm o}}) \otimes \widehat\alpha(e^{(l)}))q_{\widehat\beta , \widehat\alpha} = q_{\widehat\beta , \widehat\alpha},$$\noindent
d'o\`u le r\'esultat.
\end{proof}
\begin{corollary}\label{ShatS} $S \widehat S$ est une C*-alg\`ebre et on a $S \widehat S = U {\cal C}(V) U^*$.
\end{corollary}
 
\noindent
Le corollaire se d\'eduit directement des trois lemmes suivants que nous allons \'etablir.

\begin{lemme} Posons $B = [\delta(x) (1_H \otimes y) \,\,|\,\, x\in S \,, \, y\in\widehat S]$. Nous avons :
\begin{enumerate}
\item $B$  est stable par involution ;
\item Il existe une forme positive normale $\omega\in B(H)_\ast^+$,  v\'erifiant $S\widehat S  = (\omega \otimes {\rm id} ) (W B W^*)$ ;
\item $S \widehat S$ est stable par involution.
\end{enumerate}
\end{lemme}

\begin{proof}
 Par (\ref{2.9} c)), on a $[V^* V , S \otimes 1_H] = 0$.  Soit $\omega\in B(H)_\ast$, posons $y = ({\rm id} \otimes   \omega)(V)$. Pour tout tout $x\in S$, on a  
\begin{align*} (1_H \otimes y) \delta(x) &= (1_H \otimes ({\rm id}  \otimes \omega)(V)) \delta(x) = ({\rm id}  \otimes {\rm id}  \otimes \omega)(V_{23} \delta(x)_{12}) \\
&= ({\rm id}  \otimes {\rm id}  \otimes \omega)(V_{23} V_{23}^*V_{23}\delta(x)_{12}) = ({\rm id}  \otimes {\rm id}  \otimes \omega)(V_{23} \delta(x)_{12}V_{23}^*V_{23}) \\
&= ({\rm id}  \otimes {\rm id}  \otimes \omega)(({\rm id}  \otimes \delta)(\delta(x)) V_{23}) = ({\rm id}  \otimes {\rm id}  \otimes \omega)((\delta \otimes {\rm id} )(\delta(x)) V_{23}).
\end{align*}

Posons $\omega = \omega's$ avec $\omega'\in B(H)_\ast$ et $s\in S$. On d\'eduit   :
\begin{align}
(1_H \otimes y) \delta(x) &= ({\rm id} \otimes {\rm id}  \otimes \omega')((q_{\beta, \alpha} \otimes s)(\delta \otimes {\rm id} )(\delta(x)) V_{23}) \nonumber \\
&= ({\rm id} \otimes {\rm id}  \otimes \omega')(( (\delta \otimes {\rm id} )((1_S \otimes s)\delta(x)) V_{23}).\nonumber 
\end{align}
Comme $[(1_S \otimes S)\delta(S)] \subset S \otimes S$, on a $({\rm id} \otimes {\rm id}  \otimes \omega')((\delta \otimes {\rm id} )((1_S \otimes s)\delta(x)) V_{23})\in B$, ce qui prouve le a).
\hfill\break
On a $W W^* = q_{\alpha, \widehat \beta}$ (\ref{2.9} c)), donc $[W W^*, 1_H \otimes S \widehat S] =0$. On en d\'eduit  
en utilisant (\ref{2.10} c)) :
$$(1_H \otimes S \widehat S) W W^* = W W^* (1_H \otimes S \widehat S) W W^* =  W  B W^*.$$
\noindent
Soit   $(e_{ij}^{(l)})_{i,j = 1, \cdots, n_l, l = 1, \cdots ,k}$ un s.u.m.\ pour la C*-alg\`ebre  $N = \oplus_{l=1}^k M_{n_l}$. Il est clair qu'il existe une forme $\omega\in B(H)_\ast^+$ v\'erifiant $\omega(\alpha(e_{ij}^{(l)})) = \delta_i^j n_l$, pour tout $1\leqslant l\leqslant k$ et $1\leqslant i,j\leqslant n_l$. On a donc $(\omega \otimes {\rm id} ) (q_{\alpha, \widehat \beta}) = 1_H$. Il en r\'esulte que $S \widehat S =  (\omega \otimes {\rm id} ) ((1_H \otimes S \widehat S) W W^*) = (\omega \otimes {\rm id} )(W  B W^*)$, d'o\`u le b).
\hfill\break
Le c) est une cons\'equence du a) et du b).
\end{proof}

\begin{lemme}\label{2.17}On a $S \widehat S = [({\rm id} \otimes \omega) (W V)\,\,|\,\,\omega\in B(H)_\ast]$.
\end{lemme}

\begin{proof} Soient $s =({\rm id} \otimes \omega)(W)$ et $x\in \widehat S$. En utilisant (\ref{2.9} c)), on a :
$$s x  = ({\rm id} \otimes \omega)(W (x \otimes 1)) = ({\rm id} \otimes \omega)(W V V^*(x \otimes 1)).$$
\noindent
En posant $\omega = s'\omega' \,$ avec $s'\in S$, on en d\'eduit que $s  x\in [({\rm id} \otimes \omega') (W V)x'\,\,|\,\,\omega'\in B(H)_\ast \,,\,x'\in\widehat S]$.

Pour $\omega\in B(H)_\ast , x \in \widehat S $, on  a $({\rm id} \otimes \omega)(W V) x   = ({\rm id} \otimes \omega)(W V V^* V( x \otimes 1))$. 
En posant $\omega  = x' \omega'$ avec  $x'\in\widehat S$, on obtient avec (\ref{2.10} d)) que 
$({\rm id} \otimes \omega)(W V V^* V( x \otimes 1))\in [({\rm id} \otimes \omega')(W V\widehat\delta(x'))\,\,|\,\,x'\in\widehat S \,,\,\omega'\in B(H)_\ast]$.
Mais pour $\omega'\in B(H)_\ast  ,  x' \in \widehat S $, on a 
$({\rm id} \otimes \omega')(W V\widehat\delta(x')) =({\rm id} \otimes \omega')(W VV^*(1 \otimes x')V)) = ({\rm id} \otimes \omega' x')(W V)$,  
 d'o\`u les   inclusions $S \widehat S \subset 
[({\rm id} \otimes \omega) (W V) x\,\,|\,\,\omega\in B(H)_\ast \,,\,x \in \widehat S ] \subset  [({\rm id} \otimes \omega) (W V)\,\,|\,\,\omega\in B(H)_\ast]$.
\hfill\break
Montrons maintenant  l'inclusion 
  $[({\rm id} \otimes \omega) (W V)\,\,|\,\,\omega\in B(H)_\ast] \subset S \widehat S$. 

Soit $\omega\in B(H)_\ast$. Posons 
$\omega  =   x' \omega' x$ avec  $\omega'\in B(H)_\ast$ et $x,  x' \in \widehat S $. On a   
\begin{align} ({\rm id} \otimes \omega) (W V)& = ({\rm id} \otimes x'\omega')(W   (1 \otimes x) V)  
=  ({\rm id} \otimes x'\omega')(W V V^* (1 \otimes x) V) \nonumber \\
&= ({\rm id} \otimes \omega')(W V\widehat\delta(x)(1 \otimes x')\in [({\rm id} \otimes \omega) (W V) x\,\,|\,\,\omega\in B(H)_\ast \,,\,x \in \widehat S ].\nonumber
\end{align}
Finalement, pour $\omega\in B(H)_\ast \,,\,x \in \widehat S $, posons $\omega = s\, \omega'$ avec $s\in S$.  On a 
$$({\rm id} \otimes \omega) (W V) x = ({\rm id} \otimes \omega') (W V(x \otimes s))\in S \widehat S.$$
\end{proof}

\begin{lemme}  $S \widehat S = U^* [({\rm id} \otimes \omega) (V^* \Sigma )\,\,|\,\,\omega\in B(H)_\ast] U = U^* \cC(V) U = U \cC(V) U^* $.

En particulier, $\cC(V) =  [({\rm id} \otimes \omega) (\Sigma V )\,\,|\,\,\omega\in B(H)_\ast]$ est une C*-alg\`ebre.
\end{lemme}

\begin{proof}
En  utilisant  (\ref{2.9} c), \ref{th13}), on a
\begin{align} W V &= W V V^*V = W V\widetilde V \widetilde V^*  = 
\Sigma (1_H \otimes U^*)(\sum_{l=1}^k \lambda_l (\widehat\beta(e^{(l){\rm o}}) \otimes \widehat\alpha(e^{(l)}))q_{\widehat\beta , \widehat\alpha})
\widetilde V^* 
\nonumber \\
&=\Sigma (1_H \otimes U^*)(\sum_{l=1}^k \lambda_l (\widehat\beta(e^{(l){\rm o}}) \otimes \widehat\alpha(e^{(l)}))\widetilde V^*)  
=
\Sigma (1_H \otimes U^*)(\sum_{l=1}^k \lambda_l (\widehat\beta(e^{(l){\rm o}}) \otimes 1_H))\widetilde V^*) \nonumber \\
&= 
(1_H \otimes \sum_{l=1}^k \lambda_l \widehat\beta(e^{(l){\rm o}})) \Sigma (1_H \otimes U^*) \widetilde V^* 
= (1_H \otimes \sum_{l=1}^k \lambda_l \widehat\beta(e^{(l){\rm o}})) (U^* \otimes U^*) V^* \Sigma (U \otimes 1_H). \nonumber
\end{align}

Comme $\sum_{l=1}^k \lambda_l \widehat\beta(e^{(l){\rm o}}) U^*$ est un unitaire, on d\'eduit de (\ref{2.17}) que
$$S \widehat S = 
[({\rm id} \otimes \omega) (W V)\,\,|\,\,\omega\in B(H)_\ast] = U^* [({\rm id} \otimes \omega) (V^* \Sigma )\,\,|\,\,\omega\in B(H)_\ast] U.$$
\noindent 
Mais $S \widehat S$ \'etant une C*-alg\`ebre, on a aussi $S \widehat S = U^* \cC(V) U  = U \cC(V) U^*  $.
\end{proof}

\begin{remark}\label{Sl(hatS)} Gr\^ace aux deux relations 
$V^* = (J \otimes \widehat J) V (J \otimes \widehat J)$ et $W^* = (\widehat J \otimes J) W (\widehat J \otimes J)$, 
on peut proc\'eder comme dans \cite{BaSkaV} pour prouver que 
$${\cal C}(\widetilde V):=   [({\rm id} \otimes \omega) (\Sigma \widetilde V )\,\,|\,\,\omega\in B(H)_\ast]\quad {\rm et}\quad {\cal C}(W) := [({\rm id} \otimes \omega) (\Sigma W )\,\,|\,\,\omega\in B(H)_\ast]$$
sont des C*-alg\`ebres et v\'erifient 
$$   S\lambda(\widehat S)  = {\cal C}(\widetilde V) \quad \text{et} \quad  R(S) \widehat S   = {\cal C}(W).$$
\end{remark}

\subsection{Groupo\"ide mesur\'e associ\'e \`a une  \'equivalence mono\"idale}\label{grequivmono}

\noindent
Il s'agit d'un groupo\"ide \mq  de base $N =\C^2$,  introduit par De Commer \cite{DeC1,DeC2}. Avant de donner sa d\'efinition, 
  nous rappelons   les d\'efinitions et les r\'esultats fondamentaux  de \cite{DeC1} portant sur l'\'equivalence mono\"idale des groupes quantiques localement compacts,   que nous utiliserons.

Soit  $G$ un groupe  quantique  l.c.
\begin{definition}
Une action galoisienne \`a droite du groupe $G$ dans une alg\`ebre de von Neumann $N$,  est une coaction 
$\alpha_N : N \rightarrow N \otimes  L^\infty(G)$   ergodique, int\'egrable et telle que le produit crois\'e $N \rtimes G$ soit un facteur de type I.
\hfill\break
Le couple $(N, \alpha_N)$ est appel\'e objet galoisien \`a droite du groupe $G$.
\hfill\break
De fa\c con similaire, on d\'efinit aussi les actions galoisiennes et les objets galoisiens \`a gauche  pour $G$. 
\end{definition}

Fixons un objet galoisien \`a droite $(N, \alpha_N)$ du groupe $G$.  Dans sa th\`ese \cite{DeC1,DeC2}, De Commer construit    un groupe quantique \lc  $H$, muni d'une action galoisienne  \`a gauche $\gamma_N : N \rightarrow  L^\infty(H)  \otimes N$ qui commute avec $\alpha_N$.  
 
De fa\c con canonique, il associe aussi \`a l'objet galoisien \`a droite $(N, \alpha_N)$ de $G$, un objet galoisien \`a droite $(O, \alpha_O)$ pour le groupe $H$, et une action  galoisienne \`a  gauche $\gamma_O : O \rightarrow L^\infty(G) \otimes O$ pour le groupe $G$, qui commute avec $\alpha_O$. 
 
Finalement, De Commer a construit \cite{DeC1,DeC2}  un groupo\"ide mesur\'e quantique 
${\cal G}_{H, G} = (\C^2,  M , \alpha , \beta , \delta , T , T', \epsilon)$\index{ga@${\cal G}_{H, G}$}, o\`u $M := L^\infty(H) \oplus  N \oplus  O \oplus  L^\infty(G)$, dont l'axiomatique traduit  toutes les propri\'et\'es des coactions ou  coproduits des quatre alg\`ebres $L^\infty(H)$, $N$, $O$ et $L^\infty(G)$, ainsi que   celles des poids invariants ou de Haar sous-jacents. Plus pr\'ecis\`ement, le coproduit $\delta : M \rightarrow M \otimes M$ est form\'e par les coactions ou coproduits des quatre alg\`ebres composant $M$. Les poids op\'eratoriels $T$ et $T'$ sont d\'eduits des poids invariants par les coactions ou de Haar. De plus,  les *-morphismes $\alpha$ et $\beta$ sont \`a valeurs dans le centre de $M$ et leurs images engendrent une copie de l'alg\`ebre $\C^4$.

R\'eciproquement, il est facile de voir qu'un groupo\"ide ${\cal G}  = (\C^2,  M , \alpha , \beta , \Gamma , T , T', \epsilon)$, o\`u les applications source et but v\'erifient les hypoth\`eses pr\'ec\'edentes, est de fa\c con unique,  de la forme ${\cal G} = {\cal G}_{H, G}$ avec $H$ et $G$ des groupes quantiques \lc uniques \`a isomorphisme pr\`es.

Dans ce qui suit, on se fixe un  groupo\"ide  mesur\'e ${\cal G}  = (\C^2,  M , \alpha , \beta , \delta , T , T', \epsilon)$,   o\`u les *-morphismes $\alpha$ et $\beta$ sont \`a valeurs dans le centre de $M$ avec   des images qui engendrent une copie de l'alg\`ebre $\C^4$.  

\begin{lemme}\label{2.21} 
On a $\alpha = \widehat \beta$ et $\beta = \widehat \alpha$. 
\end{lemme}

\begin{proof}
On a (\cite{E} 3.6) $\widehat \beta(n) = J \alpha(n)^* J  = \alpha(n)$  car $\alpha(n)$ est central dans $M$. 

De m\^eme, on a (\cite{E} 3.8) $J \beta(n)^* J  = \beta(n)   =  \widehat J \alpha(n)^* \widehat J$, d'o\`u 
$\beta = {\rm Ad}\,U \circ \alpha = \widehat \alpha$.
\end{proof}

Avec les notations introduites dans le paragraphe \ref{2.2}, on se propose de pr\'eciser les repr\'esentations r\'eguli\`eres $V$ et $W$ de ${\cal G}$, ainsi que les deux groupes quantiques \lc  que d\'etermine ${\cal G}$. 
 
D'abord, on identifie $M$  \`a son image dans $B(H)$ par la repr\'esentation GNS du poids $\varphi  = \epsilon \circ \alpha^{-1} \circ T$. Posons $\psi =  \epsilon \circ \beta^{-1} \circ T'$. 

\begin{notations} Soit $(\varepsilon_1, \varepsilon_2)$\index{ec@$\varepsilon_j$} la base canonique de l'espace vectoriel $\C^2$.   Pour tout $i , j=1,2$, posons :\index{pb@$p_{ij}$}\index{mc@$M_{ij}$}\index{h@$H_{ij}$}\index{pc@$\varphi_{ij}$, $\psi_{ij}$}
$$p_{ij}  := \alpha(\varepsilon_i) \beta(\varepsilon_j) , \quad M_{ij} := p_{ij} M  , \quad H_{ij} = p_{ij} H  , \quad \varphi_{ij} := \restr{\varphi}{M_{ij}}
 , \quad \psi_{ij} := \restr{\psi}{M_{ij}}.$$
\end{notations}

On a facilement :
$$\alpha(\varepsilon_i) = p_{i1} + p_{i2} \,,\,\beta(\varepsilon_j) = p_{1j} + p_{2j} \quad ; \quad M = \bigoplus_{i,j=1,2}  M_{ij} \quad ; \quad
\delta(p_{ij})  = \sum_{k = 1,2} p_{ik} \otimes p_{kj}\quad ; \quad H = \bigoplus_{i,j=1,2} H_{ij}.$$
Le poids  $\varphi_{ij}$ est  nsff sur l'alg\`ebre de von Neumann $M_{ij}$ et sa repr\'esentation GNS s'obtient par restriction \`a $M_{ij} \subset M$, de la repr\'esentation GNS 
  de $\varphi$.  En particulier,  l'espace de Hilbert $L^2(M_{ij} , \varphi_{ij})$ s'identifie \`a    $H_{ij}$.

En utilisant (\ref{2.10} b)), on obtient :
\begin{proposition} Pour tout $i , j , k , l=1,2$, on a :
$$    (p_{ij} \otimes 1_H) V (p_{kl} \otimes 1_H) = \delta_k^i (p_{ij} \otimes p_{jl}) V (p_{il} \otimes p_{jl}) 
 , \quad (1_H \otimes p_{ij}) W (1_H \otimes p_{kl}) = \delta_j^l (p_{ik} \otimes p_{ij}) W (p_{ik} \otimes p_{kj}).$$
\end{proposition}

\begin{notations}\label{Vind}$V_{jl}^i := (p_{ij} \otimes p_{jl}) V (p_{il} \otimes p_{jl})$, $W_{ik}^j :=   (p_{ik} \otimes p_{ij}) W (p_{ik} \otimes p_{kj})$, $\widetilde V_{ki}^j :=   (p_{ki} \otimes p_{ji}) \widetilde V (p_{ki} \otimes p_{jk})$.\index{vc@$V_{jl}^i$, $W_{ik}^j$, $\widetilde V_{ki}^j$}
\end{notations}

Les op\'erateurs  unitaires 
$$V_{jl}^i : H_{il} \otimes H _{jl} \rightarrow H_{ij} \otimes H_{jl} ,\quad W_{ik}^j : H_{ik} \otimes H_{kj} \rightarrow H_{ik} \otimes H_{ij} \quad \text{et} \quad \widetilde V_{ki}^j : H_{ki} \otimes H_{jk}\rightarrow H_{ki} \otimes H_{ji}$$
\noindent
v\'erifient les relations suivantes  :
\begin{proposition}\label{pent} Soient $i , j , k , l, l', s = 1 , 2$.
\begin{enumerate}
\item Relations pentagonales :   
$$(V_{jk}^i)_{12} (V_{kl}^i)_{13} (V_{kl}^j)_{23}  = (V_{kl}^j)_{23} (V_{jl}^i)_{12},  \quad
(W_{ij}^k)_{12} (W_{ij}^l)_{13} (W_{jk}^l)_{23} = (W_{ik}^l)_{23}(W_{ij}^k)_{12},$$
$$(\widetilde V_{ji}^k)_{12} (\widetilde V_{ji}^l)_{13} (\widetilde V_{kj}^l)_{23}  = (\widetilde V_{ki}^l)_{23}(\widetilde V_{ji}^k)_{12}.$$
\item Relations de commutation :  
$V_{kj, 23}^l W_{ll', 12}^j = W_{ll', 12}^k V_{kj, 23}^{l'}$, \quad  $V_{kj, 12}^s \widetilde V_{kj, 23}^i = \widetilde V_{kj, 23}^i V_{kj, 12}^s$.
\item 
Pour tout $\omega \in B(H)_\ast$, on a  :
$$({\rm id}  \otimes p_{jl} \omega p_{jl}) (V_{jl}^i) = p_{ij} ({\rm id}  \otimes \omega)(V) p_{il}  , \quad  (p_{ik} \omega p_{ik} \otimes {\rm id} )(W_{ik}^j) = 
p_{ij} (\omega \otimes {\rm id} )(W) p_{kj}.$$
\end{enumerate}
\end{proposition}  
 
Les relations pentagonales a) et les relations de commutation b) r\'esultent de celles v\'erifi\'ees par $V$ et $W$ (\ref{2.10} a)). La preuve de c) s'obtient par calcul direct en utilisant (\ref{2.10} b)). 

Rappelons que le coproduit $\delta$ (resp. $\widehat\delta$) de la C*-alg\`ebre $S$ (resp. $\widehat S$) est donn\'e par : 
\begin{equation}\label{coprod}\delta(x) =  V ( x \otimes 1) V^*  = W^* (1 \otimes x) W  , \quad x \in S \quad 
({\rm resp.} \,\,\widehat\delta(x) = V^* (1 \otimes x) V = \widetilde V (x \otimes 1) \widetilde V^*, \quad x \in \widehat S).
\end{equation}
\begin{notations}
\begin{enumerate}
\item Posons $S_{ij} = p_{ij} S$\index{sd@$S_{ij}$}. On identifie  les  C*-alg\`ebres $M(S_{ij})$ (resp.   $M(S_{ij} \otimes S_{kl})$) et 
$p_{ij} M(S) \subset  B(H)$ (resp. $(p_{ij} \otimes p_{kl})M(S \otimes S) \subset B(H \otimes H)$).
\item
 Pour tout $x\in S_{ij}$, posons  
$\delta_{ij}^k(x)= (p_{ik} \otimes p_{kj})  \delta(x)\in M(S_{ik} \otimes S_{kj})$.\index{db@$\delta_{ij}^k$}
\end{enumerate}
\end{notations}

\begin{proposition}\label{coprodS} \cite{DeC1, DeC2}
Soient $i, j, k, l =1, 2$.
\begin{enumerate}
\item $(\delta_{ik}^l \otimes {\rm id}_{S_{kj}})\circ \delta_{ij}^k =  ({\rm id}_{S_{il}} \otimes \delta_{lj}^k)  \circ \delta_{ij}^l$.
\item 
Pour tout $x\in S_{ij}$, on a 
$\delta_{ij}^k(x) = (W_{ik}^j)^* (1_{H_{ik}} \otimes x) W_{ik}^j = V_{kj}^i (x \otimes 1_{H_{kj}}) (V_{kj}^i)^*$.
\end{enumerate}
\end{proposition}

Le r\'esultat suivant joue un r\^ole central dans la suite : 
 
\begin{proposition} \label{moritaS}
Les projecteurs $\beta(\varepsilon_1), \beta(\varepsilon_2)$ de la C*-alg\`ebre $M(\widehat S)$ v\'erifient :
\begin{enumerate}
\item  $\beta(\varepsilon_1) + \beta(\varepsilon_2) = 1_{\widehat S} \quad ; \quad \widehat\delta(\beta(\varepsilon_j)) = \beta(\varepsilon_j) \otimes \beta(\varepsilon_j), \quad j=1,2 \quad ;$
\item  $[\widehat S \beta(\varepsilon_j) \widehat S] = \widehat S, \quad j=1,2$.  
\end{enumerate}
\end{proposition}

\begin{proof} 
Le a) est \'evident. Rappelons que  dans \cite{DeC1,DeC2}, De Commer a montr\'e que l'espace  vectoriel engendr\'e par $\widehat M  \alpha(\varepsilon_1) \widehat M$ est faiblement dense dans $\widehat M$.  Il en r\'esulte que le b) est donn\'e par  le r\'esultat plus g\'en\'eral qui suit.
\end{proof}

\begin{proposition} Soit $n\in N$ un projecteur. Supposons que l'espace vectoriel engendr\'e par  $\widehat M  \alpha(n) \widehat M$ soit faiblement dense dans $\widehat M$. Alors l'espace vectoriel    engendr\'e par $\widehat S \widehat \alpha(n) \widehat S$ est normiquement dense dans $\widehat S$.
\end{proposition}

\begin{proof} Montrons d'abord que $[(1 \otimes \lambda(\widehat S))V(1 \otimes \lambda(\widehat S)) V^*] = (\rho(\widehat S) \otimes \lambda(\widehat S)) q_{\beta,\alpha}$.

D'apr\`es \cite{DeC2}, on a $[(\rho(\widehat S) \otimes 1)\widehat\delta(\rho(\widehat S))] = (\rho(\widehat S)\otimes \rho(\widehat S)) q_{\widehat\alpha,\beta}$.

Mais pour tout $x\in \widehat S$, on a (\ref{2.4} 2)) $V(1 \otimes \lambda(x)) V^* = (1 \otimes U)\Sigma V^*(1 \otimes \rho(x)) V \Sigma (1 \otimes U^*)$. Il en r\'esulte que $[(1 \otimes \lambda(\widehat S))V(1 \otimes \lambda(\widehat S)) V^*] = (\rho(\widehat S) \otimes \lambda(\widehat S)) q_{\beta,\alpha}$.

Pour tout $x , x' \in\widehat S$, on a 
$(1 \otimes \lambda(x))V (1 \otimes \lambda(x')) = (1 \otimes \lambda(x))V (1 \otimes \lambda(x')) q_{\widehat \alpha, \beta} = (1 \otimes \lambda(x))V (1 \otimes \lambda(x')) V^*V$.

On en d\'eduit que $[(1 \otimes \lambda(\widehat S))V (1 \otimes \lambda(\widehat S))] = 
(\rho(\widehat S) \otimes \lambda(\widehat S)) q_{\beta,\alpha} V = (\rho(\widehat S) \otimes \lambda(\widehat S)) V$, donc
$$[(\id \otimes \omega)(1 \otimes \lambda(\widehat S))V (1 \otimes \lambda(\widehat S))\,|\,\omega\in B(H)_\ast] = \widehat S.$$ 
Gr\^ace \`a cette \'egalit\'e et l'hypoth\`ese, on a 
$$[({\rm id}  \otimes \omega_{\xi, \eta}) (1 \otimes \lambda(x))  V (1 \otimes \lambda(x') \alpha(n)) \,\,| \,\, x , x'\in \widehat S \, | \,\xi \, , \,\eta \in H] =\widehat S.$$
\noindent

Mais par le calcul pr\'ec\'edent, on a pour $x , x'\in \widehat S$ :
$$(1 \otimes \lambda(x))  V (1 \otimes \lambda(x')\alpha(n)) = (1 \otimes \lambda(x))  V (1 \otimes \lambda(x')) V^* V (1 \otimes \alpha(n) \in  (\widehat S \otimes \lambda(\widehat S)) V (1 \otimes \alpha(n)).$$ 
Mais on a par (\ref{2.10} b)) $(\widehat \alpha(n) \otimes 1) V    = V (1 \otimes \alpha(n))$, donc  :
$$[({\rm id}  \otimes \omega_{\xi, \eta}) (1 \otimes \lambda(x))  V (1 \otimes \lambda(x') \alpha(n)) \,\,| \,\, x , x'\in \widehat S \, | \,\xi \, , \,\eta \in H] \subset [\widehat S \widehat\alpha(n) \widehat S].$$
\end{proof}

\begin{notations}\label{notE}  Pour tout $i , j , k  = 1,2$, $x\in \widehat S$ et $y\in\lambda(\widehat S)$, posons  :
\begin{enumerate}
\item $x_{ij} := \beta(\varepsilon_i) x \beta(\varepsilon_j)\in \widehat S$, $\widehat \pi_i(x_{jk}) := p_{ij} x p_{ik}$ ; $y_{ij} := \alpha(\varepsilon_i) y \alpha(\varepsilon_j)\in \lambda(\widehat S)$, $\widehat \pi^i(y_{jk}) := p_{ji} y p_{ki} $ ;\index{pd@$\widehat{\pi}_i$, $\widehat{\pi}^i$}
\item
$E_{jk}^i := [ \widehat\pi_i(x_{jk})\,|\,x \in \widehat S] \subset B(H_{ik}, H_{ij})$, $E_{jk, \lambda}^i := U E_{jk}^i U^* = [ \widehat\pi^i(\lambda(x_{jk}))\,|\, x \in \widehat S] \subset B(H_{ki}, H_{ji})$.\index{ed@$E_{jk}^i$, $E_{jk,\lambda}^i$}
\end{enumerate}
\end{notations}

\begin{proposition} \label{Emorita}
Pour $\omega\in B(H)_\ast$, posons $x = ({\rm id}  \otimes \omega)(V)$. Pour tout $i , j , k, l = 1, 2$, nous avons :
\begin{enumerate}
\item $\widehat \pi_i(x_{jl}) = ({\rm id}  \otimes p_{jl} \omega p_{jl}) (V_{jl}^i), \quad 
\widehat \pi^j(\lambda(x_{ik}))   = (p_{ik} \omega p_{ik} \otimes {\rm id} )(W_{ik}^j)$ \quad ;
\item 
$(\widehat\pi_k \otimes \widehat\pi_l)\,\widehat\delta(x_{ij}) =  (V_{li}^k)^* (1_{H_{kl}}  \otimes \widehat\pi_l(x_{ij})) V_{lj}^k = \widetilde V_{ki}^l (\widehat\pi_k(x_{ij}) \otimes 1_{H_{lk}}) ( \widetilde V_{kj}^l)^*  , \quad x \in \widehat S  $ \quad ;
\item $[E_{jl}^i\,H_{il}] = H_{ij}$ \quad ;
\item  $(E_{jl}^i) ^* = E_{lj}^i  , \quad [E_{jk}^i \circ E_{kl}^i] =   E_{jl}^i  $. 
\end{enumerate}
\end{proposition}
\noindent
\begin{proof} Le a) s'obtient par un calcul direct.
\hfill\break
Le coproduit $\widehat\delta : \widehat S \rightarrow M(\widehat S \otimes \widehat S)$ v\'erifie $\widehat\delta(x) = V^* (1 \otimes  x)V = \widetilde V (x \otimes 1) \widetilde V^*$, d'o\`u le b). 
\hfill\break
Le c)  
  r\'esulte   du fait que chaque op\'erateur $V_{jl}^i$ est un unitaire.
\hfill\break
L'\'egalit\'e $(E_{jl}^i) ^* = E_{lj}^i$    et l'inclusion  $[E_{jk}^i \circ E_{kl}^i] \subset   E_{jl}^i  $ sont imm\'ediates.
\hfill\break
Pour tout  $\omega , \omega'\in B(H)_\ast$, consid\'erons la forme normale   $\Phi\in B(H)_\ast$    d\'efinie par  : 
\begin{equation}\label{eq1}  \Phi : x \mapsto (p_{jk}\omega p_{jk} \otimes p_{kl}\omega'p_{kl})(V(x \otimes 1)V^*) = 
  (p_{jk}\omega p_{jk} \otimes p_{kl}\omega' p_{kl})(V_{kl}^j(p_{jl}xp_{jl} \otimes 1)(V_{kl}^j)^*).
\end{equation}
On a $\Phi = p_{jl} \Phi p_{jl}$ et comme  $V_{kl}^j : H_{jl} \otimes H_{kl}  \rightarrow H_{jk} \otimes H_{kl}$ est unitaire, il est clair que   l'espace vectoriel engendr\'e par ces formes normales $\Phi$,  est dense dans $B(H_{jl})_\ast$.

On en d\'eduit que 
$ E_{jl}^i = [(\id \otimes \Phi)(V)  \, |\, \Phi  \,\,\hbox{comme dans} \,\ref{eq1}]$.

Avec  $\omega , \omega'\in B(H)_\ast$ et $\Phi$ comme dans \ref{eq1}, posons $x = ({\rm id} \otimes \omega)(V)\,,\,y = ({\rm id} \otimes \omega')(V)$  et $z = ({\rm id} \otimes \Phi) (V)$.  la relation 
$(V_{jk}^i)_{12} (V_{kl}^i)_{13}   = (V_{kl}^j)_{23} (V_{jl}^i)_{12}(V_{kl}^j)_{23}^*$, nous donne imm\'ediatement 
$\widehat\pi_i(x_{jk})\widehat\pi_i(y_{kl}) = \widehat\pi_i(z_{jl})$, d'o\`u l'inclusion 
$E_{jl}^i \subset [\,E_{jk}^i \circ E_{kl}^i]$.
\end{proof}

\begin{corollary}   Pour tout $i = 1 ,2$, on a  des repr\'esentations fid\`eles et non d\'eg\'en\'er\'ees\index{pd@$\widehat{\pi}_i$, $\widehat{\pi}^i$} 
$$\widehat \pi_i : \widehat S \rightarrow B(H_{i1} \oplus H_{i2}) \quad \text{et} \quad  
 \widehat \pi^i : \lambda(\widehat S) \rightarrow B(H_{1i} \oplus H_{ 2i})$$
\noindent
d\'efinies par :
 $$ \widehat \pi_i(x) = \begin{pmatrix}p_{i1} x p_{i1} & p_{i1} x p_{i2} \\ 
p_{i2} x p_{i1} & p_{i2} x p_{i2} \end{pmatrix}, \quad x\in\widehat{S}\quad ;\quad  
 \widehat \pi^i(y) = \begin{pmatrix}p_{1i} x p_{1i} & p_{1i} x p_{2i} \\ 
p_{2i} x p_{1i} & p_{2i} x p_{2i} \end{pmatrix},\quad y\in\lambda(\widehat{S}). $$
\end{corollary}

\begin{proof}  Il r\'esulte de (\ref{Emorita} c)) que la  repr\'esentation $\widehat \pi_i$ est non d\'eg\'en\'er\'ee. 
Soit $x \in \widehat S $ tel que $\widehat \pi_i(x) = 0$. On a donc $\alpha(\varepsilon_i) x = 0$. Comme $1 - \alpha(\varepsilon_i) \in \widehat M$, ce projecteur est limite faible de sommes finies d'\'el\'ements de la forme $y \alpha(\varepsilon_i)  y'$ avec $y , y' \in \widehat M$, donc $x=0$.
\hfill\break
La preuve pour la repr\'esentation $\widehat \pi^i$ est similaire.
\end{proof}
%\begin{remark}
%\end{remark}

 \begin{corollary} Pour tout $i , j  , k , l  = 1 , 2$, $E_{lk}^i \otimes S_{jl}$ (resp. $S_{ik} \otimes E_{kl,\lambda}^j$) est un C*-module sur 
$E_{kk}^i \otimes S_{jl}$ (resp. $S_{ik} \otimes E_{ll,\lambda}^j$). De plus, on a 
$$V_{jl}^i \in {\cal L}(E_{lk}^i \otimes S_{jl}, E_{jk}^i \otimes S_{jl}) \quad \text{et} \quad W_{ik}^j \in {\cal L}(S_{ik} \otimes E_{kl,\lambda}^j , S_{ik} \otimes E_{il,\lambda}^j).$$
\end{corollary}

\begin{proof}
Ce sont des cons\'equences de $V\in M(\widehat S \otimes S)$ et $W\in M(S \otimes \lambda(\widehat S))$.
\end{proof}

\begin{proposition} \label{grco}
\begin{enumerate}
 \item   $G_i:=(M_{ii} , \delta_{ii}^i , \varphi_{ii}, \psi_{ii})$, est un groupe quantique localement compact.
\item  $(M_{ij} , \delta_{ij}^j)$ est une action galoisienne \`a droite pour le groupe quantique $G_j$ et $V_{jj}^i$ est l'impl\'ementation canonique  \cite{V}  de  cette action. 
\item   $(M_{ij} , \delta_{ij}^i)$ est une action galoisienne \`a gauche  pour le groupe quantique $G_i$ et $W_{ii}^j$ est 
l'impl\'ementation canonique  \cite{V}  de  cette action.
\item  Les deux actions $\delta_{ij}^j$ et $\delta_{ij}^i$ sur $M_{ij}$, commutent.
\item $W_{12}^2 = \Sigma \widetilde G ^* \Sigma$, o\`u $\widetilde G$ est l'op\'erateur de Galois introduit 
\cite{DeC1}  par De Commer.
\end{enumerate}
\end{proposition}
La preuve de ces r\'esultats se d\'eduit de l'axiomatique du groupo\"ide sous-jacent, de la d\'efinition  et des  propri\'et\'es des repr\'esentations r\'eguli\`eres $V$ et $W$. Donnons par exemple la preuve que la repr\'esentation normale 
$ \sigma :  M_{ij} \rtimes G_j  \rightarrow  B(H_{ij})$  d\'efinie \cite{V}  par l'impl\'ementation canonique, est fid\`ele.
\hfill\break
Pour tout $x\in M_{ij}$, on a $\delta_{ij}^j(x) = (W_{ij}^j)^* (1_{H_{ij}} \otimes x) W_{ij}^j = V_{jj}^i (x \otimes 1_{H_{jj}}) (V_{jj}^i)^*$.
\hfill\break 
La relation \ref{pent} b)    
$W_{ij, 12}^j V_{jj, 23}^j = V_{jj, 23}^i W_{ij, 12}^j$ entra\^ine que pour tout $x\in M_{ij}$ et $y \in (\widehat M_{jj})'$, on a :
$$W_{ij, 12}^j \, \delta_{ij}^j(x)(1_{H_{ij}} \otimes y) \,(W_{ij, 12}^j)^* = 1_{H_{ij}} \otimes \sigma(\delta_{ij}^j(x)(1 \otimes y))$$

 \noindent
Terminons ce paragraphe par la notation-d\'efinition suivante :

\noindent
\begin{notation-definition}\label{defco}Un groupo\"ide mesur\'e quantique de la forme   ${\cal G}  = (\C^2,  M , \alpha , \beta , \delta , T , T', \epsilon)$,   o\`u les *-morphismes $\alpha$ et $\beta$ sont \`a valeurs dans le centre de $M$ avec   des images qui engendrent une copie de l'alg\`ebre $\C^4$, sera not\'e  ${\cal G}_{G_1, G_2}$\index{gb@${\cal G}_{G_1, G_2}$}, o\`u  $G_i:=(M_{ii} , \delta_{ii}^i , \varphi_{ii}, \psi_{ii})$  et sera appel\'e un groupo\"ide  de co-liaison.
\end{notation-definition}

\begin{definition} Deux groupes quantiques localement compacts $H$ et $G$ sont dits mono\"idalement \'equivalents s'il existe un groupo\"ide mesur\'e quantique de co-liaison ${\cal G}_{G_1, G_2}$ tel que $H$ (resp. $G$) soit isomorphe \`a $G_1$ (resp. $G_2$).
\end{definition}

\subsection{R\'egularit\'e}

Dans ce paragraphe,  nous introduisons la notion de groupo\"ide mesur\'e quantique semi-r\'egulier. Dans le cas d'un groupo\"ide de co-liaison ${\cal G}_{G_1, G_2}$, nous montrons que la semi-r\'egularit\'e (resp. r\'egularit\'e) des deux groupes quantiques $G_1$ et $G_2$ est \'equivalente \`a la semi-r\'egularit\'e (resp. r\'egularit\'e) de ${\cal G}_{G_1, G_2}$.
\hfill\break
La notion de groupo\"ide r\'egulier a \'et\'e introduite dans le cas compact par \cite{Ec} et g\'en\'eralis\'ee au cadre des C*-unitaires pseudo-multiplicatifs  par \cite{Tim1,Tim2}.
\hfill\break
Notons que De Commer a donn\'e un exemple \cite{DeC3} de groupo\"ide de co-liaison ${\cal G}_{G_1, G_2}$, o\`u $G_1$ est r\'egulier et $G_2$ semi-r\'egulier (et non r\'egulier). 

Dans ce qui suit, on se fixe un groupo\"ide ${\cal G} = (N, M,  \alpha, \beta, \delta, T, T', \epsilon)$ de base $N =  \oplus_{l=1}^k M_{n_l}$ (\cf\ref{m.qfini}).

Dans l'appendice, nous avons donn\'e la d\'efinition des op\'erateurs $R_\xi^\alpha  , L_\xi^\beta \in B(H_\epsilon, H)$ associ\'es \`a tout $\xi \in H$. Notons que  l'espace vectoriel normiquement  ferm\'e $[R_\xi^\alpha (R_\eta^\alpha)^* \,\,|\,\,\xi , \eta \in H]$
(resp. $[L_\xi^\beta (L_\eta^\beta)^* \,\,|\,\,\xi , \eta \in H]$), est une sous-C*-alg\`ebre de ${\cal K}(H)$, dont l'adh\'erence faible est $\alpha(N)'$ (resp. $\beta(N^{\rm o})'$).

\noindent
\begin{definition} \label{reg1}Le  groupo\"ide ${\cal G}$   est dit semi-r\'egulier (resp. r\'egulier)  si on a $[L_\xi^\beta (L_\eta^\beta)^* \,\,|\,\,\xi , \eta \in H] \subset {\cal C}(V)$ (resp. $[L_\xi^\beta (L_\eta^\beta)^* \,\,|\,\,\xi , \eta \in H] = {\cal C}(V)$).
\end{definition}
 Nous avons la caract\'erisation suivante de la semi-r\'egularit\'e (resp. r\'egularit\'e) : 

\begin{proposition}\label{dualreg} Il y a \'equivalence entre les conditions suivantes :
\begin{enumerate}
\item Le  groupo\"ide ${\cal G}$ est semi-r\'egulier (resp. r\'egulier).
\item $[L_\xi^{\widehat\beta} (L_\eta^{\widehat\beta})^* \,\,|\,\,\xi , \eta \in H]  \subset S \widehat S $ (resp. $[L_\xi^{\widehat\beta} (L_\eta^{\widehat\beta})^* \,\,|\,\,\xi , \eta \in H]  = S \widehat S $).
\item Le groupo\"ide dual $\widehat{{\cal G}}$ est semi-r\'egulier (resp. r\'egulier).
\item $[R_\xi^\alpha (R_\eta^\alpha)^* \,\,|\,\,\xi , \eta \in H] \subset {\cal C}(W)$ (resp.  $[R_\xi^\alpha (R_\eta^\alpha)^* \,\,|\,\,\xi , \eta \in H] = {\cal C}(W)$).
\end{enumerate}
\end{proposition}

\begin{proof}
  L'\'equivalence des conditions a) et b) r\'esulte (\ref{ShatS}) des \'egalit\'es  :
$$S \widehat S = U {\cal C}(V) U^*  , \quad  
 U [L_\xi^\beta (L_\eta^\beta)^* \,\,|\,\,\xi , \eta \in H ]U^* = [L_\xi^{\widehat\beta} (L_\eta^{\widehat\beta})^* \,\,|\,\,\xi , \eta \in H].$$
\noindent
Les C*-alg\`ebres ``$S$\,'' et ``$\widehat S$\,'' correspondant au groupo\"ide $\widehat{{\cal G}}$, sont respectivement $\widehat S$ et $R(S) = U S U^*$. Par l'\'equivalence des conditions a) et b) appliqu\'ee au groupo\"ide $\widehat{{\cal G}}$, la semi-r\'egularit\'e (resp. r\'egularit\'e) de $\widehat{{\cal G}}$  est \'equivalente \`a :
$$[R_\xi^\alpha (R_\eta^\alpha)^* \,\,|\,\,\xi , \eta \in H] \subset R(S) \widehat S \quad \quad  \text{(resp. }  [R_\xi^\alpha (R_\eta^\alpha)^* \,\,|\,\,\xi , \eta \in H] = R(S) \widehat S).$$
\noindent 
Par ailleurs, on a $J S \widehat S J  = R(S) \widehat S = {\cal C}(W)$ et $J [L_\xi^{\widehat\beta} (L_\eta^{\widehat\beta})^* \,\,|\,\,\xi , \eta \in H] J =  [R_\xi^\alpha (R_\eta^\alpha)^* \,\,|\,\,\xi , \eta \in H]$, d'o\`u l'\'equivalence des quatre conditions.
\end{proof}

\noindent
Pour \'etudier la (semi-)r\'egularit\'e d'un groupo\"ide de co-liaison ${\cal G}_{G_1, G_2}$ associ\'e \`a deux groupes quantiques \lc mono\"idalement \'equivalents, nous avons besoin du lemme qui va suivre, dont les notations sont les suivantes:
\hfill\break
Soient $H , K$ et $L$ des espaces de Hilbert,   $X : K \otimes H \rightarrow L \otimes H$  et $Y : H \otimes K \rightarrow H \otimes L$ des op\'erateurs born\'es. Posons\index{c@$\cC(-)$} : 
$${\cal C}(X) := [({\rm id} \otimes   \omega ) (\Sigma_{L \otimes H}\, X)\,\,|\,\, \omega\in B(H, L)_\ast ] , \quad {\cal C}(Y) := [({\rm id} \otimes   \omega ) (\Sigma_{H \otimes L}\, Y)\,\,|\,\, \omega\in B(K, H)_\ast ].
 $$

\begin{lemme} \label{reg}Soient $E , E'$ des espaces de Hilbert non nuls. Les  conditions suivantes sont \'equivalentes  :
\begin{enumerate}
\item  ${\cal K}(H, L) \subset {\cal C}(Y)\, $ (resp. ${\cal K}(K, H) \subset {\cal C}(X)$).
\item  ${\cal K}(H \otimes E , E' \otimes L) \subset [({\cal K}(H, E') \otimes 1_L) Y (1_H \otimes {\cal K}(E, K))]$ \hfill\break
(resp. 
${\cal K}(K \otimes E', E \otimes H) \subset 
[({\cal K}(L, E) \otimes 1_H) X (1_K \otimes {\cal K}(E', H))]$).
\end{enumerate}
\begin{proof}
On a  par un calcul direct :
$$(1_L \otimes \theta_{e',\xi}) \Sigma_{H \otimes L}\, Y (1_H \otimes \theta_{\eta, e})  =  
({\rm id} \otimes   \omega_{\xi, \eta} ) (\Sigma_{H \otimes L}\, Y) \otimes \theta_{e', e},\quad
\xi \in H \, , \,\eta \in K \, , \, e \in E \, , \,e' \in E'.$$
\noindent
Il en r\'esulte  :
$$[(1_L \otimes {\cal K}(H, E'))\Sigma_{H \otimes L}\, Y (1_H \otimes {\cal K}(E, K))] = [{\cal C}(Y) \odot {\cal K}(E, E')] \subset B(H \otimes E , L \otimes E').$$
\noindent
L'\'equivalence des conditions a) et b) est alors une cons\'equence de :
\noindent 
$$\Sigma_{L \otimes E'} \, (1_L \otimes {\cal K}(H, E')\,\Sigma_{H \otimes L} = {\cal K}(H, E') \otimes 1_L.$$
\noindent
L'\'equivalence des conditions  (resp.) s'obtient en appliquant l'\'equivalence des conditions a) et b) \`a 
$$Y' := \Sigma_{K \otimes H}  X^*\Sigma_{H \otimes L}\, ,\, K' := L\, ,\, L' := K$$
\noindent
car  on a $\cC(Y')  = \cC(X)^*$.
\end{proof}
\end{lemme}

\noindent
On va appliquer le lemme pr\'ec\'edent aux op\'erateurs unitaires  
$$V_{rj}^i : H_{ij} \otimes H _{rj} \rightarrow H_{ir} \otimes H_{rj} \quad \text{et} \quad W_{ij}^k : H_{ij} \otimes H_{jk} \rightarrow H_{ij} \otimes H_{ik}.$$
\begin{corollary}\label{cor-reg}
 Soient $i , j , k , r = 1 , 2$.
\begin{enumerate}
\item Il y a \'equivalence entre les conditions suivantes :\begin{enumerate}\renewcommand\theenumii{\roman{enumii}}
\renewcommand\labelenumii{\rm ({\theenumii})}
\item
${\cal K}(H_{ij} , H_{rj}) \subset {\cal C}(V_{rj}^i)$ \quad (resp. ${\cal K}(H_{ij} , H_{rj}) = {\cal C}(V_{rj}^i)$).
\item ${\cal K}(H_{ij} \otimes H_{ir} , H_{rj} \otimes H_{rj}) \subset    [({\cal K}(H_{ir} , H_{rj}) \otimes 1_{H_{rj}}) V_{rj}^i (1_{H_{ij}} \otimes {\cal K}(H_{ir} , H_{rj}))]$\hfill\break
(resp.
${\cal K}(H_{ij} \otimes H_{ir} , H_{rj} \otimes H_{rj}) =    [({\cal K}(H_{ir} , H_{rj}) \otimes 1_{H_{rj}}) V_{rj}^i (1_{H_{ij}} \otimes {\cal K}(H_{ir} , H_{rj}))]$).
\end{enumerate}
\item Il y a \'equivalence entre les conditions suivantes :\begin{enumerate}\renewcommand\theenumii{\roman{enumii}}
\renewcommand\labelenumii{\rm ({\theenumii})}
\item ${\cal K}(H_{ij} , H_{ik}) \subset  {\cal C}(W_{ij}^k)$ \quad (resp. ${\cal K}(H_{ij} , H_{ik})=  {\cal C}(W_{ij}^k)$).
\item ${\cal K}(H_{ij} \otimes H_{jk} , H_{ik} \otimes H_{ik}) \subset 
[({\cal K}(H_{ij}, H_{ik})  \otimes  1_{H_{ik}})\, W_{ij}^k (1_{H_{ij}} \otimes {\cal K}(H_{jk}))] $
\hfill\break
(resp. ${\cal K}(H_{ij} \otimes H_{jk} , H_{ik} \otimes H_{ik}) = 
[({\cal K}(H_{ij}, H_{ik})  \otimes  1_{H_{ik}})\, W_{ij}^k (1_{H_{ij}} \otimes {\cal K}(H_{jk}))]$).
\end{enumerate}
\end{enumerate}
\end{corollary}

\begin{proof} Pour le a), appliquer   \ref{reg} avec $X:= V_{rj}^i$, $H:= H_{rj}$, $K:=H_{ij}$, $L:=H_{ir}$ et aussi $E' := H_{ir}$, $E:=H_{rj}$.
\hfill\break
Pour le b), appliquer \ref{reg} avec $Y:= W_{ij}^k$, $H:= H_{ij}$, $K:=H_{jk}$, $L:=H_{ik}$ et aussi $E' := H_{ik}$, $E:=H_{jk}$.
\end{proof}

Par un calcul direct, on obtient :
\begin{lemme} Pour tout $\omega\in B(H)_\ast$   et  tout $i , j , k ,  r = 1 , 2$, nous avons :
$$p_{ik}({\rm id} \otimes \omega)(\Sigma W) p_{ij} =     p_{ik}({\rm id} \otimes p_{jk}\omega p_{ij})(\Sigma_{ij \otimes ik} W_{ij}^k)  p_{ij}  , \quad p_{ik} [R_\xi^\alpha (R_\eta^\alpha)^* \,\,|\,\,\xi , \eta \in H] p_{ij} = {\cal K}(H_{ij} , H_{ik})\, ;$$
\noindent
$$p_{rj}({\rm id} \otimes \omega)(\Sigma V) p_{ij} =     p_{rj}({\rm id} \otimes p_{rj} \omega p_{ir}) (\Sigma_{ir \otimes rj}\, V_{rj}^i)  p_{ij} , \quad p_{rj} [L_\xi^\beta (L_\eta^\beta)^* \,\,|\,\,\xi , \eta \in H] p_{ij} = {\cal K}(H_{ij} , H_{rj})$$
(On note $\Sigma_{ij\otimes ik}:=\Sigma_{H_{ij}\otimes H_{ik}}$ et $\Sigma_{ir\otimes rj}:=\Sigma_{H_{ir}\otimes H_{rj}}$).
\end{lemme}

\begin{corollary}\label{reg group}
Il y a \'equivalence entre les conditions suivantes :
\begin{enumerate}
\item
${\cal G}$ est semi-r\'egulier (resp. r\'egulier).
\item Pour tout $i , j , r$, on a  
${\cal K}(H_{ij} , H_{rj}) \subset {\cal C}(V_{rj}^i)$ (resp. ${\cal K}(H_{ij} , H_{rj}) = {\cal C}(V_{rj}^i)$).
\item Pour tout $i , j , k$, on a
${\cal K}(H_{ij} , H_{ik}) \subset  {\cal C}(W_{ij}^k)$ (resp. ${\cal K}(H_{ij} , H_{ik})=  {\cal C}(W_{ij}^k)$).
\end{enumerate}
\begin{proof} Appliquer \ref{dualreg} et le lemme pr\'ec\'edent.
\end{proof}
\end{corollary} 
\begin{lemme} \label{cv}
Soient $i , j ,  r , s = 1 , 2$.
\begin{enumerate}
\item ${\cal C}(V_{rj}^i) H_{ij} = H_{rj}$.
\item
${\cal C}(V_{rj}^i)^* =  {\cal C}(V_{ij}^r)  , \quad  [{\cal C}(V_{rj}^i)  {\cal C}(V_{ij}^s)] ={\cal C}(V_{rj}^s)$.
\end{enumerate}
\begin{proof} 
Il r\'esulte de l'\'egalit\'e $S \widehat S = U {\cal C}(V) U^*$, qu'on a :
$${\cal C}(V_{rj}^i) =  [R(S)_{rj} E_{ri, \lambda}^j] \quad   \quad {\rm avec} \quad    R(S)_{rj}:= R(S)p_{rj} 
  \quad {\rm et} \quad E_{ri, \lambda}^j = p_{rj}  \lambda(\widehat S)p_{ij}.$$
\noindent
On d\'eduit alors de \ref{ShatS} qu'on a :
$$[L(S)_{ji} E_{ir}^j]   =  [E_{ir}^j L(S)_{jr}]\quad   \quad {\rm avec} \quad    L(S)_{ji}:= S p_{ji} 
  \quad {\rm et} \quad E_{ir }^j = p_{ji}   \widehat S p_{jr}.$$
\noindent
En appliquant ${\rm Ad} U$  \`a l'\'egalit\'e  pr\'ec\'edente, on obtient 
$[R(S)_{ij} E_{ir, \lambda}^j]   =  [E_{ir, \lambda}^j R(S)_{rj}]$.\hfill\break
De (\ref{Emorita} c) et d)), il r\'esulte alors :
\hfill\break
- ${\cal C}(V_{rj}^i) H_{ij} = [R(S)_{rj} E_{ri, \lambda}^j] H_{ij} =  H_{rj}$, car $R(S)_{rj}$ est une sous-C*-alg\`ebre non d\'eg\'en\'er\'ee de $B(H_{rj})$.\hfill\break
 - ${\cal C}(V_{rj}^i)^* = [R(S)_{rj} E_{ri, \lambda}^j]^* = [E_{ir, \lambda}^j R(S)_{rj}] = [R(S)_{ij} E_{ir, \lambda}^j] = {\cal C}(V_{ij}^r)$.\hfill\break
 - ${\cal C}(V_{rj}^i)  {\cal C}(V_{ij}^s) = [R(S)_{rj} E_{ri, \lambda}^j][R(S)_{ij} E_{is, \lambda}^j] =
 [R(S)_{rj} E_{ri, \lambda}^j R(S)_{ij} E_{is, \lambda}^j] 
= [R(S)_{rj} R(S)_{rj} E_{ri, \lambda}^j E_{is, \lambda}^j] =[R(S)_{rj}E_{rs, \lambda}^j] = {\cal C}(V_{rj}^s)$.
\end{proof}
\end{lemme}

\begin{proposition} \label{reg-car}Pour $i , j = 1 , 2$, il y a \'equivalence entre les conditions suivantes : 
\begin{enumerate}
\item
$ {\cal K}(H_{jj}) \subset  {\cal C}(V_{jj}^j)\,$ (resp. $ {\cal K}(H_{jj}) = {\cal C}(V_{jj}^j)\,$).
\item  $ {\cal K}(H_{jj} , H_{ij}) \subset  {\cal C}(V_{ij}^j)$ (resp. $ {\cal K}(H_{jj} , H_{ij}) =  {\cal C}(V_{ij}^j)\,$).
\item  Le groupe quantique $G_j$ est semi-r\'egulier (resp. r\'egulier).
\end{enumerate}
\begin{proof} l'implication  $a) \Longrightarrow b)$  r\'esulte  \ref{cv} de ${\cal C}(V_{ij}^j) H_{jj} = H_{ij}$ et de  
${\cal C}(V_{ij}^j) = [{\cal C}(V_{ij}^j) {\cal C}(V_{jj}^j)]$.
\hfill\break
l'implication  $b) \Longrightarrow a)$  r\'esulte \ref{cv} de ${\cal C}(V_{ij}^j)^* H_{ij} = {\cal C}(V_{jj}^i)H_{ij} = H_{jj}$ et de  
${\cal C}(V_{jj}^j) = [{\cal C}(V_{jj}^i) {\cal C}(V_{ij}^j)]$.
\hfill\break
Le c) r\'esulte de l'\'equivalence des conditions a) et b).
\end{proof}
\end{proposition}

\noindent
Le r\'esultat principale de ce paragraphe est le suivant :

\noindent
\begin{theorem}\label{thregu}Soit ${\cal G}_{G_1, G_2}$ un groupo\"ide de co-liaison associ\'e \`a deux groupes quantiques \lc  $G_1$ et $G_2$ mono\"idalement \'equivalents. Alors il y a \'equivalence entre les conditions suivantes :
\begin{enumerate}
\item $G_1$ et $G_2$ sont semi-r\'eguliers (resp. r\'eguliers).
\item Le groupo\"ide ${\cal G}_{G_1, G_2}$ est semi-r\'egulier (resp. r\'egulier).
\end{enumerate}
\begin{proof}
l'implication  $b) \Longrightarrow a)$  r\'esulte de \ref{reg group}.

R\'eciproquement, supposons que pour $j \in \{ 1 , 2\}$,  on a ${\cal K}(H_{jj}) \subset{\cal C}(V_{jj}^j)$ (resp. 
${\cal K}(H_{jj}) ={\cal C}(V_{jj}^j)$) et montrons  (\ref{reg group} b)) que  pour tout $i ,  r = 1 , 2$, on a 
${\cal K}(H_{ij}, H_{rj}) \subset {\cal C}(V_{rj}^i)$ (resp.  ${\cal K}(H_{ij}, H_{rj}) = {\cal C}(V_{rj}^i)$).
\hfill\break
Il r\'esulte de \ref{cv} que pour tout $i , j , r= 1,2$, on a  
   ${\cal C}(V_{rj}^i) H_{ij} = H_{rj}$. On en d\'eduit :
$${\cal C}(V_{rj}^j) H_{jj} = H_{rj} , \quad {\cal C}(V_{ij}^j)H_{jj} = H_{ij}.$$
\noindent
Tout op\'erateur compact  $k\in {\cal K}(H_{ij}, H_{rj})$ est alors de la forme  :
$$k = t \, k' \, t'^*  \quad ;  \quad 
t\in {\cal C}(V_{rj}^j) , \, k'\in {\cal K}(H_{jj}) , \, t'\in {\cal C}(V_{ij}^j).$$
\noindent
Mais, par  \ref{cv}, on a $[{\cal C}(V_{rj}^j) {\cal C}(V_{jj}^j) {\cal C}(V_{ij}^j)^*] = [{\cal C}(V_{rj}^j) {\cal C}(V_{jj}^j) {\cal C}(V_{jj}^i)] ={\cal C}(V_{rj}^i) $, d'o\`u l'implication $a) \Longrightarrow b)$.
\end{proof}
\end{theorem}
\noindent
\begin{corollary}  Soit ${\cal G}_{G_1, G_2}$ un groupo\"ide de co-liaison associ\'e \`a deux groupes quantiques \lc  $G_1$ et $G_2$ mono\"idalement \'equivalents. Alors $G_1$ et $G_2$ sont semi-r\'eguliers (resp. r\'eguliers) si et seulement si pour tout $i , j , k = 1 , 2$, on a
$${\cal K}(H_{ij} \otimes H_{jk} , H_{ik} \otimes H_{ik}) \subset 
[({\cal K}(H_{ij}, H_{ik})  \otimes  1_{H_{ik}})\, W_{ij}^k (1_{H_{ij}} \otimes {\cal K}(H_{jk}))]$$
\noindent
$$(\text{resp. }{\cal K}(H_{ij} \otimes H_{jk} , H_{ik} \otimes H_{ik}) = 
[({\cal K}(H_{ij}, H_{ik})  \otimes  1_{H_{ik}})\, W_{ij}^k (1_{H_{ij}} \otimes {\cal K}(H_{jk}))]).$$
\end{corollary}

\begin{proof}C'est une cons\'equence de \ref{cor-reg} ,  \ref{reg group} et \ref{thregu}.\end{proof}

\begin{remarks}\label{larem} Soient   $G_1$ et $G_2$ deux groupes quantiques mono\"idalement \'equivalents.
\begin{enumerate} 
\item Nous verrons que l'inclusion
$${\cal K}(H_{21} \otimes H_{21} , H_{21} \otimes H_{11}) \subset [(1_{H_{21}} \otimes {\cal K}(H_{11})) (W_{21}^1)^* ({\cal K}(H_{21})  \otimes 1_{H_{21}})]$$
 \noindent
joue un r\^ole crucial pour \'etablir, dans le cas r\'egulier,  l'\'equivalence des actions continues de $G_1$ et $G_2$.
\item  Le c) de \ref{reg-car},  montre que l'\'egalit\'e ${\cal K}(H_{22} , H_{12}) ={\cal C}(V_{12}^2)$ n'entra\^ine pas en g\'en\'eral que $G_1$ est r\'egulier,   ce qui permet de r\'epondre 
n\'egativement \`a une question de \cite{NT}.
\end{enumerate}
\end{remarks}

\section {Action continue d'un groupo\"ide  quantique sur une base de dimension finie}

Dans ce paragraphe, nous introduisons la notion d'action continue dans une C*-alg\`ebre, d'un groupo\"ide \mq    de base   une
C*-alg\`ebre  de dimension finie. Nous d\'efinissons le produit crois\'e correspondant,  qu'on munit d'une action continue du groupo\"ide dual. 

Dans le cas d'un groupo\"ide de co-liaison, nous montrons qu'une action continue d'un tel groupo\"ide, nous donne une action continue de chacun des groupes quantiques \lc sous-jacents. G\'en\'eralisant le cas de l'action triviale \cite{DeC1}, nous montrons que les produits crois\'es correspondants \`a ces actions de groupes quantiques, sont Morita \'equivalents.

Enfin, nous g\'en\'eralisons la bidualit\'e de Takesaki-Takai pour les actions continues d'un groupo\"ide \mq r\'egulier. Nous terminons ce paragraphe par une description utile pour la suite, du double produit crois\'e par une action d'un groupo\"ide de co-liaison.

\begin{notation} On se fixe un groupo\"ide mesur\'e  ${\cal G} = (N, M,  \alpha, \beta, \delta, T, T', \epsilon)$  de base   $N = \oplus_{l=1}^k M_{n_l}$ munie de la trace $\epsilon = \oplus_{l=1}^k n_l {\rm Tr}_{M_{n_l}}$ et on reprend les notations de 
\ref{2.2}. 
\end{notation}

\subsection{Action continue. Produit crois\'e}

\begin{definition}\label{ac} Soit $A$ une C*-alg\`ebre. Une action continue du groupo\"ide \mq   ${\cal G}$ dans $A$ est un couple $(\beta_A, \delta_A)$, o\`u $\beta_A : N^{\rm o} \rightarrow  M(A)$ est un *-morphisme non d\'eg\'en\'er\'e et $\delta_A : A  \rightarrow   M(A \otimes    S)$ est un *-morphisme injectif  v\'erifiant : 
\begin{enumerate}
\item $ \displaystyle \delta_A(1_A) = q_{\beta_A, \alpha} = \sum_{l=1}^k \frac{1}{n_l} \sum_{i, j=1}^{n_l}    
 \beta_A(e_{ji}^{(l){\rm o}}) \otimes \alpha(e_{ij}^{(l)})$ \quad  (\cf\ref{d(1)}) ;
\item  $(\delta_A \otimes {\rm id}_S) \circ \delta_A =  ({\rm id}_A \otimes \delta)\circ \delta_A$ ; 
\item Pour tout $n\in N$, on a $\delta_A \circ\beta_A(n^{\rm o}) = q_{\beta_A, \alpha} (1_A \otimes \beta(n^{\rm o}))$ ;
\item   $[\delta_A(A)(1_A \otimes S)] = q_{\beta_A, \alpha} (A\otimes S)$.
\end{enumerate}
Une C*-alg\`ebre $A$ munie d'une action continue $(\beta_A, \delta_A)$ du groupo\"ide ${\cal G}$ est appel\'ee une 
${\cal G}$-alg\`ebre.
\end{definition}

\begin{remark} La condition a)   est \'equivalente \`a   $q_{\beta_A, \alpha}(A \otimes S) = [\delta_A(A)(A \otimes S)]$. Elle entra\^ine que les  *-homomorphismes  $ \delta_A \otimes {\rm id}_S$ et ${\rm id}_A \otimes \delta$,  se prolongent de fa\c con unique  \`a $M(A \otimes  S)$, en des *-homomorphismes strictement continus et v\'erifiant respectivement $(\delta_A \otimes {\rm id}_S)(1_A) = q_{\beta_A, \alpha, 12}$ et $({\rm id}_A \otimes \delta)(1_A) = q_{\beta, \alpha, 23}.$
 
\noindent
On a alors  : \hfill\break
 - $(\delta_A \otimes {\rm id}_S)\delta_A(1_A) = ({\rm id}_A \otimes \delta)\delta_A(1_A) = q_{\beta_A, \alpha, 12} \,\, q_{\beta, \alpha, 23}  = q_{\beta, \alpha, 23} \,\, q_{\beta_A, \alpha, 12}$ ;\hfill\break
 - $\delta_A \circ\beta_A(n^{\rm o}) = (1_A \otimes \beta(n^{\rm o})) q_{\beta_A, \alpha} =(1_A \otimes \beta(n^{\rm o}))  \delta_A(1_A) , \quad n\in N$ ;\hfill\break
 - $[(1_A \otimes S) \delta_A(A)] =  (A \otimes S)q_{\beta_A, \alpha} =  (A \otimes S) \delta_A(1_A)$.
 \end{remark}

\begin{examples}\label{2ex}
\begin{enumerate}
 \item  L'action  triviale du groupo\"ide ${\cal G}$. Il s'agit de l'action    de ${\cal G}$  dans la C*-alg\`ebre $A:= N^{\rm o}$ d\'efinie par les *-morphismes :
$$\beta_{N^{\rm o}} : N^{\rm o} \rightarrow N^{\rm o} : n^{\rm o} \mapsto n^{\rm o} \quad ; \quad \delta_{N^{\rm o}} : N^{\rm o} \rightarrow M(N^{\rm o} \otimes  S) : x^{\rm o}  \mapsto  q_{\beta_{N^{\rm o}},\alpha}(1_{N^{\rm o}} \otimes  \beta(x^{\rm o})).$$
 \item L'action    du groupo\"ide ${\cal G}$ dans lui-m\^eme. On prend $A:= S , \beta_A := \beta$ et $\delta_A := \delta$ le coproduit de $S$.
\end{enumerate}
\end{examples}

\begin{definition}\label{morgrou}  Soient $(A, \beta_A, \delta_A)$ et $(B, \beta_B, \delta_B)$ deux ${\cal G}$-alg\`ebres. On dit qu'un *-morphisme non d\'eg\'en\'er\'e 
$f : A \rightarrow M(B)$ est ${\cal G}$-\'equivariant si on a :
$$(f \otimes {\rm id}_S) \delta_A = \delta_B \circ f \quad ; \quad  f \circ \beta_A  =\beta_B.$$
\noindent
\end{definition}
\begin{notation} On note $A^{\cal G}$\index{ad@$A^{\cal G}$} la cat\'egorie dont les objets sont les ${\cal G}$-alg\`ebres,  et les morphismes sont les *-morphismes 
$f : A \rightarrow M(B)$ non d\'eg\'en\'er\'es et ${\cal G}$-\'equivariants, \ie $f \circ \beta_A  =\beta_B \,$ et $(f \otimes \id)\delta_A  = \delta_B \circ f$.  
\end{notation}

Fixons dans ce qui suit  une action continue $(\beta_A, \delta_A)$ du groupo\"ide ${\cal G}$ dans une C*-alg\`ebre $A$.
Comme cons\'equence  de la condition (\ref{ac} a)), on a facilement :
\begin{lemme-notation}\label{piL}
La repr\'esentation de $A$ dans le C*-module $A \otimes H$ d\'efinie par  $\pi_L := ({\rm id}_A \otimes L) \circ \delta_A $\index{pf@$\pi_L$, $\widehat{\theta}$} se prolonge de fa\c con unique en une repr\'esentation  $\pi_L : M(A) \rightarrow   {\cal L}(A \otimes H)$ fid\`ele,  continue pour les topologies stricte/*-forte et v\'erifiant 
$\pi_L(1_A) =   
 q_{\beta_A, \alpha}$. De plus, on a  
$$\pi_L(m) =  \pi_L(m) q_{\beta_A, \alpha} = q_{\beta_A, \alpha} \pi_L(m),\quad m\in M(A).$$
\end{lemme-notation}
Introduisons le C*-module\index{ee@${\cal E}_{A,L}$} 
\begin{equation}\label{eL} {\cal  E}_{A, L} := q_{\beta_A, \alpha} (A \otimes H).
\end{equation} 
Par restriction de  $\pi_L$, on obtient une repr\'esentation unitale, fid\`ele et continue pour les topologies stricte/*-forte, not\'ee
\begin{equation}\label{piMA} \pi : M(A) \rightarrow {\cal L}( {\cal  E}_{A, L})   : m \mapsto \restr{\pi_L(m)}{{\cal  E}_{A,L}}.
\end{equation}
Pour tout $T\in M(\widehat S)$ , on a $[1_A \otimes T , q_{\beta_A, \alpha}] = 0$ (\ref{2.10} a)). On en d\'eduit 
  une repr\'esentation unitale   continue pour les topologies stricte/*-forte :
\begin{equation}\label{thetaMwS} \widehat\theta : M(\widehat S) \rightarrow {\cal L}( {\cal  E}_{A,L}) :  T \mapsto\widehat\theta(T) : = \restr{(1_A \otimes T)}{{\cal  E}_{A,L}}.
\end{equation}  

\noindent
\begin{definition}\label{pc} On appelle  produit crois\'e de la C*-alg\`ebre $A$ par l'action $(\beta_A, \delta_A)$ du groupo\"ide mesur\'e ${\cal G}$,    la sous-C*-alg\`ebre not\'ee $A \rtimes {\cal G}$\index{ae@$A \rtimes {\cal G}$}  de ${\cal L}( {\cal  E}_{A, L})$ engendr\'ee par les produits $\pi(a) \widehat\theta(x)$, o\`u $a\in A$ et $x\in\widehat S$.
\end{definition}

\begin{lemme}\label{pclem} Le produit crois\'e   $A \rtimes {\cal G}$ est l'adh\'erence de l'espace vectoriel ferm\'e engendr\'e par les produits $\pi(a) \widehat\theta(x)$, o\`u $a\in A$ et $x\in\widehat S$.
\end{lemme}

\begin{proof}
Soient $\omega\in B(H)_\ast$ et $a\in A$. Posons $x = (\id \otimes \omega)(V)$. Nous avons :
\begin{align} (1_A \otimes x) \pi_L(a) &= (\id_A \otimes \id_H \otimes \omega) (V_{23} \pi_L(a)_{12}) 
= (\id_A \otimes \id_H \otimes \omega)(V_{23}V_{23}^*V_{23} \pi_L(a)_{12}) \nonumber\\
&= (\id_A \otimes \id_H \otimes \omega)(V_{23} \pi_L(a)_{12}V_{23}^*V_{23}) \quad (\ref{2.9}\, c)) \nonumber \\
&= 
(\id_A \otimes \id_H \otimes \omega)[(\id_A \otimes L \otimes L)(\id_A \otimes \delta)(\delta_A(a))] 
= (\id_A \otimes \id_H \otimes \omega)[(\pi_L \otimes L)(\delta_A(a))]. \nonumber
\end{align}
En posant $\omega =   \omega'  s$ avec $\omega'\in B(H)_\ast$ et $s \in S$, on obtient  que 
$ (1_A \otimes x) \pi_L(a)$ est limite normique de sommes finies $\sum_i \, \pi_L(a_i) (1_A \otimes x_i)$ avec  $a_i\in A$ et $x_i \in \widehat S$.
\end{proof}

 Le r\'esultat pr\'ec\'edent montre clairement que $\pi$ (resp.\ $\widehat\theta$)  d\'efinit un *-morphisme 
$\pi : M(A) \rightarrow  M(A \rtimes {\cal G})$ (resp.\  $\widehat\theta : M(\widehat S) \rightarrow  M(A \rtimes {\cal G})$)  
unital et   strictement continu.\index{pg@$\pi$, $\widehat{\theta}$} 

Notons aussi qu'on a une unique repr\'esentation fid\`ele, continue pour les topologies stricte/*-forte :\index{ja@$j_{A\rtimes{\cal G}}$}
\begin{equation} 
j_{A\rtimes{\cal G}} : M(A\rtimes{\cal G}) \rightarrow {\cal L}(A \otimes H) , \quad j_{A\rtimes{\cal G}}(1_{A\rtimes{\cal G}}) = q_{\beta_A,\alpha}, 
\end{equation} 
v\'erifiant :  
\begin{equation} 
(j_{A\rtimes{\cal G}} \circ \pi)(m) = ({\rm id}_A \otimes L) \delta_A(m),\quad m\in M(A)\quad ; \quad
  (j_{A\rtimes{\cal G}} \circ \widehat\theta)(T) = q_{\beta_A, \alpha} (1_A \otimes T),\quad T\in M(\widehat S).
\end{equation}  

\noindent
On d\'efinit de m\^eme le produit crois\'e d'une C*-alg\`ebre par une action continue   du groupo\"ide mesur\'e dual  $\widehat{\cal  G}$ (\cf\ref{RemNot}).

\noindent
\begin{definition} Soit $B$ une C*-alg\`ebre. Une action continue  du groupo\"ide mesur\'e dual  $\widehat{\cal  G}$ dans $B$ est un couple $(\alpha_B, \delta_B)$, où  $\alpha_B : N  \rightarrow  M(B)$  est un *-morphisme non d\'eg\'en\'er\'e et $\delta_B : B  \rightarrow   M(B \otimes    \widehat S)$ est un *-morphisme injectif v\'erifiant :
\begin{enumerate}
\item  $ \displaystyle \delta_B(1_B) = q_{\alpha_B, \beta} = \sum_{l=1}^k \frac{1}{n_l}\sum_{i, j=1}^{n_l}
 \alpha_B(e_{ji}^{(l)}) \otimes \beta(e_{ij}^{(l){\rm o}})$ \quad (\cf\ref{d(1)}) ;
\item  $(\delta_B \otimes {\rm id}_{\widehat S}) \circ \delta_B =  ({\rm id}_B \otimes \widehat\delta)\circ \delta_B$ ;
\item  Pour tout $n\in N$, on a $\delta_B \circ\alpha_B(n) = q_{\alpha_B, \beta} (1_B \otimes  \widehat\alpha(n))$ ;
\item  $[\delta_B(B)(1_B \otimes \widehat S)] = q_{\alpha_B, \beta} (B\otimes  \widehat S)$.
\end{enumerate}
Une C*-alg\`ebre $B$ munie d'une action continue $(\alpha_B, \delta_B)$ du groupo\"ide $\widehat{\cal G}$ est appel\'ee une 
$\widehat{\cal G}$-alg\`ebre.
\end{definition}
La condition a)  entra\^ine que les  *-morphismes  $ \delta_B \otimes {\rm id}_{\widehat S}$ et ${\rm id}_B \otimes \widehat\delta$ se prolongent de fa\c con unique  \`a  $M(B \otimes  \widehat S)$, en des *-morphismes strictement continus v\'erifiant :
$$(\delta_B \otimes {\rm id}_{\widehat S})(1_B) = q_{\alpha_B, \beta, 12} , \quad ({\rm id}_B \otimes \widehat\delta)(1_B) = q_{\widehat\alpha, \beta, 23}.$$
\noindent
Pour d\'efinir le produit crois\'e  $B \rtimes \widehat{\cal  G}$, on introduit le  C*-module\index{ef@${\cal E}_{B,\lambda}$} 
\begin{equation}\label{Elambda}  
{\cal  E}_{B, \lambda}: = q_{ \alpha_B, \widehat\beta}(B   \otimes H).
\end{equation} 
Exactement comme dans le cas d'une action du groupo\"ide ${\cal G}$, on a :\hfill\break  
  -  un   *-morphisme  $\widehat\pi  : M(B) \rightarrow {\cal L}( {\cal  E}_{B, \lambda})$ unital, injectif et continu pour les topologies stricte/*-forte  v\'erifiant $\widehat\pi(m) = \restr{ ({\rm id  }_B \otimes \lambda) \delta_B(m) }{{\cal  E}_{B, \lambda}}$ ;\hfill\break
 -   un   *-morphisme  $\theta   : M(S) \rightarrow {\cal L}( {\cal  E}_{B, \lambda})$ unital  et continu pour les topologies stricte/*-forte v\'erifiant $\theta(T) = \restr{({\rm id  }_B \otimes T)}{ {\cal  E}_{B, \lambda}}$.\index{ph@$\widehat{\pi}$, $\theta$}\hfill\break
Comme dans le cas d'une action de ${\cal G}$, on a :

\begin{proposition-definition}\label{copc}
L'adh\'erence du sous-espace vectoriel de $\cL({\cal  E}_{B, \lambda})$   engendr\'e par les produits $\widehat\pi(b)  \theta(x) \,,\, b\in B \,{\rm et }\,\, x\in  S$,  est une sous-C*-alg\`ebre, qu'on appelle produit crois\'e de $B$ par l'action de $\widehat {\cal G}$ et qu'on note   $B \rtimes \widehat{\cal  G}$.\index{bb@$B\rtimes\widehat{\cal G}$}
\end{proposition-definition}

\subsection{Action duale}

\subsubsection{Cas d'une action continue du groupo\"ide \texorpdfstring{${\cal G}$}{G}}

Soit $(\beta_A, \delta_A)$ une action continue du groupo\"ide ${\cal G}$ dans une C*-alg\`ebre $A$.\hfill\break
Posons $B = A \rtimes {\cal G}$ et $ {\cal  E}_A :=  {\cal  E}_{A, L}$. Remarquons (\ref{2.10} b)) qu'on a 
$[q_{\beta_A, \alpha,12} , \widetilde V_{23}] = 0$ dans $\cL(A \otimes H \otimes H)$. Soit $X\in {\cal L}( {\cal  E}_A \otimes  H)$ l'isom\'etrie partielle d\'efinie par 
$$X:=  \restr{\widetilde V_{23}}{ {\cal E}_{A}\otimes H}. $$
On a   (\ref{2.9} c)) :
\begin{equation} 
X^*X = \restr{q_{\widehat \beta  , \widehat \alpha ,23}}{ {\cal E}_{A}\otimes H}  , \quad X X^* = \restr{q_{\widehat \alpha , \beta ,23}}{ {\cal E}_{A}\otimes H} = (\widehat\theta \otimes {\rm id}_{\widehat S})(q_{\widehat\alpha , \beta}).
\end{equation}
 Soient  $ \delta_B : B \rightarrow {\cal L}( {\cal  E}_A \otimes H)\,$   et  $\alpha_B : N  \rightarrow {\cal L}({\cal E}_A)\,$, les applications lin\'eaires   d\'efinies par :
$$\delta_B(b) = X(b \otimes 1_H) X^*,\quad b\in B \quad ; \quad \alpha_B(n) = \widehat\theta(\widehat\alpha(n)) = \restr{(1_A  \otimes \widehat\alpha(n))}{{\cal E}_{A}},\quad n\in N. $$\noindent
On a :

\noindent
\begin{proposition} Le couple $(\alpha_B, \delta_B)$ d\'efinit une action continue du groupo\"ide mesur\'e $\widehat{\cal  G}$ dans la C*-alg\`ebre $B = A \rtimes {\cal G}$.
\begin{proof} Il est clair que $\alpha_B : N  \rightarrow M(B)$ est un *-morphisme et qu'on a $XX^* = 
q_{\alpha_B, \beta} =  (\widehat\theta \otimes {\rm id}_{\widehat S})(q_{\widehat\alpha , \beta})$.\hfill\break
Pour voir que $\delta_B$ est \`a valeurs dans $M(B \otimes \widehat S)$, consid\'erons $M(B \otimes \widehat S)$ comme une sous-C*-alg\`ebre de ${\cal L}({\cal E}_A \otimes H)$. Pour tout $a\in A$ et tout $x\in \widehat S$, on a  (\ref{mult} et \ref{2.10} d)) :
$$X (\pi(a) \otimes 1_H) X^* = (\pi(a) \otimes 1_{\widehat S}) (\widehat\theta \otimes {\rm id}_{\widehat S})(q_{\widehat\alpha, \beta}),\quad X (\widehat\theta(x) \otimes 1_H) X^* = (\widehat\theta \otimes {\rm id}_{\widehat S}) \widehat\delta(x).$$\noindent
Comme $X^*X = \restr{q_{\widehat \beta  , \widehat \alpha ,23}}{ {\cal E}_{A}\otimes H}$, on a   (\ref{2.10} b))  pour tout  
$a\in A$, $x\in \widehat S$ et $b\in M(B)$ :
\begin{equation}\label{acd} [X^*X, b \otimes 1_H] = 0    \quad {\rm et} \quad  X (\pi(a) \widehat\theta(x) \otimes 1_H) X^* =  (\pi(a) \otimes 1_{\widehat S}) (\widehat\theta \otimes {\rm id}_{\widehat S}) \widehat\delta(x).
\end{equation}
Il en r\'esulte que $ \delta_B : B \rightarrow M(B \otimes \widehat S)\,$  est un *-morphisme et on a $\delta_B(1_B) = q_{\alpha_B, \beta}$. \hfill\break
Montrons qu'il est injectif. Soit $b\in B$ tel que $\delta_B(b) = 0$.\hfill\break
On a alors $X^*X (b \otimes 1_H) X^* = (b \otimes 1_H) X^*X X^* = (b \otimes 1_H) X^*=0$. Il en r\'esulte que pour tout $\omega\in B(H)_\ast$, on a $\,\,b ({\rm id} \otimes \omega)(X^*) = 0$. Comme $[({\rm id} \otimes \omega)(X^*) \,|\, 
 \omega\in B(H)_\ast] = \restr{(1_A \otimes R(S))}{ {\cal E}_{A} }$, on a $b=0$. \hfill\break
Montrons la continuit\'e de $\delta_B$. \hfill\break
Pour $a\in A$ et $x , x'\in\widehat S$, on a :
$$\delta_B(\pi(a) \widehat\theta(x))(1_B \otimes x') = (\pi(a) \otimes 1_{\widehat S}) (\widehat\theta \otimes {\rm id}_{\widehat S}) (\widehat\delta(x) (1_{\widehat S} \otimes x')),$$\noindent
donc $[\delta_B(B) (1_B \otimes \widehat S) ] \subset \delta_B(1_B) (B \otimes \widehat S) $.\hfill\break
Pour montrer l'inclusion $\delta_B(1_B) (B \otimes \widehat S) \subset 
[\delta_B(B) (1_B \otimes \widehat S) ]$, remarquons  qu'on a : 
\begin{align} \delta_B(1_B) (B \otimes \widehat S) &= \delta_B(1_B) [(\pi(a) \otimes 1_{\widehat S}) (\widehat\theta(x) \otimes x') |\,\, a\in A \,\,;\,\,x\,,\,x'\in \widehat S] \nonumber\\
&=   [ \delta_B(1_B) (\pi(a) \otimes 1_{\widehat S}) (\widehat\theta \otimes {\rm id}_{\widehat S})(q_{\widehat\alpha , \beta} (x  \otimes x')) \,|\,\, a\in A \,\,;\,\,x\,,\,x'\in \widehat S]\nonumber\\
&= [\delta_B(1_B) (\pi(a) \otimes 1_{\widehat S}) (\widehat\theta \otimes {\rm id}_{\widehat S})(\widehat\delta(u) (1_{\widehat S} \otimes u'))\, |  \,\, a\in A \,\,;\,\,u\,,\,u'\in \widehat S] \quad  (\ref{simp})\nonumber\\
 &=
   [\delta_B(\pi(a)\widehat\theta(u)) (1_B \otimes u')\, |  \,\, a\in A \,\,,\,\,u\,,\,u'\in \widehat S] = [\delta_B(B) (1_B \otimes \widehat S) ]. \nonumber
   \end{align}
Montrons la coassociativit\'e de $\delta_B$. Pour tout $c\in M(B \otimes \widehat S)$,  on a  :
$$(\delta_B \otimes {\rm id}_{\widehat S})(c)  =  X_{12} c_{13} X_{12}^* \quad {\rm  et } \quad 
({\rm id}_B \otimes \widehat\delta)(c) = \widetilde V_{23} c_{12} \widetilde V_{23}^*.$$\noindent
La coassociativit\'e de $\delta_B$ r\'esulte de la relation pentagonale v\'erifi\'ee par $\widetilde V$  et  des relations (\cf\ref{2.10} b)):     
$$X_{12} X_{13} = \widetilde V_{23} X_{12} \widetilde V_{23}^* \quad {\rm et } \quad [X_{12},  \widetilde V_{23}^*\widetilde V_{23}] = 0 \quad {\rm dans  } \quad
 {\cal L}({\cal E}_A \otimes H \otimes H).$$\noindent
Finalement, pour tout $n\in N$, on a  \ref{2.10} b) :\hfill\break
$\delta_B(\alpha_B(n)) = X (\widehat\theta(\widehat\alpha(n)) \otimes 1_H) X^* = 
(\widehat\theta \otimes {\rm id}_{\widehat S})\widehat\delta(\widehat\alpha(n)) = \delta_B(1_B)(1_B \otimes 
\widehat\alpha(n)).$
\end{proof}
\end{proposition}
\begin{definition}On appelle action duale du   groupo\"ide  $\widehat{\cal  G}$ dans le produit crois\'e $B=A \rtimes {\cal G}$, l'action d\'efinie par le couple $(\alpha_B, \delta_B)$.
\end{definition}

\subsubsection{Cas d'une action continue du groupo\"ide dual \texorpdfstring{$\widehat{\cal G}$}{}} 

Soit $(\alpha_B, \delta_B)$ une action continue    du groupo\"ide   $\widehat{\cal  G}$ dans une C*-alg\`ebre $B$.\hfill\break
Posons $C = B \rtimes \widehat{\cal  G}$ et ${\cal E}_B := {\cal E}_{B, \lambda}$ (\cf \ref{Elambda}). Remarquons (\ref{2.10} b)) qu'on a $[q_{\alpha_B, \widehat\beta,12} ,  V_{23}] = 0$ dans $\cL(B \otimes H \otimes H)$.
Soit $Y\in {\cal L}({\cal E}_B \otimes  H)$ l'isom\'etrie partielle d\'efinie par 
$$Y:=  \restr{V_{23}}{ {\cal E}_{B}\otimes H }. $$\noindent
On a   (\ref{2.9} c)) :
\begin{equation}
Y^*Y = \restr{q_{\widehat \alpha  , \beta ,23}}{ {\cal E}_{B}\otimes H } , \quad Y Y^* =\restr{q_{ \beta  , \alpha ,23}}{ {\cal E}_{B}\otimes H } = (\theta \otimes {\rm id}_{S})(q_{\beta, \alpha}). 
\end{equation}
Soient  $ \delta_C : C \rightarrow {\cal L}({\cal E}_B \otimes H)$   et  $\beta_C : N^{\rm o}  \rightarrow {\cal L}({\cal E}_B)$, les applications lin\'eaires d\'efinies par :
$$\delta_C(c) = Y(c \otimes 1_H) Y^*, \quad c\in C \quad ; \quad \beta_C(n^{\rm o}) = \theta(\beta(n^{\rm o}))  = \restr{(1_B  \otimes \beta(n^{\rm o}))}{ {\cal E}_{B} },\quad n\in N. $$\noindent
Comme dans le cas d'une action continue du groupo\"ide ${\cal G}$, on \'etablit les formules :
$$Y (\widehat\pi(b) \otimes 1_H) Y^* = (\widehat\pi(b) \otimes 1_S) (\theta \otimes {\rm id}_{S})(q_{\beta, \alpha}),\quad Y (\theta(s) \otimes 1_H) Y^* = (\theta \otimes {\rm id}_{S})  \delta(s),\quad b\in B \,,\, s \in S\, ;$$\noindent
\begin{equation}\label{acbd}
[Y^*Y, c \otimes 1_H] = 0,  \quad Y (\widehat\pi(b) \theta(s)  \otimes 1_H) Y^* = 
(\widehat\pi(b) \otimes 1_S) (\theta \otimes {\rm id}_{S})\delta(s),\quad c\in M(C) \,,\, b\in B \,,\, s \in S\, ;
\end{equation}
$$Y_{12} Y_{13} =  V_{23} Y_{12}   V_{23}^* \quad {\rm et } \quad [Y_{12},    V_{23}^*  V_{23}] = 0 \quad {\rm dans  } \quad
 {\cal L}({\cal E}_B \otimes H \otimes H).$$\noindent
Il en r\'esulte qu'on a  :

\noindent
\begin{proposition} Le couple $(\beta_C, \delta_C)$ d\'efinit une action continue du groupo\"ide mesur\'e ${\cal G}$ dans la C*-alg\`ebre $C = B  \rtimes \widehat{\cal  G}$  qu'on appelle l'action duale de l'action $(\alpha_B, \delta_B)$.
\end{proposition}

\subsubsection{Action continue d'un groupo\"ide de co-liaison}

Soient $G_1$ et $G_2$ deux groupes quantiques \lc mono\"idalement \'equivalents. Dans ce qui suit, on se fixe un groupo\"ide de co-liaison ${\cal G} : = {\cal G}_{G_1, G_2} = (\C^2,  M , \alpha , \beta , \delta , T , T', \epsilon)$, dont les coins sont identifi\'es aux       groupes quantiques    $G_1$ et $G_2$. 

Pour \'etudier les actions continues de ce groupo\"ide, nous reprenons pour ${\cal G}$, les notations introduites au paragraphe \ref{grequivmono} et nous identifions les  C*-alg\`ebres $M(S_{ij})$ (resp. $M(S_{ij} \otimes  S_{kl})$) et 
$p_{ij} M(S) \subset  B(H)$ (resp. $(p_{ij} \otimes p_{kl})M(S \otimes S) \subset B(H \otimes H)$).
%{\it Notations et rappels } Dans ce paragraphe, notons :
%\hfill\break
% - $S$ la C*-alg\`ebre  de Hopf, munie du coproduit (non unital) $\delta : S \rightarrow M(S \otimes S)$ associ\'e \`a un groupo\"ide de liaison ${\cal G} : = {\cal G}_{G_1, G_2} = ({\bf C}^2,  M , \alpha , \beta , \delta , T , T', \varepsilon)$, dont les coins sont les     groupes quantiques lc mono\"idalement \'equivalents  $G_1$ et $G_2$.
%\hfill\break
% - $(\varepsilon_1, \varepsilon_2)$ est la base canonique de l'espace vectoriel ${\bf C}^2$ et 
%pour $i , j = 1,2$, posons : 
%$$p_{ij} :=  \alpha(\varepsilon_i)  \beta(\varepsilon_j)\in Z(M(S))\, ,\,S_{ij} :=p_{ij} S$$
%\noindent
% - On identifie  les  C*-alg\`ebres $M(S_{ij})$, ( resp.   $M(S_{ij} \otimes  S_{kl})$) et 
%$p_{ij} M(S) \subset  B(H)$ , (resp. $(p_{ij} \otimes p_{kl})M(S \otimes S) \subset B(H \otimes H)$).
%\hfill\break
% - Pour tout $x\in S_{ij}$, posons  
%$\delta(x)= (p_{ik} \otimes p_{kj})  \delta(x)\in M(S_{ik} \otimes S_{kj})$, on a $\delta_{ij}^k(x) = (W_{ik}^j)^* (1_{H_{ik}} \otimes x) W_{ik}^j$.

\noindent
Soit $A$  une C*-alg\`ebre.  Nous allons commencer par donner une description   \'equivalente d'une action continue    $(\beta_A, \delta_A)$ du groupo\"ide 
${\cal G}$, en termes d'actions continues des groupes quantiques $G_1$ et $G_2$. 
\hfill\break
On  rappelle que $\beta_A : \C^2 \rightarrow M(A)$ est un *-morphisme unital et 
$\delta_A : A \rightarrow M(A\otimes S)$ un *-morphisme injectif v\'erifiant les conditions \ref{ac}.

\noindent
Remarquons d'abord que le *-morphisme 
$\beta_A : \C^2 \rightarrow M(A)$ est \`a valeurs dans le centre de $M(A)$. En effet, pour tout $n\in \C^2 , \,x\in A$, on a :
\begin{equation}\label{centre}
\delta_A(\beta_A(n) x) = \delta_A(1_A) (1_A \otimes \beta(n)) \delta_A(x)  =  \delta_A(x) \delta_A(1_A) (1_A \otimes \beta(n))  = \delta_A(x\beta_A(n)).
\end{equation}
\noindent
Par injectivit\'e de $\delta_A$, on d\'eduit $[\beta_A(n) , x ] = 0$.

\noindent
\begin{notations}\label{not1} Posons :
\begin{enumerate} 
\item  $q_j := \beta_A(\varepsilon_j)$,\index{qb@$q_j$} $A_j = q_j A$.  Il est clair que $A_1$ et $A_2$ sont des id\'eaux  bilat\`eres ferm\'es de la C*-alg\`ebre $A$ et on a $A = A_1\oplus A_2$ ;
\item Pour $j,k = 1,2$, notons  $\pi_j^k : M(A_k \otimes  S_{kj}) \rightarrow M(A \otimes S)$ le prolongement strictement continu de l'injection canonique $A_k \otimes  S_{kj} \rightarrow A \otimes S$ v\'erifiant  $\pi_j^k(1_{A_k \otimes  S_{kj}}) = q_k \otimes p_{kj}$.\index{pi@$\pi_j^k$}
\end{enumerate}
\end{notations}
Par d\'efinition \ref{q} du projecteur $q_{\beta_A, \alpha}$, on a :
\begin{equation} \label{qba} 
q_{\beta_A, \alpha} = q_1 \otimes \alpha(\varepsilon_1) + q_2  \otimes \alpha(\varepsilon_2).
  \end{equation}
On d\'eduit alors de (\ref{ac} c))  que pour $j = 1 , 2$, on a :
\begin{equation} \label{delq}
\delta_A(q_j) = \sum_{k=1,2}   q_k \otimes p_{kj}.
  \end{equation}
Par un calcul direct, on obtient facilement :
\begin{lemme}\label{com1}
 Pour tout $a\in A$ et $j,k = 1,2$, on a 
\begin{equation} (q_k \otimes 1_S) \delta_A(q_j a) =  (1_A \otimes \alpha(\varepsilon_k)) \delta_A(q_j a) = (q_k \otimes p_{kj}) \delta_A(a).
\nonumber \end{equation}
\end{lemme}
\begin{proposition} \label{acc}Pour $j,k =1,2$, il existe un unique *-morphisme injectif et non d\'eg\'en\'er\'e\index{dc@$\delta_{A_j}^k$/$\delta_{B_j}^k$} 
\begin{equation} \delta_{A_j}^k : A_j \rightarrow  M(A_k \otimes  S_{kj})
\nonumber \end{equation}
v\'erifiant   $\pi_j^k \circ \delta_{A_j}^k(x) = (q_k \otimes 1_S) \delta_A(x) =  (1_A \otimes \alpha(\varepsilon_k)) \delta_A(x) = (q_k \otimes p_{kj}) \delta_A(x)$ pour tout $x\in A_j$.
\hfill\break 
De plus, on a : 
\begin{enumerate}
\item Pour tout $a\in A$, on a $\delta_A(a)  = \sum_{k,j} \pi_j^k \circ \delta_{A_j}^k(a q_j)$ ;
\item  Pour tout $j, k, l =1, 2$, on a 
 $(\delta_{A_k}^l \otimes {\rm id}_{S_{kj}}) \delta_{A_j}^k = ({\rm id}_{A_l} \otimes \delta_{lj}^k) \delta_{A_j}^l $ ; 
\item $[\delta_{A_j}^k(A_j)(1_{A_k} \otimes S_{kj})] = A_k \otimes S_{kj}$, en particulier on a :
\begin{equation} 
M(\delta_{A_j}^k(A_j)) \subset 
M(A_k \otimes  S_{kj}) , \quad  
 A_k = [({\rm id}_{A_k} \otimes \omega) \delta_{A_j}^k(A_j)\, | \,\omega\in B(H_{kj})_\ast] \quad \text{;}\nonumber 
\end{equation}
\item $\delta_{A_i}^i : A_i  \rightarrow M(A_i \otimes S_{ii})$ est une action continue du groupe quantique $G_i$ dans la C*-alg\`ebre $A_i$. 
\end{enumerate}
\begin{proof} On a ${\rm Im}\,(\pi_j^k) =  (q_k \otimes p_{kj}) M(A \otimes S)$. Par ailleurs,    pour tout $x\in A_j$, on a 
$$\delta_A(x) = \delta_A(q_j x) = \sum_k   (q_k \otimes p_{kj}) \delta_A(x),$$\noindent
d'o\`u l'existence des *-morphismes $\delta_{A_j}^k$, qui v\'erifient :
\begin{equation} \pi_j^k \circ \delta_{A_j}^k(x) = (q_k \otimes p_{kj}) \delta_A(x),\quad x\in A_j.
\nonumber 
\end{equation} 
Le a)    r\'esulte  de l'\'egalit\'e $A = A_1\oplus A_2$. Par ailleurs, pour $j, k, l =1, 2$, nous avons :
$$(\pi_k^l  \otimes {\rm id} _{S_{kj}}) (\delta_{A_k}^l \otimes  {\rm id} _{S_{kj}})(T) = (q_l \otimes p_{lk} \otimes p_{kj})(\delta_A \otimes {\rm id} _S)(\pi_j^k(T)), \quad T\in M(A_k \otimes S_{kj})\quad ;$$\noindent
$$(\pi_k^l  \otimes {\rm id} _{S_{kj}}) ({\rm id} _{A_l} \otimes  \delta_{lj}^k)(T) = (q_l \otimes p_{lk} \otimes p_{kj})({\rm id} _A \otimes \delta)(\pi_j^l(T)), \quad T\in M(A_l \otimes S_{lj}).$$\noindent
Le b) r\'esulte alors de ces deux relations et de la coassociativit\'e de $\delta_A$. Notons que le b) entra\^ine l'injectivit\'e des *-morphismes  $\delta_{A_j}^k$.
\hfill\break
Le c) se d\'eduit de la condition de continuit\'e  (\ref{ac} d)) de la coaction $\delta_A$. Le d) est une cons\'equence de b)  et de c).
\end{proof}
\end{proposition}
\noindent
\begin{corollary}\label{isoGj} Soient $j , k = 1 , 2$ avec $j\not= k$.
\begin{enumerate}
\item  Pour tout $x\in \delta_{A_j}^k(A_j)$, on a $({\rm id}_{A_k} \otimes \delta_{kj}^j)(x) \in M(\delta_{A_j}^k(A_j) \otimes S_{jj})$.
\item  $\delta_{A_j}^k(A_j) \rightarrow M(\delta_{A_j}^k(A_j) \otimes S_{jj}) : x \mapsto ({\rm id}_{A_k} \otimes \delta_{kj}^j)(x)$ est une action continue du groupe quantique ${G_j}$  dans la C*-alg\`ebre $\delta_{A_j}^k(A_j)$.
\item  $A_j \rightarrow \delta_{A_j}^k(A_j) : x \mapsto \delta_{A_j}^k(x)$ est un *-isomorphisme $G_j$-\'equivariant.
\end{enumerate}
\begin{proof}  Les assertions du corollaire
    sont des cons\'equences de (\ref{acc}  b), c) et d)) faciles \`a \'etablir.
\end{proof}
\end{corollary}
%posons  ${\rm Ind}_{G_k}^{G_j}\, (A_k \delta_{A_k}^k) := \delta_{A_j}^k(A_j) \in M(A_k \otimes S_{kj})$.
\begin{remarks}
\begin{enumerate} 
\item 
Dans le cas \ref{2ex} de l'action triviale, les C*-alg\`ebres  $A_j$ s'identifient \`a $\C$ et le d) de \ref{acc}, correspond \`a l'action triviale de $G_1$ et $G_2$.
\item Dans le cas \ref{2ex} de l'action    du groupo\"ide ${\cal G}$ dans lui-m\^eme. On a 
$$A_1 = S_{11} \oplus S_{21}   , \quad A_2 = S_{12} \oplus S_{22} , \quad \delta_{A_j}^k = 
\delta_{1j}^k  \oplus \delta_{2j}^k.$$
\item Dans le cas o\`u les groupes quantiques $G_1$ et $G_2$ sont r\'eguliers, nous verrons (\ref{isoind}) que la $G_j$-alg\`ebre $\delta_{A_j}^k(A_j)$ peut \^etre obtenue directement \`a partir de la $G_k$-alg\`ebre $A_k$, dont elle est en fait une d\'eformation.
\end{enumerate}
\end{remarks} 

De cette    description  concr\`ete  d'une action continue $(A, \beta_A, \delta_A)$ du groupo\"ide ${\cal G}$ en fonction des *-morphis\-mes $\delta_{A_j}^k$, on d\'eduit  une d\'efinition pratique des *-morphismes ${\cal G}$-\'equivariants.
\begin{lemme}\label{recmor}  Soient $(A, \beta_A, \delta_A)$ et $(B, \beta_B, \delta_B)$ deux ${\cal G}$-alg\`ebres. Pour $k=1,2$, posons $q_k := \beta_A(\varepsilon_k)$ et soit  $\iota_k : M(B_k) \rightarrow M(B)$ le prolongement strictement continu de l'inclusion $B_k \subset B$.
\begin{enumerate} 
\item  Si    $f : A \rightarrow M(B)$ est un *-morphisme   (\ref{morgrou})   ${\cal G}$-\'equivariant, alors pour tout $j=1, 2$, il existe un unique *-morphisme   non d\'eg\'en\'er\'e  
 $f_j  : A_j \rightarrow  M(B_j)$ v\'erifiant :
\begin{equation} \label{morph} (f_k \otimes {\rm id} _{S_{kj}})\circ  \delta_{A_j}^k = \delta_{B_j}^k \circ f_j  , \quad k, j =1, 2. \end{equation}
De plus, pour tout $a\in A$, nous avons $f(a)  = \sum_j \iota_j \circ  f_j(aq_j)$.
\item R\'eciproquement, si  pour $j=1, 2$, nous avons un *-morphisme  non d\'eg\'en\'er\'e $f_j  : A_j \rightarrow  M(B_j)$  v\'erifiant \ref{morph}, alors le *-morphisme non d\'eg\'en\'er\'e 
$$f : A \rightarrow M(B) : a \mapsto f(a) = \sum_j \iota_j \circ  f_j(aq_j)$$
est ${\cal G}$-\'equivariant.
\end{enumerate} 
\begin{proof}  De la seconde condition de la ${\cal G}$-\'equivariance de $f$, il r\'esulte que  $f(\beta_A(\varepsilon_j)) = \beta_B(\varepsilon_j)$, donc 
$f(A_j) \subset     \beta_B(\varepsilon_j)M(B) = \iota_j(M(B_j))$ pour tout $j=1,2$,  d'o\`u l'existence des *-morphismes non d\'eg\'en\'er\'es   $f_j  : A_j \rightarrow  M(B_j)$.
\hfill\break
Le a) se d\'eduit   ais\'ement de la premi\`ere condition de la ${\cal G}$-\'equivariance de $f$. Le b) est facile \`a v\'erifier. 
\end{proof}
\end{lemme}

\begin{corollary}\label{cor1} 
Les   deux    correspondances  :
$$A^{\cal G} \rightarrow A^{G_i}  :  (A, \beta_A, \delta_A) \mapsto (A_i,  \delta_{A_i}^i), \quad  i =1, 2,$$
\noindent 
  sont fonctorielles.
\end{corollary}
\begin{proof} Il est clair que si $f :  (A, \beta_A, \delta_A) \rightarrow (B, \beta_B, \delta_B)$ est un *-morphisme ${\cal G}$-\'equivariant, alors $f_i : (A_i, \delta_{A_i}^i) \rightarrow (B_i, \delta_{B_i}^i)$ est $G_i$-\'equivariant.
\end{proof}

Pour montrer que les deux correspondances $A^{\cal G} \rightarrow A^{G_i}  :  (A, \beta_A, \delta_A) \mapsto (A_i,  \delta_{A_i}^i)$ sont biunivoques dans le cas o\`u $G_1$ et $G_2$ sont r\'eguliers, nous  aurons besoin d'une  r\'eciproque de  \ref{acc} :  

\noindent
\begin{lemme}\label{recacc}  Soient $A_1$ et $A_2$ deux C*-alg\`ebres  munies de *-morphismes     injectifs
\begin{equation} 
\delta_{A_j}^k : A_j \rightarrow  M(A_k \otimes  S_{kj}) , \quad j,k =1,2, \nonumber
\end{equation} 
v\'erifiant les conditions (\ref{acc} b) et c)).
\hfill\break
Posons $A := A_1 \oplus A_2 $ et avec les notations de \ref{not1}, d\'efinissons   les *-morphismes : 
$$\delta_A : A \rightarrow M(A \otimes S) : a =(a_1,a_2) \mapsto \delta_A(a) := \sum_{k,j} \pi_j^k \delta_{A_j}^k(a_j)  , \quad 
\beta_A :\C^2 \rightarrow M(A) : (\lambda, \mu) \mapsto  \begin{pmatrix}\lambda & 0 \\
0 & \mu \end{pmatrix}.$$
\noindent 
Alors $(\beta_A , \delta_A)$ est une action continue du groupo\"ide ${\cal  G}$ dans la C*-alg\`ebre $A = A_1 \oplus A_2$.
\end{lemme}

{\bf \'Equivalence de Morita des produits crois\'es  des C*-alg\`ebres $A_1 \rtimes G_1$ et $A_2 \rtimes G_2$}

\noindent
Soit $(A,  \beta_A, \delta_A)$ une action continue du groupo\"ide ${\cal G}:= {\cal G}_{G_1, G_2}$. Avec les notations de \ref{acc}, nous allons montr\'e que les produits crois\'es $A_1 \rtimes G_1$ et $A_2 \rtimes G_2$ sont Morita \'equivalents.

Commen\c cons par expliciter le produit crois\'e  (\cf\ref{pc})  $A \rtimes {\cal G}$.

On a $${\cal E}_A  := {\cal E}_{A, L} =  (q_1 \otimes \alpha(\varepsilon_1) + q_2  \otimes \alpha(\varepsilon_2)) (A \otimes H) = {\cal E}_{A, 1} \oplus {\cal E}_{A, 2},$$
\noindent
avec ${\cal E}_{A, k} := A_k \otimes  \alpha(\varepsilon_k) H =  A_k \otimes H_{k,1} \oplus A_k \otimes H_{k,2}$, $k=1,2$.\index{eg@${\cal E}_{A,k}$}
\hfill\break
Posons $B  = A \rtimes {\cal G} \subset {\cal L}({\cal E}_A)$. On rappelle qu'on a $B = [\pi(a) \widehat\theta(x) \,\,| \,\, a\in A \,,\,x \in \widehat S ]$, o\`u :
$$\pi : M(A)   \rightarrow M(B) : m \mapsto \restr{({\rm id}_A \otimes L)\delta_A(m)}{ {\cal E}_A } \quad ; \quad \widehat\theta : M(\widehat S)  \rightarrow M(B) :  T \mapsto \restr{(1_A \otimes T)}{ {\cal E}_A }.$$
Il est clair que l'application $a \mapsto {\pi(a)}\!\!\!\restriction_{{\cal E}_{A, k}}$ d\'efinit une repr\'esentation fid\`ele 
$\pi_k  : A \rightarrow {\cal L}({\cal E}_{A, k})$ de la C*-alg\`ebre $A$ et on a pour tout $a\in A$ :
$$\pi_k(a) = \begin{pmatrix}({\rm id}_{A_k} \otimes L_{k1})\delta_{A_1}^k(aq_1) & 0 \\
0 & ({\rm id}_{A_k} \otimes L_{k2})\delta_{A_2}^k(aq_2)\end{pmatrix}.$$
\noindent
De m\^eme, pour $k=1,2$, on a une repr\'esentation fid\`ele $\widehat\theta_k  : \widehat S \rightarrow {\cal L}({\cal E}_{A, k}) : x  \mapsto \restr{(1_A \otimes \rho(x))}{ {\cal E}_{A, k} }$ de la  C*-alg\`ebre $\widehat S$ et pour tout $x \in \widehat S $, on a :\index{pj@$\pi_k$, $\widehat{\theta}_k$}
$$\widehat\theta_k(x) = 
\begin{pmatrix} 1_{A_k} \otimes \widehat\pi_k(x_{11}) & 1_{A_k} \otimes \widehat\pi_k(x_{12}) \\
1_{A_k} \otimes \widehat\pi_k(x_{21}) & 1_{A_k} \otimes \widehat\pi_k(x_{22})\end{pmatrix}, \quad  x_{ij} := \beta(\varepsilon_i) x \beta(\varepsilon_j) , \quad \widehat\pi_k(x_{ij}) := p_{ki} x p_{kj}.$$
\noindent
%On va utiliser le r\'esultat suivant :
%\hfill\break
%{\bf Lemme} Pour $\omega\in B(H)_\ast$, posons $x =({\rm id} \otimes \omega)(V)$. on a  $\widehat\pi_k(x_{ij}) = ({\rm id} \otimes p_{ij}\omega p_{ij})(V_{ij}^k)$
%\hfill\break
Notons $\pi_B$ la repr\'esentation du  produit crois\'e $B = A \rtimes\cG$ d\'efinie par l'inclusion 
$B   \subset  {\cal L}({\cal E}_A)$. 

Par restriction de la repr\'esentation $\pi_B$ au sous C*-module ${\cal E}_{A, k}$, on obtient pour tout $k=1,2$,   une repr\'esentation 
$\pi_{B,k}  : B \rightarrow {\cal L}({\cal E}_{A, k})$. Pour tout $a\in A$ et $x \in \widehat S $, posons $b = \pi(a) \widehat\theta(x)$.  On a  :
$$\pi_{B,k}(b) = \pi_k(a) \widehat\theta_k(x) =\begin{pmatrix}({\rm id}_{A_k} \otimes L_{k1})\delta_{A_1}^k(aq_1)(1_{A_k} \otimes \widehat\pi_k(x_{11})) & ({\rm id}_{A_k} \otimes L_{k1})\delta_{A_1}^k(aq_1)(1_{A_k} \otimes \widehat\pi_k(x_{12})) \\
({\rm id}_{A_k} \otimes L_{k2})\delta_{A_2}^k(aq_2)(1_{A_k} \otimes \widehat\pi_k(x_{21})) & ({\rm id}_{A_k} \otimes L_{k2})\delta_{A_2}^k(aq_2)(1_{A_k} \otimes \widehat\pi_k(x_{22}))
\end{pmatrix}.$$
\noindent
Pour tout $i , j , k  = 1,2$, posons\index{eh@${\cal E}_{jk}^i$} 
\begin{equation}\label{Er} 
{\cal E}_{jk}^i := [ \pi_i(aq_j) \widehat\theta_i(x_{jk})\,\,|\,\, a\in A \,\,, \,\, x \in \widehat S 
] \subset {\cal L}(A_i\otimes  H_{ik}, A_i \otimes H_{ij}).
\end{equation}
 Remarquons que pour tout $i=1,2$, on a  ${\cal E}_{ii}^i = A_i \rtimes_{\delta_{A,i}} G_i$.

\begin{theorem} Pour tout $i, j, k, l =1, 2$, on a  : 
$${\cal E}_{jl}^i\,(A_i \otimes H_{il}) = A_i \otimes H_{ij}  , \quad 
 [{\cal E}_{jk}^i \circ {\cal E}_{kl}^i] =   {\cal E}_{jl}^i  , \quad ({\cal E}_{jl}^i) ^* = {\cal E}_{lj}^i.$$\noindent
En particulier, on a :
$$[{\cal E}_{jk}^i \circ ({\cal E}_{jk}^i)^*] = {\cal E}_{jj}^i   , \quad
[({\cal E}_{jk}^i)^* \circ ({\cal E}_{jk}^i)]  = {\cal E}_{kk}^i.$$
\end{theorem}
Pour la preuve de ce th\'eor\`eme, nous avons besoin du lemme suivant :
\begin{lemme} \label{inj} Pour tout $i, j, k =1,2$, notons $\iota_{jk}^i$ l'injection canonique d\'efinie par la composition :
\begin{equation}
{\cal L}(A_i \otimes H_{ik} , A_i \otimes H_{ij}) \rightarrow  {\cal L}({\cal E}_{A, i}) \rightarrow  {\cal L}({\cal E}_A).\nonumber 
\end{equation}
Pour tout $a\in A$ et tout $x \in \widehat S $, on a  :  
\begin{equation}
\iota_{jk}^i(\widehat\theta_i(x_{jk}) \pi_i(aq_k)) = (q_i \otimes p_{ij}) \widehat\theta(x) \pi(a) (q_i \otimes p_{ik}) , \quad 
\iota_{jk}^i(\pi_i(aq_j) \widehat\theta_i(x_{jk})) = (q_i \otimes p_{ij})   \pi(a) \widehat\theta(x) (q_i \otimes p_{ik}).\nonumber
\end{equation}
\end{lemme}
Un calcul direct permet d'\'etablir ce lemme.
\begin{proof}[D\'emonstration du th\'eor\`eme] Montrons l'inclusion $({\cal E}_{jl}^i) ^* \subset {\cal E}_{lj}^i$.\hfill\break
Soient $a\in A$ et $x \in \widehat S $. On a  $(\pi_i(aq_j)\widehat\theta_i(x_{jl}))^* = \widehat\theta_i(x^*_{lj}) \pi_i(a^*q_j)$. Posons 
$$\widehat\theta(x^*) \pi(a^*)  = \lim \sum_{\rm finie}   \pi(a_s) \widehat\theta(x_s),\text{ avec }a_s\in A \text{ et } x_s\in\widehat S.$$ 
Il r\'esulte de \ref{inj}  qu'on a 
$$\widehat\theta_i(x^*_{lj}) \pi_i(a^*q_j) = \lim \sum_{\rm finie} \pi_i(a_sq_l) \widehat\theta_i(x_{lj,s}),$$
\noindent
d'o\`u l'inclusion $({\cal E}_{jl}^i) ^* \subset {\cal E}_{lj}^i$  et finalement l'\'egalit\'e.

L'inclusion $[{\cal E}_{jk}^i \circ {\cal E}_{kl}^i ]\subset    {\cal E}_{jl}^i$ r\'esulte de l'\'egalit\'e 
$({\cal E}_{kj}^i) ^* = {\cal E}_{jk}^i$. 
\hfill\break
Soient $a\in A$ et $x\in \widehat S$, montrons que $\pi_i(aq_j)\widehat\theta_i(x_{jl})\in [{\cal E}_{jk}^i \circ {\cal E}_{kl}^i ]$.
\hfill\break
Posons  $a = a_1 a_2$ avec $a_i\in A$.  D'apr\`es (\ref{Emorita} d)), on peut supposer que $\widehat\theta_i(x_{jl}) = \widehat\theta_i(y_{jk})\widehat\theta_i(z_{kl})$ avec $y , z \in\widehat S$.  
On a alors 
$\pi_i(aq_j)\widehat\theta_i(x_{jl}) = \pi_i(a_1q_j)\pi_i(a_2q_j)\widehat\theta_i(y_{jk})\widehat\theta_i(z_{kl})$. 

Comme 
$\pi_i(a_2q_j)\widehat\theta_i(y_{jk}) \in {\cal E}_{jk}^i = ({\cal E}_{kj}^i) ^* $, on a 
$\pi_i(aq_j)\widehat\theta_i(x_{jl}) \in [{\cal E}_{jk}^i \circ {\cal E}_{kl}^i]$.
\hfill\break
L'\'egalit\'e $[{\cal E}_{jl}^i\,(A_i \otimes H_{il})] = A_i \otimes H_{ij}$ r\'esulte de   (\ref{acc} c))  et   de l'\'egalit\'e  $[{E}_{jl}^i\,H_{il}] = H_{ij}$.
\end{proof}

\begin{corollary} Pour tout $i,j=1,2$, ${\cal E}_{jj}^i$  est une C*-alg\`ebre et ${\cal E}_{ij}^i$ est une \'equivalence de Morita entre les C*-alg\`ebres ${\cal E}_{ii}^i =  A_i \rtimes_{\delta_{A,i}} G_i$ et ${\cal E}_{ij}^i$.
\begin{proof}
On d\'eduit imm\'ediatement du th\'eor\`eme que ${\cal E}_{jj}^i$ est une C*-alg\`ebre qui agit de fa\c con non d\'eg\'en\'er\'ee dans le $A_i$-module hilbertien $A_i \otimes H_{ij}$ et on peut consid\'erer $M({\cal E}_{jj}^i)$ comme une sous-C*-alg\`ebre de ${\cal L}(A_i \otimes H_{ij}) = M(A_i \otimes {\cal K}(H_{ij}))$.
\hfill\break
Il est clair que ${\cal E}_{ij}^i$, muni de l'action \`a droite de ${\cal E}_{jj}^i$ et du produit scalaire $\langle \xi , \eta \rangle := \xi^* \eta$ est un C*-module plein et on a   ${\cal E}_{ii}^i  = A_i \rtimes_{\delta_{A,i}} G_i  \simeq {\cal K}({\cal E}_{ij}^i)$.
\end{proof}
\end{corollary}

\begin{proposition} On a un *-isomorphisme $\mu_{ji} : {\cal E}_{jj}^i \rightarrow  {\cal E}_{jj}^j$ v\'erifiant pour tout $x\in  {\cal E}_{jj}^i$  : 
\begin{equation}(\delta_{A_i}^j \otimes {\rm id}_{{\cal K}(H_{ij})})(x) = (W_{ji, 23}^j)^*\mu_{ji}(x)_{13}W_{ji, 23}^j.\nonumber 
\end{equation}
De plus,   pour tout $a\in A$ et $x \in \widehat S $, on a  $\mu_{ji}(\pi_i(a q_j)\widehat\theta_i(x_{jj})) = \pi_j(a q_j)\widehat\theta_j(x_{jj})$.
\end{proposition}

Pour \'etablir  cette  proposition, on va d'abord prouver le lemme suivant :
 
\begin{lemme} Soient $i, j, k = 1, 2$.
\begin{enumerate}
\item $(W_{ik, 12}^j)^* \, V_{jj, 23}^i \, =  V_{jj, 23}^k\, (W_{ik, 12}^j)^*$.
\item Pour tout $x \in \widehat S $, on a $(W_{ik}^j)^*(1_{H_{ik}} \otimes \widehat\pi_i(x_{jj}))W_{ik}^j = 
 1_{H_{ik}} \otimes \widehat\pi_k(x_{jj})$. 
\item Pour tout $a\in A$, on a 
$(\delta_{A_i}^j \otimes {\rm id}_{S_{ij}})(\delta_{A_j}^i(aq_j))  = 
(W_{ji, 23}^j)^*\delta_{A_j}^j(aq_j)_{13}W_{ji, 23}^j$.
\end{enumerate}
\end{lemme}

\begin{proof}
Le a) r\'esulte de (\ref{pent} b)).   Pour $\omega\in B(H)_\ast$, posons $x =({\rm id} \otimes \omega)(V)$. On a  :
$$\widehat\pi_i(x_{jj}) = ({\rm id} \otimes p_{jj}\omega p_{jj})(V_{jj}^i)  , \quad \widehat\pi_k(x_{jj}) = ({\rm id} \otimes p_{jj}\omega p_{jj})(V_{jj}^k).$$
\noindent
Le b) est donc une cons\'equence du a). Le c) r\'esulte de (\ref{coprodS} b)) et (\ref{acc} b)).
\end{proof}
\noindent
\begin{proof}[D\'emonstration  de la proposition] Soient $a \in A$ et $x \in \widehat S$. Posons $u =\pi_i(a q_j) \widehat\theta_i(x_{jj})\in {\cal E}_{jj}^i$. Nous avons :
\begin{align} 
(\delta_{A_i}^j \otimes {\rm id}_{{\cal K}(H_{ij})})(u) &= ({\rm id}_{A_j} \otimes \delta_{jj}^i)(\pi_j(a q_j))(1_{A_j} \otimes1_{H_{ji}} \otimes  \widehat\pi_i(x_{jj}))\nonumber \\
& = (W_{ji, 23}^j)^*  \pi_j(a q_j)_{13} W_{ji, 23}^j (1_{A_j} \otimes1_{H_{ji}} \otimes  \widehat\pi_i(x_{jj}))= (W_{ji, 23}^j)^*(\pi_j(a q_j)\widehat\theta_j(x_{jj}))_{13}W_{ji, 23}^j.\nonumber
\end{align}
On en d\'eduit l'existence de l'isomorphisme $\mu_{ji} : {\cal E}_{jj}^i \rightarrow  {\cal E}_{jj}^j$ et la relation 
\begin{equation}
\mu_{ji}(\pi_i(a q_j)\widehat\theta_i(x_{jj})) = \pi_j(a q_j)\widehat\theta_j(x_{jj})  \quad   ; \quad   a \in A \, , \,x\in \widehat S.\nonumber
\end{equation}
\end{proof}

On a donc obtenu le résultat suivant :

\begin{corollary} \label{pcM}Les C*-alg\`ebres $A_1 \rtimes_{\delta_{A_1}^1} G_1$ et $A_2 \rtimes_{\delta_{A_2}^2} G_2$ sont Morita \'equivalentes.
\end{corollary}

Pour tout $x\in {\cal E}_{jj}^i$, posons 
\begin{equation}
\delta_{{\cal E}_{jj}^i}(x) :=  \widetilde V_{ij, 23}^j (x \otimes 1_{H_{ji}})(\widetilde V_{ij, 23}^j)^*.
\end{equation}
Alors $\delta_{{\cal E}_{jj}^i}$ est  la   restriction de la coaction duale du syst\`eme dynamique $(A, \beta_A, \delta_A)$ \`a la C*-alg\`ebre ${\cal E}_{jj}^i$. Elle   d\'efinit une action continue   du groupe quantique $\widehat{G_j}$ dans la C*-alg\`ebre 
${\cal E}_{jj}^i$.

\begin{proposition} L'isomorphisme 
$\mu_{ji} : ({\cal E}_{jj}^i , \delta_{{\cal E}_{jj}^i}) \rightarrow  ({\cal E}_{jj}^j , \delta_{{\cal E}_{jj}^j})$ est $\widehat {G_j}$-\'equivariant.
\begin{proof}
Posons $u = \pi_i(aq_j) \widehat\theta_i(x_{jj})$. En utilisant (\ref{pent} b)) et (\ref{Emorita} b)),  on obtient : 
$$\widetilde V_{ij, 23}^j (\pi_i(aq_j) \widehat\theta_i(x_{jj})  \otimes 1_{H_{ji}})(\widetilde V_{ij, 23}^j)^* = 
(\pi_i(aq_j) \otimes 1_{H_{jj}}) (\widehat\theta_i \otimes {\rm id }_{\widehat S_{jj}})\widehat\delta(x_{jj}).$$
\noindent
En r\'eutilisant (\ref{pent} b)) et (\ref{Emorita} b)),  on  d\'eduit :
\begin{align} \widetilde V_{jj, 23}^j (\mu_{ji}(u) \otimes 1_{H_{jj}})(\widetilde V_{jj, 23}^j)^* &= 
\widetilde V_{jj, 23}^j (\pi_j(aq_j) \widehat\theta_j(x_{jj}) \otimes 1_{H_{jj}})(\widetilde V_{jj, 23}^j)^* \nonumber \\
& = 
(\pi_j(aq_j) \otimes 1_{H_{jj}}) (\widehat\theta_j \otimes {\rm id}_{\widehat S_{jj}})\widehat\delta(x_{jj}) \nonumber \\
& = 
(\mu_{ji} \otimes {\rm id}_{\widehat S_{jj}})((\pi_i(aq_j) \otimes 1_{H_{jj}}) (\widehat\theta_i \otimes {\rm id}_{\widehat S_{jj}})\widehat\delta(x_{jj})) =
(\mu_{ji} \otimes {\rm id}_{\widehat S_{jj}})\delta_{{\cal E}_{jj}^i}(u).\nonumber 
\end{align}
\end{proof}
\end{proposition}

\begin{remarks}
\begin{enumerate}
\item Le corollaire  \ref{pcM}  a \'et\'e \'etabli (avec des notations et des conventions diff\'erentes) dans  \cite{DeC1} dans le cas de l'action triviale du groupo\"ide ${\cal G}$ dans la C*-alg\`ebre $A:= N^{\rm o} = \C^2$.\hfill\break
Pour cette action, notons que : les C*-alg\`ebres  $A_1$ et $A_2$ s'identifient \`a $\C$ ; pour $j,k=1,2$, les *-morphismes
$\delta_{A_j}^k : \C \rightarrow M(\C \otimes S_{kj})  = M(S_{kj})$ v\'erifient  $\delta_{A_j}^k(1) =    p_{kj}$ ; pour tout $i , j , k=1,2$, on a ${\cal E}_{jk}^i = E_{jk}^i$, avec les notations \ref{notE} et \ref{Er}.
\hfill\break
Le produit crois\'e $B = A \rtimes {\cal G}$ est canoniquement isomorphe \`a $\widehat S$. Plus pr\'ecis\'ement, nous avons que $\pi_B : (B, \alpha_B, \delta_B) \rightarrow (\widehat S , \widehat\alpha, \widehat\delta)$   est un isomorphisme de $\widehat{\cal G}$-alg\`ebres.
\item  Dans le cas \ref{2ex} de l'action    du groupo\"ide ${\cal G}$ dans lui-m\^eme, on obtient que les C*-alg\`ebres  
$$(S_{11} \rtimes_{\delta_{11}^1} G_1) \oplus (S_{21} \rtimes_{\delta_{21}^1} G_1) \quad \text{et} \quad
(S_{12} \rtimes_{\delta_{12}^2} G_2) \oplus (S_{22} \rtimes_{\delta_{22}^2} G_2)$$ 
sont Morita \'equivalentes.
\end{enumerate}
\end{remarks}
\bigskip 
\subsection{Bidualit\'e}

\noindent
Soit $(\beta_A, \delta_A)$ une action continue du groupo\"ide mesur\'e ${\cal G}$ dans une C*-alg\`ebre $A$.
Posons $B = A \rtimes {\cal G}$ muni de l'action duale $(\alpha_B, \delta_B)$ du groupo\"ide $\widehat{\cal  G}$ et $C = B \rtimes \widehat{\cal  G}$ muni de l'action (bi)duale $(\beta_C, \delta_C)$.\hfill\break
Dans une premi\`ere partie, nous montrons que  la C*-alg\`ebre $C$  s'identifie canoniquement \`a une   sous-C*-alg\`ebre 
$D \subset {\cal L}(A \otimes H)$. Nous notons alors  $(\beta_D, \delta_D)$ l'action continue de ${\cal G}$ dans $D$, obtenue par transport  de structure.
\hfill\break
Dans le cas d'un groupo\"ide mesur\'e ${\cal G}$ r\'egulier, nous montrons  l'\'egalit\'e  $D =   q_{\beta_A, \widehat\alpha} (A \otimes {\cal K}(H))q_{\beta_A, \widehat\alpha}$. Nous d\'ecrivons aussi explicitement 
 l'action continue                   $(\beta_D, \delta_D)$ du groupo\"ide mesur\'e ${\cal G}$ dans  $D$, \`a l'aide de l'action initiale  $(\beta_A, \delta_A)$ et de la repr\'esentation r\'eguli\`ere droite du groupo\"ide ${\cal G}$.\hfill\break
Nous examinons ensuite le cas d'une action continue d'un groupo\"ide de co-liaison.

\subsubsection{Cas g\'en\'eral}

Commen\c cons par d\'efinir $D$ et \'etablir un *-isomorphisme $C = A \rtimes {\cal G}\rtimes  \widehat{\cal  G}  \rightarrow D$.
\begin{lemme-notations}\label{DEAR} \strut 
\begin{enumerate}
\item Posons $D := [\,({\rm id}_A \otimes R) \delta_A(a)  \,  (1_A \otimes \lambda(x) L(s))\,| \, a\in A,\, x\in\widehat S,\, s\in S]  \subset  {\cal L}(A \otimes H)$. Alors $D$ est une sous-C*-alg\`ebre de ${\cal L}(A \otimes H)$.
\item Il existe un unique *-morphisme   $\pi_R : M(A) \rightarrow {\cal L}(A \otimes H)$\index{pk@$\pi_R$} injectif et continu pour les topologies stricte/*-forte v\'erifiant :
$$\pi_R(m) =  ({\rm id}_A \otimes R) \delta_A(m), \quad m\in M(A) \quad ; \quad \pi_R(1_A) = 
q_{\beta_A, \widehat\alpha}.$$
\item Pour tout $d\in D$, on a  $q_{\beta_A, \widehat\alpha}\, d =  d =  d \,q_{\beta_A, \widehat\alpha}$. De plus,    $D  (A\otimes H) = 
 q_{\beta_A, \widehat\alpha} (A \otimes H)$. 
\item Il existe un unique *-morphisme  $j_D : M(D) \rightarrow {\cal L}(A \otimes H)$\index{jb@$j_D$} fid\`ele,  continu pour les topologies stricte/*-forte,  prolongeant ${\rm id}_D$ et v\'erifiant $j_D(1_D) =  q_{\beta_A, \widehat\alpha}$.
\end{enumerate}
Dans la suite, on pose ${\cal E}_{A, R} := q_{\beta_A, \widehat\alpha} (A \otimes H)$.\index{ei@${\cal E}_{A,R}$}
\begin{proof}
Pour prouver le a), on peut proc\'eder comme dans \ref{pclem}. On peut aussi le d\'eduire de  la proposition \ref{isofi} qui va suivre.
\hfill\break
Le b) est une cons\'equence de la d\'efinition (\ref{ac} a)).\hfill\break 
Le b)  et (\ref{2.10} b)) entra\^inent que pour tout  $d\in D$, on a  $q_{\beta_A, \widehat\alpha} d =  d =  d q_{\beta_A, \widehat\alpha}$, donc $D(A \otimes H) \subset q_{\beta_A, \widehat\alpha} (A \otimes H)$. Pour prouver l'\'egalit\'e, il suffit de remarquer que la C*-alg\`ebre (\ref{Sl(hatS)}) $\lambda(\widehat S) S$ agit de fa\c con non d\'eg\'en\'er\'ee dans $H$, et d'utiliser la continuit\'e de l'action (\cf\ref{ac} d)). \hfill\break
Le d) est une cons\'equence du c).
\end{proof}
\end{lemme-notations}

Avec les notations de \ref{pc} et \ref{copc}, on a :

\begin{proposition} \label{isofi} Il existe un unique *-isomorphisme 
$\phi : C \rightarrow D$ v\'erifiant  :
$$\phi(\widehat \pi(\pi(a) \widehat\theta(x)) \theta(s)) =  \pi_R(a) (1_A  \otimes \lambda(x) L(s))  ,\quad a\in A,\, x\in \widehat S ,\, s\in S.$$
\noindent 
Posons 
$$\delta_D := (\phi \otimes {\rm id} _S) \circ \delta_C \circ \phi^{-1} \quad \text{et}  \quad \beta_D := \phi \circ \beta_C.$$
\hfill\break
Alors $(\beta_D, \delta_D)$ est une action continue du groupo\"ide ${\cal G}$ et 
pour tout $a\in A$, $x\in \widehat S$ et $s\in S$, nous avons  :
$$(j_D \otimes {\rm id}_S)\delta_D(\pi_R(a)(1_A \otimes \lambda(x) L(s))) =  (\pi_R(a) \otimes 1_S) (1_A \otimes [(\lambda(x) \otimes 1_S) (L\otimes {\rm id}_S) \delta(s)])\, ;$$\noindent
$$j_D(\beta_D(n^{\rm o}))  = q_{\beta_A, \widehat\alpha} (1_A \otimes \beta(n^{\rm o})), \quad n\in N.$$\noindent
\end{proposition}
\begin{proof} 
Le  produit crois\'e  $B = A \rtimes {\cal G}$ agit de fa\c con fid\`ele et  non d\'eg\'en\'er\'ee (\ref{pclem}) dans le $A$-module hilbertien ${\cal E}_A = q_{\beta_A, \alpha} (A \otimes H)$. Dans la suite, on consid\`ere $M(B)$ comme une sous-C*-alg\`ebre de la C*-alg\`ebre ${\cal L}({\cal E}_A)$.
\hfill\break  
Le  produit crois\'e  $C = B \rtimes \widehat{\cal  G}$ agit de fa\c con fid\`ele et  non d\'eg\'en\'er\'ee (\ref{copc}) dans le $B$-module hilbertien ${\cal E}_B = q_{\alpha_B,\widehat\beta} (B \otimes H)$ et on a $M(C) \subset {\cal L}({\cal E}_B)$.\hfill\break 
Introduisons  le  *-morphisme injectif,  continu pour les topologies stricte/*-forte :    
$$\phi_0 : M(C)  \rightarrow {\cal L}({\cal E}_B \otimes_B {\cal E}_A) :  T \mapsto T \otimes_B 1_{{\cal E}_A}.$$ \noindent 
Par restriction au sous-C*-A-module  ${\cal E}_B \otimes_B {\cal E}_A$, l'unitaire  
$$(B \otimes H) \otimes_B {\cal E}_A 
\rightarrow {\cal E}_A \otimes H : (b \otimes \eta) \otimes _B \xi \mapsto b\xi \otimes \eta$$
\noindent  
d\'efinit un unitaire $u : {\cal E}_B \otimes_B {\cal E}_A \rightarrow q_{\alpha_B, \widehat\beta}({\cal E}_A \otimes H)$. Soit 
$$\phi_1 : C  \rightarrow {\cal L}(q_{\alpha_B, \widehat\beta}({\cal E}_A \otimes H)) : c \mapsto u \phi_0(c) u^*.$$ \noindent
Posons $C_1 = \phi_1(C) \,,\,\beta_{C_1} =  \phi_1 \circ   \beta_C \,,\,\delta_{C_1} = (\phi_1 \otimes {\rm id}_S) \circ \delta_C \circ \phi_1^{-1}$.\hfill\break
Pour tout  $a\in A \,,\, x\in\widehat S \, $ et $s\in S$, on v\'erifie facilement qu'on a :
$$\phi_1(\widehat \pi(\pi(a) \widehat\theta(x)) \theta(s)) = (\pi(a)\otimes 1_H)\,     (\widehat\theta \otimes \lambda)\widehat\delta(x) \, (1_{{\cal E}_A} \otimes L(s))q_{\alpha_B, \widehat\beta} \in {\cal L}(q_{\alpha_B, \widehat\beta}({\cal E}_A \otimes H)).$$
\noindent
En utilisant (\ref{2.10} b)), on obtient $[q_{\beta_A, \alpha, 12}  , (1_H \otimes U) V (1_H \otimes  U^*)_{23}] = 0$ dans $\cL(A \otimes H \otimes H)$. Posons $Z := \restr{(1_H \otimes U) V (1_H \otimes  U^*)_{23}}{ {\cal E}_{A}\otimes H}\,\in {\cal L}({\cal E}_A \otimes H)$. \`A l'aide de (\ref{2.9} c)) et de (\ref{2.10} b)), on a :
$$Z^* Z = q_{\alpha_B, \widehat\beta}  =  (\widehat\theta \otimes \lambda)( q_{\widehat\alpha , \beta})  , \quad Z Z^* = (\widehat\theta \otimes R)( q_{\beta , \alpha}).$$\noindent
Consid\'erons alors le *-morphisme injectif :
$$\phi_Z : C_1  \rightarrow  {\cal L}({\cal E}_A \otimes H): c' \mapsto Z c'  Z^*.$$\noindent
Alors $\phi_Z$ s'\'etend de fa\c con unique en un *-morphisme injectif,  continu pour les topologies stricte/*-forte $\phi_Z : M(C_1)  \rightarrow  {\cal L}({\cal E}_A \otimes H)$  v\'erifiant $\phi_Z(1_{C_1}) = Z Z^* = (\widehat\theta \otimes R)( q_{\beta , \alpha} )$. \hfill\break 
Posons $C_2 := \phi_Z(C_1)$, $\beta_{C_2} :=  \phi_Z \circ   \beta_{C_1}$ et $\delta_{C_2} := (\phi_Z \otimes {\rm id}_S)   \circ \delta_{C_1} \circ \phi_Z^{-1}$.\hfill\break
Pour tout $a\in A \,,\,x\in\widehat S$ et $s\in S$, on a  (\cf \ref{2.10} b), c) et d)):
$$(\widehat\theta \otimes \lambda) \widehat\delta(x)  = Z^* (1_{{\cal E}_A} \otimes \lambda(x)) Z  , \quad (\pi \otimes R) \circ \delta_A(a) = Z (\pi(a) \otimes 1_H) Z^*  , \quad [Z, 1_{{\cal E}_A} \otimes L(s)] = 0.$$\noindent
Il en r\'esulte que $\phi_Z \circ \phi_1(\widehat \pi(\pi(a) \widehat\theta(x)) \theta(s)) = (\pi \otimes R) \circ \delta_A(a) \, (1_{{\cal E}_A} \otimes \lambda(x))\,(1_{{\cal E}_A} \otimes L(s))$. \hfill\break 
En tenant compte du sous-espace initial et du sous-espace final de $Z$, on obtient :
$$C_2 = [\,(\pi \otimes R) \circ \delta_A(a) \, (1_{{\cal E}_A} \otimes \lambda(x))\,(1_{{\cal E}_A} \otimes L(s))\,\,|\,\,\,a\in A \,,\,x\in\widehat S\,,\,s\in S] \subset {\cal L}({\cal E}_A \otimes H). $$\noindent
Soit $j_{C_2} : M(C_2) \rightarrow {\cal L}({\cal E}_A \otimes H)$ l'unique *-morphisme continu pour les topologies stricte/*-forte,  prolongeant ${\rm id}_{C_2}$ et v\'erifiant $j_{C_2}(1_{C_2}) =  (\widehat\theta \otimes R)( q_{\beta , \alpha})$.\hfill\break
Pour tout $a\in A \,,\,x\in\widehat S$ et $s\in S$,  posons $c = (\pi \otimes R) \circ \delta_A(a) \, (1_{{\cal E}_A} \otimes \lambda(x))\,(1_{{\cal E}_A} \otimes L(s))$.
\hfill\break 
On a dans ${\cal L}({\cal E}_A \otimes  H \otimes S)$, \cf (\ref{acbd}) : \hfill\break
 - $(j_{C_2} \otimes {\rm id}_S) \delta_{C_2}(c) =  ((\pi \otimes R) \circ \delta_A(a) \otimes 1_S) \, (1_{{\cal E}_A} \otimes  \lambda(x) \otimes 1_S) \,(1_{{\cal E}_A} \otimes (L \otimes {\rm id}_S) \delta(s))$ ;\hfill\break
 - $j_{C_2} \circ \beta_{C_2} (n^{\rm o}) = Z Z^* (1_{{\cal E}_A} \otimes \beta(n^{\rm o}))  =   (\widehat\theta \otimes R) q_{\beta , \alpha}(1_{{\cal E}_A} \otimes \beta(n^{\rm o}))\in {\cal L}({\cal E}_A \otimes H)$, pour tout $n\in N$.\hfill\break
Finalement, introduisons l'unitaire  
$$v : A \otimes H \otimes_\pi {\cal E}_A \rightarrow {\cal E}_A \otimes H : (a \otimes \eta) \otimes_\pi  \xi  \mapsto \pi(a) \xi \otimes \eta.$$
\noindent
Pour tout $T\in B(H) \,,\,m\in M(A \otimes S)$, on a facilement : \hfill\break
- $v^* (1_{{\cal E}_A} \otimes  T) v = (1_A \otimes T) \otimes _\pi 1_{{\cal E}_A}$ ;\hfill\break
 - $v^*\, (\pi \otimes R)(m) \,  v = ({\rm id}_A \otimes R) (m)    \otimes_\pi \,1_{{\cal E}_A}$.\hfill\break
Il en r\'esulte que pour tout $a\in A \,,\,x\in\widehat S$ et $s\in S$, on a 
$$v^*\,  (\pi \otimes R) \circ \delta_A(a) \, (1_{{\cal E}_A} \otimes \lambda(x) L(s) )\,  v =  
\pi_R(a)  \,  (1_A \otimes \lambda(x) L(s))  \otimes_\pi \,1_{{\cal E}_A}.$$\noindent
Pour tout $c\in C_2$, posons $v^* c v =  \phi_2(c) \otimes_\pi 1_{{\cal E}_A}$. On obtient un *-isomorphisme 
$\phi_2 : C_2 \rightarrow D$.\hfill\break
Pour tout $a\in A \,,\,x\in\widehat S$ et $s\in S$, posons $c = (\pi \otimes R) \circ \delta_A(a) \, (1_{{\cal E}_A} \otimes \lambda(x))\,(1_{{\cal E}_A} \otimes L(s))$. On a  alors :
$$\phi_2(c) = 
\pi_R(a)  \,  (1_A \otimes \lambda(x) L(s))\,,\,(j_{D} \otimes {\rm id}_S) (\phi_2 \otimes {\rm id}_S)\delta_{C_2}(c)) =  (\pi_R(a) \otimes 1_S)  \,(1_A \otimes [(\lambda(x) \otimes 1_S)(L \otimes {\rm id}_S) \delta(s)]),$$\noindent
$$j_D\circ \phi_2 \circ   \beta_{C_2}(n^{\rm o}) =  q_{\beta_A, \widehat\alpha} (1_A \otimes \beta(n^{\rm o}))  , \quad n\in N.$$
\noindent
Le *-isomorphisme $\phi := \phi_2 \circ \phi_Z \circ \phi_1$ convient.
\end{proof}

\noindent
Dans ce qui suit, nous allons d\'ecrire l'action continue $(\beta_D, \delta_D)$ \`a l'aide de l'action initiale $(\beta_A, \delta_A)$ et de la repr\'esentation r\'eguli\`ere droite du groupo\"ide ${\cal G}$.
\begin{notations}
\begin{enumerate} 
\item \'Etant donnée la repr\'esentation r\'eguli\`ere droite (\ref{mult}) $V$ de ${\cal G}$, on note ${\cal V}$\index{vd@${\cal V}\in{\cal L}(H\otimes S)$} l'unique isom\'etrie partielle de  ${\cal L}(H \otimes S)$ v\'erifiant 
$({\rm id}_H \otimes L) ({\cal V})  = V$.
\item La C*-alg\`ebre ${\cal K}(H)$ est not\'ee ${\cal K}$.
\item On note  $\delta_0 : A \otimes {\cal K} \rightarrow M(A \otimes {\cal K} \otimes S)$\index{dd@$\delta_0$} le *-morphisme injectif v\'erifiant :
\begin{equation}\delta_0(a \otimes k)  = \delta_A(a)_{13} (1_A  \otimes k \otimes 1_S),  \quad a\in A \, , \, k\in{\cal K}.
\end{equation}
\item Pour tout $x\in M(A \otimes {\cal K})$  et  tout $n\in N$, posons  :\index{df@$\delta_{A\otimes{\cal K}}$, $\beta_{A\otimes{\cal K}}$}
\begin{equation}\label{bak}
\delta_{A \otimes {\cal K}}(x):={{\cal V}  }_{23}  \delta_0(x) {{\cal V}  }_{23}^* \quad ; \quad \beta_{A \otimes {\cal K}}(n^{\rm o}) := q_{\beta_A, \widehat\alpha} (1_A \otimes \beta(n^{\rm o})).
\end{equation}
\end{enumerate}
\end{notations}
 % et l'application lin\'eaire :
%$$\delta_{A \otimes H}  : A \otimes H \rightarrow {\cal L}(A \otimes S, 
%A \otimes H \otimes S) : a \otimes \xi \mapsto {{\cal V}  }_{23} \delta_A(a)_{13} (1_A \otimes \xi \otimes 1_S)
% \quad a\in A \, , \, \xi \in H$$
%\noindent
%L'existence  de l'application $\delta_{A \otimes H}$ et sa continuit\'e r\'esulte de l'\'egalit\'e :
%$$\delta_{A \otimes H}(a \otimes \xi)^* \delta_{A \otimes H}(a' \otimes \xi')  =  
% \delta_A(\langle q_{\beta_A, \widehat\alpha}(a \otimes \xi) \, , \, q_{\beta_A, \widehat\alpha}(a' \otimes \xi') \rangle
% \quad a, a'\in A \, ; \, \xi , \xi' \in H$$\noindent
%qui s'obtient \`a partir de l'\'egalit\'e dans ${\cal L}(A \otimes H \otimes S)$    :
%$$({\rm id}_A \otimes R \otimes {\rm id}_S) (\sigma_{23}[(\delta_A \otimes {\rm id}_S) (a^* \otimes 1_S) 
%q_{\beta_A, \alpha} (a' \otimes 1_S))] =  \delta_A(a)_{13}^* {{\cal V}  }_{23}^*{{\cal V}  }_{23} \delta_A(a')_{13} 
%\quad a \, ,  \, a'\in A$$\noindent
%Par ailleurs, il est clair que dans ${\cal L}(A \otimes S, 
%A \otimes H \otimes S)$, on a  :
%$$\delta_{A \otimes H}(\xi . a)  = \delta_{A \otimes H}(\xi) .\delta_A(a) \quad 
%\xi\in A \otimes H \, , \, a\in A$$\noindent 

Alors $\delta_0$ (resp. $\delta_0 \otimes {\rm id}_S$) se prolonge de fa\c con  unique en un *-morphisme  $M(A \otimes {\cal K}) \rightarrow 
M(A \otimes {\cal K} \otimes S)\,$ (resp.  $M(A \otimes {\cal K} \otimes S) \rightarrow 
M(A \otimes {\cal K} \otimes S \otimes S)$), strictement continu et injectif   v\'erifiant 
\begin{equation}
\delta_0(1_{A \otimes {\cal K}}) = q_{\beta_A, \alpha, 13} \quad \hbox{(resp. } (\delta_0 \otimes {\rm id}_S)(1_{A \otimes {\cal K} \otimes S}) = q_{\beta_A, \alpha, 13}{\rm )}.
\end{equation}
De m\^eme, on a un *-morphisme strictement continu et injectif :
\begin{equation}
{\rm id}_{A \otimes {\cal K}} \otimes \delta  :  
M(A \otimes {\cal K} \otimes S) \rightarrow 
M(A \otimes {\cal K} \otimes S \otimes S) , \quad ({\rm id}_{A \otimes {\cal K}} \otimes \delta )(1_{A \otimes {\cal K} \otimes S}) = q_{\beta, \alpha, 34}.
\end{equation}

\begin{lemme} Pour tout $x\in M(A \otimes {\cal K})$, on a 
\begin{equation}q_{\beta_A, \alpha,13} (\delta_0 \otimes {\rm id}_S) \delta_0(x) = (\delta_0 \otimes {\rm id}_S) \delta_0(x) =  ({\rm id}_{A \otimes {\cal K}} \otimes \delta) \delta_0(x)  =q_{\beta, \alpha,34} ({\rm id}_{A \otimes {\cal K}} \otimes \delta) \delta_0(x).\nonumber
\end{equation}
\begin{proof}Pour tout $a\in A$, on a  : 
\begin{equation}
(\delta_0 \otimes {\rm id}_S) (\delta_A(a)_{13}) = [(\delta_A \otimes {\rm id}_S)\delta_A(a)]_{134} =[({\rm id}_A \otimes \delta) \delta_A(a)]_{134},  \nonumber
\end{equation}
\begin{equation}
({\rm id}_{A \otimes {\cal K}} \otimes \delta)(\delta_A(a)_{13}) =  [({\rm id}_A \otimes \delta) \delta_A(a)]_{134}.\nonumber
\end{equation}
\noindent
Pour tout $k\in {\cal K}$, on a :
\begin{equation}
(\delta_0 \otimes {\rm id}_S) (1_A \otimes k  \otimes 1_S) = q_{\beta_A, \alpha, 13}(1_A \otimes k  \otimes 1_{S \otimes S}),\nonumber
\end{equation}
\begin{equation}
({\rm id}_{A \otimes {\cal K}} \otimes \delta)(1_A \otimes k  \otimes 1_S) =
q_{\beta, \alpha, 34} (1_A \otimes k  \otimes 1_{S \otimes S}).\nonumber
\end{equation}
\end{proof}
\end{lemme}
Notons que 
$ \beta_{A \otimes {\cal K}} : N^{\rm o} \rightarrow M(A \otimes {\cal K})$  est un *-morphisme car $[\widehat\alpha(N), \beta(N^{\rm o})] = 0$.

\begin{lemme} On a $\delta_{A \otimes {\cal K}}(1_{A \otimes {\cal K}}) = q_{\beta_A, \widehat\alpha, 12} \, q_{\beta, \alpha, 23}$. Pour tout $n\in N$, on a  :
\begin{equation}
\delta_{A \otimes {\cal K}}(\beta_{A \otimes {\cal K}}(n^{\rm o})) =   
\delta_{A \otimes {\cal K}}(1_{A \otimes {\cal K}}) (1_{A \otimes {\cal K}} \otimes \beta(n^{\rm o})).
\end{equation}
\end{lemme}

\begin{proof}
En utilisant (\ref{2.10} b)) et (\ref{2.9} c)), on  a 
\begin{equation}
\delta_{A \otimes {\cal K}}(1_{A \otimes {\cal K}}) ={\cal V}_{23}  \delta_0(1_{A \otimes {\cal K}})
{\cal V}_{23}^* =  {\cal V}_{23}  q_{\beta_A,\alpha, 13}
{\cal V}_{23}^*   = q_{\beta_A, \widehat\alpha, 12} {\cal V}_{23}{\cal V}_{23}^* = q_{\beta_A, \widehat\alpha, 12}  q_{\beta, \alpha, 23}
\end{equation}
et \`a l'aide de (\ref{ac} c)), on d\'eduit  
\begin{equation}\label{d01}
\delta_0(q_{\beta_A,  \widehat\alpha}) = q_{\beta_A,\alpha, 13} \, q_{\widehat\alpha, \beta, 23} = q_{\beta_A,\alpha, 13}{{\cal V}  }_{23}^*{{\cal V}  }_{23}.
\end{equation}
Pour $n\in N$, on a 
\begin{align}\delta_{A \otimes {\cal K}}(\beta_{A \otimes {\cal K}}(n^{\rm o})) &= {\cal V}_{23}  \delta_0(q_{\beta_A,  \widehat\alpha})(1_A \otimes \beta(n^{\rm o}) \otimes 1_S)    {\cal V}_{23}^*  =  {\cal V}_{23}  q_{\beta_A,\alpha, 13}{{\cal V}  }_{23}^*{{\cal V}  }_{23}(1_A \otimes \beta(n^{\rm o}) \otimes 1_S)    {\cal V}_{23}^*\nonumber \\
& = q_{\beta_A, \widehat\alpha, 12}  q_{\beta, \alpha, 23} {\cal V}_{23}(1_A \otimes \beta(n^{\rm o}) \otimes 1_S){\cal V}_{23}^* = q_{\beta_A, \widehat\alpha, 12}q_{\beta, \alpha, 23}(1_{A\otimes \cK}\otimes \beta(n^{\rm o})) \nonumber \\
&= \delta_{A \otimes {\cal K}}(1_{A \otimes {\cal K}}) (1_{A \otimes {\cal K}} \otimes \beta(n^{\rm o})).\nonumber
\end{align}
\end{proof}
%et que 
%$\delta_{A \otimes {\cal K}} : M(A \otimes {\cal K}) \rightarrow 
%M(A \otimes {\cal K} \otimes S)$
%est  une application compl\`etement positive   strictement continue et v\'erifiant :  
%\begin{equation}\delta_{A \otimes {\cal K}}(1_{A \otimes {\cal K}}) = q_{\beta_A, \widehat\alpha, 12}  q_{\beta, \alpha, 23}\nonumber\end{equation}

%De m\^eme, l'application compl\`etement positive 
%$\delta_{A \otimes {\cal K}} \otimes {\rm id}_S : A \otimes {\cal K} \otimes S  \rightarrow 
%M(A \otimes {\cal K} \otimes S  \otimes S)$,   se prolonge de fa\c con     unique en une application compl\`etement positive   $M(A \otimes {\cal K} \otimes S) \rightarrow 
%M(A \otimes {\cal K} \otimes S  \otimes S)$, strictement continue et v\'erifiant 
%$$(\delta_{A \otimes {\cal K}} \otimes {\rm id}_S) (1_{A \otimes {\cal K}}) = q_{\beta_A, \widehat\alpha, 12} q_{\beta, \alpha, 23}$$
%\noindent
%Comme $\delta_0(q_{\beta_A,  \widehat\alpha}) = q_{\beta_A,\alpha, 13} \, q_{\widehat\alpha, \beta, 23} = q_{\beta_A,\alpha, 13}{{\cal V}  }_{23}^*{{\cal V}  }_{23}$, on a :
%$$\delta_{A \otimes {\cal K}}(\beta_{A \otimes {\cal K}}(n^{o})) = q_{\beta_A, \widehat\alpha, 12} q_{\beta, \alpha, 23} (1_{A \otimes {\cal K}} \otimes \beta(n^{o})) =   
%\delta_{A \otimes {\cal K}}(1_{A \otimes {\cal K}}) (1_{A \otimes {\cal K}} \otimes \beta(n^{o}))$$\noindent

\noindent
Nous avons : 
\begin{lemme} Pour tout $x\in M(A \otimes {\cal K})$, on a 
$(\delta_{A \otimes {\cal K}} \otimes {\rm id}_S) \circ\delta_{A \otimes {\cal K}}(x)  = ({\rm id}_A \otimes \delta) \circ \delta_{A \otimes {\cal K}}(x)$.
\begin{proof}
Pour tout $x \in M(A \otimes {\cal K})$,  on a 
\begin{align}
(\delta_{A \otimes {\cal K}} \otimes {\rm id}_S)(\delta_{A \otimes {\cal K}}(x)) &= 
{{\cal V}  }_{23} (\delta_0 \otimes {\rm id}_S)(\delta_{A \otimes {\cal K}}(x)) {{\cal V}  }_{23}^* = 
{{\cal V}  }_{23}{{\cal V}  }_{24} q_{\beta_A, \alpha, 13}(\delta_0 \otimes {\rm id}_S) \delta_0(x)q_{\beta_A, \alpha, 13}
{{\cal V}  }_{24}^*{{\cal V}  }_{23}^*\nonumber \\ 
&=
{{\cal V}  }_{23}{{\cal V}  }_{24}({\rm id}_{A \otimes {\cal K}} \otimes \delta) \delta_0(x) {{\cal V}  }_{24}^*{{\cal V}  }_{23}^* = 
({\rm id}_{A \otimes {\cal K}} \otimes \delta)({{\cal V}  }_{23} \delta_0(x) {{\cal V}  }_{23}^*) \nonumber \\
&= 
({\rm id}_{A \otimes {\cal K}} \otimes \delta) \delta_{A \otimes {\cal K}}(x).\nonumber 
\end{align}
\end{proof}
\end{lemme}
\noindent
Dans ce qui suit, on identifie $M(A \otimes {\cal K})$ (resp. $M(A \otimes {\cal K} \otimes S)$) \`a 
${\cal L}(A \otimes H)$ (resp. ${\cal L}(A \otimes H \otimes S)$). 

\noindent
\begin{proposition}\label{bidu0} Soient $a\in A$, $T\in M(A)$, $ x\in \widehat S$ et $s\in S$.
\begin{enumerate}
\item $({\rm id}_{A \otimes H} \otimes L)\delta_0(\pi_R(a)) = q_{\beta_A, \alpha, 13}(1_A \otimes U\otimes 1_H)
\Sigma_{23} V_{23} (\pi_L(a) \otimes 1_H) V_{23}^*\Sigma_{23} (1_A \otimes U^* \otimes 1_H)q_{\beta_A, \alpha, 13}$.
\item $\delta_{A \otimes {\cal K}}(\pi_R(a)) = q_{\beta_A, \widehat\alpha, 12} q_{\beta, \alpha, 23} (\pi_R(a) \otimes 1_S), \,\, \delta_{A \otimes {\cal K}}(\pi_R(T)) = q_{\beta_A, \widehat\alpha, 12} q_{\beta, \alpha, 23} (\pi_R(T) \otimes 1_S)$.
\item $\begin{aligned}[t]
\delta_{A \otimes {\cal K}}(1_A \otimes \lambda(x)L(s)) & = q_{\beta_A, \widehat\alpha, 12} q_{\beta, \alpha, 23} (1_A \otimes \lambda(x) \otimes 1_S) ( 1_A \otimes [(L \otimes {\rm id}_S) \delta(s)])  
\nonumber\\ 
& = q_{\beta_A, \widehat\alpha, 12}  (1_A \otimes \lambda(x) \otimes 1_S) ( 1_A \otimes [(L \otimes {\rm id}_S) \delta(s)]).
\end{aligned}$
\item $\begin{aligned}[t]
\delta_{A \otimes {\cal K}}(\pi_R(a) (1_A \otimes \lambda(x)L(s)) )  &=
q_{\beta_A, \widehat\alpha, 12} \pi_R(a)_{12} (1_A \otimes \lambda(x) \otimes 1_S) ( 1_A \otimes [(L \otimes {\rm id}_S) \delta(s)]) \nonumber\\
& = (\pi_R(a) \otimes 1_S) (1_A \otimes [(\lambda(x) \otimes 1_S) (L\otimes {\rm id}_S) \delta(s)]).
\end{aligned} $
\end{enumerate}
\end{proposition}
\begin{proof} 
 Reprenons (\ref{piL} , \ref{DEAR} b)) les repr\'esentations 
 $\pi_R , \pi_L : M(A)\rightarrow {\cal L}(A \otimes H)$. \hfill\break
On a $[V^*V, S \otimes 1_H] = 0$ (\cf \ref{2.9} c) et \ref{2.10} b)). On en d\'eduit que pour tout $a\in A$, on a :
 \begin{equation}
 ({\rm id}_{A \otimes H} \otimes L) \delta_0(\pi_L(a)) = \Sigma_{23}
V_{23} (\pi_L(a) \otimes 1_H) V_{23}^*\Sigma_{23}.
\end{equation} 
On en d\'eduit : 
\begin{equation}
({\rm id}_{A \otimes H} \otimes L)\delta_0(\pi_R(a)) = q_{\beta_A, \alpha, 13}(1_A \otimes U\otimes 1_H)
\Sigma_{23} V_{23} (\pi_L(a) \otimes 1_H) V_{23}^*\Sigma_{23} (1_A \otimes U^* \otimes 1_H)q_{\beta_A, \alpha, 13},
\end{equation}
d'o\`u le a).
Il s'ensuit qu'on a  aussi 
\begin{align} 
({\rm id}_{A \otimes H} \otimes L)\delta_{A \otimes {\cal K}}&(\pi_R(a))\nonumber \\
&= q_{\beta_A, \widehat\alpha, 12}V_{23} (1_A \otimes U\otimes 1_H)
\Sigma_{23} V_{23} (\pi_L(a) \otimes 1_H) V_{23}^*\Sigma_{23} (1_A \otimes U^* \otimes 1_H) V_{23}^*q_{\beta_A, \widehat\alpha, 12}\nonumber \\
& = q_{\beta_A, \widehat\alpha, 12}V_{23}\widetilde V_{23}  \pi_L(a)_{13} \widetilde V_{23}^*V_{23}^*q_{\beta_A, \widehat\alpha, 12}. 
\end{align}
Posons  avec les notations de \ref{th13}, $T_{\widehat\beta, \widehat\alpha} := \sum_{l=1}^k \lambda_l (\widehat\beta(e^{(l)})  \otimes \widehat\alpha(e^{(l)}))$.\index{t@$T_{\widehat\beta, \widehat\alpha}$} \hfill\break 
Il r\'esulte alors de (\ref{th13} ,  \ref{2.9} c)  et  \ref{2.10} b)) que  
 \begin{align}
 ({\rm id}_{A \otimes H} \otimes L)\delta_{A \otimes {\cal K}}&(\pi_R(a))\nonumber\\
 &= q_{\beta_A, \widehat\alpha, 12}W_{23}^*\Sigma_{23} (1_A \otimes 1_H \otimes U^*) 
T_{\widehat\beta, \widehat\alpha, 23}\pi_L(a)_{13} T_{\widehat\beta, \widehat\alpha, 23}^*(1_A \otimes 1_H \otimes U) \Sigma_{23}
W_{23} q_{\beta_A, \widehat\alpha, 12}\nonumber\\
  &= q_{\beta_A, \widehat\alpha, 12} \pi_R(a)_{12} 
W_{23}^*\Sigma_{23} (1_A \otimes 1_H \otimes U^*) 
T_{\widehat\beta, \widehat\alpha, 23} T_{\widehat\beta, \widehat\alpha, 23}^*(1_A \otimes 1_H \otimes U) \Sigma_{23}
W_{23} q_{\beta_A, \widehat\alpha, 12}\nonumber\\
& =   q_{\beta_A, \widehat\alpha, 12} q_{\beta, \alpha, 23}\pi_R(a)_{12},\nonumber
\end{align}
ce qui \'etablit la premi\`ere relation du b). Par continuit\'e stricte, on d\'eduit aussi la seconde relation.
\hfill\break
Pour la preuve du c), on a (\cf \ref{2.10} b) et c))  
\begin{align}
\delta_{A \otimes {\cal K}}(1_A \otimes \lambda(x)L(s))  &= 
{{\cal V}  }_{23} \delta_0(1_A \otimes \lambda(x)L(s)) {{\cal V}  }_{23}^*   
= {{\cal V}  }_{23} q_{\beta_A, \alpha, 13}(1_A \otimes \lambda(x)L(s) \otimes 1_S) {{\cal V}  }_{23}^*  
 \nonumber \\  
&= q_{\beta_A, \widehat\alpha, 12}   
(1_A \otimes \lambda(x) \otimes 1_S) ( 1_A \otimes [(L \otimes {\rm id}_S) \delta(s)]).\nonumber
\end{align}
En remarquant (\ref{2.9}  c)) qu'on a 
\begin{equation}  
({\rm id}_{A \otimes H} \otimes L)\delta_0(\pi_R(a)) = ({\rm id}_{A \otimes H} \otimes L)\delta_0(\pi_R(a)) 
q_{\widehat\alpha, \beta,23} =  ({\rm id}_{A \otimes H} \otimes L)\delta_0(\pi_R(a))  V_{23}^*V_{23}, \nonumber
\end{equation}
on d\'eduit le d)  du b) et du c).
\end{proof}
\noindent
\begin{corollary}\label{bidu} L'action biduale  $(\beta_D, \delta_D)$ (\cf \ref{isofi}) du groupo\"ide m.q.\ ${\cal G}$ dans le double produit crois\'e 
$A \rtimes {\cal G} \rtimes \widehat{\cal  G} \simeq D$ est donn\'ee par :
\begin{equation}
(j_D \otimes {\rm id} _S)\delta_D(d) = \delta_{A \otimes {\cal K}}(d) = {{\cal V}  }_{23}  \delta_0(d) {{\cal V}  }_{23}^*  \,,\,d \in D \;\; ; \;\; j_D \beta_D(n^{\rm o}) = \beta_{A \otimes {\cal K}}(n^{\rm o}) = q_{\beta_A, \widehat\alpha} (1_A \otimes \beta(n^{\rm o})) \,,\, n\in N.\nonumber 
\end{equation}
\begin{proof} Pour $d = \pi_R(a)(1_A \otimes \lambda(x) L(s))$ avec $a\in A$, $x\in \widehat S$ et $s\in S$, il suffit d'appliquer \ref{isofi}. La seconde relation d\'ecoule de \ref{isofi} et de la d\'efinition de $\beta_{A \otimes {\cal K}}$.
\end{proof}
\end{corollary}

\begin{remark}\label{remc} Par continuit\'e stricte, on a aussi 
$$(j_D \otimes {\rm id} _S)\delta_D(T) = \delta_{A \otimes {\cal K}}(j_D(T)) = {{\cal V}  }_{23}  \delta_0(j_D(T)) {{\cal V}  }_{23}^*, \quad T \in M(D).$$
\end{remark}

\begin{theorem}\label{thbidu}  Soit $(\beta_A, \delta_A)$ une action continue du groupo\"ide mesur\'e r\'egulier ${\cal G}$ dans une C*-alg\`ebre $A$. Alors le double produit crois\'e $(A \rtimes {\cal G}) \rtimes \widehat{\cal  G}$, muni de 
$(\beta_{(A \rtimes {\cal G}) \rtimes \widehat{\cal  G}}, \delta_{(A \rtimes {\cal G}) \rtimes \widehat{\cal  G}})$, est canoniquement  isomorphe  \`a la C*-alg\`ebre $D   = q_{\beta_A,  \widehat\alpha} (A \otimes {\cal K}) q_{\beta_A,  \widehat\alpha}$ munie de l'action continue $(\beta_D, \delta_D)$ de ${\cal G}$ d\'efinie par : 
$$(j_D \otimes {\rm id} _S)\delta_D(d) = \delta_{A \otimes {\cal K}}(d) = {{\cal V}  }_{23}  \delta_0(d) {{\cal V}  }_{23}^*,\,d \in D \;\; ; \;\; j_D(\beta_D(n^{\rm o})) = \beta_{A \otimes {\cal K}}(n^{\rm o}) = q_{\beta_A, \widehat\alpha} (1_A \otimes \beta(n^{\rm o})),\, n\in N.$$
\end{theorem}

{\it D\'emonstration.} Tenant compte de (\ref{isofi} et \ref{bidu}), il nous reste \`a montrer l'\'egalit\'e $D = q_{\beta_A,  \widehat\alpha} (A \otimes {\cal K}) q_{\beta_A,  \widehat\alpha}$.
\hfill\break
Il r\'esulte facilement de la   continuit\'e (\ref{ac} d)) de l'action $(\beta_A, \delta_A)$ et de l'inclusion $ \lambda(\widehat S) S  \subset {\cal K}$ (\cf \ref{Sl(hatS)} et \ref{dualreg}) que
$$D \subset q_{\beta_A,  \widehat\alpha} (A \otimes {\cal K}) q_{\beta_A,  \widehat\alpha}.$$
\noindent
Pour d\'emontrer l'autre inclusion, on va montrer que 
$\delta_0(q_{\beta_A,  \widehat\alpha} (A \otimes {\cal K}) q_{\beta_A,  \widehat\alpha)}) \subset \delta_0(D)$.\hfill\break
Nous avons besoin du lemme suivant :

\begin{lemme}\label{reglem}
Soient $\alpha  :  N  \rightarrow B(H)$, $\beta  : N^{\rm o} \rightarrow B(K)$  des repr\'esentations    unitales. 
\begin{enumerate}
\item  Pour tout $\xi , \eta \in H$, nous avons  $q_{\alpha, \beta}  (R_\xi^\alpha (R_\eta^\alpha)^* \otimes 1_K) = (R_\xi^\alpha (R_\eta^\alpha)^* \otimes 1_K)  q_{\alpha, \beta}  = q_{\alpha, \beta} ( \theta_{\xi, \eta} \otimes 1_K)q_{\alpha, \beta}$.
\item Pour tout $\xi , \eta \in K$, nous avons  $q_{\alpha, \beta}  (1_H \otimes L_\xi^\beta (L_\eta^\beta)^*) =
(1_H \otimes L_\xi^\beta (L_\eta^\beta)^*) q_{\alpha, \beta}
= q_{\alpha, \beta} (1_H \otimes \theta_{\xi, \eta})q_{\alpha, \beta}$.
\end{enumerate}
\end{lemme}
Pour la preuve, voir l'appendice (\cf\ref{proput}).
\begin{proof}[Fin de la d\'emonstration du th\'eor\`eme]
  (\ref{d01})  nous donne $\delta_0(q_{\beta_A,  \widehat\alpha})  = q_{\beta_A, \alpha, 13}q_{\widehat\alpha, \beta, 23}$ et, par continuit\'e de l'action, on a  
$$q_{\beta_A,  \widehat\alpha} (A \otimes {\cal K}) q_{\beta_A,  \widehat\alpha} \subset 
[q_{\beta_A,  \widehat\alpha}\pi_R(a) (1_A \otimes k) q_{\beta_A,  \widehat\alpha}\,\,| \,\,a\in A \,,\,k\in{\cal K}].$$
\noindent
On remarque que $\pi_R(a) = q_{\beta_A,  \widehat\alpha}\pi_R(a) = \pi_R(a)q_{\beta_A,  \widehat\alpha}$.
Pour  $a\in A$ et $k\in{\cal K}$, on a 
$$\delta_0(\pi_R(a) q_{\beta_A,  \widehat\alpha} (1_A \otimes k) q_{\beta_A,  \widehat\alpha}) = 
\delta_0(\pi_R(a))q_{\beta_A, \alpha, 13}q_{\widehat\alpha, \beta, 23} (1_A \otimes k \otimes 1_S)
q_{\widehat\alpha, \beta, 23}q_{\beta_A, \alpha, 13}.$$
En appliquant \ref{reglem} \`a $\widehat\alpha$ et $\beta$, on obtient 
$$q_{\widehat\alpha, \beta} ({\cal K} \otimes 1_S)q_{{\widehat\alpha}, \beta} = q_{\widehat\alpha, \beta} ([R_\xi^{\widehat\alpha} (R_\eta^{\widehat\alpha})^*  \,\,|\,\, \xi , \eta \in H] \otimes 1_S) = q_{\widehat\alpha, \beta} (U {\cal C}(W) U^* \otimes 1_S).$$
\noindent 
Il en r\'esulte que 
$$\delta_0(\pi_R(a) q_{\beta_A,  \widehat\alpha} (1_A \otimes k) q_{\beta_A,  \widehat\alpha})\in 
[\delta_0(\pi_R(a)) q_{\beta_A, \alpha, 13}q_{\widehat\alpha, \beta, 23}(1_A \otimes \lambda(x) L(s) \otimes 1_S)  \,\,| \,\, x\in \widehat S\,,\,s\in S].$$
Finalement, on a 
$$\delta_0(\pi_R(a)) q_{\beta_A, \alpha, 13}q_{\widehat\alpha, \beta, 23}(1_A \otimes \lambda(x) L(s) \otimes 1_S) = \delta_0(\pi_R(a) (1_A \otimes \lambda(x) L(s)) \in \delta_0(D).$$
\end{proof}
\noindent

\noindent
\subsubsection{Structure du double produit crois\'e \texorpdfstring{$A \rtimes {\cal G}_{G_1, G_2} \rtimes \widehat{\cal G}_{G_1, G_2} $}{}}\label{dpco}

\noindent
Dans cette section, nous d\'eveloppons les cons\'equences du th\'eor\`eme \ref{thbidu} dans le cas d'une action continue \ref{acc}  $(A, \beta_A, \delta_A)$,  d'un groupo\"ide de co-liaison \ref{defco} ${\cal G} := {\cal G}_{G_1, G_2}$, associ\'es \`a deux groupes quantiques \lc $G_1$ et $G_2$.

Pour d\'ecrire la ${\cal G}$-alg\`ebre double produit crois\'e $(D, \beta_D, \delta_D)$, nous conservons les notations de la section pr\'ec\'edente et nous  utilisons les notations \ref{not1} et \ref{acc} pour la ${\cal G}$-alg\`ebre  $(A, \beta_A, \delta_A)$. Avec ces notations, on v\'erifie facilement qu'on a 
\begin{equation}\label{betD}  
j_D(\beta_D(\varepsilon_j))  = q_j  \otimes \beta(\varepsilon_j) \in\cL(A \otimes H), \quad j = 1 , 2.  \end{equation}
Par  \ref{acc} appliqu\'ee \`a la ${\cal G}$-alg\`ebre $D$, nous avons 
\begin{equation} 
D = D_1 \oplus D_2, \quad \text{o\`u}\quad  D_j:=\beta_D(\varepsilon_j) D = (q_j \otimes \beta(\varepsilon_j)) D \subset \cL(A \otimes H),\quad j=1,2.\nonumber
\end{equation}
La coaction $\delta_D : D \rightarrow M(D \otimes S)$ est d\'etermin\'ee (\ref{acc} a)) par les *-morphismes 
$\delta_{D_j}^k  : D_j \rightarrow M(D_k \otimes S_{kj})$, $j,k=1,2$,
d\'efinis par 
$$\pi_j^k\circ\delta_{D_j}^k(x) = (\beta_D(\varepsilon_k) \otimes p_{kj})\delta_D(x), \quad x\in D_j,$$
\noindent
 o\`u $\pi_j^k :  M(D_k \otimes S_{kj}) \rightarrow  M(D \otimes S)$ est le prolongement strictement continu de l'inclusion 
$D_k \otimes S_{kj} \subset D \otimes S$ v\'erifiant $\pi_j^k(1_{D_k \otimes S_{kj}})= \beta_D(\varepsilon_k) \otimes p_{kj}$.

\begin{notations} \label{not2}Soient 
$j, k, l = 1,2$.
\begin{enumerate}
\item $U_{jk} : H_{jk} \rightarrow H_{kj}$\index{u@$U_{jk}$} est la restriction de l'unitaire  $U$ au sous-espace $H_{jk}$. On note 
$R_{jk} :  S_{jk}  \rightarrow B(H_{kj})$\index{ra@$R_{jk}$} la repr\'esentation d\'efinie par  $R_{jk}(x) = U_{jk}  L(x) U_{jk}^*$.
\item 
   $\iota_j : M(D_j) \rightarrow M(D)$\index{i@$\iota_j$} est le prolongement strictement continu de l'inclusion $D_j \subset D$ v\'erifiant 
$\iota_j(1_{D_j}) = \beta_D(\varepsilon_j)$.
%\hfill\break
%$\iota_{kj} =  (j_D  \otimes L) \circ  \pi_j^k : M(D_k \otimes S_{kj}) \rightarrow \cL(A \otimes H)$ est la repr\'esentation  continue pour les topologies stricte/*-forte v\'erifiat 
%$\iota_{kj}(1_{D_k \otimes S_{kj}}) = q_k \otimes \beta(\varepsilon_k) \otimes p_{kj}$.
% On a $\iota_j(M(D_j)) = \beta_D(\varepsilon_j)M(D)$.
\item ${\cal E}_{A,R}  :=q_{\beta_A, \widehat\alpha}(A \otimes H) =  (q_1 \otimes\beta(\varepsilon_1) + q_2 \otimes \beta(\varepsilon_2)) (A \otimes H) = {\cal E}_{A, R, 1} \oplus {\cal E}_{A, R, 2}$\index{ej@${\cal E}_{A,R,k}$}
avec 
\begin{equation}
{\cal E}_{A, R,  k} := A_k \otimes  \beta(\varepsilon_k)  H =  (A_k \otimes H_{1k}) \oplus (A_k \otimes H_{2k}),\quad k=1,2.
\end{equation}
\item $q_{lkj} := q_k \otimes p_{lk} \otimes p_{kj}\in {\cal L}(A \otimes H \otimes H)$.\index{qc@$q_{lkj}$}
 \end{enumerate}
\end{notations}
% - $\pi_L = ({\rm id}_A \otimes L)\delta_A : A \rightarrow {\cal L}(A \otimes H)$  et aussi $\pi_R = ({\rm id}_A \otimes R)\delta_A : A \rightarrow {\cal L}(A \otimes H)$.\hfill\break
% - ${\cal E}_{A,R}  =q_{\beta_A, \widehat\alpha}(A \otimes H) =  (q_1 \otimes\beta(\varepsilon_1) + q_2 \otimes \beta(\varepsilon_2)) (A \otimes H) = {\cal E}_{A, R, 1} \oplus {\cal E}_{A, R, 2}$
%\hfill\break
%avec ${\cal E}_{A, R,  k} = A_k \otimes  \beta(\varepsilon_k)  H =  A_k \otimes H_{1,k} \oplus A_k \otimes H_{2,k}\,,\, k=1,2$. 
%\hfill\break
%Notons que ${\cal K}({\cal E}_{A,R}) = {\cal K}({\cal E}_{A,R, 1}) \oplus {\cal K}({\cal E}_{A,R, 2})$ car $A_1$ et $A_2$ sont des id\'eaux bilat\`eres ferm\'es de $A$ telles que $A_ 1 A_2 = \{ 0 \}$.
%\hfill\break
%- On rappelle (?) que  pour $x\in \widehat S$ et $l , l' =1 , 2$,  on a pos\'e $x_{ll'} =\beta(\varepsilon_l) x \beta(\varepsilon_{l'})$ et on a $p_{lj} \lambda(x) p_{l'j} = \widehat\pi^j(\lambda(x_{ll'}))$.
%\hfill\break

\begin{lemme}\label{etap1} Soient  $a\in A$, $x \in \widehat S$, $s\in S$. Posons :
\begin{equation} d := \pi_R(a)(1_A \otimes \lambda(x) L(s)) = d_1 + d_2 \quad ; \quad d_j = (q_j \otimes \beta(\varepsilon_j) ) d,\quad j=1,2.\nonumber 
\end{equation} 
Nous avons :
\begin{enumerate}
\item $d_j = \sum_{l,l'=1,2} d_{ll',j}$, o\`u  $d_{ll',j} = \pi_R(aq_l)(1_A \otimes p_{lj} \lambda(x)p_{l'j} L(s)p_{l'j})  =  \pi_R(aq_l)(1_A \otimes \widehat\pi^j(\lambda(x_{ll'})) L(s)p_{l'j})$ ; 
\item $(q_k \otimes \beta(\varepsilon_k) \otimes p_{kj})  (j_D \otimes L)\delta_D(d_j)  = $
$$\!\!\! \sum_{l, l'=1,2}
q_{lkj}  V_{23} (q_k \otimes p_{lj} \otimes p_{kj})
\Sigma_{23} (\pi_L \otimes R)\delta_A(a q_l) \Sigma_{23} 
(1_A \otimes p_{lj} \lambda(x) p_{l'j} L(s)\otimes 1_H)
(q_k \otimes p_{l'j} \otimes p_{kj}) (V_{23})^* q_{l'kj}.$$
\end{enumerate}
\end{lemme}

\begin{proof} Par d\'efinition de la repr\'esentation (\ref{notE} b)) $\widehat\pi^j$, on a  
$p_{lj} \lambda(x)p_{l'j} = \widehat\pi^j(\lambda(x_{ll'}))$. En utilisant (\ref{Eq:commut}), on obtient : 
\begin{align} 
d_j = (q_j \otimes\beta(\varepsilon_j)) \pi_R(a)(1_A \otimes \lambda(x) L(s)) &=
(q_j \otimes\beta(\varepsilon_j)) \pi_R(a)(1_A \otimes \beta(\varepsilon_j)\lambda(x) \beta(\varepsilon_j)L(s))  \nonumber\\
& = \sum_{l,l'} (q_j \otimes\beta(\varepsilon_j)) \pi_R(a)(1_A \otimes p_{lj}\lambda(x) p_{l'j}L(s)) p_{l'j}. 
\end{align}
\hfill\break
Par \ref{com1}, on a  $(q_j \otimes\beta(\varepsilon_j)) \pi_R(a)(1_A \otimes p_{lj}) = \pi_R(aq_l)(1_A \otimes p_{lj})$, d'o\`u le a).
\hfill\break
Par \ref{bidu},  nous avons 
\begin{equation}\label{f1}(j_D \otimes L)\delta_D(d_{ll',j}) = V_{23} ({\rm id}_A \otimes L)\delta_0(d_{ll',j})V_{23}^*. 
\end{equation}
\noindent
Par le (\ref{bidu0} a)), nous avons  
\begin{align}({\rm id}_{A \otimes H} \otimes L)\delta_0(\pi_R(a)) &= q_{\beta_A, \alpha, 13}(1_A \otimes U\otimes 1_H)
\Sigma_{23} V_{23} (\pi_L(a) \otimes 1_H) V_{23}^*\Sigma_{23} (1_A \otimes U^* \otimes 1_H)q_{\beta_A, \alpha, 13}\nonumber\\
&= q_{\beta_A, \alpha, 13} \Sigma_{23} (\pi_L \otimes R)\delta_A(a) \Sigma_{23}q_{\beta_A, \alpha, 13}. \nonumber
\end{align}
On en d\'eduit :
\begin{align} ({\rm id}_A \otimes L)\delta_0(d_{ll',j}) & = ({\rm id}_A \otimes L)\delta_0(\pi_R(aq_l))q_{\beta_A, \alpha, 13} (1_A \otimes p_{lj}\lambda(x) p_{l'j}L(s) \otimes 1_H)\nonumber \\ 
&  = q_{\beta_A, \alpha, 13} \Sigma_{23} (\pi_L \otimes R)\delta_A(a q_l) \Sigma_{23}q_{\beta_A, \alpha, 13}
(1_A \otimes p_{lj}\lambda(x) p_{l'j}L(s) \otimes 1_H).\nonumber 
\end{align}
En utilisant (\ref{2.10} b)), la relation (\ref{f1}) devient :
\begin{equation} (j_D \otimes L)\delta_D(d_{ll',j}) = q_{\beta_A, \widehat\alpha, 12} V_{23} \Sigma_{23} (\pi_L \otimes R)\delta_A(a q_l) \Sigma_{23} (1_A \otimes p_{lj}\lambda(x) p_{l'j}L(s) \otimes 1_H) V_{23}^*q_{\beta_A, \widehat\alpha, 12}.\nonumber 
\end{equation}
Mais on a $\widehat\alpha = \beta$ et les projecteurs $q_k , p_{lj}$ sont centraux dans $M(A)$ et $M(S)$ respectivement.  En tenant compte des relations de commutation (\ref{2.10} b)) et de la notation (\ref{not2} d)), on obtient  :
\begin{align} (q_k \otimes \beta(\varepsilon_k) \otimes p_{kj})V_{23} (1_A \otimes p_{lj} \otimes 1_H) & = 
 (q_k \otimes \alpha(\varepsilon_l) \beta(\varepsilon_k) \otimes p_{kj})V_{23} (1_A \otimes p_{lj} \otimes 1_H) \nonumber \\
&= 
q_{lkj}V_{23}(q_k \otimes p_{lj} \otimes p_{kj})\nonumber\end{align}
et aussi  $(q_k \otimes p_{l'j} \otimes p_{kj}) V_{23}^*q_{\beta_A, \widehat\alpha, 12} = (q_k \otimes p_{l'j} \otimes p_{kj})V_{23}^*q_{l'kj}$, d'o\`u le b).
\end{proof}
\noindent
Avec les notations (\ref{not2} c)), introduisons les représentations $\pi_{D_j} $ et $\pi_{D_k} \otimes L_{kj}$ d\'efinies par : 
\begin{equation}\label{equ:piDj}
\pi_{D_j} : D_j \rightarrow  {\cal L}({\cal E}_{A, R, j}) : u \mapsto  \pi_{D_j}(u) = \restr{u}{ {\cal E}_{A, R, j} } 
 \quad ; \quad \pi_{D_k} \otimes L_{kj} : D_k \otimes S_{kj} \rightarrow {\cal L}({\cal E}_{A, R, k} \otimes H_{kj}).
\end{equation}
\noindent
Les repr\'esentations $\pi_{D_j} $ et $\pi_{D_k} \otimes L_{kj}$ sont fid\`eles et non d\'eg\'en\'er\'ees.  Les  prolongements strictement continus aux C*-alg\`ebres des multiplicateurs,  sont donn\'es par :
$$\pi_{D_j} : M(D_j) \rightarrow  {\cal L}({\cal E}_{A, R, j}) : T \mapsto  \restr{j_D(\iota_j(T))}{ {\cal E}_{A, R, j} }   , \quad \iota_j : M(D_j) \rightarrow M(D) ,\quad  \iota_j(1_{D_j}) = \beta_D(\varepsilon_j)\, ;$$
\noindent
$$\pi_{D_k} \otimes L_{kj} : M(D_k \otimes S_{kj}) \rightarrow {\cal L}({\cal E}_{A, R, k} \otimes H_{kj})  : T \mapsto 
 (\pi_{D_k} \otimes L_{kj})(T) =  \restr{(j_D \otimes L) \pi_j^k(T)}{ {\cal E}_{A, R, k} \otimes H_{kj} }. $$  
%d\'efinies par :
%$$\pi_{D_j}(u) = u_{|_{{\cal E}_{A, R, j}}} \quad , \quad
%(\pi_{D_k} \otimes L_{kj})(T)  = (j_D \otimes L) \pi_j^k(T)_{|_{{\cal E}_{A, R, k} \otimes H_{kj}}}$$

\begin{proposition} \label{etap2}  Soient  $a\in A \,,\, x \in \widehat S  \,, \, s\in S$.  Posons 
\begin{equation} d = \pi_R(a)(1_A \otimes \lambda(x) L(s)) = d_1 + d_2, \quad {\rm avec} \quad d_j = (q_j \otimes \beta(\varepsilon_j) ) d,\quad j=1,2.\nonumber
\end{equation} 
Nous avons avec les notations    (\ref{notE} b) et   \ref{not2}) :
\begin{enumerate}
\item Dans  ${\cal L}({\cal E}_{A,R,j}) = \bigoplus_{l, l'=1,2} {\cal L}(A_j \otimes H_{l'j}   ,  A_j \otimes H_{lj})$, on a 
$$\pi_{D_j}(d_j) =   \sum_{l, l'=1,2} ({\rm id}_{A_j} \otimes R_{jl})(\delta_{A_l}^j(a q_l))(1_{A_j} \otimes \widehat\pi^j (\lambda(x_{ll'})) L(s)p_{l'j})\quad \text{;}$$
\item  Dans  ${\cal L}({\cal E}_{A, R, k} \otimes H_{kj}) =  \bigoplus_{l, l'=1,2} {\cal L}(A_k \otimes H_{l'k} \otimes H_{kj} ,  A_k \otimes H_{lk} \otimes H_{kj})$, on a : 
$$(\pi_{D_k}  \otimes L_{kj}) \delta_{D_j}^k(d_j) = $$
\noindent 
$$\sum_{l, l'=1,2}
V_{kj, 23}^l  \Sigma_{23} \hfill\break ({\rm id}_{A_k} \otimes L_{kj} \otimes R_{jl})[(\delta_{A_j}^k  \otimes  {\rm id}_{S_{jl}}) \delta_{A_l}^j (aq_l)]
(1_{A_k} \otimes   1_{H_{kj}} \otimes \widehat\pi^j(\lambda(x_{ll'})) L(s) p_{l'j})  \Sigma_{23}  
 (V_{kj, 23}^{l'})^*.
$$
\end{enumerate}
\end{proposition}
\noindent
\begin{proof} La   d\'efinition des *-morphismes \ref{acc} $\delta_{A_l}^j$ entra\^ine avec les notations \ref{not2} :
$$\pi_{D_j}(d_{ll',j}) =  ({\rm id}_{A_j} \otimes R_{jl})(\delta_{A_l}^j(a q_l))(1_{A_j} \otimes \widehat\pi^j (\lambda(x_{ll'})) L(s)p_{l'j})\in {\cal L}(A_j \otimes H_{l'j} , A_j \otimes H_{lj}),$$
\noindent
d'o\`u le a).

Pour tout $T\in M(D_k \otimes S_{kj})$, on a $(\pi_{D_k} \otimes L_{kj})(T) = \restr{(j_D \otimes L) \pi_j^k(T)}{ {\cal E}_{A, R, k} \otimes H_{kj} }$.
\hfill\break
Par d\'efinition \ref{acc} des *-morphismes $\delta_{D_j}^k$  et     (\ref{betD})  de $\beta_D$, on a 
$$(j_D \otimes L) \pi_j^k \delta_{D_j}^k(d_{ll',j}) = (j_D \otimes L)(\beta_D(\varepsilon_k) \otimes p_{kj}) \delta_D(d_{ll',j}) = 
(q_k \otimes \beta(\varepsilon_k) \otimes p_{kj}) (j_D \otimes L)\delta_D(d_{ll',j}).$$
\noindent
Le b) de \ref{etap1} nous donne
$$(j_D \otimes L) \pi_j^k \delta_{D_j}^k(d_{ll',j}) = $$
\noindent 
$$
q_{lkj}  V_{23} (q_k \otimes p_{lj} \otimes p_{kj})
\Sigma_{23} (\pi_L \otimes R)\delta_A(a q_l) \Sigma_{23}
(1_A \otimes p_{lj} \lambda(x) p_{l'j} L(s)\otimes 1_H)
(q_k \otimes p_{l'j} \otimes p_{kj}) (V_{23})^* q_{l'kj}. $$\noindent
Comme $q_{lkj}  = q_k \otimes p_{lk} \otimes p_{kj}$ et en tenant compte des notations \ref{Vind},  
on  d\'eduit  : 
$$(\pi_{D_k}  \otimes L_{kj}) \delta_{D_j}^k(d_j) = $$
\noindent 
$$\sum_{l, l'}
V_{kj, 23}^l  \Sigma_{23} \hfill\break ({\rm id}_{A_k} \otimes L_{kj} \otimes R_{jl})[(\delta_{A_j}^k  \otimes  {\rm id}_{S_{jl}}) \delta_{A_l}^j (aq_l)]
(1_{A_k} \otimes   1_{H_{kj}} \otimes \widehat\pi^j(\lambda(x_{ll'})) L(s) p_{l'j})  \Sigma_{23}  
 (V_{kj, 23}^{l'})^*.
$$\noindent
\end{proof}
Dans ce qui suit, nous allons montrer que chaque $G_j$-alg\`ebre  $(D_j, \delta_{D_j}^j)$ est une alg\`ebre de liaison.

\begin{lemme}\label{liaison1} Soient $l, j = 1 ,2$.
\begin{enumerate}
\item Il existe un unique projecteur   $e_{l, j}\in M(D_j)$ v\'erifiant $j_D(\iota_j(e_{l, j})) = q_j \otimes p_{lj}$.\index{ek@$e_{l,j}$}
\item On a $e_{1, j} + e_{2, j} = 1_{D_j}  , \quad [D_j e_{l, j} D_j] = D_j   , \quad 
\delta_{D_j}^j(e_{l, j}) = e_{l, j} \otimes 1_{S_{jj}}.$
\item Pour $k = 1 , 2$,  on a $\delta_{D_j}^k(e_{l, j}) = e_{l, k} \otimes 1_{S_{kj}}.$
\item 
 $\pi_{D_j}(e_{l, j}) = e_{l, A_j}$, o\`u $e_{l, A_j}\in {\cal L}({\cal E}_{A, R, j})$ est le projecteur sur le sous-module 
$A_j \otimes H_{lj}$.
\end{enumerate}
\end{lemme}

\begin{proof} 
On a $j_D(M(D)) = \{T\in {\cal L}(A \otimes H) \, | \, T D \subset D \,, \, D T \subset D \, , \,T j_D(1_D) = 
j_D(1_D) T = T\}$. 
On d\'eduit alors  de la d\'efinition de $D$ et de $j_D(1_D)$ (\cf\ref{DEAR} a) et d)) que $T :=  q_j \otimes p_{lj} \in j_D(M(D))$.

Comme $T j_D(\beta_{D}(\varepsilon_j)) = 
T (q_j \otimes \beta(\varepsilon_j)) = T$, on a 
$q_j \otimes p_{lj} \in j_D(\beta_{D}(\varepsilon_j) M(D))$. L'existence et l'unicit\'e de $e_{l, j}\in M(D_j)$ r\'esulte alors de l'\'egalit\'e 
$\iota_j(M(D_j)) = \beta_{D}(\varepsilon_j) M(D)$, d'o\`u le a).

Le *-morphisme $j_D \circ \iota_j$ \'etant fid\`ele et prolongeant l'inclusion $D_j \subset D$, il est clair qu'on a 
$e_{1, j} + e_{2, j} = 1_{D_j}$. L'\'egalit\'e $[D_j e_{l, j} D_j] = D_j$ se d\'eduit de (\ref{etap1} a)) et de    $[E_{l'l, \lambda}^j  E_{ll'', \lambda}^j] = E_{l'l'', \lambda}^j$  (\cf\ref{Emorita}). 

L'\'egalit\'e  $\delta_{D_j}^k(e_{l, j}) = e_{l, k} \otimes 1_{S_{kj}}$ est \'equivalente \`a :
\begin{equation}\label{pikj} 
(j_D \otimes L)\pi_j^k \delta_{D_j}^k(e_{l, j}) = q_k \otimes p_{lk} \otimes p_{kj}.
\end{equation}
\noindent 
On a :
\begin{align}
(j_D \otimes L)\pi_j^k \delta_{D_j}^k(e_{l, j})  &= (j_D \otimes L)
(\delta_D(\iota_j(e_{l, j})) (\beta_D(\varepsilon_k)  \otimes p_{kj})) &  (\ref{not2} \,b)) \nonumber \\
& =
V_{23} ({\rm id} \otimes L)\delta_0(j_D(\iota_j(e_{l, j}))) V_{23}^* (q_k \otimes \beta(\varepsilon_k) \otimes p_{kj}) 
& (\ref{remc}\;\text{et} \;\ref{betD}) \nonumber \\
& = 
V_{23} ({\rm id} \otimes L)\delta_0(q_j \otimes p_{l,j}) V_{23}^* (q_k \otimes \beta(\varepsilon_k) \otimes p_{kj}) &  (\ref{liaison1} \, a)) \nonumber \\
&  =
 \sum_{s=1,2} V_{23} q_{\beta_A, \alpha, 13} (q_s \otimes p_{lj} \otimes p_{sj}) V_{23}^* (q_k \otimes \beta(\varepsilon_k) \otimes p_{kj}) 
&  (\ref{delq}) \nonumber \\
& =
  V_{23} (q_k \otimes p_{lj} \otimes p_{kj})V_{23}^* (q_k \otimes \beta(\varepsilon_k) \otimes p_{kj}) 
&   (\ref{qba}) \nonumber \\
&=
  [VV^* (1_H \otimes \beta(\varepsilon_j))VV^*]_{23} (q_k \otimes p_{lk} \otimes p_{kj})  &  (\ref{moritaS} \, a))\nonumber \\
& =
q_{\beta, \alpha, 23} (q_k \otimes p_{lk} \otimes p_{kj}) = q_k \otimes p_{lk} \otimes p_{kj}.  & \nonumber 
\end{align} 
 
Pour la preuve du d), on a $\pi_{D_j}(e_{l, j}) = \restr{j_D(\iota_j(e_{l, j}))}{ {\cal E}_{A, R, j} } =
\restr{(q_j \otimes p_{lj})}{ {\cal E}_{A, R, j} } = e_{l, A_j}$.
\end{proof}

\noindent
\begin{notations} Pour $j , l , l'  = 1, 2$, posons :\index{dg@$D_{ll', j}$, $D_{l,j}$}
$$D_{ll', j}  :=  e_{l,j}  D_j e_{l',j}  
  , \quad  D_{l,j}  := D_{ll,j}.   $$ 
\end{notations}
\begin{corollary}\label{etap3} Soient $j ,k, l , l'  = 1, 2$.
\begin{enumerate}
\item Par restriction de la structure de $G_j$-alg\`ebre de $D_j$, 
$D_{ll', j}$  est un $D_{l,j}-D_{l',j}$   bimodule hilbertien  $G_j$-\'equivariant (\cf \cite{BaSka1}).
\item Dans le cas $j\not= k$, nous avons :
\begin{enumerate}\renewcommand\theenumii{\roman{enumii}}
\renewcommand\labelenumii{\rm ({\theenumii})}
\item  Le *-morphisme $D_j \rightarrow \delta_{D_j}^k(D_j)  : x  \mapsto \delta_{D_j}^k(x)$ est un *-isomorphisme  d'alg\`ebres de liaison ;
\item  $\delta_{D_j}^k(D_{ll', j})$ est un $\delta_{D_j}^k(D_{l, j})-\delta_{D_j}^k(D_{l', j})$ bimodule hilbertien et par restriction de l'isomorphisme du (i),  le  $D_{l,j}-D_{l',j}$   bimodule hilbertien  $G_j$-\'equivariant $D_{ll', j}$ est canoniquement isomorphe au $\delta_{D_j}^k(D_{l, j})-\delta_{D_j}^k(D_{l', j})$ bimodule hilbertien $\delta_{D_j}^k(D_{ll', j})$ au dessus des *-isomorphismes 
$D_{l,j}   \rightarrow \delta_{D_j}^k(D_{l,j})$  et  $D_{l',j}   \rightarrow \delta_{D_j}^k(D_{l',j})$.
 \end{enumerate}
\end{enumerate}
\end{corollary}
\begin{proof} Le corollaire est une cons\'equence imm\'ediate de \ref{liaison1}.
\end{proof}
 Nous verrons dans le cas o\`u les groupes quantiques
sont r\'eguliers que le  bimodule hilbertien  $G_j$-\'equivariant $\delta_{D_j}^k(D_{ll', j})$ s'obtient directement \`a l'aide du bimodule $D_{ll', k}$  par un proc\'ed\'e d'induction.

\noindent
{\bf Cas o\`u $G_1$ et $G_2$ sont r\'eguliers}

\noindent  
Dans ce ce cas on a \ref{thbidu}   
$D = q_{\beta_A, \widehat\alpha} ( A \otimes {\cal K} )  q_{\beta_A, \widehat\alpha}  \subset {\cal L}(A \otimes H)$. Comme  $q_{\beta_A, \widehat\alpha}(A \otimes H) = {\cal E}_{A,R} = {\cal E}_{A,R, 1} \oplus {\cal E}_{A,R, 2}$, on en d\'eduit que   
$$\pi_{D_j} :  D_j   \rightarrow  {\cal K}({\cal E}_{A,R, j}),\quad \text{o\`u}\quad   {\cal E}_{A,R, j}  = (A_j \otimes H_{1j}) \oplus (A_j \otimes H_{2j}), \quad  j = 1 , 2,$$
\noindent
est un *-isomorphisme d'alg\`ebres de liaison. 
 
Dans ce qui suit,  
on identifie les C*-alg\`ebres ${\cal K}({\cal E}_{A,R, j})$ et 
$A_j \otimes {\cal K}(H_{1j} \oplus H_{2j})$ et on introduit les notations suivantes :

\begin{notations} \label{not4} Soient $j , k , l , l'  = 1 , 2$.
\begin{enumerate}
 \item  ${\cal B}_k :=\pi_{D_k}(D_k) =  A_k \otimes {\cal K}(H_{1k} \oplus H_{2k})$, ${\cal B}_{ll',k} := A_k \otimes  {\cal K}(H_{l'k}, H_{lk})$, ${\cal B}_{l',k}  := {\cal B}_{ l'l',k}  = A_k \otimes {\cal K}(H_{l'k})$.\index{bc@${\cal B}_k$, ${\cal B}_{ll',k}$, ${\cal B}_{l',k}$}
\item  $\delta_{{\cal B}_j}^k := (\pi_{D_k} \otimes {\rm id}_{S_{kj}}) \circ  \delta_{D_j}^k \circ \pi_{D_j}^{-1} :  {\cal B}_j \rightarrow M({\cal B}_k \otimes S_{kj})$.
\item $\delta_{{\cal B}_{j}, 0}^k : {\cal B}_{j} \rightarrow M(A_{k} \otimes {\cal K}(H_{1j} \oplus H_{2j}) \otimes  S_{kj}) $  est le *-morphisme injectif  d\'efini par : 
$$\delta_{{\cal B}_{j}, 0}^k(a \otimes T):=
\delta_{A_j}^k(a)_{13} (1_{A_k} \otimes T \otimes 1_{S_{kj}}) \quad ; \quad a\in A_j, \quad
T\in {\cal K}(H_{1j} \oplus H_{2j}).$$  
\item $\delta_{{\cal B}_{ll',j}, 0}^k : {\cal B}_{ll',j} \rightarrow {\cal L}(A_k\otimes{\cal K}(H_{l'j})\otimes S_{kj},A_k\otimes{\cal K}(H_{l'j},H_{lj})\otimes S_{kj})$ est l'application lin\'eaire d\'efinie par : 
$$\delta_{{\cal B}_{ll',j}, 0}^k(a \otimes T):=
\delta_{A_j}^k(a)_{13} (1_{A_k} \otimes T \otimes 1_{S_{kj}}) \quad ; \quad a\in A_j, \quad
T\in {\cal K}(H_{l'j}, H_{lj}).$$
 \item $\delta_{{\cal B}_{ll',j}}^k  : {\cal B}_{ll',j} \rightarrow  {\cal L}({\cal B}_{l',k}  \otimes   S_{kj} , {\cal B}_{ll',k}  \otimes S_{kj}) :
T \mapsto V_{kj, 23}^l  \delta_{{\cal B}_{ll',j}, 0}^k(T) (V_{kj, 23}^{l'})^* $. 
 \item $\delta_{{\cal B}_{l,j}}^k :=  \delta_{{\cal B}_{ll,j}}^k :  {\cal B}_{l,j} 
\rightarrow  {\cal L}( {\cal B}_{l,k} \otimes    S_{kj})$.\index{dh@$\delta_{{\cal B}_j}^k$, $\delta_{{\cal B}_{j}, 0}^k$, $\delta_{{\cal B}_{ll',j}, 0}^k$, $\delta_{{\cal B}_{ll',j}}^k$, $\delta_{{\cal B}_{l,j}}^k$} 
\end{enumerate}
\end{notations}
%{\cal L}(A_k \otimes {\cal K}(H_{l'k}) \otimes   S_{kj} , A_k \otimes {\cal K}(H_{l'k}, H_{lk}) \otimes S_{kj}) :
\noindent
\begin{proposition}\label{etap4} Soient $j , k , l , l' = 1 , 2$.
\begin{enumerate}
\item  ${\cal B}_{l,j}$ est une C*-alg\`ebre et ${\cal B}_{ll',j}$, muni des actions \'evidentes de ${\cal B}_{l,j}$ et ${\cal B}_{l',j}$, est un ${\cal B}_{l,j} - {\cal B}_{l',j}$ bimodule hilbertien.
\item     $\pi_{D_j}(e_{l, j} D_j e_{l', j}) = A_j \otimes {\cal K}(H_{l'j}, H_{lj})  = {\cal B}_{ll',j}$.
\item  Pour tout $j , k , l , l' = 1 , 2$, on a :
$$\delta_{{\cal B}_{ll',j}}^k(\xi) =  V_{kj, 23}^l  ({\rm id}_{A_k} \otimes \sigma)(\delta_{A_j}^k  \otimes {\rm id}_{ {\cal K}(H_{l'j}, H_{lj})})(\xi) (V_{kj, 23}^{l'})^*, \quad  \xi\in {\cal B}_{ll',j},$$
\noindent
o\`u $\sigma  : S_{kj} \otimes {\cal K}(H_{l'j}, H_{lj})  \rightarrow {\cal K}(H_{l'j}, H_{lj}) \otimes  S_{kj} : a \otimes T \mapsto  T \otimes  a$.
\end{enumerate}
\end{proposition} 

\begin{proof} Le a) est \'evident et b) r\'esulte \ref{thbidu}  de l'\'egalit\'e $D = q_{\beta_A, \widehat\alpha}  (A \otimes {\cal K})   q_{\beta_A, \widehat\alpha} = D_1 \oplus D_2$ et de (\ref{liaison1} d)). On peut aussi d\'eduire directement le b) en utilisant la r\'egularit\'e \ref{reg group}  du groupo\"ide ${\cal G}$, la continuit\'e \ref{acc} de la coaction    $\delta_A$ et (\ref{liaison1} d)).
 
Pour \'etablir le c), Il suffit de v\'erifier l'\'egalit\'e cherch\'ee pour les $\xi\in{\cal B}_{ll',j}$ de la forme $\xi = \pi_{D_j}(d_{ll',j})$, o\`u 
$d_{ll',j}  = \pi_R(aq_l)(1_A \otimes p_{lj} \lambda(x)p_{l'j} L(s)p_{l'j})$
  (\ref{etap1} a)). Dans ce cas le c) est exactement  (\ref{etap2} b)).
\end{proof} 

\begin{corollary} \label{etap5} Soient $j ,   l, l' = 1 , 2$.
\begin{enumerate}
\item  $\delta_{{\cal B}_{l,j}}^j :=  \delta_{{\cal B}_{ll,j}}^j :  A_j \otimes {\cal K}(H_{lj})
\rightarrow  {\cal L}( A_j \otimes {\cal K}(H_{lj}) \otimes    S_{jj})$ est une action continue du groupe quantique $G_j$ dans la C*-alg\`ebre ${\cal B}_{l,j}$.
\item  Dans le cas $l=j$, la coaction $\delta_{{\cal B}_{j,j}}^j$ co\"incide avec la coaction biduale du double produit crois\'e $A_j \rtimes G_j  \rtimes \widehat{G_j}$, apr\`es identification (\cf\cite{BaSka2}) avec  $A_j \otimes {\cal K}(H_{jj})$.
\item  Supposons $l \not=l'$. Alors $({\cal B}_{ll',j} , \delta_{{\cal B}_{ll',j}}^j)$ d\'efinit une \'equivalence de Morita $G_j$-\'equivariante des $G_j$-alg\`ebres  $A_j \otimes {\cal K}(H_{lj})$ et $A_j \otimes {\cal K}(H_{l'j})$.
\end{enumerate}
\end{corollary}

\begin{proof} Il d\'ecoule de (\ref{etap3} a), \ref{etap4} b)) que le *-isomorphisme 
$\pi_{D_j} :  D_j   \rightarrow A_j \otimes {\cal K}(H_{1j} \oplus H_{2j})$ 
r\'ealise par restriction au  $D_{l,j}-D_{l',j}$ bimodule hilbertien $G_j$-\'equivariant $D_{ll',j}$, un isomorphisme de bimodule hilbertien $G_j$-\'equivariant de 
$D_{ll',j}$ sur $({\cal B}_{ll',j} , \delta_{{\cal B}_{ll',j}}^j)$.
\end{proof}

\begin{notations} \label{not3} Pour tout $j , l , l'  = 1 , 2$, nous notons
$\gamma_{ll',j} := ({\cal B}_{l,j} ,  {\cal B}_{ll',j} , {\cal B}_{l',j})$,\index{gc@$\gamma_{ll',j}$} 
l'\'equivalence de Morita $G_j$-\'equivariante (\ref{etap5} c)) des $G_j$-alg\`ebres ${\cal B}_{l,j}$ et ${\cal B}_{l',j}$ d\'efinie par le bimodule ${\cal B}_{ll',j}$.
\end{notations}

Pour le produit interne de bimodules $G_j$-\'equivariants, voir \cite{BaSka1}.

\begin{proposition} Pour tout $j , l , l' , l'' = 1 , 2$, nous avons :
\begin{equation}\label{moritaG}\gamma_{ll'',j} =  \gamma_{ll',j} \otimes_{{\cal B}_{l',j}} \gamma_{l'l'',j}.
\end{equation}
\begin{proof}
On v\'erifie facilement  que l'application 
\[
 \pi : {\cal B}_{ll',j} \otimes_{{\cal B}_{l',j}}  {\cal B}_{l'l'',j}  \rightarrow  {\cal B}_{ll'',j}   :  \\
 \xi  \otimes_{{\cal B}_{l',j}} \eta  \mapsto \xi \circ \eta
\]
est un isomorphisme de ${\cal B}_{l,j} - {\cal B}_{l'',j}$ bimodules hilbertiens.\hfill\break
Posons ${\cal E}_1 := {\cal B}_{ll',j}$ et ${\cal E}_2 := {\cal B}_{l'l'',j}$  et soient $\xi_i \in {\cal E}_i \, ,\, i = 1 , 2$. 
 
Avec les notations de (\cite{BaSka1} (2.10)), on v\'erifie sans peine qu'on a 
\[
(\pi \otimes {\rm id}_{S_{jj}})(\Delta(\xi_1 , \xi_2)) = \delta_{{\cal B}_{ll'',j}}^j(\pi(\xi_1 \otimes_{{\cal B}_{l',j}} \xi_2)).
\]
\end{proof}
\end{proposition}
\noindent
Pour $j , k =1, 2 \, ,\, j\not= k$, posons 
$${\cal E}_{ll',k}^j := \delta_{{\cal B}_{ll',j}}^k({\cal B}_{ll',j})
\subset {\cal L}({\cal B}_{l',k}  \otimes   S_{kj} , {\cal B}_{ll',k}  \otimes S_{kj}).$$
\noindent
Les relations :
$$\delta_{{\cal B}_{l,j}}^k(a)\delta_{{\cal B}_{ll',j}}^k(\xi)  =  \delta_{{\cal B}_{ll',j}}^k( a\xi), \quad \xi\in {\cal B}_{ll',j}, \quad a\in  {\cal B}_{j, l}\quad ;$$
\noindent
$$\delta_{{\cal B}_{ll',j}}^k(\xi) \delta_{{\cal B}_{l',j}}^k(a) =  \delta_{{\cal B}_{ll',j}}^k(\xi a), \quad \xi\in {\cal B}_{ll',j}, \quad a\in  {\cal B}_{l',j}\quad ;$$
\noindent
$$\langle \delta_{{\cal B}_{ll',j}}^k(\xi) , \delta_{{\cal B}_{ll',j}}^k(\eta) \rangle  := \delta_{{\cal B}_{l'l,j}}^k(\xi^*)\delta_{{\cal B}_{ll',j}}^k(\eta) = \delta_{{\cal B}_{l',j}}^k(\xi^*\circ  \eta), \quad \xi , \eta \in {\cal B}_{ll',j},
$$
\noindent
permettent  de munir ${\cal E}_{ll',k}^j$ d'une structure de $\delta_{{\cal B}_{l,j}}^k({\cal B}_{l,j}) - \delta_{{\cal B}_{l',j}}^k({\cal B}_{l',j})$ bimodule hilbertien.
\hfill\break
En fait, ${\cal E}_{ll',k}^j$ correspond au bimodule hilbertien 
$\delta_{D_j}^k(D_{ll',j})$ par l'identification apparue dans (\ref{not4} a)). On se propose de donner une formule directe de  la coaction de ${\cal E}_{ll',k}^j$ qu'on obtient \`a partir de celle de ${\cal B}_{ll',j}$, par transport de structure (\ie l'isomorphisme de bimodules  $\delta_{{\cal B}_{ll',j}}^k$).
\hfill\break
Rappelons (\ref{coprodS} b)) que la coaction $\delta_{kj}^j : S_{kj} \rightarrow M(S_{kj} \otimes S_{jj})$ est donn\'ee par :
$$\delta_{kj}^j(x) = V_{jj}^k (x \otimes 1_{S_{jj}}) (V_{jj}^k)^*, \quad x\in S_{kj}.$$

\begin{lemme}\label{coact} Soient $j , k , l , l' = 1 , 2$.
\begin{enumerate}
\item L'inclusion $\delta_{{\cal B}_{l',j}}^k({\cal B}_{l',j}) 
\subset M({\cal B}_{l',k}  \otimes   S_{kj})$  d\'efinit un *-morphisme injectif et non d\'eg\'en\'er\'e et on  a 
$$M(\delta_{{\cal B}_{l',j}}^k({\cal B}_{l',j})) 
\subset M({\cal B}_{l',k}  \otimes   S_{kj}).$$
\item On a les deux  inclusions canoniques  :
$${\cal L}(\delta_{{\cal B}_{l',j}}^k({\cal B}_{l',j}) , \delta_{{\cal B}_{ll',j}}^k({\cal B}_{ll',j}))
\subset {\cal L}({\cal B}_{l',k}  \otimes   S_{kj} , {\cal B}_{ll',k}  \otimes   S_{kj}),$$
\noindent
 $${\cal L}(\delta_{{\cal B}_{l',j}}^k({\cal B}_{l',j}) \otimes S_{jj} , \delta_{{\cal B}_{ll',j}}^k({\cal B}_{ll',j}) \otimes S_{jj})
\subset {\cal L}({\cal B}_{l',k}  \otimes   S_{kj}  \otimes S_{jj} , {\cal B}_{ll',k}  \otimes   S_{kj}  \otimes S_{jj}).$$
\end{enumerate}
\begin{proof} On a (\ref{acc} c))  $[\delta_{D_j}^k(D_j)(1_{D_k} \otimes S_{kj})] = D_k \otimes S_{kj}$. On en d\'eduit alors (\ref{liaison1} c)) que
$
[\delta_{D_j}^k(D_{ll',j})(1_{D_k} \otimes S_{kj})] =    D_{ll',k}  \otimes S_{kj}  = [\delta_{D_j}^k(D_{ll',j})(1_{D_{ll',k}} \otimes S_{kj})]. 
$
En composant avec $\pi_{D_k} \otimes {\rm id}_{S_{kj}}$, on obtient 
 $$[\delta_{{\cal B}_{ll',j}}^k({\cal B}_{ll',j})(1_{{\cal B}_{l',k}} \otimes S_{kj})] = {\cal B}_{ll',k} \otimes S_{kj}.$$
\noindent
On en d\'eduit le a) et une inclusion canonique 
$${\cal L}(\delta_{{\cal B}_{l',j}}^k({\cal B}_{l',j}) , \delta_{{\cal B}_{ll',j}}^k({\cal B}_{ll',j}))
\subset {\cal L}({\cal B}_{l',k}  \otimes   S_{kj} , {\cal B}_{ll',k}  \otimes   S_{kj}),$$
\noindent
qui permet d'\'etablir le b) par produit tensoriel avec la C*-alg\`ebre $S_{jj}$.
\end{proof}
\end{lemme}
\noindent
L'inclusion ${\cal E}_{ll',k}^j := \delta_{{\cal B}_{ll',j}}^k({\cal B}_{ll',j}) \subset 
{\cal L}(\delta_{{\cal B}_{l',j}}^k({\cal B}_{l',j}) , \delta_{{\cal B}_{ll',j}}^k({\cal B}_{ll',j}))$ et  (\ref{coact} b)),  nous donnent  une injection   
$${\cal E}_{ll',k}^j \rightarrow {\cal L}({\cal B}_{l',k}  \otimes   S_{kj}  \otimes S_{jj} , {\cal B}_{ll',k}  \otimes   S_{kj}  \otimes S_{jj}) : T \mapsto T_{12},$$
\noindent
qui permet finalement de d\'efinir
 une application lin\'eaire\index{di@$\delta_{ll',k}^j$}
$$\delta_{ll',k}^j : {\cal E}_{ll',k}^j \rightarrow {\cal L}({\cal B}_{l',k}  \otimes   S_{kj} \otimes S_{jj} , {\cal B}_{ll',k}  \otimes S_{kj} \otimes S_{jj}) : 
T \mapsto  V_{jj, 23}^k T_{12}  (V_{jj, 23}^k)^*.$$

\begin{proposition}\label{indbimod} Soient $j , k =1, 2$ tels que $j\not= k$.
\begin{enumerate}
\item  L'application lin\'eaire $\delta_{ll',k}^j$ est \`a valeurs dans ${\cal L}(\delta_{{\cal B}_{l',j}}^k({\cal B}_{l',j}) \otimes S_{jj} , {\cal E}_{ll',k}^j \otimes S_{jj})$.
\item  Le morphisme $({\cal B}_{ll',j} , \delta_{{\cal B}_{ll',j}}^j) \rightarrow ({\cal E}_{ll',k}^j ,  \delta_{ll',k}^j) : \xi \mapsto \delta_{{\cal B}_{ll',j}}^k(\xi)$ est un isomorphisme de bimodules $G_j$-\'equivariants au dessus des *-isomorphismes 
 $$\delta_{{\cal B}_{l,j}}^k :   {\cal B}_{l,j}  \rightarrow \delta_{{\cal B}_{l,j}}^k({\cal B}_{l,j}) 
 \quad \text{et} \quad \delta_{{\cal B}_{l',j}}^k :   {\cal B}_{l',j}  \rightarrow \delta_{{\cal B}_{l',j}}^k({\cal B}_{l',j}). $$
\end{enumerate}
\begin{proof}
En appliquant le corollaire (\ref{isoGj} c)) \`a l'action continue (biduale) $(\beta_D, \delta_D)$ de ${\cal G}$ dans la C*-alg\`ebre $D$, on obtient un *-isomorphisme $G_j$-\'equivariant :  
$$(D_j , \delta_{D_j}^j) \rightarrow \delta_{D_j}^k(D_j) : x \mapsto \delta_{D_j}^k(x).$$
\noindent
La C*-alg\`ebre $\delta_{D_j}^k(D_j)$ \'etant munie  (\ref{isoGj} b)) de la coaction ${\rm id}_{D_k} \otimes \delta_{kj}^j$, on d\'eduit de (\ref{liaison1} c), d))  et  (\ref{etap3} b), i)) que  
$$(\pi_{D_k} \otimes {\rm id}_{S_{kj}}) \delta_{D_j}^k : D_j  \rightarrow M(A_k \otimes {\cal K}(H_{1k} \oplus H_{2k}) \otimes S_{kj})$$
\noindent 
est un *-morphisme injectif d'alg\`ebres de liaison et on a :
$$(\pi_{D_k} \otimes {\rm id}_{S_{kj}}) \delta_{D_j}^k (e_{l, j}de_{l', j}) = \delta_{{\cal B}_{ll',j}}^k(\pi_{D_j}(e_{l, j} de_{l', j})), \quad  d\in D_j.$$
\noindent
En utilisant (\ref{acc} b)), on a aussi :
$$({\rm id}_{D_{ll',k}} \otimes \delta_{kj}^j) \delta_{D_j}^k (e_{l, j}de_{l', j}) = 
(\delta_{D_j}^k  \otimes {\rm id}_{S_{jj}}) \delta_{D_j}^j (e_{l, j}de_{l', j}), \quad  d\in D_j.$$
\noindent
Par composition avec $(\pi_{D_k} \otimes {\rm id}_{S_{kj}} \otimes {\rm id}_{S_{jj}})$, on obtient 
$$({\rm id}_{{\cal B}_{ll',k}} \otimes \delta_{kj}^j)\delta_{{\cal B}_{ll',j}}^k(\pi_{D_j}(e_{l, j} de_{l', j})) =
(\delta_{{\cal B}_{ll',j}}^k \otimes {\rm id}_{S_{jj}})\delta_{{\cal B}_{ll',j}}^j(\pi_{D_j}(e_{l, j}de_{l', j})).
$$
\noindent
Soit $\Phi : {\cal B}_{ll',j} \rightarrow {\cal E}_{ll',k}^j : x \mapsto \Phi(x):= \delta_{{\cal B}_{ll',j}}^k(x)$. Nous avons montr\'e que
$$ \delta_{ll',k}^j(\xi)  = (\Phi \otimes {\rm id}_{S_{jj}}) \delta_{{\cal B}_{ll',j}}^j(\Phi^{-1}(\xi)), \quad \xi\in{\cal E}_k^j,$$
\noindent
ce qui entra\^ine  que $\Phi : {\cal B}_{ll',j} \rightarrow {\cal E}_{ll',k}^j$ est un isomorphisme de bimodules $G_j$-\'equivariants au dessus des C*-isomorphismes 
$$\Psi_{l} : {\cal B}_{l,j} \rightarrow \delta_{{\cal B}_{l,j}}^k({\cal B}_{l,j}) : x \mapsto \delta_{{\cal B}_{l,j}}^k(x)   , \quad \Psi_{l'} : {\cal B}_{l',j} \rightarrow \delta_{{\cal B}_{l',j}}^k({\cal B}_{l',j}) : x \mapsto \delta_{{\cal B}_{l',j}}^k(x).$$
\end{proof}
\end{proposition}
\begin{remark} Nous verrons dans le paragraphe suivant une construction du bimodule hilbertien
 ${\cal E}_{ll',k}^j := \delta_{{\cal B}_{ll',j}}^k({\cal B}_{ll',j})$ \`a partir du bimodule ${\cal B}_{ll',k}$ dont il est une d\'eformation.
\end{remark}

\noindent
\section{ Induction d'actions }\label{indaction}

Dans ce paragraphe on se fixe un groupo\"ide de co-liaison  ${\cal G} := {\cal G}_{G_1, G_2}$ associ\'e  \`a deux groupes quantiques \lc $G_1$ et $G_2$ mono\"idalement \'equivalents et r\'eguliers. Nous conservons les   notations des paragraphes  (\ref{Hopff} et \ref{grequivmono}))  pour tous les objets associ\'es \`a ${\cal G}$.
\hfill\break
Nous avons d\'emontr\'e \ref{acc}  qu'\`a toute action continue $(\beta_A, \delta_A)$  du groupo\"ide m.q.\ ${\cal G}$ dans une  C*-alg\`ebre $A = A_1 \oplus A_2$, on associe une action continue $(A_i, \delta_{A_i}^i)$ du   groupe quantique $G_i$.   Dans le cas o\`u le groupo\"ide ${\cal G}$ est r\'egulier, nous allons montrer  que la correspondance $(A, \beta_A, \delta_A) \rightarrow (A_1, \delta_{A_1}^1)$ est biunivoque. Plus pr\'ecis\'ement, \`a toute action continue 
$(A_1, \delta_{A_1})$ du groupe  quantique r\'egulier $G_1$, on associe canoniquement une action continue  $(A_2, \delta_{A_2})$ du groupe  quantique    $G_2$. Nous munirons alors  la C*-alg\`ebre   $A := A_1 \oplus A_2$ d'une action continue du groupo\"ide ${\cal G}$, ce qui permet de construire la correspondance r\'eciproque de la correspondance 
$(A, \beta_A, \delta_A) \rightarrow (A_1, \delta_{A_1}^1)$. 

Notons que la correspondance bijective $(A_1, \delta_{A_1}) \rightarrow (A_2, \delta_{A_2})$ et son application \`a la $KK$-th\'eorie \'equivariante de Kasparov (\cf\cite{BaSka1,K2}), g\'en\'eralisent au cas des groupes quantiques \lc r\'eguliers, deux r\'esultats du cas compact,  dus respectivement \`a \cite{RV} et \cite{V1}. Pour les actions dans une alg\`ebre de von Neumann, la correspondance 
$(A_1, \delta_{A_1}) \rightarrow (A_2, \delta_{A_2})$  a  \'et\'e \'etablie 
dans \cite{DeC2}.

\subsection{\'Equivalence des actions continues de \texorpdfstring{$G_1$}{G1} et \texorpdfstring{$G_2$}{G2}}

Le r\'esultat le plus important de ce paragraphe est le suivant :
\begin{proposition}\label{ind} Soit $\delta_{A_1}  : A_1 \rightarrow M(A_1 \otimes S_{11})$ une action continue du groupe quantique r\'egulier $G_1$.  Posons :\index{dj@$\delta_{A_1}^{(2)}$}   
\begin{equation}
\delta_{A_1}^1 := \delta_{A_1}, \quad 
\delta_{A_1}^{(2)} := ({\rm id}_{A_1} \otimes \delta_{11}^2) \circ \delta_{A_1} : A_1 \rightarrow M(A_1 \otimes S_{12}  \otimes S_{21}), \nonumber 
\end{equation}
\begin{equation}
{\rm Ind }_{G_1}^{G_2}\, A_1 := [({\rm id} \otimes {\rm id} \otimes \omega)\delta_{A_1}^{(2)}(a) \,\,|\,\, a\in A_1 \,,\,\omega\in B(H_{21})_\ast] \subset M(A_1 \otimes S_{12}).\nonumber 
\end{equation}
Alors  ${\rm Ind }_{G_1}^{G_2}\, A_1$\index{i@${\rm Ind }_{G_1}^{G_2}\, A_1$} est une  sous-C*-alg\`ebre de $M(A_1 \otimes S_{12})$.
\end{proposition}

\begin{proof}
Il est clair que ${\rm Ind }_{G_1}^{G_2}\, A_1$ est stable par involution. 
\hfill\break 
Le *-morphisme  $({\rm id} \otimes \delta_{11}^2) \circ \delta_{A_1} : A_1 \rightarrow M(A_1 \otimes S_{12} \otimes S_{21})$ est  non d\'eg\'en\'er\'e et   $S_{21}$ \'etant une sous-C*-alg\`ebre non d\'eg\'en\'er\'ee de $B(H_{21})$, on a :  
$$[({\rm id} \otimes {\rm id} \otimes \omega)\delta_{A_1}^{(2)}(a) \,\,|\,\, a\in A_1 \,,\,\omega\in B(H_{21})_\ast] \subset M(A_1 \otimes S_{12}).$$
\noindent 
Posons  ${\cal A} := {\rm vect}\,\{({\rm id} \otimes {\rm id} \otimes \omega)\delta_{A_1}^{(2)}(a) \,\,|\,\, a\in A_1 \,,\,\omega\in B(H_{21})_\ast\}$, il reste \`a montrer   que 
${\cal A} {\cal A} \subset {\rm Ind }_{G_1}^{G_2}\, A_1$.\hfill\break
Soient  $x =({\rm id}_{A_{1} \otimes S_{12}} \otimes \omega_{\xi, \eta})\delta_{A_1}^{(2)}(a) \,,\, x' =({\rm id}_{A \otimes S_{12}} \otimes \omega_{\xi', \eta'})\delta_{A_1}^{(2)}(a')$ avec $\xi, \eta, \xi', \eta'\in H_{21}$,  on a 
$$x x' = ({\rm id}_{A_1 \otimes S_{12}} \otimes \omega_{\xi, \eta} \otimes   \omega_{\xi', \eta'}) \delta_{A_1}^{(2)}(a)_{123} \delta_{A_1}^{(2)}(a')_{124}.$$\noindent
Comme $H_{11} \not= \{0\}$ et $H_{21} \not= \{0\}$, on peut  supposer que $\eta = k(\eta_1)$, où $k\in{\cal K}(H_{21})$, $\eta_1\in H_{21}$ et 
$\xi' = l(\xi_1')$, où $l\in {\cal K}(H_{11}, H_{21})$, $\xi_1'\in H_{11}$. Il en r\'esulte que
$$x x' = ({\rm id}_{A_{1} \otimes S_{12}} \otimes \omega_{\xi, \eta_1} \otimes   \omega_{\xi_1', \eta'}) \delta_{A_1}^{(2)}(a)_{123} (k \otimes l^*)_{34} \delta_{A_1}^{(2)}(a')_{124}.$$\noindent
Comme $k \otimes l^*\in {\cal K}(H_{21} \otimes H_{21}, H_{21} \otimes H_{11})$, par (\ref{larem} a)), on  peut supposer qu'on a 
$$k \otimes l^* = (1_{H_{21}} \otimes k') (W_{21}^1)^* (k'' \otimes  1_{H_{21}}),\quad {\rm avec}\quad k'\in {\cal K}(H_{11}),\, k''\in {\cal K}(H_{21}).$$\noindent
Finalement, on peut supposer que $x x'$ est limite normique d'\'el\'ements de la forme :
$$y = ({\rm id}_{A_{1} \otimes S_{12}} \otimes \omega \otimes \omega')(\delta_{A_1}^{(2)}(a)_{123}  (W_{21, 34}^1)^*  \delta_{A_1}^{(2)}(a')_{124} W_{21, 34}^1),$$\noindent
avec 
$\omega\in B(H_{21})_\ast$, $\omega'\in B(H_{11})_\ast$ et $a, a' \in A_1$.\hfill\break
On a (\ref{coprodS} b)) :
$$(W_{21, 34}^1)^* \delta_{A_1}^{(2)}(a')_{124} W_{21, 34}^1 =  
  [({\rm id}_{A_1 \otimes S_{12}}  \otimes \delta_{21}^1) 
\delta_{A_1}^{(2)}(a')]_{1234} =  
 [({\rm id}_{A_1 \otimes S_{12}}  \otimes \delta_{21}^1)({\rm id}_{A_1} \otimes \delta_{11}^2) \delta_{A_1}(a')]_{1234}.$$
\noindent
En utilisant (\ref{coprodS} a)), on obtient   
\begin{align*}
(W_{21, 34}^1)^* \delta_{A_1}^{(2)}(a')_{124} W_{21, 34}^1 & = 
 [({\rm id}_{A_1} \otimes \delta_{11}^2 \otimes {\rm id}_{S_{11}}) ({\rm id}_{A_1} \otimes \delta_{11}^1) \delta_{A_1}(a')]_{1234}  \\
 &=[({\rm id}_{A_1} \otimes \delta_{11}^2 \otimes {\rm id}_{S_{11}}) (\delta_{A_1} \otimes {\rm id}_{S_{11}}) \delta_{A_1}(a')]_{1234}.
\end{align*}
Posons $\omega'  = \omega''s$ avec $\omega''\in B(H_{11})_\ast$ et $s\in S_{11}$. Par continuit\'e de la coaction $\delta_{A_1}$, on voit que $y$ est limite normique d'\'el\'ements de la forme  
$$z =  ({\rm id}_{A_{1} \otimes S_{12}} \otimes \omega \otimes \omega')(\delta_{A_1}^{(2)}(a)_{123} [({\rm id}_{A_1} \otimes \delta_{11}^2 \otimes {\rm id}_{S_{11}})
 (\delta_{A_1}(a') \otimes 1_{S_{11}})]_{1234}) = ({\rm id}_{A_{1} \otimes S_{12}} \otimes \omega \otimes \omega')(\delta_{A_1}^{(2)}(aa')_{123})$$\noindent
avec 
$\omega\in B(H_{21})_\ast$, $\omega'\in B(H_{11})_\ast$ et $a, a' \in A_1$, d'o\`u le r\'esultat.
\end{proof}

\begin{remarks}
\begin{enumerate} 
\item En fait nous avons montr\'e que pour tout $T\in M(A_1)$ et tout $a\in A_1$, on a 
\begin{equation}\label{imp}
({\rm id}_{A_1 \otimes S_{12}} \otimes \omega \otimes   \omega') \delta_{A_1}^{(2)}(T)_{123} \delta_{A_1}^{(2)}(a)_{124} \in{\rm Ind }_{G_1}^{G_2}\, A_1,\quad \omega,\,\omega'\in B(H_{21})_{\ast}. 
\end{equation}
\item La proposition \ref{ind} reste vraie pour une action fortement continue d'un groupe quantique \lc semi-r\'egulier.
\item  L'id\'ee de  la preuve de  \ref{ind}, est la m\^eme que celle de (\cite{BaSkaV}, Proposition 5.8).
\end{enumerate}
\end{remarks}

\begin{proposition}\label{ind3} Posons $A_2 := {\rm Ind }_{G_1}^{G_2}\, A_1$.  Nous avons :
\begin{enumerate}
\item  $[A_2(1_{A_1} \otimes  S_{12})]  =  [(1_{A_1} \otimes  S_{12}) A_2] = A_1 \otimes S_{12}$. En particulier, l'inclusion  $A_2 \subset M(A_1 \otimes S_{12})$ d\'efinit  
  un *-morphisme injectif et non d\'eg\'en\'er\'e  et on a $M(A_2) \subset M(A_1 \otimes S_{12})$ ;
\item  Posons $\delta_{A_2} := \restr{({\rm id}_{A_1} \otimes \delta_{12}^2)}{A_2}$. Alors, le *-morphisme   $\delta_{A_2}$ est \`a valeurs dans $M(A_2 \otimes  S_{22})$ et   $\delta_{A_2} : A_2 \rightarrow  M(A_2 \otimes  S_{22})$ est une action continue 
 du groupe quantique $G_2$ ;
\item  La correspondance $A^{  G_1} \rightarrow A^{  G_2}  : (A_1, \delta_{A_1})  \mapsto (A_2, \delta_{A_2})$ est fonctorielle.
\end{enumerate}  
\begin{proof}  Soit  $x = ({\rm id}_{A_1 \otimes S_{12}} \otimes  \omega)\delta_{A_1}^{(2)}(a) \,;\,a\in A_1\,,\, \omega\in B(H_{21})_\ast$. Posons  $\omega = s' \omega'  \,,\,s'\in S_{21}$ et soit $s\in S_{12}$. 
On a   
$$x(1_{A_1} \otimes s) = ({\rm id}_{A_1} \otimes {\rm id}_{S_{12}} \otimes \omega')(\delta_{A_1}^{(2)}(a) (1_{A_1} \otimes s \otimes s')).$$\noindent
Mais il r\'esulte facilement   de (\ref{simp}) qu'on a $[\delta_{11}^2(S_{11}) (S_{12} \otimes 1_{S_{21}})] = S_{12} \otimes S_{21}$.
Comme  $s \otimes s'\in S_{12} \otimes S_{21}$, on d\'eduit   que $x(1_{A_1} \otimes  s)$ est limite normique d'\'el\'ements de la forme :
$$y  =  ({\rm id}_{A_1 \otimes S_{12}} \otimes \omega) [({\rm id}_{A_1} \otimes  \delta_{11}^2) (\delta_{A_1}(a')(1_{A_1} \otimes s'))] (1_{A_1} \otimes s''),\quad a'\in A_1,\, s'\in S_{11},\, s'' \in S_{12}.$$\noindent
Par continuit\'e de la coaction $\delta_{A_1}$, on a alors $y\in A_1 \otimes S_{12}$, donc $[A_2(1_{A_1} \otimes  S_{12})] \subset A_1 \otimes S_{12}$ et aussi $[(1_{A_1} \otimes  S_{12}) A_2] \subset A_1  \otimes  S_{12}$.\hfill\break
Pour montrer l'inclusion $A_1 \otimes S_{12} \subset [A_2(1_{A_1} \otimes  S_{12})]$, remarquons qu'on a  (\ref{simp}) : 
$$[\delta_{11}^2(S_{11}) (1_{S_{12}} \otimes  S_{21} )] = S_{12} \otimes S_{21} \, , \, S_{12} = [({\rm id} \otimes \omega) \delta_{11}^2 (S_{11}) \,|\, \omega\in B(H_{21})_\ast]\, , \,A_1 \otimes S_{11} = [\delta_{A_1}(A_1) (1_{A_1} \otimes S_{11})],$$
\noindent
ce qui donne facilement l'inclusion 
$A_1 \otimes S_{12} \subset [A_2(1_{A_1} \otimes  S_{12})]$, d'o\`u le a).

\noindent
Soit $x = ({\rm id}_{A_1 \otimes S_{12}} \otimes  \omega)\delta_{A_1}^{(2)}(a)$, où $a\in A_1$, $\omega\in B(H_{21})_\ast$. Montrons que 
$$({\rm id}_{A_1} \otimes \delta_{12}^2)(x) \in M(A_2 \otimes S_{22}) \subset M(A_1 \otimes S_{12} \otimes S_{22}).$$
\noindent
En utilisant (\ref{coprodS} a) et b)), on d\'eduit  :
\begin{align} 
({\rm id}_{A_1} \otimes \delta_{12}^2)(x) &=    ({\rm id}_{A_1 \otimes S_{12} \otimes S_{22}} \otimes \omega) ({\rm id}_{A_1} \otimes \delta_{12}^2 \otimes {\rm id}_{S_{21}})({\rm id}_{A_1} \otimes  \delta_{11}^2) \delta_{A_1}(a) \nonumber \\
& =  
  ({\rm id}_{A_1 \otimes S_{12} \otimes S_{22}} \otimes \omega) ({\rm id}_{A_1 \otimes S_{12}}  \otimes \delta_{21}^2)  ({\rm id}_{A_1} \otimes  \delta_{11}^2) \delta_{A_1}(a)  \nonumber \\
&=    
 ({\rm id}_{A_1 \otimes S_{12} \otimes S_{22}} \otimes \omega)((W_{22, 34}^1)^*  \delta_{A_1}^{(2)}(a)_{124} W_{22, 34}^1).\nonumber 
\end{align}
Comme $W_{22}^1 \in M(S_{22} \otimes \cK(H_{21}))$, on a donc $({\rm id}_{A_1} \otimes \delta_{12}^2)(x) \in M(A_1 \otimes S_{12} \otimes S_{22})$.

Posons $\omega  =   \omega' u$ avec $u \in \cK(H_{21})$ et soit  $s\in S_{22}$. En utilisant de nouveau que   $W_{22}^1 \in M(S_{22} \otimes \cK(H_{21}))$, on voit que  
$(1_{A_2} \otimes s) ({\rm id}_{A_1} \otimes \delta_{12}^2)(x)$ est limite normique d'\'el\'ements $y$ de la forme  :
$$y =   ({\rm id}_{A_1 \otimes S_{12} \otimes S_{22}} \otimes \omega) (\delta_{A_1}^{(2)}(a)_{124} [(s \otimes  1_{H_{21}})W_{22}^1 (1_{S_{22}} \otimes v)]_{34}),$$\noindent
avec  $ \omega\in B(H_{21})_\ast \,,\,s\in S_{22} \,,\,v\in \cK(H_{21})$.\hfill\break
La r\'egularit\'e de ${G}_2$  (\cf \cite{BaSka2}, A4) entra\^ine que 
\begin{equation}\label{eq:corep} 
[(s \otimes 1_{H_{21}})W_{22}^1 (1_{S_{22}} \otimes v) \,\,|\,\,s\in S_{22} \,,\,v\in \cK(H_{21})] = S_{22}  \otimes  \cK(H_{21}),
\end{equation}
d'o\`u l'inclusion
$(1_{A_2} \otimes S_{22})\delta_{A_2}(A_2) \subset A_2 \otimes S_{22}$. Comme le morphisme $\delta_{A_2}$ est involutif, on a \'egalement
$\delta_{A_2}(A_2) (1_{A_2} \otimes S_{22}) \subset A_2 \otimes S_{22}$.
\hfill\break
Montrons la continuit\'e de $\delta_{A_2}$, \ie  $ A_2 \otimes S_{22} = [(1_{A_2} \otimes S_{22})\delta_{A_2}(A_2)]$.
\hfill\break
Soit $a\in A_1 , \omega'\in B(H_{21})_\ast\,$ et $s\in S_{22}$. Posons  
$\omega' =  v  \omega  u$ avec $u , v \in \cK(H_{21})$. On a :
$$[({\rm id}_{A_1 \otimes S_{12}} \otimes \omega')(\delta_{A_1}^{(2)}(a))] \otimes s =
({\rm id}_{A_1 \otimes S_{12} \otimes S_{22}} \otimes \omega)((1_{A_1 \otimes S_{12}} \otimes 1_{S_{22}}\otimes u)  \delta_{A_1}^{(2)}(a)_{124} 
(1_{A_1 \otimes S_{12}} \otimes s \otimes v)).$$
En utilisant de nouveau  que $W_{22}^1 \in M(S_{22} \otimes \cK(H_{21}))$ et l'\'egalit\'e (\ref{eq:corep}), 
on voit alors facilement que
$ [({\rm id}_{A_1 \otimes S_{12}} \otimes \omega')(\delta_{A_1}^{(2)}(a))] \otimes s$ est limite normique d'\'el\'ements $y$ de la forme :
$$ y = ({\rm id}_{A_1 \otimes S_{12} \otimes S_{22}} \otimes \omega)((1_{A_1 \otimes S_{12}} \otimes s' \otimes u') (W_{22, 34}^1)^*  \delta_{A_1}^{(2)}(a)_{124} W_{22, 34}^1),$$\noindent
avec   $s'\in S_{22}$, $u'\in 
\cK(H_{21})$ et $\omega\in B(H_{21})_\ast$.

En utilisant (\ref{coprodS} a) et b)), on obtient que   
  $[({\rm id}_{A_1 \otimes S_{12}} \otimes \omega)(\delta_{A_1}^{(2)}(a))] \otimes s\in [(1_{A_2} \otimes S_{22})\delta_{A_2}(A_2)]$, d'o\`u la continuit\'e de $\delta_{A_2}$.

On v\'erifie facilement gr\^ace  \`a (\ref{coprodS} a)), qu'on a $(\delta_{A_2} \otimes {\rm id}_{S_{22}})\delta_{A_2} =({\rm  id}_{A_2} \otimes \delta_{22}^2) \delta_{A_2}$, d'o\`u le b).

Soient $(A_1,  \delta_{A_1}, G_1)$ et $(B_1,  \delta_{B_1}, G_1)$ deux syst\`emes dynamiques. 
Posons $A_2 = {\rm Ind }_{G_1}^{G_2}\, A_1$ et $B_2 = {\rm Ind }_{G_1}^{G_2}\, B_1$.
\hfill\break
Soit $f_1: (A_1, \delta_{A_1})  \rightarrow  (B_1, \delta_{B_1})$ un *-morphisme $G_1$-\'equivariant.   Notons  $f_1 \otimes {\rm id}_{S_{12}}$ le prolongement unital et strictement continu du *-morphisme non d\'eg\'en\'er\'e $A_1 \otimes  S_{12}  \rightarrow M(B_1 \otimes S_{12})  : x \mapsto (f_1 \otimes {\rm id}_{S_{12}})(x)$.\hfill\break
Pour tout $a\in A_1$ et $\omega\in B(H_{21})_\ast$, on a  (\ref{imp}):
$$(f_1 \otimes {\rm id}_{S_{12}})({\rm id}_{A_1 \otimes S_{12}} \otimes  \omega)({\rm id}_{A_1} \otimes \delta_{11}^2) \circ \delta_{A_1}(a) = 
({\rm id}_{B_1 \otimes S_{12}} \otimes  \omega)({\rm id}_{B_1} \otimes \delta_{11}^2)\delta_{B_1}(f_1(a))\in M(B_2).$$
\noindent
Pour tout $x\in A_2$, posons $f_2(x) := (f_1 \otimes {\rm id}_{S_{12}})(x)$. Il est imm\'ediat que $f_2 : A_2 \rightarrow M(B_2)$ est un *-morphisme non d\'eg\'en\'er\'e. Montrons que $f_2$ est $G_2$-\'equivariant.  
Soient $a\in A_1$ et $\omega\in B(H_{21})_\ast$. Posons $x:= ({\rm id}_{A_1 \otimes S_{12}} \otimes  \omega)({\rm id}_{A_1} \otimes \delta_{11}^2) \circ \delta_{A_1}(a)$. Nous avons en utilisant (\ref{coprodS} a)) :
\begin{align} 
(f_2 \otimes {\rm id}_{S_{22}})\delta_{A_2}(x) &= (f_1 \otimes {\rm id}_{S_{12}\otimes S_{22}}) ({\rm id}_{A_1} \otimes \delta_{12}^2  \otimes  \omega)({\rm id}_{A_1} \otimes \delta_{11}^2)  \delta_{A_1}(a) \nonumber \\
&=  
(f_1 \otimes {\rm id}_{S_{12}\otimes S_{22}})
({\rm id}_{A_{1} \otimes S_{12}\otimes S_{22}} \otimes  \omega) ({\rm id}_{A_1} \otimes [(\delta_{12}^2 \otimes {\rm id}_{S_{21}})\delta_{11}^2])\delta_{A_1}(a)) \nonumber \\
&=(f_1 \otimes {\rm id}_{S_{12}\otimes S_{22}})
({\rm id}_{A_{1} \otimes S_{12}\otimes S_{22}} \otimes  \omega) ({\rm id}_{A_1} \otimes [({\rm id}_{S_{12}} \otimes \delta_{21}^2)\delta_{11}^2])\delta_{A_1}(a) \nonumber \\
&= ({\rm id}_{B_{1} \otimes S_{12}\otimes S_{22}} \otimes  \omega) ({\rm id}_{B_1} \otimes [({\rm id}_{S_{12}} \otimes \delta_{21}^2)\delta_{11}^2])\delta_{B_1}(f_1(a)) \nonumber \\
&=({\rm id}_{B_{1} \otimes S_{12}\otimes S_{22}} \otimes  \omega) ({\rm id}_{B_1} \otimes [(\delta_{12}^2 \otimes {\rm id}_{S_{21}})\delta_{11}^2])\delta_{B_1}(f_1(a)) = 
\delta_{B_2}(f_2(x)).\nonumber 
\end{align}
\end{proof}
\end{proposition}
\noindent
%R\'eciproquement, partons d'une action continue $\delta_{A_2}  : A_2 \rightarrow M(A_2 \otimes S_{22})$ du groupe $G_2$.  Posons  : 
%$$\delta_{A_2}^2 := \delta_{A_2} \, ,\, \delta_{A_2}^{(1)} := ({\rm id} \otimes \delta_{22}^1) \circ \delta_{A_2} : A_2 \rightarrow M(A_2 \otimes S_{21}  \otimes S_{12})$$\noindent
%Par le m\^eme raisonnement, on obtient que 
%$$A_1 := {\rm Ind }_{{\bf G}_2}^{{\bf G}_1}\, A_2  := [\{({\rm id} \otimes {\rm id} \otimes \omega)\delta_{A_2}^{(1)}(a) \,\,|\,\, a\in A_2 \,,\,\omega\in B(H_{12})_\ast\}]
%\subset M(A_2 \otimes S_{21})$$\noindent
%est une sous-C*alg\`ebre de $M(A_2 \otimes S_{21})$ et 
%$$\delta_{A_1} := {\rm id}_{A_2} \otimes \delta_{21}^1 :   A_1 \rightarrow  M(A_1 \otimes  S_{11})$$
%\noindent
%est  une  action continue  du groupe quantique $G_1$ dans la C*-alg\`ebre $A_1$. De plus, on a :
%$$[A_1(1_{A_2} \otimes  S_{21})]  =  [(1_{A_2} \otimes  S_{21}) A_1] = A_2 \otimes S_{21}$$\noindent 
%En particulier, l'inclusion  $A_1 \subset M(A_2 \otimes S_{21})$ d\'efinit  
%  un *-morphisme injectif et non d\'eg\'en\'er\'e.
%\hfill\break 
%On a donc une correspondance   fonctorielle : 
%$$A^{  G_2} \rightarrow A^{  G_1}  : (A_2, \delta_{A_2})  \mapsto (A_1, \delta_{A_1})$$

\noindent
Si on part d'une action continue $\delta_{A_2}  : A_2 \rightarrow M(A_2 \otimes S_{22})$ du groupe quantique $G_2$, posons :\index{dk@$\delta_{A_2}^{(1)}$} 
$$\delta_{A_2}^2 := \delta_{A_2},\quad \delta_{A_2}^{(1)} := ({\rm id} \otimes \delta_{22}^1) \circ \delta_{A_2} : A_2 \rightarrow M(A_2 \otimes S_{21}  \otimes S_{12}).$$\noindent
Avec une preuve similaire,  nous avons :

\noindent
\begin{proposition}\label{ind4} Soit $A_1 := {\rm Ind }_{G_2}^{G_1}\, A_2  := [({\rm id}_{A_2} \otimes {\rm id}_{S_{21}} \otimes \omega)\delta_{A_2}^{(1)}(a) \,\,|\,\, a\in A_2 \,,\,\omega\in B(H_{12})_\ast]$.\index{i@${\rm Ind }_{G_2}^{G_1}\, A_2$}
\begin{enumerate}
\item $A_1$ est une sous-C*-alg\`ebre de $M(A_2 \otimes S_{21})$.
\item $[A_1(1_{A_2} \otimes  S_{21})]  =  [(1_{A_2} \otimes  S_{21}) A_1] = A_2 \otimes S_{21}$. En particulier, l'inclusion  $A_1 \subset M(A_2 \otimes S_{21})$ d\'efinit  
  un *-morphisme injectif et non d\'eg\'en\'er\'e et on a $M(A_1) \subset M(A_2 \otimes S_{21})$.
\item Posons $\delta_{A_1} := \restr{({\rm id}_{A_2} \otimes \delta_{21}^1)}{A_1}$. Alors, le *-morphisme   $\delta_{A_1}$ est \`a valeurs dans $M(A_1 \otimes  S_{11})$ et   $\delta_{A_1} : A_1 \rightarrow  M(A_1 \otimes  S_{11})$ est une action    continue du groupe quantique $G_1$.
\item La correspondance $A^{  G_2} \rightarrow A^{  G_1}  : (A_2, \delta_{A_2})  \mapsto (A_1, \delta_{A_1})$ est fonctorielle. 
\end{enumerate}
%\begin{proof}[]
%\end{proof}

\end{proposition}
\noindent
Dans les deux propositions suivantes, \'etudions les compos\'ees de correspondances 
$$A^{  G_1} \rightarrow A^{  G_2} \rightarrow A^{  G_1}\quad {\rm et} \quad  
A^{  G_2} \rightarrow A^{  G_1} \rightarrow A^{  G_2}.$$

\begin{proposition}\label{ind5}
Soit  $\delta_{A_1}  : A_1 \rightarrow M(A_1 \otimes S_{11})$ une action continue du groupe quantique $G_1$ dans la C*-alg\`ebre $A_1$. Posons $A_2 = {\rm Ind }_{G_1}^{G_2}\, A_1$, $\delta_{A_2} = \restr{({\rm id}_{A_1} \otimes \delta_{12}^2)}{A_1}$ et 
$C = {\rm Ind }_{G_2}^{G_1}\, A_2 \subset M(A_2 \otimes S_{21})$ munie de $\delta_C = \restr{({\rm id}_{A_2} \otimes \delta_{21}^1)}{C}$. \hfill\break
Avec ces notations, on a :
\begin{enumerate}
\item  Pour $\omega\in B(H_{12})_\ast$ et $\omega'\in B(H_{21})_\ast$,  on a  
$$({\rm id}_{A_1 \otimes S_{12} \otimes S_{21}} \otimes \omega)  ({\rm id}_{A_1 \otimes  S_{12}} \otimes \delta_{22}^1) ({\rm id}_{A_1} \otimes \delta_{12}^2) ({\rm id}_{A_1 \otimes S_{12}} \otimes \omega')
 \delta_{A_1}^{(2)}   =   
\delta_{A_1}^{(2)} ({\rm id}_{A_1} \otimes \omega \otimes \omega')\delta_{A_1}^{(2)}.$$\noindent
\item $C \subset M(A_2 \otimes S_{21}) \subset M(A_1 \otimes S_{12} \otimes S_{21})$ et $C= \delta_{A_1}^{(2)} (A_1)$.
\item Le *-morphisme
$$\pi_1 : (A_1, \delta_{A_1})  \rightarrow  (C, \delta_C) :  a \mapsto  \delta_{A_1}^{(2)} (a) = ({\rm id}_{A_1} \otimes \delta_{11}^2)\delta_{A_1}(a)$$
est un isomorphisme $G_1$-\'equivariant.\index{pl@$\pi_1$, $\pi_2$}
\item Le *-morphisme   $\delta_{A_1}^{2}  : A_1 \rightarrow M(A_2 \otimes S_{21}) : a \mapsto  \delta_{A_1}^{(2)} (a)$ est   injectif, non d\'eg\'en\'er\'e et on a $[\delta_{A_1}^{2}(A_{1})(1_{A_2} \otimes S_{21})] = A_2 \otimes S_{21}$.
\end{enumerate}
\end{proposition}

\begin{proof} Pour $\omega\in B(H_{12})_\ast$ et $\omega'\in B(H_{21})_\ast$,  on a  en utilisant (\ref{coprodS} a)) :
\begin{align}({\rm id}_{A_1 \otimes S_{12} \otimes S_{21}} \otimes \omega) & ({\rm id}_{A_1 \otimes  S_{12}} \otimes \delta_{22}^1) ({\rm id}_{A_1} \otimes \delta_{12}^2) ({\rm id}_{A_1 \otimes S_{12}} \otimes \omega')
 \delta_{A_1}^{(2)} \nonumber \\
&  =  
({\rm id}_{A_1 \otimes S_{12} \otimes S_{21}} \otimes \omega \otimes \omega') 
  ({\rm id}_{A_1 \otimes  S_{12}} \otimes \delta_{22}^1 \otimes {\rm id}_{S_{21}}) ({\rm id}_{A_1} \otimes \delta_{12}^2\otimes {\rm id}_{S_{21}}) \delta_{A_1}^{(2)} \nonumber \\
& =  ({\rm id}_{A_1 \otimes S_{12} \otimes S_{21}} \otimes \omega \otimes \omega') 
  ({\rm id}_{A_1} \otimes [ ({\rm id}_{S_{12}} \otimes \delta_{22}^1 \otimes {\rm id}_{S_{21}}) (\delta_{12}^2\otimes {\rm id}_{S_{21}})]) \delta_{A_1}^{(2)} \nonumber \\
& = 
({\rm id}_{A_1 \otimes S_{12} \otimes S_{21}} \otimes \omega \otimes \omega') ({\rm id}_{A_1} \otimes \delta_{11}^2 \otimes {\rm id}_{S_{12} \otimes S_{21}})
({\rm id}_{A_1} \otimes \delta_{12}^1\otimes {\rm id}_{S_{21}}) \delta_{A_1}^{(2)} \nonumber \\
&=   ({\rm id}_{A_1 \otimes S_{12} \otimes S_{21}} \otimes \omega \otimes \omega') ({\rm id}_{A_1} \otimes \delta_{11}^2 \otimes {\rm id}_{S_{12} \otimes S_{21}})
({\rm id}_{A_1} \otimes [(\delta_{12}^1\otimes {\rm id}_{S_{21}})\delta_{11}^2]) \delta_{A_1}\nonumber \\ 
&=   ({\rm id}_{A_1 \otimes S_{12} \otimes S_{21}} \otimes \omega \otimes \omega') ({\rm id}_{A_1} \otimes \delta_{11}^2 \otimes {\rm id}_{S_{12} \otimes S_{21}})
({\rm id}_{A_1} \otimes [({\rm id}_{S_{11}}\otimes \delta_{11}^2)\delta_{11}^1]) \delta_{A_1} \nonumber \\
& =  
({\rm id}_{A_1 \otimes S_{12} \otimes S_{21}} \otimes \omega \otimes \omega') ({\rm id}_{A_1} \otimes \delta_{11}^2 \otimes \delta_{11}^2)
(\delta_{A_1} \otimes {\rm id}_{S_{11}})\delta_{A_1} \nonumber \\
& = \delta_{A_1}^{(2)} ({\rm id}_{A_1} \otimes \omega \otimes \omega')\delta_{A_1}^{(2)},
\nonumber 
\end{align} d'o\`u le a).
\hfill\break
Par d\'efinition, la C*-alg\`ebre $C$ est le sous-espace vectoriel ferm\'e engendr\'e par les \'el\'ements de la forme : 
$$({\rm id}_{A_1 \otimes S_{12} \otimes S_{21}} \otimes \omega)  ({\rm id}_{A_1 \otimes  S_{12}} \otimes \delta_{22}^1) ({\rm id}_{A_1} \otimes \delta_{12}^2) ({\rm id}_{A_1 \otimes S_{12}} \otimes \omega')
 \delta_{A_1}^{(2)}(a),$$
o\`u $a\in A_1$, $\omega\in B(H_{12})_\ast$ et $\omega'\in B(H_{21})_\ast$. Donc $C\subset M(A_1 \otimes S_{12} \otimes S_{21})$, et pour montrer  le b), il suffit, en utilisant le a),  d'\'etablir  : 
$$A_1 = [({\rm id}_{A_1} \otimes \omega \otimes \omega')\delta_{A_1}^{(2)}(a)\,\,|\,\,a\in A_1  ,   \omega\in B(H_{12})_\ast , \omega'\in B(H_{21})_\ast].$$ \noindent
Pour $a\in A_1$, $\omega\in B(H_{12})_\ast$ et $\omega'\in B(H_{21})_\ast$, on a en utilisant (\ref{coprodS} b)) :
\hfill\break
$({\rm id}_{A_1} \otimes \omega \otimes \omega')\delta_{A_1}^{(2)}(a)  = ({\rm id}_{A_1} \otimes \omega \otimes \omega')
((W_{12, 23}^1)^*\delta_{A_1}(a)_{13} W_{12, 23}^1) = ({\rm id}_{A_1} \otimes W_{12}^1(\omega \otimes \omega')(W_{12}^1)^*)
(\delta_{A_1}(a)_{13})$.\hfill\break
Comme $W_{12}^1 : H_{12} \otimes H_{21}  \rightarrow H_{12} \otimes H_{11}$ est un unitaire, on a 
\begin{align*}
[({\rm id}_{A_1} \otimes \omega \otimes \omega')\delta_{A_1}^{(2)}(a)\,\,|\,\,a\in A_1  ,   \omega\in B(H_{12})_\ast , \omega'\in B(H_{21})_\ast] &=   [({\rm id}_{A_1} \otimes \omega) \delta_{A_1} (a)\,\,,\,\,a\in A_1  ,   \omega\in B(H_{11})_\ast ]\\
& =: A_1.
\end{align*}
Il en r\'esulte que  
$$C = \delta_{A_1}^{(2)} ([({\rm id}_{A_1} \otimes \omega \otimes \omega')\delta_{A_1}^{2}(A_1)\,\,|\,\,\omega\in B(H_{12})_\ast , \omega'\in B(H_{21})_\ast]) = \delta_{A_1}^{(2)}(A_1).$$
\noindent
On a donc que $\pi_1  : (A_1, \delta_{A_1})  \rightarrow  (C, \delta_C) :  a \mapsto \delta_{A_1}^{(2)}(a)$ est un *-isomorphisme.
\hfill\break 
Pour montrer que cet isomorphisme   est $G_1$-\'equivariant, il suffit d'\'etablir pour tout $a\in A_1$, l'\'egalit\'e 
$$ \delta_C(\delta_{A_1}^{(2)}(a)) = (\delta_{A_1}^{(2)} \otimes {\rm id}_{S_{11}})\delta_{A_1}(a)$$
\noindent
dans  $M(A_1 \otimes S_{12} \otimes S_{21} \otimes S_{11})$. On a en utilisant (\ref{coprodS} a)) :
\begin{align} 
\delta_C(\delta_{A_1}^{(2)}(a)) &= ({\rm id}_{A_1 \otimes S_{12}} \otimes \delta_{21}^1)({\rm id}_{A_1} \otimes \delta_{11}^2) \delta_{A_1}(a) \nonumber \\
&  =  ({\rm id}_{A_1} \otimes \delta_{11}^2 \otimes {\rm id}_{S_{11}})({\rm id}_{A_1} \otimes \delta_{11}^1) \delta_{A_1}(a) \nonumber \\
&=({\rm id}_{A_1} \otimes \delta_{11}^2 \otimes {\rm id}_{S_{11}})(\delta_{A_1} \otimes {\rm id}_{S_{11}}) \delta_{A_1}(a) = (\delta_{A_1}^{(2)} \otimes {\rm id}_{S_{11}})\delta_{A_1}(a). \nonumber 
\end{align}
Le d) se d\'eduit du b) et de (\ref{ind4} b)).
\end{proof}

\noindent
En permutant les r\^oles de $G_1$ et $G_2$, on a avec une preuve similaire, en partant d'une action continue du groupe $G_2$:

\noindent
\begin{proposition}\label{ind6} Soit  $\delta_{A_2}  : A_2 \rightarrow M(A_2 \otimes S_{22})$ une action continue du groupe quantique $G_2$ dans la C*-alg\`ebre $A_2$. Posons  $A_1 := {\rm Ind }_{G_2}^{G_1}\, A_2$,  $\delta_{A_1} := \restr{({\rm id}_{A_2} \otimes \delta_{21}^1)}{A_1}$ et 
$D = {\rm Ind }_{G_1}^{G_2}\, A_1 \subset M(A_1 \otimes S_{12})$ munie de $\delta_D = \restr{({\rm id}_{A_1} \otimes \delta_{12}^2)}{D}$.
\begin{enumerate}
\item  $D \subset M(A_1 \otimes S_{12}) \subset M(A_2 \otimes S_{21} \otimes S_{12})$ et $D= \delta_{A_2}^{(1)} (A_2)$.
\item Le *-morphisme
$$\pi_2 : (A_2, \delta_{A_2})  \rightarrow  (D, \delta_D) :  a \mapsto  \delta_{A_2}^{(1)} (a) = ({\rm id}_{A_2} \otimes \delta_{22}^1)\delta_{A_2}(a)$$
est un isomorphisme $G_2$-\'equivariant.\index{pl@$\pi_1$, $\pi_2$}
\item Le *-morphisme   $\delta_{A_2}^{1}  : A_2 \rightarrow M(A_1 \otimes S_{12}) : a \mapsto  \delta_{A_2}^{(1)} (a)$ est injectif, non d\'eg\'en\'er\'e et on a $[\delta_{A_2}^{1}(A_{2})(1_{A_1} \otimes S_{12})] = A_1 \otimes S_{12}$.
\end{enumerate}
%\begin{proof}[]
%\end{proof}
\end{proposition}
\noindent
On a donc \'etabli : 

\noindent 
\begin{theorem}  Les correspondances $$(A_1, \delta_{A_1}) \mapsto (A_2 := {\rm Ind }_{G_1}^{G_2}\, A_1 ,\, \delta_{A_2} := \restr{({\rm id}_{A_1} \otimes \delta_{12}^2)}{A_2} ),$$ 
$$(A_2, \delta_{A_2}) \mapsto (A_1 := {\rm Ind }_{G_2}^{G_1}\, A_2,\, \delta_{A_1} := \restr{({\rm id}_{A_2} \otimes \delta_{21}^1)}{A_1})$$
\noindent 
sont r\'eciproques l'une de l'autre.
%\begin{proof}[ ]
%\end{proof}
\end{theorem}

\noindent
En utilisant le proc\'ed\'e d'induction d\'evelopp\'e   pr\'ec\'edemment, nous allons  construire   une correspondance  fonctorielle   $A^{G_1} \rightarrow A^{{\cal G}}$ r\'eciproque de      la correspondance  \ref{cor1} :
$$A^{\cal G} \rightarrow A^{G_1} : (A, \beta_A, \delta_A) \mapsto (A_1, \delta_{A_1}^1).$$
\noindent
Commen\c cons par d\'efinir une correspondance $ A^{G_1} \rightarrow A^{{\cal G}}  : (B_1, \delta_{B_1}) \mapsto (B, \beta_B, \delta_B)$.

\noindent
Soit  $\delta_{B_1}  : B_1 \rightarrow M(B_1 \otimes S_{11})$ une action continue du groupe quantique $G_1$. 

On a associ\'e \`a $(B_1, \delta_{B_1})$, une action continue  $(B_2, \delta_{B_2})$   du groupe quantique $G_2$,  avec $B_2 = {\rm Ind }_{G_1}^{G_2}\, B_1$ et $\delta_{B_2} = \restr{(\id_{B_1} \otimes \delta_{12}^2)}{B_2}$,  et nous avons  les quatre *-morphismes\index{dc@$\delta_{A_j}^k$/$\delta_{B_j}^k$} 
$$\delta_{B_j}^k : B_j  \rightarrow M(B_k \otimes S_{kj}),\quad j,k=1,2,$$
\noindent
suivants :\hfill\break
 -  $\delta_{B_1}^1 := \delta_{B_1}, \quad \delta_{B_2}^2:= \delta_{B_2}$ ;
\hfill\break 
 - $\delta_{B_1}^2 : B_1 \rightarrow M(B_2 \otimes S_{21})$ est donn\'e   par (\ref{ind5} d)) : 
\begin{equation} b \mapsto \delta_{B_1}^2(b) = \delta_{B_1}^{(2)}(b) :=({\rm id}_{B_1} \otimes \delta_{11}^2)\delta_{B_1}^1(b) \in M(B_2 \otimes S_{21})\,\text{;}\nonumber 
\end{equation}
\hfill\break 
- $\delta_{B_2}^1   : B_2 \rightarrow M(B_1 \otimes S_{12})$ est d\'etermin\'e (\ref{ind6} c)) et  la relation  (\cf \ref{ind5} c)) :
\begin{equation}\label{pi1}\delta_{B_2}^{(1)}  = (\pi_1 \otimes {\rm id}_{S_{12}}) \delta_{B_2}^1  , \quad \pi_1 : B_1  \rightarrow  {\rm Ind }_{G_2}^{G_1}\, B_2 : b \mapsto \pi_1(b) := \delta_{B_1}^{(2)}(b),
\end{equation}
\noindent
avec $\delta_{B_2}^{(1)} := ({\rm id}_{B_2} \otimes \delta_{22}^1) \circ \delta_{B_2}^2$.

\begin{lemme}\label{lesmorph} Soient $j, k, l=1,2$. 
\begin{enumerate}
\item  $(\delta_{B_k}^l \otimes {\rm id}_{S_{kj}}) \delta_{B_j}^k = ({\rm id}_{B_l} \otimes \delta_{lj}^k) \delta_{B_j}^l$.
\item  $[\delta_{B_k}^l(B_k)(1_{B_l} \otimes S_{lk})] = B_l \otimes S_{lk}$.
\item $B_l = [ ({\rm id}_{B_l} \otimes \omega) (\delta_{B_k}^l(B_k)) \,\,| \,\,\omega \in B(H_{lk})_\ast ]$.
\end{enumerate}
\end{lemme}

\begin{proof}
La preuve du a) s'appuie uniquement sur (\ref{coprodS} a)) et les d\'efinitions des *-morphismes $\delta_{B_k}^l$. Plus pr\'ecis\'ement :
 \hfill\break
- Si $j= 1$, seul le cas $k=2$ et $l=1$ n'est pas imm\'ediat. On l'\'etablit par composition avec 
$\delta_{B_1}^2 \otimes {\rm id}_{S_{12}\otimes S_{21}}$ ;
\hfill\break
 - Si $j=2$. Dans le cas $k=l=1$, on \'etablit la formule par composition avec 
$\delta_{B_1}^2 \otimes {\rm id}_{S_{11} \otimes S_{12}}$. Si  $k=2$ et $l=1$, on \'etablit la formule par composition avec 
$\delta_{B_1}^2 \otimes {\rm id}_{S_{12} \otimes S_{22}}$. Les autres cas sont imm\'ediats.
\hfill\break
Pour montrer le b) rappelons que $B_2 = {\rm Ind }_{G_1}^{G_2}\, B_1$. Posons $C = {\rm Ind }_{G_2}^{G_1}\, B_2$   et $D={\rm Ind }_{G_1}^{G_2}\, C$.  On a  :
\hfill\break
 - $[C(1_{B_2} \otimes S_{21})] = B_2 \otimes S_{21}$ (\ref{ind4} b)) et    $C =  \delta_{B_1}^2(B_1)$, donc 
$[\delta_{B_1}^2(B_1)(1_{B_2} \otimes S_{21})] = B_2 \otimes S_{21}$ ;\hfill\break
 -  $[D(1_{C} \otimes S_{12})] = C \otimes S_{12}$ (\ref{ind3} a)) et   $D =  \delta_{B_2}^{(1)}(B_2)$. 
\hfill\break
Par composition avec le *-isomorphisme $\pi_1^{-1} \otimes {\rm id}_{S_{12}}$ et en remarquant qu'on a  $(\pi_1^{-1} \otimes {\rm id}_{S_{12}}) \delta_{B_2}^{(1)} = \delta_{B_2}^{1}$, on d\'eduit l'\'egalit\'e $[\delta_{B_2}^1(B_2)(1_{B_1} \otimes S_{12})] = B_1 \otimes S_{12}$.\hfill\break
Les deux autres relations restantes d\'ecoulent de la   continuit\'e   (\ref{ind3} b)) des coactions  $\delta_{B_1}$ et $\delta_{B_2}$. Le c) est une cons\'equence du b).
\end{proof}

En appliquant \ref{recacc}, on  a donc montr\'e :
\begin{proposition}\label{reccorresp} Soit  $\delta_{B_1}  : B_1 \rightarrow M(B_1 \otimes S_{11})$ une action continue du groupe quantique $G_1$.
Posons $B_2 = {\rm Ind }_{G_1}^{G_2}\, B_1$ et  $B := B_1 \oplus B_2$. D\'efinissons les deux *-morphismes :
$$\delta_B : B \rightarrow M(B \otimes S) : b =(b_1,b_2) \mapsto \delta_B(b) := \sum_{k,j} \pi_j^k \delta_{B_j}^k(b_j) \quad ; \quad 
\beta_B : \C^2 \rightarrow M(B) : (\lambda, \mu) \mapsto  \begin{pmatrix}\lambda & 0\\
0 & \mu \end{pmatrix},$$\noindent
o\`u   $\pi_j^k : M(B_k \otimes  S_{kj}) \rightarrow M(B \otimes S)$ d\'esigne le prolongement strictement continu de l'injection canonique $B_k \otimes  S_{kj} \rightarrow B \otimes S$.
\hfill\break 
Nous avons :
\begin{enumerate}
\item $(\beta_B, \delta_B)$ est une action continue du groupo\"ide ${\cal G}$ ;
\item La correspondance $A^{G_1} \rightarrow A^{\cal G}  : (B_1, \delta_{B_1})  \mapsto  (B,  \beta_B, \delta_B)$ est fonctorielle.
\end{enumerate}
\end{proposition}

\begin{proof}
Le a) d\'ecoule de \ref{recacc}, le b) de \ref{recmor}; notons que les isomorphismes $\pi_1$ et $\pi_2$ ((\ref{ind5} c)) et (\ref{ind6} b)) respectivement), v\'erifient une propri\'et\'e de naturalit\'e vis-\`a-vis des morphismes de $A^{G_1}$ et $A^{G_2}$. 
\end{proof}

\begin{theorem}
La correspondance $A^{ G_1} \rightarrow A^{\cal G}$ est la correspondance r\'eciproque de la correspondance 
$A^{\cal G} \rightarrow A^{G_1} : (A, \beta_A, \delta_A) \mapsto (A_1, \delta_{A_1}^1)$.
\end{theorem}

Il est clair que la compos\'ee de corrrespondances $A^{G_1} \rightarrow A^{\cal G} \rightarrow A^{G_1}$ est la correspondance identique.
Pour montrer qu'il en est de m\^eme pour  la compos\'ee $A^{\cal G} \rightarrow A^{G_1} \rightarrow A^{\cal G}$, nous avons besoin du r\'esultat suivant :

\begin{proposition}\label{isoind} Soit $(\beta_A, \delta_A)$ une action continue  du groupo\"ide ${\cal G}$ dans une C*-alg\`ebre $A$. Avec les notations de \ref{not1} et \ref{acc}, posons $\widetilde{A_2} := {\rm Ind}_{G_1}^{G_2} (A_1, \delta_{A_1}^1)$ et 
$\widetilde{A_1} := {\rm Ind}_{G_2}^{G_1} (A_2, \delta_{A_2}^2)$. Alors on a :
\begin{enumerate}
\item  Pour tout $x\in A_2$, on a $\delta_{A_2}^1(x) \in \widetilde{A_2}$. De plus, 
$\widetilde\pi_2 : (A_2, \delta_{A_2}^2) \rightarrow (\widetilde{A_2}, \delta_{\widetilde{A_2}}) : x \mapsto  \delta_{A_2}^1(x)$ est un *-isomorphisme $G_2$-\'equivariant. 
\item Pour tout $x\in A_1$, on a $\delta_{A_1}^2(x) \in \widetilde{A_1}$. De plus, 
$\widetilde\pi_1 : (A_1, \delta_{A_1}^1) \rightarrow (\widetilde{A_1}, \delta_{\widetilde{A_1}}) : x \mapsto  \delta_{A_1}^2(x)$ est un *-isomorphisme $G_1$-\'equivariant.\index{pm@$\widetilde{\pi}_1$, $\widetilde{\pi}_2$}
\end{enumerate}
\end{proposition}

\begin{proof}
 Pour la preuve du a), on a (\ref{acc} c)), 
   $A_2 =[({\rm id}_{A_2} \otimes  \omega)\delta_{A_1}^2(a) \,\,|\,\, a\in A_1 \,,\,\omega\in B(H_{21})_\ast]$. 
\hfill\break
Soit $a\in A_1$ et $\omega\in B(H_{21})_\ast$, on a 
$$\delta_{A_2}^1(({\rm id}_{A_2} \otimes  \omega)\delta_{A_1}^2(a)) = 
({\rm id}_{{A_1} \otimes S_{12}} \otimes \omega)(\delta_{A_2}^1 \otimes {\rm id}_{S_{21}}) \delta_{A_1}^2(a) = 
({\rm id}_{{A_1} \otimes S_{12}} \otimes \omega)({\rm id}_{A_1} \otimes \delta_{11}^2) \delta_{A_1}^1(a)\in \widetilde {A_2}.$$
On en d\'eduit que $\widetilde\pi_2  :  A_2   \rightarrow    \widetilde {A_2}  : x \mapsto \delta_{A_2}^1(x)$ est un *-isomorphisme. 
Pour tout $x\in A_2$, on a 
$$\delta_{\widetilde{A_2}}(\widetilde\pi_2(x))  = \delta_{\widetilde{A_2}}(\delta_{A_2}^1(x)) = ({\rm id}_{A_1} \otimes \delta_{12}^2)(\delta_{A_2}^1(x)) = 
(\delta_{A_2}^1 \otimes {\rm id}_{S_{22}})(\delta_{A_2}^2(x)) = (\widetilde\pi_2  \otimes {\rm id}_{S_{22}})(\delta_{A_2}^2(x)),$$ 
d'o\`u l'\'equivariance de l'isomorphisme. La preuve du b) est similaire.
\end{proof}

\begin{proof}[D\'emonstration du th\'eor\`eme] 
Partons d'une action continue $(\beta_A, \delta_A)$ du groupo\"ide ${\cal G}$. Avec les notations de \ref{not1} et \ref{acc}, on a 
$A = A_1 \oplus A_2$ avec $A_i:= \beta_A(\varepsilon_i) A$ et les *-morphismes 
$\delta_{A_j}^k : A_j \rightarrow  M(A_k \otimes  S_{kj})$ v\'erifiant les conditions (\ref{acc} b) et c)).
\hfill\break
Posons $(B_1, \delta_{B_1}) := (A_1, \delta_{A_1}^1)$ et  
$B_2 := {\rm Ind}_{G_1}^{G_2} (A_1, \delta_{A_1}^1)$, et soit $(\beta_B, \delta_B)$ l'action continue du groupo\"ide  ${\cal G}$ 
dans la   C*-alg\`ebre $B := B_1 \oplus B_2$ associ\'ee \ref{reccorresp}  \`a $(B_1, \delta_{B_1})$. Reprenons le *-isomorphisme  $G_2$-\'equivariant $\widetilde\pi_2 : (A_2, \delta_{A_2}^2) \rightarrow (B_2, \delta_{B_2}^2)$ donn\'e par le a) de \ref{isoind}.\hfill\break
Pour tout $a = a_1 + a_2\in A$ , posons $f(a) = (a_1, \widetilde\pi_2(a_2))$. \hfill\break 
Il est clair que $f : A  \rightarrow B$ est un *-isomorphisme v\'erifiant $f \circ \beta_A = \beta_B$. Notons $f_j : A_j \rightarrow B_j$ le *-isomorphisme obtenu par restriction de $f$ \`a $A_j$.\hfill\break 
Avec les notations de \ref{lesmorph}, il est clair que la relation 
$(f \otimes {\rm id}_S) \circ \delta_A = \delta_B \circ f$ est \'equivalente aux relations  
$$(f_k \otimes {\rm id}_{S_{kj}}) \delta_{A_j}^k = 
\delta_{B_j}^k \circ f_j  , \quad j, k =1,2.$$
\noindent
Pour $j, k = 1$, il n'y a rien \`a d\'emontrer.  Pour $j = 1$ et $k = 2$, dans $M(A_1 \otimes S_{12} \otimes S_{21})$, on a 
$$\delta_{B_1}^2 = \delta_{A_1}^{(2)} :=  ({\rm id}_{A_1} \otimes \delta_{11}^2) \delta_{A_1}^1 = (\delta_{A_2}^1 \otimes {\rm id}_{S_{21}}) \delta_{A_1}^2 = (\widetilde\pi_2  \otimes {\rm id}_{S_{21}}) \delta_{A_1}^2 = (f_2  \otimes {\rm id}_{S_{21}}) \delta_{A_1}^2.$$
Pour $j, k = 2$,  la relation $(f_k \otimes {\rm id}_{S_{kj}}) \delta_{A_j}^k = 
\delta_{B_j}^k \circ f_j$ est exactement la $G_2$-\'equivariance de l'isomorphisme (\ref{isoind} a)) $\widetilde\pi_2$.
\hfill\break
Pour $j = 2$ et $k = 1$,  il faut montrer que  $\delta_{A_2}^1 = \delta_{B_2}^1 \circ \widetilde\pi_2$. 
En composant avec l'isomorphisme $\pi_1 \otimes {\rm id}_{S_{12}}$ introduit dans (\ref{pi1}), il revient au m\^eme  de montrer 
$$(\pi_1 \otimes {\rm id}_{S_{12}}) \circ \delta_{A_2}^1 = 
 ({\rm id}_{B_2} \otimes \delta_{22}^1) \circ \delta_{B_2}^2 \circ \widetilde\pi_2.$$
Dans $M(B_2 \otimes S_{21} \otimes S_{12}) \subset M(A_1 \otimes S_{12}  \otimes S_{21} \otimes S_{12})$, on a  par d\'efinition de $\pi_1$ : 
$$(\pi_1 \otimes {\rm id}_{S_{12}}) \circ \delta_{A_2}^1 = ([({\rm id}_{A_1} \otimes \delta_{11}^2) \delta_{A_1}^1] \otimes {\rm id}_{S_{12}}) \circ \delta_{A_2}^1.$$
Par \'equivariance de $\widetilde\pi_2$ (et aussi sa d\'efinition), on a    dans $M(B_2 \otimes S_{21} \otimes S_{12}) \subset M(A_1 \otimes S_{12}  \otimes S_{21} \otimes S_{12})$ : 
\begin{align}({\rm id}_{B_2} \otimes \delta_{22}^1) \circ \delta_{B_2}^2 \circ \widetilde\pi_2 &= 
({\rm id}_{B_2} \otimes \delta_{22}^1) (\widetilde\pi_2 \otimes {\rm id}_{S_{22}}) \delta_{A_2}^2 \nonumber \\
&  = (\widetilde\pi_2 \otimes {\rm id}_{S_{21} \otimes S_{12}}) ({\rm id}_{A_2} \otimes \delta_{22}^1) \delta_{A_2}^2 \nonumber \\
&= (\widetilde\pi_2 \otimes {\rm id}_{S_{21} \otimes S_{12}}) (\delta_{A_1}^2 \otimes {\rm id}_{S_{12}}) \delta_{A_2}^1\nonumber \\
&  = ([(\delta_{A_2}^1 \otimes {\rm id}_{S_{21}}) \delta_{A_1}^2] \otimes {\rm id}_{S_{12}}) \delta_{A_2}^1\nonumber\\
& =
([({\rm id}_{A_1} \otimes \delta_{11}^2) \delta_{A_1}^1] \otimes {\rm id}_{S_{12}}) \circ \delta_{A_2}^1.\nonumber\qedhere 
\end{align}
\end{proof}

\subsection {Induction et \'equivalence de Morita}

\noindent
Nous allons montrer que les \'equivalences de cat\'egories  
$$A^{\cal G} \rightarrow  A^{G_1}, \quad A^{ G_1} \rightarrow A^{G_2} \quad \text{et} \quad A^{\cal G} \rightarrow  A^{G_2}$$
\noindent  
\'echangent les \'equivalences de Morita.

\noindent
\begin{definition}  On appelle ${\cal G}$-alg\`ebre de liaison,  une C*-alg\`ebre $J$, munie d'une action continue $(\beta_J,  \delta_J)$ du groupo\"ide ${\cal G}$ et de deux projecteurs non nuls $e_1,\, e_2 \in M(J)$ v\'erifiant les conditions :
\begin{enumerate}
\item  $e_1 + e_2 = 1_J \quad ; \quad [J e_i J] = J$,\quad $i = 1, 2$.
\item $\delta_J(e_i) = q_{\beta_J,  \alpha } (e_i \otimes 1_S)$,\quad $i=1,2$.
\end{enumerate}
\end{definition}
Reprenons les notations de (\ref{not1}) :
$$q_j = \beta_J(\varepsilon_j), \quad  J_j = q_j J, \quad j=1,2 \quad ; \quad \delta_{J_j}^k : J_j \rightarrow M(J_k \otimes S_{kj}),\quad j,k=1,2.$$
\noindent
Rappelons (\ref{centre}) que $\beta_J$ prend ses valeurs dans le centre de $M(J)$. D'où $[q_j, e_i] = 0$, pour tout $i,j=1,2$. Pour $j=1,2$, notons  $\iota_j : M(J_j) \rightarrow M(J)$ le prolongement strictement continu de l'inclusion $J_j \subset J$ v\'erifiant $\iota_j(1_{J_j}) = q_j$. 
Pour $i, j = 1, 2$, notons $e_{i, j}$, l'unique projecteur de $M(J_j)$ v\'erifiant 
$\iota_j(e_{i, j}) = e_i q_j$.

Avec ces notations nous avons  facilement :

\begin{lemme} Soient  $i, j, k =1, 2$.
\begin{enumerate}
\item $e_{1, j}  +  e_{2, j}   = 1_{J_j}, \quad [J_j e_{i, j} J_j] = J_j$.
\item  $\delta_{J_j}^k(e_{i, j}) = e_{i, k} \otimes 1_{S_{kj}}$.
\end{enumerate}
\end{lemme}

Il en r\'esulte que  la $G_j$-alg\`ebre $(J_j, \delta_{J_j}^j)$, munie de $(e_{1, j}, e_{2, j})$ est une \'equivalence de Morita $G_j$-\'equivariante et l'image par le foncteur    
$$A^{\cal G} \rightarrow A^{G_j} : (J, \beta_J, \delta_J) \mapsto (J_j, \delta_{J_j}^j)$$ 
d'une \'equivalence de Morita est aussi une \'equivalence de Morita.

R\'eciproquement, soit $(J_1, \delta_{J_1}, e_{i, 1})$ une $G_1$-alg\`ebre de liaison. Notons $(J, \delta_J, \beta_J)$ la ${\cal G}$-alg\`ebre associ\'ee \ref{reccorresp} par le foncteur $A^{G_1} \rightarrow A^{\cal G}$.

Rappelons que $J= J_1 \oplus  J_2$ avec  $J_2 : ={\rm Ind}_{G_1}^{G_2} (J_1, \delta_{J_1})$ munie de la coaction $\delta_{J_2}^2:= \restr{({\rm id}_{J_1}  \otimes \delta_{12}^2)}{J_2}$. La C*-alg\`ebre 
    $M(J_2)$ est identifi\'ee \`a une sous-C*-alg\`ebre de  $M(J_1 \otimes S_{12})$ et le coproduit 
$\delta_J : J  \rightarrow M(J \otimes S)$ est donn\'e pour tout $a = (a_1, a_2)$ par la formule :
$$\delta_J(a) = \sum_{k,j} \pi_j^k(\delta_{J_j}^k( a_j)) \quad ; \quad \delta_{J_j}^k : J_j \rightarrow M(J_k \otimes S_{kj}),\quad j,k=1,2.$$
Pour  $i = 1, 2$, posons   $e_{i,2} := e_{i,1} \otimes 1_{S_{12}} \in {\rm Proj}(M(J_1 \otimes S_{12}))$.

Avec ces notations, nous avons :

\begin{proposition} \label{indu2}
Le quintuplet $(J_2, \delta_{J_2}^2, e_{1, 2},e_{2, 2})$ est une  $G_2$-alg\`ebre de liaison. De plus, pour tout 
$i, j, k =1, 2$,  nous avons 
$$\delta_{J_j}^k(e_{i, j}) = e_{i, k} \otimes 1_{S_{kj}}.$$
\begin{proof} Montrons que $e_{i, 2}\in M(J_2)$. Soient $a\in J_1$ et $\omega\in B(H_{21})_\ast$. Nous avons 
$$e_{i, 2} ({\rm id}_{J_1 \otimes S_{12}} \otimes \omega) ({\rm id}_{J_1} \otimes \delta_{11}^2) \delta_{J_1}(a) = 
({\rm id}_{J_1 \otimes S_{12}} \otimes \omega) ({\rm id}_{J_1} \otimes \delta_{11}^2) \delta_{J_1}( e_{i, 1} a).$$
Il en r\'esulte que $e_{i, 2} \in {\rm Proj}(M(J_2))$ et  on a :
$$e_{1, 2}  + e_{2, 2} = 1_{J_2} \quad ; \quad \delta_{J_j}^2(e_{i, j}) = e_{i, 2} \otimes 1_{S_{2j}}, \quad i, j = 1, 2.$$
Pour montrer   l'\'egalit\'e $[J_2 e_{i, 2} J_2] = J_2$, il suffit   de montrer \ref{ind} que pour tout $\omega \in B(H_{21})_\ast$ et pour tout  $a_1, b_1\in J_1$, on a $({\rm id}_{J_2} \otimes \omega)\delta_{J_1}^2(a_1 e_{i, 1} b_1) \in [J_2 e_{i, 2} J_2]$. On a 
$$({\rm id}_{J_2} \otimes \omega)\delta_{J_1}^2(a_1 e_{i, 1} b_1) =  ({\rm id}_{J_2} \otimes \omega) [\delta_{J_1}^2(a_1)( e_{i, 2} \otimes 1_{S_{21}}) \delta_{J_1}^2(b_1)].$$ 
Par (\ref{ind5} d)), on a $[\delta_{J_1}^2(J_1)(1_{J_2} \otimes S_{21})] = J_2 \otimes S_{21}$. Il en r\'esulte que  
 $({\rm id}_{J_2} \otimes \omega)\delta_{J_1}^2(a_1 e_{i, 1} b_1)$ est limite normique de sommes finies d'\'el\'ements de la forme 
\[
x  e_{i, 2} ({\rm id}_{J_2} \otimes \omega')\delta_{J_1}^2(b_1), \quad \text{avec} \quad x \in J_2,\; \omega'\in B(H_{21})_*.
\]
\noindent
Pour voir que $\delta_{J_2}^1(e_{i, 2}) = e_{i, 1} \otimes 1_{S_{12}}$, il suffit de composer cette \'egalit\'e avec le *-isomorphisme $\pi_1 \otimes {\rm id}_{S_{12}}$, o\`u $\pi_1$ est le *-isomorphisme (\ref{pi1}).
\end{proof}
\end{proposition}
\noindent
Pour $i = 1,2$, posons $e_i := (e_{i, 1}, e_{i, 2})\in M(J)$. Les d\'efinitions  des *-morphismes \ref{reccorresp}, $(\beta_B , \delta_B)$ et ce qui pr\'ec\`ede,  entra\^inent imm\'ediatement qu'on a :

\begin{corollary}\label{indugr} Le quintuplet $(J,  \beta_J, \delta_J, e_1, e_2)$ est une ${\cal G}$-alg\`ebre de liaison.
L'image par le foncteur  $A^{G_1} \rightarrow A^{\cal G}$ d'une $G_1$-alg\`ebre de liaison est aussi une ${\cal G}$-alg\`ebre de liaison.
\end{corollary} 

\subsection{Induction de C*-modules}

\noindent
Pour \'etendre le proc\'ed\'e d'induction  aux C*-modules \'equivariants \cite{BaSka1}, nous avons besoin du lemme suivant :
\begin{lemme}\label{indu1} Soient $(B_1, \delta_{B_1})$ et $(J_1, \delta_{J_1})$ deux $G_1$-alg\`ebres  et 
$f_1  : B_1 \rightarrow M(J_1)$ un *-morphisme v\'erifiant :
\begin{enumerate}
\item  Pour une unit\'e approch\'ee $(u_\lambda)$ de $B_1$, on suppose que $f_1(u_\lambda)  \rightarrow e_1\in M(J_1)$ pour la topologie stricte ;
\item  $(f_1\otimes {\rm id}_{S_{11}}) \delta_{B_1} = \delta_{J_1}  \circ f_1$. 
\end{enumerate}
Alors nous avons :
\begin{enumerate}[label={\rm(\roman*)}]
\item  $e_1$ est un projecteur de $M(J_1)$, qui ne d\'epend pas de l'unit\'e approch\'ee $(u_\lambda)$ de $B_1$. En outre, on a $\delta_{J_1}(e_1) = e_1 \otimes 1_{S_{11}}$ ; 
\item  Pour tout $T\in M(B_1 \otimes S_{12})$, on a $(f_1\otimes {\rm id}_{S_{12}})(T) = (e_1 \otimes 1_{S_{12}}) (f_1\otimes {\rm id}_{S_{12}})(T) (e_1 \otimes 1_{S_{12}})$.
\end{enumerate}
\end{lemme}

Remarquons que la condition b) a un sens gr\^ace \`a la condition a), $f_1 \otimes {\rm id}_{S_{11}}$ \'etant le prolongement strictement continu \`a $M(B_1 \otimes S_{11})$ du *-morphisme $f_1 \otimes {\rm id}_{S_{11}}  : B_1 \otimes S_{11}\rightarrow M(J_1 \otimes S_{11})$, v\'erifiant $(f_1 \otimes {\rm id}_{S_{11}})(1_{B_1 \otimes S_{11}}) = e_1 \otimes 1_{S_{11}}$.\begin{proof}
Pour voir que   $e_1$ est un projecteur de $M(J_1)$ qui ne d\'epend pas de l'unit\'e approch\'ee, et que pour toute unit\'e approch\'ee $(u_\lambda')$ de $B_1$, on a  $f_1(u_\lambda')  \rightarrow e_1$ pour la topologie stricte dans   $M(J_1)$,  \cf\ref{efin}.\hfill\break  
On a $(f_1 \otimes {\rm id}_{S_{11}}) \delta_{B_1}(1_{B_1}) = e_1 \otimes 1_{S_{11}}$ et $(f_1 \otimes {\rm id}_{S_{11}}) \delta_{B_1}(u_\lambda) = 
\delta_{J_1}(f_1(u_\lambda)) \rightarrow \delta_{J_1}(e_1)$ pour la topologie stricte de $M(J_1 \otimes S_{11})$, donc $\delta_{J_1}(e_1) = e_1 \otimes 1_{S_{11}}$, d'o\`u le a).
\hfill\break
Le   *-morphisme     $M(B_1 \otimes  S_{12})  \rightarrow M(J_1 \otimes  S_{12}) : T \mapsto (e_1 \otimes 1_{S_{12}}) (f_1 \otimes {\rm id}_{S_{12}})(T) (e_1 \otimes 1_{S_{12}})$  est strictement continu et co\"incide avec $f_1 \otimes {\rm id}_{S_{12}}$ sur $B_1 \otimes  S_{12}$, d'o\`u le b).
\end{proof}
\noindent
Avec les hypoth\`eses et les notations de \ref{indu1}, posons  : 
$$B_2 := {\rm Ind}_{G_1}^{G_2} (B_1, \delta_{B_1}) \subset M(B_1 \otimes S_{12}), \quad J_2 := 
{\rm Ind}_{G_1}^{G_2} (J_1, \delta_{J_1}) \subset  M(J_1 \otimes S_{12}), \quad e_2 := e_1 \otimes 1_{S_{12}}.$$

% Il est clair que 
%$(\pi_1\ot {\rm id}_{S_{12}})(B_2) \subset (e_1 \otimes 1_{S_{12}})  M(J_1 \otimes S_{12}) (e_1 \otimes 1_{S_{12}})$. L'inclusion $B_2 \subset M(B_1 \otimes  S_{12})$ \'etant non d\'eg\'en\'er\'ee, pour toute unit\'e approch\'ee $(v_\lambda)$ de $B_2$, on a 
%$(\pi_1\ot {\rm id}_{S_{12}})(v_\lambda) \rightarrow e_1 \otimes 1_{S_{12}}$ pour la topologie stricte de $M(J_1 \otimes S_{12})$.
%\hfill\break
%

%Posons alors ${\rm Ind}_{G_1}^{G_2} \pi_1:  = (\pi_1\ot {\rm id}_{S_{12}})_{|_{B_2}} : B_2  \rightarrow (e_1 \otimes 1_{S_{12}})  M(J_2) (e_1 \otimes 1_{S_{12}})$ et $e_2 := e_1 \otimes 1_{S_{12}}$.  

\begin{proposition}\label{indu3}
Avec les notations précédentes, nous avons :
\begin{enumerate}
\item  $e_2 \in M(J_2)$  et $(f_1 \otimes  {\rm id}_{S_{12}})(B_2) \subset e_2  M(J_2) e_2$.   
\item Posons $f_2 := {\rm Ind}_{G_1}^{G_2} f_1  : B_2 \rightarrow M(J_2) :  x \mapsto (f_1 \otimes {\rm id}_{S_{12}})(x)$\index{i@${\rm Ind}_{G_1}^{G_2} f_1$}. Alors $f_2$ est un *-morphisme  v\'erifiant : 
\begin{enumerate}\renewcommand\theenumii{\roman{enumii}}
 \renewcommand\labelenumii{\rm ({\theenumii})}
\item Pour toute unit\'e approch\'ee $(v_\lambda)$ de $B_2$, on a $f_2(v_\lambda)  \rightarrow e_2$ pour la topologie stricte de $M(J_2)$ ;
\item   $(f_2 \otimes {\rm id}_{S_{22}}) \delta_{B_2} = \delta_{J_2}  \circ f_2$ ; 
\item Soit $F_1 \in e_1 M(J_1) e_1$ v\'erifiant pour tout $b\in B_1$ :
\begin{equation}
[F_1, f_1(b)]  \in e_1 J_1 e_1 \,\, , \,\, f_1(b)(F_1 - F_1^*)\in e_1 J_1 e_1 \,\, , \,\,
 f_1(b)(F_1^2 - 1)\in e_1 J_1 e_1 \,\, , \,\, \delta_{J_1}(F_1)  = F_1 \otimes 1_{S_{11}}.\nonumber
\end{equation}
Alors $F_2 := F_1 \otimes 1_{S_{12}}\in e_2 M(J_2) e_2$ et 
pour tout $b\in B_2$ on a : 
\begin{equation}
[F_2, f_2(b)]  \in e_2 J_2 e_2 \,\, , \,\, f_2(b)(F_2 - F_2^*)\in e_2 J_2 e_2 \,\, , \,\,
 f_2(b)(F_2^2 - 1)\in e_2 J_2 e_2\,\, , \,\, \delta_{J_2}(F_2)  = F_2 \otimes 1_{S_{22}}.\nonumber
 \end{equation} 
\end{enumerate}
\end{enumerate}
\end{proposition}
\begin{proof}
Soit $\omega\in B(H_{21})_\ast$ un \'etat normal. Pour tout $x_1 \in J_1$ et tout $\omega'\in B(H_{21})_\ast$, on  a 
$$(e_1 \otimes 1_{S_{12}}) ({\rm id}_{J_1 \otimes S_{12}} \otimes \omega') \delta_{J_1}^{(2)}(x_1) =  ({\rm id}_{J_1 \otimes S_{12}} \otimes \omega \otimes \omega')\delta_{J_1}^{(2)}( e_1)_{123}\delta_{J_1}^{(2)}(x_1)_{124},$$
\noindent
avec $\delta_{J_1}^{(2)} = ({\rm id}_{J_1} \otimes \delta_{11}^2) \delta_{J_1}$.
 Comme $e_1\in M(J_1)$, on a (\ref{imp}) : 
$$({\rm id}_{J_1 \otimes S_{12}} \otimes \omega \otimes \omega')\delta_{J_1}^{(2)}( e_1)_{123}\delta_{J_1}^{(2)}(x_1)_{124} \in J_2,$$
\noindent
donc $e_2 :=e_1 \otimes 1_{S_{12}} \in M(J_2)$.
\hfill\break
Montrons  que $(f_1 \otimes  {\rm id}_{S_{12}})(B_2) \subset   M(J_2)$. Pour  tout $b_1\in B_1$, $x_1 \in J_1$ et tout $\omega,\,\omega'\in B(H_{21})_\ast$, on a 
$$[(f_1 \otimes  {\rm id}_{S_{12}})({\rm id}_{B_1 \otimes S_{12}} \otimes \omega) \delta_{B_1}^{(2)}(b_1)] ({\rm id}_{J_1 \otimes S_{12}} \otimes \omega') \delta_{J_1}^{(2)}(x_1) = 
({\rm id}_{J_1 \otimes S_{12}} \otimes \omega \otimes \omega')\delta_{J_1}^{(2)}( f_1(b_1))_{123}\delta_{J_1}^{(2)}(x_1)_{124},$$
avec  $\delta_{B_1}^{(2)} = ({\rm id}_{B_1} \otimes \delta_{11}^2) \delta_{B_1}\,$.
On a $ f_1(b_1) \in M(J_1)$ et  par  (\ref{imp}), on d\'eduit 
$$({\rm id}_{J_1 \otimes S_{12}} \otimes \omega \otimes \omega')\delta_{J_1}^{(2)}( f_1(b_1))_{123}\delta_{J_1}^{(2)}(x_1)_{124} \in J_2,$$
\noindent
donc $(f_1 \otimes {\rm id}_{S_{12}})(B_2) \subset   M(J_2)$. Tenant compte de (\ref{indu1} b)), on a prouv\'e le a).
\hfill\break
Pour   la preuve du (i) il suffit de montrer que $e_2 J_2 = [f_2(b_2) x_2 \,\,|\,\,b_2 \in B_2 , x_2 \in J_2]$.
L'inclusion 
$$[\pi_2(b_2) x_2 \,\,|\,\,b_2 \in B_2 , x_2 \in J_2] \subset e_2 J_2$$ 
r\'esulte du a).
Pour montrer l'\'egalit\'e, montrons d'abord que 
$$e_2  J_2 = [({\rm id}_{J_1 \otimes S_{12}} \otimes \omega)\delta_{J_1}^{(2)}(f_1(b_1) x_1) 
\,\,|\,\,b_1 \in B_1,\, x_1 \in J_1,\, \omega \in B(H_{21})_\ast].$$
On a d'abord pour tout $b_1 \in B_1$, $x_1 \in J_1$   et  $\omega \in B(H_{21})_\ast$ :
$$({\rm id}_{J_1 \otimes S_{12}} \otimes \omega)\delta_{J_1}^{(2)}(f_1(b_1) x_1)  = ({\rm id}_{J_1 \otimes S_{12}} \otimes \omega)\delta_{J_1}^{(2)}(f_1(e_1b_1) x_1)  =
e_2 ({\rm id}_{J_1 \otimes S_{12}} \otimes \omega)\delta_{J_1}^{(2)}(f_1(b_1) x_1) \in e_2 J_2.$$
\noindent
Soit $(u_\lambda)$ une unit\'e approch\'ee de $B_1$, $x_1\in J_1$ et $\omega \in B(H_{21})_\ast$. On a 
$$({\rm id}_{J_1 \otimes S_{12}} \otimes \omega)\delta_{J_1}^{(2)}(f_1(u_\lambda) x_1)  \rightarrow  e_2
({\rm id}_{J_1 \otimes S_{12}} \otimes \omega)\delta_{J_1}^{(2)}(x_1)\in e_2 J_2\quad ({\rm normiquement}).$$
Il reste \`a montrer que pour tout $b_1 \in B_1$, $x_1 \in J_1$ et tout $\omega \in B(H_{21})_\ast$, on  a 
$$({\rm id}_{J_1 \otimes S_{12}} \otimes \omega)\delta_{J_1}^{(2)}(f_1(b_1) x_1) \in  [f_2(b_2) x_2 \,\,|\,\,b_2 \in B_2 , x_2 \in J_2].$$
\noindent
En effet, on sait (\ref{ind5} d) que $[\delta_{J_1}^{2}(J_1)(1_{J_1 \otimes S_{12}} \otimes S_{21})] = J_2 \otimes S_{21}$. Posons alors $\omega = s \,\omega'$ avec $s\in S_{21}$. On a alors 
$$({\rm id}_{J_1 \otimes S_{12}} \otimes \omega)\delta_{J_1}^{(2)}(f_1(b_1) x_1) = ({\rm id}_{J_1 \otimes S_{12}} \otimes \omega')\delta_{J_1}^{(2)}(f_1(b_1))
\delta_{J_1}^{(2)}(x_1)(1_{J_1 \otimes S_{12}} \otimes s)) \in [f_2(B_2) J_2],$$
\noindent
 d'o\`u le (i). Le (ii) r\'esulte de la d\'efinition  (\ref{ind3} b)) des coactions $\delta_{B_2}$ et $\delta_{J_2}$.
\hfill\break
Le (iii) se d\'eduit facilement des \'egalit\'es  
 $f_2(B_2) = [({\rm id}_{J_1 \otimes S_{12}}  \otimes \omega)({\rm id}_{J_1} \otimes \delta_{11}^2)\delta_{J_1}(f_1(x)) | x\in B_1  , \omega \in B(H_{21})_\ast ]$ et 
$e_2 J_2 e_2 = [({\rm id}_{J_1 \otimes S_{12}}  \otimes \omega)({\rm id}_{J_1} \otimes \delta_{11}^2)\delta_{J_1}(e_1y e_1) | y\in J_1 , \omega \in B(H_{21})_\ast ].$  
\end{proof}

Soit ${\cal E}_1$ un $B_ 1$-module  hilbertien $G_1$-\'equivariant.  On munit   la $G_1$-alg\`ebre 
$J_1 := {\cal K}({\cal E}_1 \oplus B_1)$ de la coaction compatible avec  celle de $B_1$ et ${\cal E}_1$ et  des deux  projecteurs :
$$e_{1, 1} = \begin{pmatrix}1_{{\cal E}_1} & 0 \\
0 & 0 \end{pmatrix}, \quad \,e_{2, 1} = \begin{pmatrix} 0 & 0 \\
0 & 1_{B_1} \end{pmatrix}.$$
On rappelle que $(J_1, \delta_{J_1}, e_{l, 1})$ est une \'equivalence de Morita $G_1$-\'equivariante. Par d\'efinition de la coaction $\delta_{J_1}$, le *-morphisme canonique $\iota_{B_1} : B_1 \rightarrow  J_1$ v\'erifie les hypoth\`eses de \ref{indu1}. 
\hfill\break        
Posons  
$J_2 = {\rm Ind}_{G_1}^{G_2} (J_1, \delta_{J_1})$. D'apr\`es  \ref{indu2}, on a une \'equivalence de Morita $G_2$-\'equivariante 
$(J_2, \delta_{J_2}, e_{l, 2})$  avec $e_{l,2}:= e_{l,1} \otimes 1_{S_{12}} \in M(J_2)$, $l=1,2$. En particulier $e_{2, 2} J_2 e_{2, 2}$, munie de la restriction de la coaction $\delta_{J_2}$, est une $G_2$-alg\`ebre. 

Nous avons :

\begin{proposition} Posons
$B_2 = {\rm Ind}_{G_1}^{G_2} (B_1, \delta_{B_1})$. Alors ${\rm Ind}_{G_1}^{G_2} \iota_{B_1} : B_2 \rightarrow e_{2, 2} J_2 e_{2, 2}$ est un *-isomorphis\-me $G_2$-\'equivariant.
\end{proposition} 
\begin{proof}
Comme $ \iota_{B_1}$ prend ses valeurs dans $J_1$, le  *-morphisme injectif
${\rm Ind}_{G_1}^{G_2} \iota_{B_1}$ est \`a valeurs dans  
$$e_{2, 2} J_2 e_{2, 2} =
 [({\rm id}_{J_1 \otimes S_{12}}  \otimes \omega)({\rm id}_{J_1} \otimes \delta_{11}^2)\delta_{J_1}(\iota_{B_1}(b))\, |\, b\in B_1,\, \omega \in B(H_{21})_\ast ].$$
\noindent
Il en r\'esulte que  ${\rm Ind}_{G_1}^{G_2} \iota_{B_1}$ est un *-isomorphisme. La $G_2$-\'equivariance r\'esulte facilement  de (\ref{indu3} (ii)).
\end{proof}
\noindent
Dans la suite, nous identifions les $G_2$-alg\`ebres $B_2 = {\rm Ind}_{G_1}^{G_2} e_{2, 1} J_1 e_{2, 1}$ et $e_{2, 2}J_2 e_{2, 2}= 
 e_{2, 2}({\rm Ind}_{G_1}^{G_2} J_1) e_{2, 2}$
\begin{definition} On appelle module induit du $B_1$-module hilbertien $G_1$-\'equivariant ${\cal E}_1$, le $B_2$-module hilbertien $G_2$-\'equivariant $ {\rm Ind}_{G_1}^{G_2} {\cal E}_1 := 
e_{1, 2} J_2 e_{2, 2} = e_{1, 2}({\rm Ind}_{G_1}^{G_2} J_1) e_{2, 2}$.\index{i@${\rm Ind}_{G_1}^{G_2} {\cal E}_1$}
\end{definition}
\noindent
%Par d\'efinition, $ {\rm Ind}_{G_1}^{G_2} {\cal E}_1$ est un $B_2$- module hilbertien $G_2$-\'equivariant?

\noindent
\subsection{Induction et bidualit\'e }

\noindent
Soit $(A_1,  \delta_{A_1})$ une $G_1$-alg\`ebre. Posons $(A_2 , \delta_{A_2}):=  {\rm Ind}_{G_1}^{G_2}(A_1,  \delta_{A_1})$. Dans ce paragrahe,   nous montrons  que la $G_2$-alg\`ebre ${\rm Ind}_{G_1}^{G_2}(A_1 \rtimes G_1 \rtimes \widehat{G_1})$ est $G_2$-Morita \'equivalente \`a la $G_2$-alg\`ebre $A_2 \rtimes G_2 \rtimes \widehat{G_2}$. Ce r\'esultat est obtenu par application du th\'eor\`eme \ref{thbidu}  au double  produit  crois\'e  $A \rtimes {\cal G} \rtimes \widehat{\cal G}$, o\`u $(A, \beta_A, \delta_A)$ est l'image  de 
$(A_1, \delta_{A_1})$ par la correspondance  \ref{reccorresp}  $A^{G_1} \rightarrow A^{\cal G}$. 
 \hfill\break
 Nous examinerons ensuite le cas o\`u $A_1 = J_1$ est une $G_1$-alg\`ebre de liaison.

\noindent
Fixons une  $G_1$-alg\`ebre $(A_1,  \delta_{A_1})$.    Posons $(A_2,  \delta_{A_2}) : = {\rm Ind}_{G_1}^{G_2}(A_1,  \delta_{A_1})$. Rappelons que $A:= A_1 \oplus A_2$ est munie de l'action   $(\beta_A, \delta_A)$ d\'efinie dans \ref{reccorresp}.
\hfill\break
 Identifions alors les ${\cal G}$-alg\`ebres $A  \rtimes {\cal G}  \rtimes \widehat{\cal G}$ et   $(D , \beta_D , \delta_D)$ comme dans \ref{thbidu} et reprenons int\'egralement les notations  du paragraphe \ref{dpco} et en particulier les notations \ref{not4}.
 
\begin{rappels}
Soient $j,l,l'=1,2$ avec $l\not=l'$. Par identification des bimodules hilbertiens \'equivariants $D_{ll',j} :=e_{l,j}D_j e_{l',j}$ et 
${\cal B}_{ll',j} := A_j \otimes {\cal K}(H_{l'j}, H_{lj})$, nous savons \ref{etap4} que :
\begin{enumerate}  
\item  le    
$A_j \otimes {\cal K}(H_{lj})-A_j \otimes {\cal K}(H_{l'j})$  bimodule hilbertien $G_j$-\'equivariant $A_j \otimes {\cal K}(H_{l'j}, H_{lj})$ est une \'equivalence de Morita $G_j$-\'equivariante des $G_j$-alg\`ebres  $A_j \otimes {\cal K}(H_{lj})$ et $A_j \otimes {\cal K}(H_{l'j})$ ;
\item  la coaction de $A_j \otimes {\cal K}(H_{jj})$ co\"incide avec la coaction biduale du double produit crois\'e $A_j \rtimes G_j    \rtimes \widehat{G_j}$, apr\`es identification  \cite{BaSka2}  avec  $A_j \otimes {\cal K}(H_{jj})$ ; 
\item  ${\cal E}_{ll',k}^j := \delta_{{\cal B}_{ll',j}}^k({\cal B}_{ll',j})$, muni de la coaction 
$$\delta_{ll',k}^j : \xi  \mapsto  V_{jj, 23}^k \xi_{12}  (V_{jj, 23}^k)^* $$
\noindent
est un $\delta_{{\cal B}_{l,j}}^k({\cal B}_{j, l}) - \delta_{{\cal B}_{j,l'}}^k({\cal B}_{j, l'})$ bimodule hilbertien $G_j$-\'equivariant.
\end{enumerate}
\end{rappels}

\begin{theorem}\label{thindbimod}  Soient  $j , k , l , l' = 1 , 2\,,\, j\not= k$. 
\begin{enumerate}
\item  ${\cal E}_{ll',k}^j := \delta_{{\cal B}_{ll',j}}^k({\cal B}_{ll',j}) = {\rm Ind}_{G_k}^{G_j} {\cal B}_{ll',k}$. En particulier, 
$\delta_{{\cal B}_{l,j}}^k({\cal B}_{l,j})  = {\rm Ind}_{G_k}^{G_j} {\cal B}_{l,k}$.
\item La coaction $\delta_{ll',k}^j$ du bimodule ${\cal E}_{ll',k}^j $ co\"incide avec la coaction induite par celle de ${\cal B}_{ll',k}$.
\item Le morphisme ${\cal B}_{ll',j}  \rightarrow {\rm Ind}_{G_k}^{G_j} {\cal B}_{ll',k}    : \xi \mapsto \delta_{{\cal B}_{ll',j}}^k(\xi)$ est un isomorphisme de bimodules $G_j$-\'equivariants au dessus des *-isomorphismes 
 $$\delta_{{\cal B}_{l,j}}^k :   {\cal B}_{l,j}  \rightarrow {\rm Ind}_{G_k}^{G_j }{\cal B}_{l,k} \quad \text{et} \quad \delta_{{\cal B}_{l',j}}^k :   {\cal B}_{l',j}  \rightarrow {\rm Ind}_{G_k}^{G_j }{\cal B}_{l',k}. $$
\end{enumerate}
\end{theorem}

\begin{proof}  Par    restriction de l'isomorphisme \ref{isoind}  $G_j$-\'equivariant 
$$\widetilde\pi_j : D_j \rightarrow  {\rm Ind}_{G_k}^{G_j}  D_k : x \mapsto \delta_{D_j}^k(x),$$
\noindent
on obtient   des  isomorphismes  
$$ D_{ll',j}  \rightarrow  {\rm Ind}_{G_k}^{G_j}  D_{ll',k},\quad l,l'=1,2,$$
\noindent
de bimodules hilbertiens  $G_j$-\'equivariants au dessus des isomorphismes $G_j$-\'equivariants 
$$D_{l,j}   \rightarrow {\rm Ind}_{G_k}^{G_j}  D_{l,k}, \quad D_{l',j}   \rightarrow {\rm Ind}_{G_k}^{G_j}  D_{l',k},\quad l,l'=1,2. $$
En effet, on a  $\delta_{D_j}^k(e_{l,j}) = e_{l,k} \otimes 1_{S_{kj}}$ (\cf\ref{liaison1} c)). Il en r\'esulte    qu'on a :
$$\delta_{D_j}^k(D_{ll',j}) = (e_{l,k} \otimes 1_{S_{kj}})({\rm Ind}_{G_k}^{G_j}  D_k) (e_{l',k} \otimes 1_{S_{kj}}) =
{\rm Ind}_{G_k}^{G_j} D_{ll',k}, \quad \delta_{D_j}^k(D_{l,j})  = {\rm Ind}_{G_k}^{G_j}  D_{l,k}. $$

\noindent
Par identification des $G_j$-alg\`ebres $D_j$ et ${\cal B}_j := A_j \otimes {\cal K}(H_{1j} \oplus H_{2j})$, on d\'eduit le a).
\hfill\break
Il est clair que la coaction $\delta_{ll',k}^j : \xi  \mapsto  V_{jj, 23}^k \xi_{12}  (V_{jj, 23}^k)^* $ du ${\rm Ind}_{G_k}^{G_j} {\cal B}_{l,k}$-module  
${\cal E}_{ll',k}^j$ est la coaction induite du   ${\cal B}_{l',k}$-module ${\cal B}_{ll',k}$.
\hfill\break
Le c) r\'esulte de  a) , b) et (\ref{indbimod} b)).
\end{proof}
\noindent
\begin{corollary}\label{ind+iso} Pour $j ,k , l , l'= 1,2$ avec $j\not=k$, nous avons :
\begin{enumerate}
\item  Les $G_j$-alg\`ebres  $A_j \otimes {\cal K}(H_{lj})$ et 
${\rm Ind}_{G_k}^{G_j}  A_k \otimes {\cal K}(H_{lk})$   sont canoniquement isomorphes. En particulier, la 
$G_j$-alg\`ebre  $A_j \otimes {\cal K}(H_{jj})$ est canoniquement isomorphe \`a la $G_j$-alg\`ebre
${\rm Ind}_{G_k}^{G_j}  A_k \otimes {\cal K}(H_{jk})$.
\item  Les $A_j \otimes {\cal K}(H_{lj})-A_j \otimes {\cal K}(H_{l'j})$ bimodules hilbertiens $G_j$-\'equivariants ${\rm Ind}_{G_k}^{G_j}  A_k \otimes {\cal K}(H_{l'k} ,H_{lk})$ et $A_j \otimes {\cal K}(H_{l'j} ,H_{lj})$ sont canoniquement isomorphes. 
\item  La $G_2$-alg\`ebre $A_2  \rtimes G_2  \rtimes \widehat{G_2}$ est $G_2$-Morita \'equivalente \`a la 
$G_2$-alg\`ebre ${\rm Ind}_{G_1}^{G_2} A_1  \rtimes G_1  \rtimes \widehat{G_1}$.
\end{enumerate}
\end{corollary}

\begin{proof} Les  \'enonc\'es  a) et   b) d\'ecoulent de l'\'enonc\'e (plus g\'en\'eral) (\ref{thindbimod} c)). 
\hfill\break
En prenant $j = l =2$ et $l' = 1$ dans (\ref{etap5} c)), on obtient une $G_2$-\'equivalence de Morita des $G_2$-alg\`ebres 
$A_2 \otimes {\cal K}(H_{22})$  et $A_2 \otimes {\cal K}(H_{12})$.
\hfill\break
En prenant $j =2 , l = k = 1$ dans le a),  on d\'eduit que les   $G_2$ -alg\`ebres  $A_2 \otimes {\cal K}(H_{12})$ et ${\rm Ind}_{G_1}^{G_2}  A_1 \otimes {\cal K}(H_{11})$ sont canoniquement isomorphes. 
\end{proof}

\noindent
{\bf Cas d'une alg\`ebre de liaison}

\noindent
Soit $(J_1, \delta_{J_1}^1, e_1^1, e_1^2)$ une $G_1$-alg\`ebre de liaison. Posons 
$(J_2 , \delta_{J_2}^2) = {\rm Ind}_{G_1}^{G_2} (J_1, \delta_{J_1}^1)$, qu'on munit de sa structure de $G_2$-alg\`ebre de liaison d\'efinie  \ref{indu2} par les projecteurs :
$$e_2^1 := e_1^1 \otimes 1_{S_{12}}, \quad e_2^2 := e_1^2 \otimes 1_{S_{12}}.$$
\noindent
Soit $J= (J_1 \oplus J_2,  \beta_J, \delta_J)$  la ${\cal G}$-alg\`ebre de liaison  \ref{indugr}  correspondant \`a $J_1$ et $J_2$, d\'efinie par    les deux projecteurs de $M(J)$ :
$$e^1 := (e_1^1, e_2^1), \quad e^2 := (e_1^2, e_2^2).$$
\noindent
Pour expliciter la structure de ${\cal G}$-alg\`ebre de liaison du double produit crois\'e $J  \rtimes {\cal G}  \rtimes \widehat{\cal G} $ {\it correspondant aux projecteurs $e^1 , e^2$}, nous reprenons les notations du paragraphe \ref{dpco},  {\it en rempla\c cant dans toutes ces notations $A$ par $J$.}

Comme dans \ref{thbidu}, on  identifie $J  \rtimes {\cal G}  \rtimes \widehat{\cal G} $ \`a  la C*-alg\`ebre $D \subset {\cal L}(J \otimes H)$.
\begin{lemme}\label{esD} Pour $i = 1 , 2$, il existe un unique projecteur $e_D^i\in M(D)$\index{el@$e_D^i$} v\'erifiant 
$$j_D(e_D^i) = q_{\beta_J, \widehat\alpha}(e^i \otimes 1_H)\in {\cal L}(J \otimes H) \quad \quad (\widehat\alpha = \beta).$$
\noindent
On a : $e_D^1 + e_D^2 = 1_D \quad ;\quad [D e_D^i D]  =D, \,\,i = 1 , 2  \quad ; \quad \delta_D(e_D^i) = \delta_D(1_D)(e_D^i \otimes 1_S)$.
\end{lemme}
\begin{proof} 
Rappelons qu'on a  
$$D  =  [ ({\rm id}_J \otimes R)\delta_J(a)(1_J \otimes \lambda(x) L(s)) \,\,|\,\, a\in J , x\in \widehat S , s\in S ]  = D_1 \oplus D_2 \subset {\cal L}(J \otimes H).$$\noindent
Il est clair que $q_{\beta_J, \widehat\alpha}(e^i \otimes 1_H)\in j_D(M(D))$, o\`u :
$$j_D(M(D)) = \{T\in {\cal L}(A \otimes H) \,|\, T D \subset D \,, \, D T \subset D \, , \,T j_D(1_D) = 
j_D(1_D) T = T\} \quad ; \quad  j_D(1_D) =q_{\beta_J, \widehat\alpha}. $$
\noindent
On en d\'eduit l'existence des projecteurs $e_D^i$, $i=1,2$.
 Les assertions restantes du lemme se d\'eduisent de la structure de ${\cal G}$-alg\`ebre de liaison de $(J,  \beta_J, \delta_J, e^1, e^2)$.
\end{proof}
\noindent
Il apparait clairement que la ${\cal G}$-alg\`ebre $D$ a en fait deux structures d'alg\`ebres de liaison compatibles : celle d\'efinie par les projecteurs $e_D^i$ et celle qui correspond aux projecteurs \ref{liaison1}  $e_{l, j}$ de $D_1$ et $D_2$ : 

\noindent
\begin{lemme}\label{eslj}Pour tout $s , j , l  = 1 , 2$, notons $e_{l,j}^s \in M(D_j) $\index{em@$e_{l,j}^s$}, l'unique projecteur v\'erifiant  
$$\iota_j(e_{l,j}^s) = e_D^s \iota_j(e_{l,j}), \quad \text{o\`u} \quad 
\iota_j : M(D_j)  \rightarrow M(D),\,\, \iota_j(1_{D_j}) = \beta_D(\varepsilon_j).$$
\noindent
Pour $s , j , k ,l  = 1 , 2$, on a $\delta_{D_j}^k(e_{l,j}^s) = e_{l,k}^s \otimes 1_{S_{kj}}$
(pour la  d\'efinition des projecteurs $e_{l,j}$, \cf\ref{liaison1}).
\end{lemme}
\begin{proof} La d\'efinition \ref{esD} des projecteurs $e_D^s$ entra\^ine clairement qu'on a 
$$[e_D^s , \iota_j(e_{l,j})]  = 0, \quad e_D^s \iota_j(e_{l,j})\beta_D(\varepsilon_j) =   e_D^s \iota_j(e_{l,j}),\quad j,l=1,2,$$
\noindent
d'o\`u  l'existence et l'unicit\'e des projecteurs $e_{l,j}^s$.

Notons maintenant par $\pi_j^k : M(D_k \otimes S_{kj}) \rightarrow M(D \otimes S)$ le *-morphisme canonique correspondant \`a l'inclusion $D_k \otimes S_{kj}  \subset D \otimes S$.  L'\'egalit\'e \`a \'etablir est \'equivalente \`a :
$$(j_D \otimes {\rm id}_S) \pi_j^k \delta_{D_j}^k(e_{l,j}^s) = (j_D \otimes {\rm id}_S) \pi_j^k(e_{l,k}^s \otimes 1_{S_{kj}}).$$
\noindent
Nous avons 
\begin{align} (j_D \otimes {\rm id}_S) \pi_j^k \delta_{D_j}^k(e_{l,j}^s) &= (j_D \otimes {\rm id}_S) (\delta_D(\iota_j(e_{l,j}^s)) (\beta_D(\varepsilon_k) \otimes p_{kj}))&\nonumber\\
& = (j_D \otimes {\rm id}_S)  (\delta_D( e_D^s \iota_j(e_{l,j})) (\beta_D(\varepsilon_k) \otimes p_{kj})) &\nonumber \\
&  = (j_D \otimes {\rm id}_S)  (\delta_D(1_D)(e_D^s \otimes 1_S)  \delta_D(\iota_j(e_{l,j}))(\beta_D(\varepsilon_k) \otimes p_{kj})) & (\ref{esD})  \nonumber  \\
 & = (j_D \otimes {\rm id}_S)  (\delta_D(1_D)(\beta_D(\varepsilon_k) \otimes p_{kj})(e_D^s \otimes 1_S)  \delta_D(\iota_j(e_{l,j}))(\beta_D(\varepsilon_k) \otimes   p_{kj})) \nonumber  &\\
& = (j_D \otimes {\rm id}_S)((\beta_D(\varepsilon_k) \otimes p_{kj})(e_D^s \otimes 1_S))  \pi_j^k (\delta_{D_j}^k(e_{l,j})) 
\nonumber & \\
 & = 
q_{\beta_J, \widehat\alpha, 12}(1_J \otimes \beta(\varepsilon_k) \otimes p_{kj})  (e^s \otimes 1_H \otimes 1_S)(q_k \otimes p_{lk}  \otimes p_{kj}).\nonumber & (\ref{pikj}) 
\end{align}
Par ailleurs, nous avons
\begin{align}
(j_D \otimes {\rm id}_S) \pi_j^k(e_{l,k}^s \otimes 1_{S_{kj}}) &= (j_D \otimes {\rm id}_S)((\iota_k(e_{l,k}^s)\otimes 1_S) (\beta_D(\varepsilon_k) \otimes p_{kj})) & \nonumber \\
& = (j_D \otimes {\rm id}_S)((e_D^s \iota_k(e_{l,k}) \otimes 1_S) (\beta_D(\varepsilon_k) \otimes p_{kj})) &\nonumber \\
& = 
q_{\beta_J, \widehat\alpha, 12} (e^s \otimes 1_H \otimes 1_S)(q_k \otimes p_{lk} \otimes 1_S) q_{\beta_J, \widehat\alpha, 12}(1_J \otimes \beta(\varepsilon_k) \otimes p_{kj}) & (\ref{liaison1} a))  \nonumber \\
& = q_{\beta_J, \widehat\alpha, 12}(1_J \otimes \beta(\varepsilon_k) \otimes p_{kj})  (e^s \otimes 1_H \otimes 1_S)(q_k \otimes p_{lk}  \otimes p_{kj}).\nonumber 
\end{align}
\end{proof}
\noindent
\begin{notations} Pour $j , l , l' , s , s' = 1, 2$, posons :\index{dl@$D_{ll', j}^{ss'}$, $D_{l,j}^{s}$}
$$D_{ll', j}^{ss'} :=  e_{l,j}^s D_j e_{l',j}^{s'}, \quad  D_{l,j}^{s} := D_{ll,j}^{ss}.$$ 
\noindent
\end{notations}
On d\'eduit de \ref{eslj} que par  restriction de la structure de $G_j$-alg\`ebre de $D_j$, on obtient que 
 $D_{ll',j}^{ss'}$ est  un 
$D_{l,j}^{s}-D_{l',j}^{s'}$ bimodule hilbertien  $G_j$-\'equivariant.

\noindent
\begin{proposition}\label{isoindF} Soient  $j , k = 1 , 2,\, j\not= k$. Par restriction de l'isomorphisme \ref{isoind} $G_j$-\'equivariant 
$$\widetilde\pi_j : D_j \rightarrow  {\rm Ind}_{G_k}^{G_j}  D_k : x \mapsto \delta_{D_j}^k(x),$$
\noindent
on obtient des  isomorphismes  : 
$$ D_{ll',j}^{ss'}  \rightarrow {\rm Ind}_{G_k}^{G_j}  D_{ll',k}^{ss'},\quad l,l',s,s'=1,2,$$
\noindent
de bimodules hilbertiens  $G_j$-\'equivariants au dessus des isomorphismes $G_j$-\'equivariants :
$$D_{l,j}^{s}  \rightarrow {\rm Ind}_{G_k}^{G_j}  D_{l,k}^{s}, \quad D_{l',j}^{s'}  \rightarrow {\rm Ind}_{G_k}^{G_j}  D_{l',k}^{s'},\quad l,l',s,s'=1,2.$$
\end{proposition}

\begin{proof} 
Il r\'esulte   \ref{eslj} des  \'egalit\'es     $\delta_{D_j}^k(e_{l,j}^s) = e_{l,k}^s \otimes 1_{S_{kj}}$, qu'on a :
$$\delta_{D_j}^k(e_{l,j}^sD_j e_{l',j}^{s'}) = (e_{l,k}^s \otimes 1_{S_{kj}})({\rm Ind}_{G_k}^{G_j}  D_k) (e_{l',k}^{s'} \otimes 1_{S_{kj}}) =
{\rm Ind}_{G_k}^{G_j} e_{l,k}^s D_k e_{l',k}^{s'}  , \quad \delta_{D_j}^k(D_{l,j}^{s})  = {\rm Ind}_{G_k}^{G_j}  D_{l,k}^{s}. $$
\noindent
\end{proof}
\noindent
Nous allons appliquer \ref{isoindF} \`a l'alg\`ebre de liaison correspondant \`a un module hilbertien \'equivariant, \cf\ref{dbimod}.

Soient $(A_1, \delta_{A_1})$ et $(B_1, \delta_{B_1})$ deux $G_1$-alg\`ebres et $({\cal E}_1, \delta_{{\cal E}_1})$ un $A_1-B_1$ bimodule hilbertien $G_1$-\'equivariant.  Posons :
$$A_2 =  {\rm Ind}_{G_1}^{G_2} (A_1, \delta_{A_1}), \quad B_2 =  {\rm Ind}_{G_1}^{G_2} (B_1, \delta_{B_1}), \quad  
{\cal E}_2 =  {\rm Ind}_{G_1}^{G_2} ({\cal E}_1, \delta_{{\cal E}_1}). $$
\noindent
Prenons pour $J_1$, la $G_1$-alg\`ebre 
$$J_1 = \begin{pmatrix} {\cal K}({\cal E}_1) & {\cal E}_1 \\
{\cal E}_1^* &  B_1\end{pmatrix},$$ 
qu'on munit de la coaction $\delta_{J_1}$ qui d\'efinit la structure de $B_1$-module $G_1$-\'equivariant 
de ${\cal E}_1$, \cf\ref{dbimod}. Alors,
$$J_2 := {\rm Ind}_{G_1}^{G_2}  J_1 = \begin{pmatrix}{\cal K}({\cal E}_2) & {\cal E}_2 \\
{\cal E}_2^* &  B_2\end{pmatrix}$$ 
est une $G_2$-alg\`ebre de liaison.
\hfill\break
Consid\'erons alors la ${\cal G}$-alg\`ebre de liaison $J := J_1\oplus J_2$  d\'efinie par $J_1$ et $J_2$, \cf\ref{indugr}. 
Identifions   les ${\cal G}$-alg\`ebres $J  \rtimes {\cal G}  \rtimes \widehat{\cal G}$ et   $(D , \beta_D , \delta_D)$ comme dans \ref{thbidu} et reprenons int\'egralement les notations  du paragraphe \ref{dpco} et en particulier les notations \ref{not4}.

Pour $j = 1 ,2$ et  par  l'identification $D_j = J_j \otimes {\cal K}(H_{1j} \oplus H_{2j})$, observons que  pour  $l , l' = 1 ,2$,  on a :
\hfill\break
1) En prenant $s=1$, la $G_j$-alg\`ebre $D_{l,j}^{1}$ s'identifie  \`a la $G_j$-alg\`ebre ${\cal K}({\cal E}_j) \otimes {\cal K}(H_{lj})$, dont la coaction est donn\'ee par la formule :
$$x \mapsto  V_{jj, 23}^l  ({\rm id}_{{\cal K}({\cal E}_j)} \otimes \sigma)(\delta_{{\cal K}({\cal E}_j)}^j  \otimes {\rm id}_{ {\cal K}(H_{lj})})(x) (V_{jj, 23}^{l})^*, \quad  x\in {\cal K}({\cal E}_j) \otimes {\cal K}(H_{lj}).$$
\noindent
2) En prenant $s=2$, la $G_j$-alg\`ebre $D_{l,j}^{2}$ s'identifie  \`a la $G_j$-alg\`ebre $B_j \otimes {\cal K}(H_{lj})$, dont la coaction est donn\'ee par la formule :
$$x \mapsto  V_{jj, 23}^l  ({\rm id}_{B_J} \otimes \sigma)(\delta_{B_j)}^j  \otimes {\rm id}_{ {\cal K}(H_{lj})})(x) (V_{jj, 23}^{l})^*, \quad  x\in B_j \otimes {\cal K}(H_{lj}).$$
\noindent
3) En prenant $s=1$ et $s'=2$, le $D_{l,j}^{1}-D_{l',j}^{2}$ bimodule hilbertien $G_j$-\'equivariant  $D_{ll',j}^{12}$
s'identifie au ${\cal K}({\cal E}_j) \otimes {\cal K}(H_{lj})- B_j \otimes {\cal K}(H_{l'j})$ bimodule hilbertien  $G_j$-\'equivariant  ${\cal E}_j \otimes {\cal K}(H_{l'j} ,H_{lj})$, dont la coaction 
est donn\'ee par la formule :
$$\xi \mapsto V_{jj, 23}^l  ({\rm id}_{{\cal E}_j} \otimes \sigma)(\delta_{{\cal E}_j}^j  \otimes {\rm id}_{ {\cal K}(H_{l'j}, H_{lj})})(\xi) (V_{jj, 23}^{l'})^*, \quad  \xi\in {\cal E}_j \otimes {\cal K}(H_{l'j} ,H_{lj}). $$
\noindent
(Tous les $\sigma$ sont des voltes \'evidentes.)

Nous avons :

\begin{theorem}\label{thliaison} Soit ${\cal E}_1$ un $A_1-B_1$ bimodule hilbertien $G_1$-\'equivariant dont l'action  \`a gauche $A_1 \rightarrow {\cal L}({\cal E}_1)$ est suppos\'ee non d\'eg\'en\'er\'ee.
Posons :
$$A_2 = {\rm Ind}_{G_1}^{G_2} A_1, \quad  B_2 =  {\rm Ind}_{G_1}^{G_2} B_1, \quad {\cal E}_2 =  {\rm Ind}_{G_1}^{G_2} {\cal E}_1.$$  
\noindent
Pour $j , k , l , l' = 1 , 2$ avec $j \not= k$, le morphisme 
$$\delta_{ll',j}^k  : {\cal E}_j \otimes {\cal K}(H_{l'j} ,H_{lj})  \rightarrow {\rm Ind}_{G_k}^{G_j}  {\cal E}_k \otimes {\cal K}(H_{l'k} ,H_{lk}) : 
\xi \mapsto    V_{kj, 23}^l  ({\rm id}_{{\cal E}_k} \otimes \sigma)(\delta_{{\cal E}_j}^k  \otimes {\rm id}_{ {\cal K}(H_{l'j}, H_{lj})})(\xi) (V_{kj, 23}^{l'})^*$$
est un isomorphisme de bimodules hilbertiens $G_j$-\'equivariants au dessus des *-isomorphismes :
$$A_j \otimes {\cal K}(H_{lj}) \rightarrow {\rm Ind}_{G_k}^{G_j}  A_k \otimes {\cal K}(H_{lk}) : x \mapsto  
V_{kj, 23}^{l}  ({\rm id}_{A_k} \otimes \sigma)(\delta_{A_j}^k  \otimes {\rm id}_{{\cal K}(H_{lj})}(x) (V_{kj, 23}^{l})^*,$$
\noindent
$$ B_j \otimes {\cal K}(H_{l'j}) \rightarrow {\rm Ind}_{G_k}^{G_j}  B_k \otimes {\cal K}(H_{l'k}) : x \mapsto  
V_{kj, 23}^{l'}  ({\rm id}_{B_k} \otimes \sigma)(\delta_{B_j}^k  \otimes {\rm id}_{{\cal K}(H_{l'j})}(x) (V_{kj, 23}^{l'})^*.$$
\end{theorem}
\begin{proof}  Il r\'esulte   de \ref{isoindF} que les assertions   qu'on obtient,  en rempla\c cant dans celles    du th\'eor\`eme,  la $G_j$-alg\`ebre $A_j$ par la   $G_j$-alg\`ebre ${\cal K}({\cal E}_j)$, sont \'etablies.
\hfill\break
L'\'equivariance de l'action \`a gauche $A_1 \rightarrow {\cal L}({\cal E}_1)$ (\cf\ref{dbimod}), d\'efinit   une action  \ref{ind3}  non d\'egen\'er\'ee $G_2$-\'equivariante $A_2 \rightarrow {\cal L}({\cal E}_2)$. Il en r\'esulte qu'on a un *-morphisme ${\cal  G}$-\'equivariant 
$A = A_1 \oplus A_2  \rightarrow M(J)$. \hfill\break
Comme dans le cas d'une action  continue d'un groupe quantique \lc dans une C*-alg\`ebre, on d\'eduit un *-morphisme ${\cal  G}$-\'equivariant   $A  \rtimes {\cal G}  \rtimes \widehat{\cal G} \rightarrow M(J  \rtimes {\cal G}  \rtimes \widehat{\cal G})$.
\hfill\break
Pour $j = 1 , 2$, posons ${\cal A}_j = A_j \otimes {\cal K}(H_{1j} \oplus H_{2j})$. 
En identifiant \ref{dpco}  les ${\cal G}$-alg\`ebres $A  \rtimes {\cal G}  \rtimes \widehat{\cal G}$ et 
${\cal A}_1 \oplus {\cal A}_2$, on obtient pour tout $j  = 1, 2$, un *-morphisme 
\begin{equation}\label{eq:morphf1}f_j : {\cal A}_j  \rightarrow M({\cal B}_j), \quad {\cal B}_j := J_j \otimes {\cal K}(H_{1j} \oplus H_{2j}),
\end{equation}
\noindent
et pour tout $j , k$, on a 
\begin{equation}\label{eq:morphf2}(f_k \otimes {\rm id}_{S_{kj}}) \circ \delta_{{\cal A}_j}^k = \delta_{{\cal B}_j}^k \circ f_j.
\end{equation}
Par restriction de $f_j$, on obtient pour $l , l'= 1 , 2$, un  *-morphisme  non d\'eg\'en\'er\'e et $G_j$-\'equivariant :   
$$A_j \otimes {\cal K}(H_{lj}) \rightarrow {\cal L}({\cal E}_j \otimes {\cal K}(H_{l'j} ,H_{lj})).$$
\noindent
%Dans le cas $l = l'$, on obtient un *-morphisme  non d\'eg\'en\'er\'e et $G_j$-\'equivariant :
%$$\pi_{l,j} : A_j \otimes {\cal K}(H_{lj}) \rightarrow {\cal L}({\cal E}_j \otimes  {\cal K}(H_{lj}))$$
L'\'equivariance de l'isomorphisme $\delta_{ll',j}^k  : {\cal E}_j \otimes {\cal K}(H_{l'j} ,H_{lj})  \rightarrow {\rm Ind}_{G_k}^{G_j}  {\cal E}_k \otimes {\cal K}(H_{l'k} ,H_{lk})$ r\'esulte alors de \ref{isoindF} et de  (\ref{eq:morphf2}).
 \end{proof}
\noindent
Pour d\'efinir  une  \'equivalence des cat\'egories $KK^{G_1}$ et $KK^{G_2}$, nous avons besoin d'expliciter quelques notations suppl\'ementaires       et d'\'etablir un lemme utile.
\begin{notations}\label{notR}
\hspace{2em}
\begin{enumerate}
\item  Pour tout $j , k , l, l' = 1 , 2$ avec $j\not=k$,  on a :
$${\cal B}_{l,j} :=  J_j \otimes {\cal K}(H_{lj}) =   \begin{pmatrix} {\cal K}({\cal E}_j) \otimes {\cal K}(H_{lj}) & {\cal E}_j \otimes {\cal K}(H_{lj}) \\
{\cal E}_j \otimes {\cal K}(H_{lj}) ^* & B_j \otimes      {\cal K}(H_{lj})  \end{pmatrix}.$$
\item  Le ${\cal B}_{l,j}-{\cal B}_{l',j}$ bimodule hilbertien  ${\cal B}_{ll',j}$ est de la forme : 
$${\cal B}_{ll',j} = \begin{pmatrix}{\cal K}({\cal E}_j) \otimes {\cal K}(H_{l'j}, H_{lj}) & {\cal E}_j \otimes {\cal K}(H_{l'j}, H_{lj}) \\
{\cal E}_j \otimes {\cal K}(H_{l'j}, H_{lj}) ^* & B_j \otimes      {\cal K}(H_{l'j}, H_{lj})  \end{pmatrix}.$$
\item Par restriction de $\delta_{{\cal B}_j}^k$, on a  les isomorphismes $G_j$-\'equivariants (\ref{isoind}, \ref{not4}) :
$$\delta_{ll',j}^k : {\cal K}({\cal E}_j) \otimes {\cal K}(H_{l'j}, H_{lj}) \rightarrow {\rm Ind}_{G_k}^{G_j} {\cal K}({\cal E}_k) \otimes {\cal K}(H_{l'k}, H_{lk})  :
x \mapsto  V_{kj, 23}^l ({\rm id}  \otimes \sigma)(\delta_{{\cal K}({\cal E}_j)}^k \otimes  {\rm id})(x) (V_{kj, 23}^{l'})^*,$$
$$\delta_{ll',j}^k : {\cal E}_j \otimes {\cal K}(H_{l'j}, H_{lj}) \rightarrow {\rm Ind}_{G_k}^{G_j} {\cal E}_k \otimes {\cal K}(H_{l'k}, H_{lk})  :
\xi \mapsto  V_{kj, 23}^l ({\rm id}  \otimes \sigma)(\delta_{{\cal E}_j}^k \otimes  {\rm id})(\xi) (V_{kj, 23}^{l'})^*,$$
$$\delta_{ll',j}^k : B_j \otimes {\cal K}(H_{l'j}, H_{lj}) \rightarrow {\rm Ind}_{G_k}^{G_j} B_k \otimes {\cal K}(H_{l'k}, H_{lk})  :
x \mapsto  V_{kj, 23}^l ({\rm id}  \otimes \sigma)(\delta_{B_j}^k \otimes  {\rm id})(x) (V_{kj, 23}^{l'})^*.$$
\end{enumerate}
\end{notations}

\begin{lemme}\label{operF} Pour tout $j , k , l = 1 , 2$ et tout $T \in M({\cal K}({\cal E}_l))$, nous avons 
$$\delta_{ll,j}^k({\rm id}_{{\cal K}({\cal E}_j)} \otimes  R_{jl})\delta_{{\cal K}({\cal E}_l)}^j(T) = ({\rm id}_{{\cal K}({\cal E}_k)} \otimes  R_{kl})\delta_{{\cal K}({\cal E}_l)}^k(T) \otimes 1_{S_{kj}}.$$
\end{lemme}
\begin{proof} Il s'agit d'un cas particulier  de (\ref{bidu0} b))  appliqu\'ee \`a l'action du groupo\"ide ${\cal G}$ dans la C*-alg\`ebre $J$.
\end{proof}

\noindent
\subsection{Application \`a la \texorpdfstring{$KK$}{KK}-th\'eorie \'equivariante}

\noindent
Soient $G_1$ et $G_2$ deux groupes quantiques l.c., mono\"idalement \'equivalents et r\'eguliers.  Fixons  $(A_1, \delta_{A_1})$ et 
$(B_1, \delta_{B_1})$   deux $G_1$-alg\`ebres et  posons :
$$(A_2 ,\delta_{A_2}) = {\rm Ind}_{G_1}^{G_2}\,(A_1, \delta_{A_1}), \quad (B_2  , \delta_{B_2}) = {\rm Ind}_{G_1}^{G_2}\,(B_1, \delta_{B_1})$$
\noindent
\`A tout $A_1-B_1$ bimodule de Kasparov $({\cal E}_1, F_1)$ nous   associons  un  $A_2-B_2$ bimodule de Kasparov $({\cal E}_2, F_2)$, ce qui nous permet  de d\'efinir un isomorphisme de groupes ab\'eliens   
$$J_{G_2, G_1} : KK^{G_1}(A_1, B_1)  \rightarrow KK^{G_2}(A_2, B_2) : x= [({\cal E}_1, F_1)] \mapsto J_{G_2, G_1}(x) := [({\cal E}_2, F_2)]$$
\noindent
dont l'isomorphisme r\'eciproque  $J_{G_1, G_2} : KK^{G_2}(A_2,  B_2)  \rightarrow KK1^{G_1}(A_1 B_1)$ est obtenu  de la m\^eme mani\`ere en permutant les r\^oles de $G_1$ et $G_2$.

\medbreak
\noindent
{\bf Rappels sur la $KK$-th\'eorie \'equivariante de Kasparov et  notations}
\medbreak
1) Soit $G$ un groupe quantique \lc r\'egulier.  \`A tout couple de $G$-alg\`ebres $A$ et $B$, on associe  \cite{BaSka1} un groupe ab\'elien not\'e $KK^G(A,B)$, dont les g\'en\'erateurs sont  les classes  de $A-B$ bimodules de Kasparov $({\cal E} , F)$  o\`u ${\cal E}$ est un $A-B$ bimodule hilbertien $G$-\'equivariant et $F\in {\cal L}({\cal E})$ v\'erifie  : 
$$[F , x] \in{\cal K}({\cal E}), \quad x(F -F^*) \in{\cal K}({\cal E}), \quad x(F^2 - 1)\in{\cal K}({\cal E}), \quad x\in A\quad;$$
$$   x(\delta_{{\cal K}({\cal E})}(F)  - F \otimes  1_{S}) \in{\cal K}({\cal E}) \otimes S, \quad x\in A \otimes S,$$
\noindent
o\`u on a pos\'e $S = C_0(G)$.
\smallbreak
2) Si $A , D , B$ sont des $G$-alg\`ebres, on a  un produit (produit interne de Kasparov) :
$$KK^G(A,D)  \times KK^G(D,B)  \rightarrow  KK^G(A,B)    : (x, y) \mapsto x\otimes_D y.$$
\smallbreak
3) Soit $A$ une $G$-alg\`ebre, identifions  \cite{BaSka2} les $G$-alg\`ebres $A  \rtimes G  \rtimes \widehat G$   et $A \otimes {\cal K}(L^2(G))$.
On note $\beta_A$ (resp. $\alpha _A$), la classe du bimodule de Kasparov  $(A \otimes L^2(G) , 0)$ (resp. $((A \otimes L^2(G))^* , 0)$). 
On a 
$\beta_A \in KK^G(A \otimes {\cal K}(L^2(G)) , A)$ et  $ \alpha _A \in KK^G(A , A \otimes {\cal K}(L^2(G)))$.
\smallbreak
4) Soient $A$ et $B$ deux $G$-alg\`ebres, on rappelle \cite{BaSka1}  que pour tout $x = [({\cal E} , F)]\in KK^G(A,B)$, on a 
$$\beta_A \otimes_A    x \otimes_B  \alpha _B = [({\cal E} \otimes  {\cal K}(L^2(G)) , ({\rm id} \otimes R)\delta_{\cal E}(F))]\in 
KK^G(A \otimes {\cal K}(L^2(G)) , B \otimes {\cal K}(L^2(G)))$$
\noindent
et que l'application $KK^G(A,B) \rightarrow KK^G(A \otimes {\cal K}(L^2(G)) , B \otimes {\cal K}(L^2(G))) : x \mapsto \beta_A \otimes_A    x \otimes_B  \alpha _B$ est un isomorphisme de groupes ab\'eliens (\cf \cite{BaSka1}).
\smallbreak
5) Pour tout $x = [({\cal E}, F)] \in KK^G(A,B)$, quitte \`a remplacer $x$ par $1_A \otimes_A x$, on peut supposer que 
l'action \`a gauche $A  \rightarrow {\cal L}({\cal E})$ est non d\'eg\'en\'er\'ee.

\noindent
{\it Notons \cite{BaSka1}  que l'op\'erateur $({\rm id}_{{\cal K}({\cal E})} \otimes R)\delta_{{\cal K}({\cal E})}(F)$ est fixe pour la coaction biduale de
${\cal K}({\cal E}) \otimes  {\cal K}(L^2(G))$.}

\medbreak

\noindent
Fixons un groupo\"ide de co-liaison ${\cal G} := {\cal G}_{G_1, G_2}$  et deux  $G_1$-alg\`ebres 
  $(A_1, \delta_{A_1})$ et 
$(B_1, \delta_{B_1})$.

Pour tout $A_1-B_1$ bimodule de Kasparov $({\cal E}_1, F_1)$, o\`u on suppose que l'action \`a gauche $A_1 \rightarrow {\cal L}({\cal E}_1)$ est non d\'eg\'en\'er\'ee, posons  \ref{thliaison} :
$$(A_2 ,\delta_{A_2}) = {\rm Ind}_{G_1}^{G_2}\,(A_1, \delta_{A_1}), \quad (B_2  , \delta_{B_2}) = {\rm Ind}_{G_1}^{G_2}\,(B_1, \delta_{B_1}), \quad ({\cal E}_2 , \delta_{{\cal E}_2}):= {\rm Ind}_{G_1}^{G_2}\,({\cal E}_1 , \delta_{{\cal E}_1}).$$
\noindent
Avec les notations de \ref{thliaison}, nous avons  :

\begin{proposition}\label{homo1} $({\cal E}_2 \otimes {\cal K}(H_{12}) , ({\rm id}_{{\cal K}({\cal E}_2)} \otimes R_{21})\delta_{{\cal K}({\cal E}_1)}^2(F_1))$ est un 
$A_2 \otimes {\cal K}(H_{12}) - B_2 \otimes {\cal K}(H_{12})$ bimodule de Kasparov $G_2$-\'equivariant.
\end{proposition}

\begin{proof}
 L'op\'erateur $({\rm id}_{{\cal K}({\cal E}_1)} \otimes R_{11})\delta_{{\cal K}({\cal E}_1)}(F_1)$ \'etant  fixe pour la coaction biduale de $G_1$ dans ${\cal K}({\cal E}_1) \otimes  {\cal K}(H_{11})$, posons $F' = ({\rm id}_{{\cal K}({\cal E}_1)} \otimes R_{11})\delta_{{\cal K}({\cal E}_1)}(F_1) \otimes  1_{S_{12}} \in {\cal L}({\cal E}_1 \otimes {\cal K}(H_{11}) \otimes  S_{12})$.
\hfill\break  
Il est clair que le  $A_1 \otimes   {\cal K}(H_{11})-B_1 \otimes   {\cal K}(H_{11})$ bimodule de Kasparov $({\cal E}_1 \otimes {\cal K}(H_{11}) , ({\rm id}_{{\cal K}({\cal E}_1)} \otimes R_{11})\delta_{{\cal K}({\cal E}_1)}(F_1))$ satisfait les conditions de \ref{indu3} b) (iii).  On en  d\'eduit que   
$({\rm Ind}_{G_1}^{G_2}\,{\cal E}_1 \otimes {\cal K}(H_{11}) , F')$ est un ${\rm Ind}_{G_1}^{G_2}\, A_1 \otimes   {\cal K}(H_{11}) - {\rm Ind}_{G_1}^{G_2}\,B_1 \otimes   {\cal K}(H_{11}) $ bimodule de Kasparov $G_2$-\'equivariant.\hfill\break
En utilisant \ref{thliaison} et \ref{operF}, on obtient que $({\cal E}_2 \otimes {\cal K}(H_{12}) , ({\rm id}_{{\cal K}({\cal E}_2)} \otimes R_{21})\delta_{{\cal K}({\cal E}_1)}^2(F_1))$ est l'ant\'ec\'edent de $({\rm Ind}_{G_1}^{G_2}\,{\cal E}_1 \otimes {\cal K}(H_{11}) , F')$ par l'isomorphisme 
$G_2$-\'equivariant $\delta_{11,2}^1$.
\end{proof}

\begin{notations} Pour $j , l , l'  = 1 ,2$, notons $\gamma_{ll',j,{\rm g}}$ (resp  $\gamma_{ll',j,{\rm d}}$), l'\'equivalence de Morita $G_j$-\'equivariante  \ref{not3}
$\gamma_{ll',j}$ relative  \`a la ${\cal G}$-alg\`ebre $A = A_1 \oplus A_2$ (resp. $B = B_1 \oplus B_2$). Nous notons aussi 
$\gamma_{ll',j,{\rm g}}$ (resp. $\gamma_{ll',j,{\rm d}}$), l'\'el\'ement du groupe de Kasparov correspondant \`a cette \'equivalence de Morita.\index{gd@$\gamma_{ll',j,{\rm g}}$, $\gamma_{ll',j,{\rm d}}$}
\end{notations}

\begin{proposition}\label{debut} Conservons les hypoth\`eses et les notations de \ref{homo1}.
\hfill\break
Il existe un op\'erateur $F_2\in {\cal L}({\cal E}_2)$ v\'erifiant :
\begin{enumerate}
\item  $({\cal E}_2, F_2)$ est un $A_2-B_2$ bimodule de Kasparov $G_2$-\'equivariant ;
\item $\beta_{A_2} \otimes_{A_2} [({\cal E}_2, F_2)] \otimes_{B_2}  \alpha _{B_2} = 
\gamma_{21,2,{\rm g}} \otimes_{A_2 \otimes \cK(H_{12})} [({\cal E}_2 \otimes {\cal K}(H_{12}) , ({\rm id}_{{\cal K}({\cal E}_2)} \otimes R_{21})\delta_{{\cal K}({\cal E}_1)}^2(F_1))] \otimes_{B_2 \otimes \cK(H_{12})} \gamma_{12,2,{\rm d}}$.
\end{enumerate}
De plus, si $F_2, F_2'\in{\cal L}({\cal E}_2)$ v\'erifient les conditions a) et b) pr\'ec\'edentes, alors on a :
$$[({\cal E}_2, F_2)] = [({\cal E}_2, F_2')] \in KK^{G_2}(A_2, B_2).$$
\end{proposition}
\begin{proof} 
Nous avons :
$$\gamma_{21,2,{\rm g}} = [(A_2 \otimes {\cal K}(H_{12}, H_{22}) ,0)], \quad \gamma_{12,2,{\rm d}} = [(B_2 \otimes {\cal K}(H_{22}, H_{12}) ,0)].$$\noindent
Il est facile de v\'erifier qu'on a :
$$A_2 \otimes {\cal K}(H_{12}, H_{22}) \otimes_{A_2 \otimes {\cal K}(H_{12})} {\cal E}_2 \otimes {\cal K}(H_{12}) = 
{\cal E}_2 \otimes   {\cal K}(H_{12}, H_{22}),$$
$${\cal E}_2 \otimes  {\cal K}(H_{12}, H_{22}) \otimes_{B_2 \otimes {\cal K}(H_{12})} 
B_2 \otimes {\cal K}(H_{22}, H_{12}) = {\cal E}_2  \otimes {\cal K}(H_{22}).$$
\noindent
Soit $({\cal E}_2  \otimes {\cal K}(H_{22}) , T_2)$ un $A_2 \otimes {\cal K}(H_{22})-B_2 \otimes {\cal K}(H_{22})$ bimodule de Kasparov $G_2$-\'equivariant v\'erifiant 
$$\gamma_{21,2,{\rm g}} \otimes [({\cal E}_2 \otimes {\cal K}(H_{12}) , ({\rm id}_{{\cal K}({\cal E}_2)} \otimes R_{21})\delta_{{\cal K}({\cal E}_1)}^2(F_1))] \otimes \gamma_{12,2,{\rm d}} =  
[({\cal E}_2  \otimes {\cal K}(H_{22}) , T_2)].$$
\noindent
L'isomorphisme \cite{BaSka1} de groupes ab\'eliens 
$$KK^{G_2}(A_2, B_2)  \rightarrow  KK^{G_2}(A_2 \otimes {\cal K}(H_{22}), B_2 \otimes {\cal K}(H_{22})) : 
x \mapsto \beta_{A_2} \otimes_{A_2} x \otimes_{B_2}  \alpha _{B_2}$$
\noindent
permet d'obtenir un op\'erateur $F_2\in {\cal L}({\cal E}_2)$ v\'erifiant a) et b). L'unicit\'e de l'\'el\'ement  du groupe de Kasparov $[({\cal E}_2,F_2)]$  corrrespondant  est claire.
\end{proof}

\begin{corollary}\label{homo2} Pour tout $x =[({\cal E}_1 , F_1)]\in KK^{G_1}(A_1, B_1)$, notons $J_{G_2, G_1} (x)\in KK^{G_2}(A_2, B_2)$, l'unique \'el\'ement du groupe $KK^{G_2}(A_2, B_2)$, v\'erifiant :
$$\beta_{A_2} \otimes_{A_2} J_{G_2, G_1}(x) \otimes_{B_2}  \alpha _{B_2} = 
\gamma_{21,2,{\rm g}} \otimes_{A_2 \otimes \cK(H_{12})} [({\cal E}_2 \otimes {\cal K}(H_{12}) , ({\rm id}_{{\cal K}({\cal E}_2)} \otimes
 R_{21})\delta_{{\cal K}({\cal E}_1)}^2(F_1))] \otimes_{B_2 \otimes \cK(H_{12})} \gamma_{12,2,{\rm d}}$$
\noindent
Alors $J_{G_2, G_1}  : KK^{G_1}(A_1, B_1) \rightarrow KK^{G_2}(A_2, B_2)$ est un morphisme de groupes ab\'eliens.
\end{corollary}

\begin{proof}
Pour tout $A_1-B_1$ bimodule de Kasparov $({\cal E}_1 , F_1)$, avec une action \`a gauche non d\'eg\'en\'er\'ee, il est clair que 
$[({\cal E}_2 \otimes {\cal K}(H_{12}) , ({\rm id}_{{\cal K}({\cal E}_2)} \otimes
 R_{21})\delta_{{\cal K}({\cal E}_1)}^2(F_1))]\in KK^{G_2}(A_2 \otimes \cK(H_{12}), B_2 \otimes \cK(H_{12}))$ ne d\'epend que de la classe de $({\cal E}_1 , F_1)$ dans $KK^{G_1}(A_1, B_1)$.
\hfill\break
L'induction de $G_1$ \`a $G_2$ \'etant compatible  avec les sommes directes, $J_{G_2, G_1}$ est un morphisme de  groupes ab\'eliens.
\end{proof}

Montrons maintenant qu'on a aussi un morphisme de groupes ab\'eliens :
$$J_{G_1, G_2} :  KK^{G_2}(A_2, B_2) \rightarrow KK^{G_1}(A_1, B_1).$$\noindent
Pour tout  $A_2-B_2$ bimodule de Kasparov $G_2$-\'equivariant $({\cal E}_2, F_2)$, posons :
$$\widetilde A_1 := {\rm Ind}_{G_2}^{G_1} A_2,\quad \widetilde B_1 := {\rm Ind}_{G_2}^{G_1} B_2,\quad J_2:= \cK({\cal E}_2 \oplus B_2),\quad \widetilde J_1 := {\rm Ind}_{G_2}^{G_1} J_2 = \cK(\widetilde{\cal E}_1 \oplus \widetilde B_1).$$\noindent
D\'efinissons d'abord un morphisme de groupes ab\'eliens 
$$\widetilde J_{G_1,G_2} : KK^{G_2}(A_2, B_2) \rightarrow KK^{G_1}(\widetilde A_1  \otimes \cK(H_{21}), \widetilde B_1 \otimes \cK(H_{21})).$$\noindent
Notons que l'action de $G_1$ dans $\widetilde A_1  \otimes \cK(H_{21})$ (resp. $\widetilde B_1  \otimes \cK(H_{21})$),  s'obtient par restriction de l'action de ${\cal G}$ dans le double produit crois\'e 
$(\widetilde A_1\oplus A_2) \rtimes {\cal G} \rtimes \widehat{\cal G}$ (resp.\ $(\widetilde B_1\oplus B_2) \rtimes {\cal G} \rtimes \widehat{\cal G}$), qui permet aussi d'obtenir un isomorphisme $G_1$-\'equivariant de 
$\widetilde A_1  \otimes \cK(H_{21})$ sur ${\rm Ind}_{G_2}^{G_1} (A_2 \otimes \cK(H_{22}))$ (resp. $\widetilde B_1  \otimes \cK(H_{21})$ sur ${\rm Ind}_{G_2}^{G_1} (B_2 \otimes \cK(H_{22}))$) (\cf\ref{reccorresp} {\it  o\`u il faut permuter les r\^oles de $G_1$ et $G_2$ et aussi  \ref{ind+iso}}).

En permutant les r\^oles de $G_1$ et $G_2$ et en utilisant une version de \ref{thliaison} et \ref{operF}, pour  la ${\cal  G}$-alg\`ebre de liaison 
$\widetilde J := \widetilde J_1 \oplus J_2$, on \'etablit  :

\begin{proposition}\label{homo3} Pour tout $A_2-B_2$ bimodule de Kasparov $G_2$-\'equivariant $({\cal E}_2, F_2)$, o\`u on suppose que l'action \`a gauche 
$A_2 \rightarrow {\cal L}({\cal E}_2)$ est non d\'eg\'en\'er\'ee, 
$(\widetilde{\cal E}_1 \otimes {\cal K}(H_{21}) , ({\rm id}_{{\cal K}(\widetilde{\cal E}_1)} \otimes R_{12})\delta_{{\cal K}({\cal E}_2)}^1(F_2))$ est un 
$\widetilde A_1 \otimes {\cal K}(H_{21}) - \widetilde B_1 \otimes {\cal K}(H_{21})$ bimodule de Kasparov $G_1$-\'equivariant.
\end{proposition}

Conservons les notations pr\'ec\'edentes. Introduisons les ${\cal G}$-alg\`ebres :
$$A :=A_1  \oplus A_2, \quad \widetilde A := \widetilde A_1  \oplus A_2\quad ; \quad 
B :=B_1  \oplus B_2, \quad \widetilde B := \widetilde B_1  \oplus B_2.$$
\noindent
On d\'eduit de (\ref{ind5}, c)) qu'on a des isomorphismes ${\cal G}$-\'equivariants :
$$f : A  \rightarrow \widetilde A  : (a_1, a_2) \mapsto  (\delta_{A_1}^{(2)}(a_1) , a_2) \quad ; \quad 
g : B  \rightarrow \widetilde B  : (b_1, b_2) \mapsto   (\delta_{B_1}^{(2)}(b_1) , b_2).$$
\noindent
Notons que $f$ (resp. $g$) {\it op\`ere par l'identit\'e} sur $A_2$ (resp. $B_2$).
\hfill\break
On d\'eduit  alors des isomorphismes ${\cal G}$-\'equivariants : 
\begin{equation}\label{eq:fg} 
f : A \rtimes {\cal G} \rtimes\widehat{\cal G} \rightarrow \widetilde A \rtimes {\cal G} \rtimes\widehat{\cal G} \quad ; \quad 
g : B \rtimes {\cal G} \rtimes\widehat{\cal G}   \rightarrow \widetilde B \rtimes {\cal G} \rtimes\widehat{\cal G},
\end{equation}
et en utilisant (\ref{etap2} a)), on obtient pour tout $l =  1,2$ des isomorphismes $G_1$-\'equivariants :
$$f_{l,1} : A_1 \otimes \cK(H_{l1}) \rightarrow \widetilde A_1 \otimes \cK(H_{l1}) \quad ; \quad 
g_{l,1} : B_1 \otimes \cK(H_{l1}) \rightarrow \widetilde B_1 \otimes \cK(H_{l1}).$$
\noindent
Comme dans \ref{homo2} et en utilisant \ref{homo3}, nous avons :
\begin{proposition}\label{homo4}
Pour tout $y =[({\cal E}_2 , F_2)]\in KK^{G_2}(A_2, B_2)$, notons $J_{G_1, G_2} (y)\in KK^{G_1}(A_1, B_1)$, l'unique \'el\'ement du groupe $KK^{G_1}(A_1, B_1)$, v\'erifiant :
$$\beta_{A_1} \otimes_{A_1} J_{G_1, G_2}(y) \otimes_{B_1}  \alpha _{B_1} = 
\gamma_{12,1,{\rm g}} \otimes f_{2,1} \otimes [(\widetilde{\cal E}_1 \otimes {\cal K}(H_{21}) , ({\rm id}_{{\cal K}(\widetilde{\cal E}_1)} \otimes R_{12})\delta_{{\cal K}({\cal E}_2)}^1(F_2))]\otimes g_{2,1}^{-1} \otimes \gamma_{21,1,{\rm d}}.$$
\noindent
Alors $J_{G_1, G_2}  : KK^{G_2}(A_2, B_2) \rightarrow KK^{G_1}(A_1, B_1)$ est un morphisme de groupes ab\'eliens.
\end{proposition}

\begin{proof}
Il est clair que 
$\widetilde J_{G_1,G_2} : KK^{G_2}(A_2, B_2) \rightarrow KK^{G_1}(\widetilde A_1  \otimes \cK(H_{21}), \widetilde B_1 \otimes \cK(H_{21})) :$
 $$y = [({\cal E}_2, F_2)] \mapsto [(\widetilde{\cal E}_1 \otimes {\cal K}(H_{21}) , ({\rm id}_{{\cal K}(\widetilde{\cal E}_1)} \otimes R_{12})\delta_{{\cal K}({\cal E}_2)}^1(F_2))]$$
\noindent
est un morphisme de groupes et on a 
$$\beta_{A_1} \otimes_{A_1} J_{G_1, G_2}(y) \otimes_{B_1}  \alpha _{B_1} = \gamma_{12,1,{\rm g}} \otimes f_{2,1} \otimes
\widetilde J_{G_1,G_2}(y) \otimes g_{2,1}^{-1} \otimes \gamma_{21,1,{\rm d}} \,\,, \,\, y\in KK^{G_2}(A_2, B_2).$$
\end{proof}

Finalement, nous avons :

\begin{theorem}  $J_{G_2, G_1} :  KK^{G_1}(A_1, B_1) \rightarrow KK^{G_2}(A_2, B_2)$ est un isomorphisme de groupes ab\'eliens  et on a $J_{G_1, G_2}  =  (J_{G_2, G_1})^{-1}$.
\end{theorem}

\begin{proof} Montrons    que $J_{G_1, G_2} \circ J_{G_2, G_1} = {\rm id}_{KK^{G_1}(A_1, B_1)}$.
\hfill\break
Soit $({\cal E}_1, F_1)$ un $A_1-B_1$ bimodule de Kasparov $G_1$-\'equivariant (avec l'action \`a gauche non d\'eg\'en\'er\'ee). Posons $x_1:=[({\cal E}_1, F_1)]\,,\, y_2:=  J_{G_2, G_1}(x_1)$ et $x_1':=  J_{G_1, G_2}(y_2)$.
\hfill\break
On sait \ref{debut} qu'il existe $F_2\in{\cal L}({\cal E}_2)$, o\`u ${\cal E}_2 := {\rm Ind}_{G_1}^{G_2}\,{\cal E}_1$, tel que
$y_2=[({\cal E}_2, F_2)] $.
\hfill\break
On a avec les notations de \ref{homo4} :
$$\beta_{A_1} \otimes_{A_1} x_1'  \otimes_{B_1}  \alpha _{B_1} = 
\gamma_{12,1,{\rm g}} \otimes f_{2,1} \otimes [(\widetilde{\cal E}_1 \otimes {\cal K}(H_{21}) , ({\rm id}_{{\cal K}(\widetilde{\cal E}_1)} \otimes R_{12})\delta_{{\cal K}({\cal E}_2)}^1(F_2))]\otimes g_{2,1}^{-1} \otimes \gamma_{21,1,{\rm d}}.$$
\noindent
Pour expliciter $J_{G_1, G_2}(y_2)$, introduisons les notations :
$$J_1:= \cK({\cal E}_1 \oplus B_1),\quad J_2:= \cK({\cal E}_2 \oplus B_2),\quad \widetilde J_1 := {\rm Ind}_{G_2}^{G_1} J_2 = \cK(\widetilde{\cal E}_1 \oplus \widetilde B_1).$$
\noindent
Nous avons alors deux ${\cal G}$-alg\`ebres $J:= J_1 \oplus J_2$ et $\widetilde J := \widetilde J_1 \oplus J_2$ avec un isomorphisme ${\cal G}$-\'equivariant \ref{ind5} :
$$h : J \rightarrow \widetilde J  : (x_1, x_2) \mapsto (\delta_{J_1}^{(2)}(x_1), x_2),$$
\noindent
qui {\it op\`ere par l'identit\'e sur $J_2$}.
En appliquant     \ref{etap2} \`a l'isomorphisme ${\cal G}$-\'equivariant
$$h : J \rtimes {\cal G} \rtimes\widehat{\cal G}   \rightarrow \widetilde J \rtimes {\cal G} \rtimes\widehat{\cal G},$$
\noindent
qui est {\it compatible  avec les ${\cal G}$-isomorphismes (\ref{eq:fg}) $f$ et $g$}, on obtient :
\begin{align}\gamma_{12,1,{\rm g}}& \otimes f_{2,1} \otimes   [(\widetilde{\cal E}_1 \otimes {\cal K}(H_{21}) , ({\rm id}_{{\cal K}(\widetilde{\cal E}_1)} \otimes R_{12})\delta_{{\cal K}({\cal E}_2)}^1(F_2))]  \otimes g_{2,1}^{-1} \otimes \gamma_{21,1,{\rm d}} \nonumber \\
&= 
\gamma_{12,1,{\rm g}}  \otimes [({\cal E}_1 \otimes {\cal K}(H_{21}) , ({\rm id}_{{\cal K}({\cal E}_1)} \otimes R_{12})\delta_{{\cal K}({\cal E}_2)}^1(F_2))]  \otimes \gamma_{21,1,{\rm d}}.\nonumber 
\end{align}
On en d\'eduit l'existence  d'un op\'erateur  $F'_1 \in {\cal L}({\cal E}_1)$   tel que $x_1' = [({\cal E}_1, F_1')]$. 
\hfill\break
Pour montrer que $ x_1 =   x_1' $, il suffit d'\'etablir dans      
 $KK^{G_1}(A_1 \otimes {\cal K}(H_{11}) , B_1 \otimes {\cal K}(H_{11}))$  :
$$[({\cal E}_1 \otimes {\cal K}(H_{11}) , ({\rm id}_{{\cal K}({\cal E}_1)} \otimes R_{11})\delta_{{\cal K}({\cal E}_1)}^1(F_1))] = 
[({\cal E}_1 \otimes {\cal K}(H_{11}) , ({\rm id}_{{\cal K}({\cal E}_1)} \otimes R_{11})\delta_{{\cal K}({\cal E}_1)}^1(F'_1))]$$
\noindent 
On a  : 
\begin{equation}\label{Eq1}\beta_{A_2} \otimes_{A_2} y_2 \otimes_{B_2}  \alpha _{B_2} = 
\gamma_{21,2,{\rm g}} \otimes x_{2} \otimes \gamma_{12,2,{\rm d}} \,, \quad x_{2}  : = [({\cal E}_2 \otimes {\cal K}(H_{12}) , ({\rm id}_{{\cal K}({\cal E}_2)} \otimes R_{21})\delta_{{\cal K}({\cal E}_1)}^2(F_1))],
\end{equation}
\begin{equation}\label{Eq2}\beta_{A_1} \otimes_{A_1} x'_1 \otimes_{B_1}  \alpha _{B_1} = 
\gamma_{12,1,{\rm g}} \otimes y_{1} \otimes \gamma_{21,1,{\rm d}} \, , \quad y_{1} := [({\cal E}_1 \otimes {\cal K}(H_{21}) , ({\rm id}_{{\cal K}({\cal E}_1)} \otimes R_{12})\delta_{{\cal K}({\cal E}_2)}^1(F_2))].
\end{equation}
\noindent
Les deux \'egalit\'es (\ref{Eq1})  et (\ref{Eq2}) sont \'equivalentes \`a :
\begin{equation}\label{Eq3}[({\cal E}_2 \otimes {\cal K}(H_{22}) , ({\rm id}_{{\cal K}({\cal E}_2)} \otimes R_{22})\delta_{{\cal K}({\cal E}_2)}^2(F_2))]   \otimes  \gamma_{21,2,{\rm d}} = 
\gamma_{21,2,{\rm g}} \otimes x_{2},
\end{equation}
\begin{equation}\label{Eq4}y_{1} \otimes \gamma_{21,1,{\rm d}}  = \gamma_{21,1,{\rm g}} \otimes [({\cal E}_1\otimes\cK(H_{11}),({\rm id}_{{\cal K}({\cal E}_1)} \otimes R_{11})\delta_{{\cal K}({\cal E}_1)}^1(F'_1))].
\end{equation}
\noindent
Pour (\ref{Eq3}), on a :
\begin{equation}\label{Eq3'}[({\cal E}_2 \otimes {\cal K}(H_{22}) , ({\rm id}_{{\cal K}({\cal E}_2)} \otimes R_{22})\delta_{{\cal K}({\cal E}_2)}^2(F_2))]   \otimes  \gamma_{21,2,{\rm d}} = [({\cal E}_2 \otimes {\cal K}(H_{12}, H_{22}) , ({\rm id}_{{\cal K}({\cal E}_2)} \otimes R_{22})\delta_{{\cal K}({\cal E}_2)}^2(F_2))].
\end{equation}
De m\^eme, pour (\ref{Eq4}), on a    :
\begin{equation}\label{Eq4'}y_{1} \otimes \gamma_{21,1,{\rm d}} = [({\cal E}_1 \otimes {\cal K}(H_{11},H_{21}),({\rm id}_{{\cal K}({\cal E}_1)} \otimes R_{12})\delta_{{\cal K}({\cal E}_2)}^1(F_2))].
\end{equation}
\noindent
Pour $j = 1 , 2$, reprenons  (\ref{eq:morphf1}) et  (\ref{eq:morphf2}) : 
$$f_j : {\cal A}_j \rightarrow {\cal L}({\cal B}_j) \quad ({\cal A}_j:= A_j \otimes {\cal K}(H_{1j} \oplus H_{2j}),\quad  {\cal B}_j:= J_j \otimes {\cal K}(H_{1j} \oplus H_{2j})),$$
\noindent 
le *-morphisme d\'efini par l'action \`a  gauche de $A_j$ dans ${\cal E}_j$.
\hfill\break 
Il r\'esulte de (\ref{Eq3}) et (\ref{Eq3'})  que  $({\rm id}_{{\cal K}({\cal E}_2)} \otimes R_{22})\delta_{{\cal K}({\cal E}_2)}^2(F_2)$ est une $({\rm id}_{{\cal K}({\cal E}_2)} \otimes R_{21})\delta_{{\cal K}({\cal E}_1)}^2(F_1)$-connexion dans le produit de Kasparov $\gamma_{21,2,{\rm g}} \otimes x_{2}$, \ie
pour tout $a\in A_2 \otimes {\cal K}(H_{12} , H_{22})$, on  a :
\begin{equation}\label{C1} ({\rm id} \otimes R_{22})\delta_{{\cal K}(E_2)}(F_2) f_2(a)  -  f_2(a) ({\rm id} \otimes R_{21})\delta_{{\cal K}({\cal E}_1)}^2(F_1)\in {\cal K}({\cal E}_2 \otimes {\cal K}(H_{12}), {\cal E}_2 \otimes {\cal K}(H_{12} , H_{22})).
\end{equation}
\noindent
Posons $d := ({\rm id} \otimes R_{22})\delta_{{\cal K}(E_2)}(F_2) f_2(a)  -  f_2(a) ({\rm id} \otimes R_{21})\delta_{{\cal K}({\cal E}_1)}^2(F_1)$. Nous avons :
\hfill\break
$\bullet$ $d\in M({\cal B}_2)$ et d\'efinit clairement un \'el\'ement $d'\in {\cal L}({\cal E}_2 \otimes {\cal K}(H_{12}), {\cal E}_2 \otimes {\cal K}(H_{12} , H_{22}))$.
\hfill\break
$\bullet$    (\ref{C1})  signifie que $d\in {\cal B}_2$ et plus exactement $d'\in {\cal K}({\cal E}_2 \otimes {\cal K}(H_{12}), {\cal E}_2 \otimes {\cal K}(H_{12} , H_{22}))$.
\hfill\break
Posons $c := \delta_{{\cal B}_2}^1(d)\in M({\cal B}_1 \otimes S_{12})$. En utilisant  (\ref{notR} c)), \ref{operF}, on obtient :
$$c = (({\rm id}_{{\cal K}({\cal E}_1)} \otimes R_{12}) \delta_{{\cal K}({\cal E}_2)}^1(F_2) \otimes 1_{S_{12}}) (f_1 \otimes {\rm id}_{S_{12}}) \delta_{{\cal A}_2}^1(a)  - 
(f_1 \otimes {\rm id}_{S_{12}}) \delta_{{\cal A}_2}^1(a) (({\rm id}_{{\cal K}({\cal E}_1)} \otimes R_{11}) \delta_{{\cal K}({\cal E}_1)}^1(F_1) \otimes 1_{S_{12}}).$$
\noindent
On d\'eduit de     (\ref{acc} c))   que pour tout $a\in A_2 \otimes {\cal K}(H_{12} , H_{22})$ et tout $\omega\in B(H_{12})_\ast$, on  a  :
$$({\rm id}_{{\cal K}({\cal E}_1)} \otimes R_{12}) \delta_{{\cal K}({\cal E}_2)}^1(F_2) f_1({\rm id}_{{\cal A}_1} \otimes \omega)\delta_{{\cal A}_2}^1(a)  - 
f_1({\rm id}_{{\cal A}_1} \otimes \omega)\delta_{{\cal A}_2}^1(a) ({\rm id}_{{\cal K}({\cal E}_1)} \otimes R_{11}) \delta_{{\cal K}({\cal E}_1)}^1(F_1) \in {\cal B}_1.$$
\noindent
Mais, il d\'ecoule de (\ref{acc} c)) et (\ref{isoindF}),  qu'on a 
$$A_1 \otimes {\cal K}(H_{11} , H_{21})  = [({\rm id}_{{\cal A}_1} \otimes \omega)\delta_{{\cal A}_2}^1(a) \,\,|\,\, a \in A_2 \otimes {\cal K}(H_{12} , H_{22}) \,,\,\omega\in B(H_{12})_\ast].$$
\noindent
Finalement, nous avons obtenu que pour tout  $b\in A_1 \otimes {\cal K}(H_{11} , H_{21})$, on a 
\begin{equation}\label{C2} ({\rm id}_{{\cal K}({\cal E}_1)} \otimes R_{12}) \delta_{{\cal K}({\cal E}_2)}^1(F_2)  f_1(b) - f_1(b)({\rm id}_{{\cal K}({\cal E}_1)} \otimes R_{11}) \delta_{{\cal K}({\cal E}_1)}^1(F_1)\in {\cal K}({\cal E}_1 \otimes {\cal K}(H_{11}) , {\cal E}_1 \otimes {\cal K}(H_{11} , H_{21})).
\end{equation}
\noindent
De m\^eme, dans (\ref{Eq4}) et (\ref{Eq4'}), on interpr\`ete $({\rm id}_{{\cal K}({\cal E}_1)} \otimes R_{12})\delta_{{\cal K}({\cal E}_2)}^1(F_2)$ comme une $({\rm id}_{{\cal K}({\cal E}_1)} \otimes R_{11})\delta_{{\cal K}({\cal E}_1)}^1(F_1')$-connexion dans le produit de Kasparov $\gamma_{21,1,{\rm g}} \otimes [({\cal E}_1 \otimes \cK(H_{11}), ({\rm id}_{{\cal K}({\cal E}_1)} \otimes R_{11})\delta_{{\cal K}({\cal E}_1)}^1(F'_1))]$, \ie 
pour tout $b\in A_1 \otimes {\cal K}(H_{11} , H_{21})$, on  a :
\begin{equation}\label{C3}({\rm id}_{{\cal K}({\cal E}_1)} \otimes R_{12})\delta_{{\cal K}({\cal E}_2)}^1(F_2) f_1(b)  -  f_1(b) ({\rm id}_{{\cal K}({\cal E}_1)} \otimes R_{11})\delta_{{\cal K}({\cal E}_1)}^1(F_1')\in {\cal K}({\cal E}_1 \otimes {\cal K}(H_{11}), {\cal E}_1 \otimes {\cal K}(H_{11} , H_{21})).
\end{equation}
\noindent
Par la diff\'erence (\ref{C2}) - (\ref{C3}), on obtient que pour tout $b\in A_1 \otimes {\cal K}(H_{11} , H_{21})$ :
$$f_1(b) ({\rm id}_{{\cal K}({\cal E}_1)} \otimes R_{11})\delta_{{\cal K}({\cal E}_1)}^1(F_1') - f_1(b)({\rm id}_{{\cal K}({\cal E}_1)} \otimes R_{11}) \delta_{{\cal K}({\cal E}_1)}^1(F_1) \in {\cal K}({\cal E}_1 \otimes {\cal K}(H_{11}), {\cal E}_1 \otimes {\cal K}(H_{11} , H_{21})).$$
\noindent
Comme $H_{21} \not= \{0\}$, on d\'eduit facilement que pour  tout  $b\in A_1 \otimes {\cal K}(H_{11})$, on a 
$$f_1(b) ({\rm id}_{{\cal K}({\cal E}_1)} \otimes R_{11})\delta_{{\cal K}({\cal E}_1)}^1(F_1') - f_1(b)({\rm id}_{{\cal K}({\cal E}_1)} \otimes R_{11}) \delta_{{\cal K}({\cal E}_1)}^1(F_1) \in {\cal K}({\cal E}_1 \otimes {\cal K}(H_{11})),$$\noindent
ce qui entra\^ine  dans      
 $KK^{G_1}(A_1 \otimes {\cal K}(H_{11}) , B_1 \otimes {\cal K}(H_{11}))$  :
$$[({\cal E}_1 \otimes {\cal K}(H_{11}) , ({\rm id}_{{\cal K}({\cal E}_1)} \otimes R_{11})\delta_{{\cal K}({\cal E}_1)}^1(F_1))] = 
[({\cal E}_1 \otimes {\cal K}(H_{11}) , ({\rm id}_{{\cal K}({\cal E}_1)} \otimes R_{11})\delta_{{\cal K}({\cal E}_1)}^1(F'_1))].$$
\noindent
On a donc \'etabli $J_{G_1, G_2} \circ J_{G_2, G_1} = {\rm id}_{KK^{G_1}(A_1, B_1)}$, ce qui montre  que $J_{G_1, G_2}$ est surjective et que $J_{G_2, G_1}$ est injective. 
\hfill\break
Pour voir  directement que $J_{G_1, G_2}$ est injective, on remarque d'abord que \ref{homo4}   d\'efinit un morphisme de groupes ab\'eliens : $\widetilde J_{G_1,G_2} : KK^{G_2}(A_2, B_2) \rightarrow KK^{G_1}(\widetilde A_1  \otimes \cK(H_{21}), \widetilde B_1 \otimes \cK(H_{21}))$.
\hfill\break
En le composant avec le morphisme (injectif)
$$J_{G_1,G_2}' : KK^{G_1}(\widetilde A_1  \otimes \cK(H_{21}), \widetilde B_1 \otimes \cK(H_{21})) \rightarrow KK^{G_2}({\rm Ind}_{G_1}^{G_2}(\widetilde A_1  \otimes \cK(H_{21})), {\rm Ind}_{G_1}^{G_2}(\widetilde B_1  \otimes \cK(H_{21})))$$
\noindent
et en identifiant ${\rm Ind}_{G_1}^{G_2}(\widetilde A_1  \otimes \cK(H_{21}))$ et ${\rm Ind}_{G_1}^{G_2}(\widetilde B_1  \otimes \cK(H_{21}))$ avec $A_2 \otimes \cK(H_{22})$ et $B_2 \otimes \cK(H_{22})$ respectivement, on voit que la compos\'ee 
$J_{G_1,G_2}' \circ \widetilde J_{G_1,G_2}$ est l'isomorphisme \cite {BaSka1}
$$KK^{G_2}(A_2, B_2) \rightarrow KK^{G_2}(A_2  \otimes \cK(H_{22}), B_2 \otimes \cK(H_{22})).$$
\end{proof}

\section{Appendice}

\subsection{Produit tensoriel relatif d'espaces de Hilbert}

Dans ce paragraphe, on se fixe une alg\`ebre de von Neumann $N$, munie d'un poids nsff $\nu$. On note $(H_\nu, \pi_\nu, \Lambda_\nu)$ la repr\'esentation GNS du poids $\nu$, et on d\'esigne par $J_\nu$, $\Delta_\nu$ et $(\sigma_t^\nu)$ respectivement l'op\'erateur de Tomita, l'op\'erateur modulaire et le groupe d'automorphismes modulaires, canoniquement associ\'es au poids $\nu$.
\begin{notations}
\begin{enumerate}
\item  On note $N^{\rm o}$ l'alg\`ebre de von Neumann oppos\'ee de $N$ qu'on munit   du poids $\nu^{\rm o}$ d\'efini par $\nu^{\rm o}(x^{\rm o}) = \nu(x)$ pour $x\in N^+$. On rappelle que la repr\'esentation GNS du poids $\nu^{\rm o}$ est $(H_\nu, \pi_{\nu^{\rm o}}, \Lambda_{\nu^{\rm o}})$, o\`u :
$$\pi_{\nu^{\rm o}}(x^{\rm o}) := J_\nu \pi_\nu(x)^* J_\nu , \quad x\in N \quad ; \quad   \Lambda_{\nu^{\rm o}}x^{\rm o} := J_\nu \Lambda_\nu x^* , \quad x\in{\mathfrak  N}_\nu^*.$$
\item   $C_N : N^{\rm o} \rightarrow \pi_\nu(N)' : x^{\rm o} \mapsto J_\nu \pi_\nu(x)^* J_\nu$ est un *-isomorphisme d'alg\`ebres de von Neumann. Il en r\'esulte que $H_\nu$ est un bimodule sur $N$ pour les actions d\'efinies par :
$$x \xi y := \pi_\nu(x) J_\nu \pi_\nu(y)^* J_\nu \xi, \quad x , y \in N, \quad \xi\in H_\nu.$$
\end{enumerate}
\end{notations}

\noindent
Dans ce qui suit, on se fixe $H , K$ des espaces de Hilbert, $\alpha : N \rightarrow B(H)$ et $\beta :  N^{\rm o}   \rightarrow B(K)$ des 
repr\'esentations normales  et unitales. Le *-morphisme $\alpha$ (resp. $\beta$) munit $H$ (resp. $K$) d'une structure de $N$-module \`a gauche (resp.  \`a droite).
\hfill\break
Pour tout  $\xi\in H$ (resp. $\xi\in K$), notons $R_\xi^\alpha$ (resp. $L_\xi^\beta$)\index{rb@$R_\xi^\alpha$, $L_\xi^\beta$}, l'op\'erateur d\'efini par : 
$$R_\xi^\alpha : \Lambda_\nu{\mathfrak N}_\nu  \rightarrow H :   \Lambda_\nu x \mapsto \alpha(x) \xi  \quad ({\rm resp.} \,\,
L_\xi^\beta : \Lambda_{\nu^{\rm o}}{\mathfrak N}_{\nu^{\rm o}}  \rightarrow K :   \Lambda_{\nu^{\rm o}}x^{\rm o} \mapsto \beta(x^{\rm o}) \xi).$$\noindent

\noindent
\begin{definition}\cite{E,DeC2}
 Le vecteur   $\xi \in H$  (resp. $\xi\in K$) est dit born\'e \`a droite (resp. \`a gauche), si l'op\'erateur 
$R_\xi^\alpha$ (resp. $L_\xi^\beta$) se prolonge en un op\'erateur born\'e $R_\xi^\alpha : H_\nu  \rightarrow H$ (resp. $L_\xi^\beta : H_\nu  \rightarrow K$).
\end{definition}
On note ${}_\nu(\alpha, H)$ (resp. $(K, \beta)_\nu$) l'espace vectoriel form\'e par les vecteurs born\'es \`a droite (resp. \`a gauche).
Il est facile de voir que  pour tout $\xi \in  {}_\nu(\alpha, H)$ (resp. $\xi \in (K, \beta)_\nu$),  l'op\'erateur 
$R_\xi^\alpha$ (resp.  $L_\xi^\beta$) est $N$-lin\'eaire \`a gauche (resp. \`a droite). Il en r\'esulte que   
  pour tout 
$\xi , \eta \in {}_\nu(\alpha, H)$ (resp. $\xi , \eta \in (K, \beta)_\nu$), on a :
$$(R_\xi^\alpha)^*R_\eta^\alpha \in \pi_\nu(N)'  ,  \quad 
R_\xi^\alpha (R_\eta^\alpha)^* \in \alpha(N)' \quad ({\rm resp.} \,\, (L_\xi^\beta)^*L_\eta^\beta \in \pi_\nu(N)  ,  \quad  L_\xi^\beta (L_\eta^\beta)^*\in \beta(N^{\rm o})').$$
 \noindent
\begin{notations} Pour tout $\xi , \eta \in {}_\nu(\alpha, H)$ (resp. $\xi , \eta \in (K, \beta)_\nu$),  posons :
$$\langle  \xi , \eta \rangle  _{N^{\rm o}} := C_N^{-1} ((R_\xi^\alpha)^*R_\eta^\alpha) \in N^{\rm o} \quad ({\rm resp.} \,\, 
\langle  \xi , \eta \rangle  _{N} := \pi_\nu^{-1} ((L_\xi^\beta)^*L_\eta^\beta) \in N).$$\noindent
\end{notations}

\begin{proposition} \cite{E}  Pour tout $\xi , \eta \in {}_\nu(\alpha, H)$ (resp. $\xi , \eta \in (K, \beta)_\nu$), et tout $y\in N$ analytique pour $(\sigma  _t^\nu)$,  nous avons :
\begin{enumerate}
\item  $\langle  \xi , \eta \rangle  _{N^{\rm o}}^* = \langle  \eta, \xi\rangle_{N^{\rm o}}  \quad ({\rm resp.} \,\, 
\langle  \xi , \eta \rangle  _{ N}^* = \langle  \eta, \xi\rangle_{N})$.
\item  $\langle  \xi , \eta y^{\rm o}\rangle_{N^{\rm o}} =  \langle  \xi , \eta \rangle  _{N^{\rm o}} (\sigma  _{i/2}^\nu(y))^{\rm o}
 \quad ({\rm resp.} \,\,  \langle  \xi , \eta y \rangle  _{N } =  \langle  \xi , \eta \rangle  _{N } \sigma  _{-{i/2}}^\nu(y))$.
\end{enumerate}
\end{proposition}

\begin{corollary} Supposons que $N$ soit de dimension finie et que $\nu$ soit une trace. Alors $(K, \langle \cdot,\cdot \rangle_N)$ est un pr\'e-C*-module hilbertien sur $N$.
\end{corollary}
\noindent
La d\'efinition du   produit tensoriel relatif \cite{E, DeC1} d'espaces de Hilbert $\reltens{K}{\beta}{\alpha}{H}$ r\'esulte du r\'esultat suivant :

\noindent
\begin{lemme}\cite{E, DeC1} Pour tout $\xi_1 , \xi_2 \in {}_\nu(\alpha, H)$  et tout $\eta_1 , \eta_2 \in (K, \beta)_\nu$, on a 
$$\langle  \eta_1 , \beta(\langle  \xi_1 ,  \xi_2 \rangle  _{N^{\rm o}})\eta_2 \rangle  _K = 
\langle  \xi_1 , \alpha(\langle  \eta_1 ,  \eta_2 \rangle  _{N })\xi_2 \rangle  _H. $$
\end{lemme}

\begin{definition} Le produit tensoriel relatif d'espaces de Hilbert $\reltens{K}{\beta}{\alpha}{H}$\index{rc@$\reltens{K}{\beta}{\alpha}{H}$} est  le s\'epar\'e-compl\'et\'e de l'espace pr\'ehilbertien $ (K, \beta)_\nu \odot {}_\nu(\alpha, H)$ pour  :
$$\langle  \eta_1 \otimes \xi_1 , \eta_2 \otimes \xi_2 \rangle   := \langle  \eta_1 , \beta(\langle  \xi_1 ,  \xi_2 \rangle  _{N^{\rm o}})\eta_2 \rangle  _K = 
\langle  \xi_1 , \alpha(\langle  \eta_1 ,  \eta_2 \rangle  _{N })\xi_2 \rangle  _H. $$\noindent
\end{definition}

\begin{remarks}
\begin{enumerate}
\item En rempla\c cant $(N, \nu)$ par $(N^{\rm o},\nu^{\rm o})$, on obtient la d\'efinition du produit tensoriel relatif 
$\reltens{H}{\alpha}{\beta}{K}$.
\item L'espace de Hilbert  $\reltens{K}{\beta}{\alpha}{H}$  est  aussi le s\'epar\'e-compl\'et\'e de l'espace pr\'ehilbertien $K \odot {}_\nu(\alpha, H)$ pour :
$$ \langle  \eta_1 \otimes \xi_1 , \eta_2 \otimes \xi_2 \rangle   := \langle  \xi_1 , \alpha(\langle  \eta_1 ,  \eta_2 \rangle  _{N })\xi_2 \rangle  _H. $$\noindent
De plus, pour tout $\xi\in K , \eta\in {}_\nu(\alpha, H)$ et $y\in N$ analytique pour  $(\sigma  _t^\nu)$, on a 
$$\reltens{\beta(y^{\rm o})\xi}{\beta}{\alpha}{\eta} =  \reltens{\xi}{\beta}{\alpha}{\alpha(\sigma_{-{i/2}}^\nu(y))\eta}.$$\noindent
\end{enumerate}
\end{remarks}

\begin{proposition}\label{pr}Supposons que $N$ soit de dimension finie et que $\nu$ soit une trace. Soit ${\cal E}$ le C*-module  sur $N$ compl\'et\'e de $(K, \langle \cdot,\cdot \rangle_N)$. Alors $\reltens{K}{\beta}{\alpha}{H}$ est le produit interne de Kasparov ${\cal E} \otimes_{\alpha} H$.
\end{proposition}

\subsection{Produit fibr\'e d'alg\`ebres de von Neumann}\label{prodfib}

Conservons les m\^emes notations.
Soient $x\in \beta(N^{\rm o})'$ et $y\in \alpha(N)'$. Il est facile de voir que l'op\'erateur 
$$(K, \beta)_\nu \odot {}_\nu(\alpha, H)  \rightarrow \reltens{K}{\beta}{\alpha}{H} : \xi \otimes \eta  \mapsto \reltens{x(\xi)}{\beta}{\alpha}{y(\eta)}$$
\noindent
se prolonge en un op\'erateur born\'e, not\'e $\reltens{x}{\beta}{\alpha}{y} \in B(\reltens{K}{\beta}{\alpha}{H})$.

\noindent
Soit $A  \subset B(H)$ (resp. $B \subset B(K)$) une alg\`ebre de von Neumann. Supposons de plus que la repr\'esentation   
$\alpha  : N  \rightarrow B(H)$ (resp. $\beta : N^{\rm o} \rightarrow B(K)$) soit \`a valeurs dans $A$ (resp. $B$).

\noindent
\begin{definition}\cite{E}  Le produit fibr\'e 
$\fprod { B } { \beta } { \alpha } { A }$\index{rd@$\fprod { B } { \beta } { \alpha } { A }$} des 
alg\`ebres de von Neumann $A$ et $B$,    est le commutant dans $B(\reltens{K}{\beta}{\alpha}{H})$ de l'ensemble 
$\{\reltens{x}{\beta}{\alpha}{y}  \,\,;\,\,  x\in B' \,,\, y\in A'\}$.
\end{definition}

{\bf Cas o\`u $N$ est de dimension finie}

 \noindent  
Supposons que  $N = \oplus_{l=1}^k M_{n_l}$ soit de dimension finie et $\nu = \oplus_{l=1}^k  {\rm Tr}\,(F_l\, \cdot)$.  Pour tout $l\in\{1, \cdots,k\}$, notons  $(F_{l, i})$ les valeurs propres de la matrice $F_l\in M_{n_l}$.
\hfill\break
Nous avons :

\begin{proposition}\cite{DeC1} Supposons que pour tout $l\in\{1, \cdots,k\}$, on ait $\sum_{i=1}^{n_l}  F_{l, i}^{-1} =1$.
\hfill\break
Soit $v : K \otimes H \rightarrow \reltens{K}{\beta}{\alpha}{H}  : \xi \otimes \eta \mapsto \reltens{\xi}{\beta}{\alpha}{\eta}$. Nous avons :
\begin{enumerate}
\item L'op\'erateur $v$ est une co-isom\'etrie.
\hfill\break
Posons $p = v^* v\in{\rm Proj}\,(B(K \otimes  H))$. Soient $A \subset \alpha(N)'$ et $B\subset \beta(N^{\rm o})'$ des alg\`ebres de von Neumann.   
\item L'application  $\fprod{B}{\beta}{\alpha}{A} \rightarrow p (B \otimes A) p : x \mapsto v^* x v $ est un *-isomorphisme d'alg\`ebres de von Neumann.
\end{enumerate}
\end{proposition}

\begin{proof}[D\'emonstration abr\'eg\'ee] Pour la preuve du a), on \'etablit que pour tout $\xi, \eta\in K$,  on a 
$$\langle  \xi , \eta \rangle  _N =  \sum_{l=1}^k \sum_{i,j  = 1}^{n_l} F_{l, i}^{-{1/2}}F_{l, j}^{-{1/2}}
\langle  \xi , \beta(e_{ji}^{(l){\rm o}}) \eta \rangle  _K\, e_{ij}^{(l)},$$\noindent
o\`u $(e_{ij}^{(l)})$ est un s.u.m.\ qui diagonalise la matrice $F_l$, \ie $F_l = \sum_{i=1}^l F_{l, i} e_{ii}^{(l)}$.
\hfill\break
Pour la preuve du b), on applique le th\'eor\`eme du commutant d'un produit tensoriel  d'alg\`ebres de von Neumann.
\end{proof}

\noindent
Conservons  $N = \oplus_{l=1}^k M_{n_l}$  et prenons pour  $\nu$, la trace de Markov non normalis\'ee 
$\epsilon :=   \oplus_{l=1}^k n_l {\rm Tr}\,_{M_{n_l}}$.   
Dans ce cas le $N$-module \`a droite $K$ muni du produit scalaire $\langle \cdot , \cdot \rangle_N$, est un C*-module sur $N$ et  $\reltens{K}{\beta}{\alpha}{H}$ est le produit de Kasparov $K \otimes_N H$. De plus, pour tout s.u.m.\ $(e_{ij}^{(l)})$, on a :
$$p = \sum_{l=1}^k \frac{1}{n_l} \sum_{i,j  = 1} \, \beta(e_{ji}^{(l){\rm o}}) \otimes \alpha(e_{ij}^{(l)}).$$
\noindent
La  repr\'esentation GNS $(H_{\epsilon}, \pi_{\epsilon}, \Lambda_{\epsilon})$ de la trace $\epsilon$ est donn\'ee par :
 $$H_{\epsilon} := \oplus_{l=1}^k \, \C^{n_l} \otimes \overline{ \C^{n_l}}, \quad \pi_{\epsilon}(x)  = \oplus _{l=1}^k\, x_l \otimes    \overline{ 1_{\C^{n_l}}}, \quad \Lambda_{\epsilon} x = \pi_{\epsilon}(x) \xi_{\epsilon}, \quad x= \oplus_{l=1}^k \, x_l\in N,$$
avec  $\displaystyle \xi_{\epsilon} = \oplus_{l=1}^k   \sqrt{n_l} \sum_{i=1}^{n_l}
\varepsilon_i^l \otimes \overline{ \varepsilon_i^l}$ et $(\varepsilon_i^l)_{1\leqslant i\leqslant n_l}$ une base orthonorm\'ee pour chaque  espace de Hilbert $\C^{n_l}$.
\hfill\break
Remarquons que si $(e_{ij}^{(l)})$ est le  s.u.m.\ d\'efini par la base orthonorm\'ee $(\varepsilon_i^l)$ de  l'espace de Hilbert $\C^{n_l}$, alors 
$(1 / \sqrt {n_l} \pi_{\epsilon}(e_{ij}^{(l)})\xi_{\epsilon})_{ 1 \leqslant i,\, j \leqslant n_l  \,; \, 1\leqslant l \leqslant k}$
est une base orthonorm\'ee pour $H_{\epsilon}$. Nous avons :
\medbreak
\begin{proposition}\label{proput}
Soient $\alpha  :  N  \rightarrow B(H)$ et $\beta  : N^{\rm o} \rightarrow B(K)$  des repr\'esentations    unitales. 
\begin{enumerate}
\item  Pour tout $\xi , \eta \in H$, nous avons  $q_{\alpha, \beta}  (R_\xi^\alpha (R_\eta^\alpha)^* \otimes 1_K) = q_{\alpha, \beta} ( \theta_{\xi, \eta} \otimes 1_K)q_{\alpha, \beta}$.
\item Pour tout $\xi , \eta \in K$, nous avons  $ q_{\alpha, \beta}  (1_H \otimes L_\xi^\beta (L_\eta^\beta)^*) = q_{\alpha, \beta} (1_H \otimes \theta_{\xi, \eta})q_{\alpha, \beta}$.
\end{enumerate}
\begin{proof} Montrons le a). 
Soient $\zeta , \zeta' \in H$. D'une part, on a
\begin{align}
\langle \zeta \otimes 1, q_{\alpha, \beta}  (R_\xi^\alpha (R_\eta^\alpha)^* \otimes 1_H)(\zeta' \otimes 1)\rangle & = 
\sum_{l=1}^k \frac{1}{n_l} \sum_{i,j=1}^{n_l} \langle (R_\xi^\alpha)^*\zeta, (R_\eta^\alpha)^* \alpha(e_{ji}^{(l)})\zeta'\rangle \beta(e_{ij}^{(l)\rm o}) \nonumber \\
& =
\sum_{l, l'=1}^k \frac{1}{n_l n_{l'}} \sum_{i,j=1}^{n_l} 
\sum_{u,v=1}^{n_{l'}}\langle (R_\xi^\alpha)^*\zeta, e_{uv}^{(l')}\xi_{\epsilon} \rangle \langle   e_{uv}^{(l')}\xi_{\epsilon}, (R_\eta^\alpha)^* \alpha(e_{ji}^{(l)})\zeta'\rangle \beta(e_{ij}^{(l)\rm o}) \nonumber \\
 & = 
\sum_{l, l'=1}^k \frac{1}{n_l n_{l'}} \sum_{i,j=1}^{n_l}\sum_{u,v=1}^{n_{l'}}\langle \zeta, \alpha(e_{uv}^{(l')})\xi\rangle \langle   \alpha(e_{uv}^{(l')})\eta,   \alpha(e_{ji}^{(l)})\zeta'\rangle \beta(e_{ij}^{(l)\rm o}) \nonumber \\
 & =  
\sum_{l =1}^k \frac{1}{n_l^2} \sum_{i,j=1}^{n_l}\sum_{ v=1}^{n_{l}}\langle \zeta, \alpha(e_{jv}^{(l)})\xi\rangle \langle   \alpha(e_{iv}^{(l)})\eta,   \zeta'\rangle \beta(e_{ij}^{(l)\rm o}).  \nonumber 
\end{align}
D'autre part, on a aussi
\begin{align}
\langle \zeta \otimes 1, q_{\alpha, \beta}(\theta_{\xi, \eta} \otimes 1) q_{\alpha, \beta}(\zeta' \otimes 1)\rangle & =  
\sum_{l, l'=1}^k \frac{1}{n_l n_{l'}} \sum_{i,j=1}^{n_l} 
 \sum_{u,v=1}^{n_{l'}}\langle \zeta, \alpha(e_{uv}^{(l')})\xi \rangle \langle \alpha(e_{ji}^{(l)})\eta,   \zeta'\rangle \beta(e_{vu}^{(l')\rm o})\beta(e_{ji}^{(l)\rm o}) \nonumber \\  
 &=  
\sum_{l =1}^k \frac{1}{n_l^2} \sum_{i,j=1}^{n_l}\sum_{ u=1}^{n_{l}}\langle \zeta, \alpha(e_{ui}^{(l)})\xi\rangle \langle   \alpha(e_{ji}^{(l)})\eta,   \zeta'\rangle \beta(e_{ju}^{(l)\rm o}) \nonumber \\
 & = 
\sum_{l =1}^k \frac{1}{n_l^2} \sum_{i,j=1}^{n_l}\sum_{ v=1}^{n_{l}}\langle \zeta, \alpha(e_{jv}^{(l)})\xi\rangle \langle   \alpha(e_{iv}^{(l)})\eta,   \zeta'\rangle \beta(e_{ij}^{(l)\rm o}).  \nonumber 
\end{align}
La preuve du b) s'obtient de fa\c con similaire.
\end{proof}

\end{proposition}
On d\'emontre de la m\^eme fa\c con le r\'esultat suivant : 
\begin{proposition} Soient $\alpha  :  N  \rightarrow B(H)$ et $\beta  : N^{\rm o} \rightarrow B(K)$  des repr\'esentations    unitales.
\begin{enumerate}
\item  Pour tout $\xi , \eta \in K$, nous avons  $q_{\beta, \alpha}  (L_\xi^\beta (L_\eta ^\beta)^* \otimes  1_H) = q_{\beta, \alpha} (\theta_{\xi, \eta } \otimes  1_H)q_{\beta, \alpha}.$
\item    Pour tout $\xi , \eta \in H$, nous avons  $q_{\beta, \alpha}  (1_K \otimes R_\xi^\alpha (R_\eta ^\alpha)^*) = q_{\beta, \alpha} ( 1_K \otimes \theta_{\xi, \eta })q_{\beta, \alpha}.$ 
 \end{enumerate}
\end{proposition}

\begin{proposition}\label{efin} Soient $A$ et $B$ deux C*-alg\`ebres et $f : A \rightarrow M(B)$ un *-morphisme. Supposons que pour une unit\'e approch\'ee $(u_\lambda)$ de $A$, on ait $f(u_\lambda)  \rightarrow e \in M(B)$ pour la topologie stricte. Alors nous avons :
\begin{enumerate}
\item $e\in {\rm Proj}\,(M(B))$ et  pour toute unit\'e approch\'ee $(v_\lambda)$ de $A$, on a $f(v_\lambda)  \rightarrow e \in M(B)$ pour la topologie stricte.
\item Le *-morphisme $f$ se prolonge de fa\c con unique en un *-morphisme $f  : M(A) \rightarrow M(B)$ strictement continu et 
v\'erifiant $f(1_A)  = e$.
\end{enumerate}
\end{proposition}
\begin{proof}
Il est facile de voir que $e$ est un projecteur et que $f$ d\'efinit un *-morphisme non d\'eg\'en\'er\'e $f  : A \rightarrow \cL(eB)$.
\end{proof}
\begin{notation}\label{d(1)}   Soient $A$ et $B$ deux C*-alg\`ebres, $f : A \rightarrow M(B)$ un *-morphisme et   $e\in {\rm Proj}\,(M(B))$. La notation 
$$f(1_A) = e$$
\noindent
signifie que pour une unit\'e approch\'ee $(u_\lambda)$ de $A$, on a $f(u_\lambda)  \rightarrow e \in M(B)$ pour la topologie stricte. On note alors $f  : M(A) \rightarrow M(B)$ le prolongement strictement continu de $f$ 
v\'erifiant $f(1_A)  = e$.
\end{notation}

\printindex

\end{document}